\documentclass[oneside,english,a4wide]{amsart}

\usepackage{graphicx}
\usepackage{float} 
\usepackage{subfigure}
\usepackage[]{caption2} 
\renewcommand{\thesubfigure}{(\roman{subfigure})}
\makeatletter \renewcommand{\@thesubfigure}{\thesubfigure \space}
\renewcommand{\p@subfigure}{} \makeatother

\usepackage[latin9]{inputenc}
\usepackage[a4paper]{geometry}
\usepackage{mathabx}
\usepackage{esvect}
\usepackage{babel}
\usepackage{mathtools}
\usepackage{amstext}
\usepackage{amsthm}
\usepackage{amssymb}
\usepackage{enumerate}
\usepackage{esint}
\usepackage{graphicx}
\usepackage{bbm}
\usepackage[all]{xy}
\usepackage{mathrsfs}
\usepackage[unicode=true,
bookmarks=true,bookmarksnumbered=true,bookmarksopen=true,bookmarksopenlevel=2,
breaklinks=false,pdfborder={0 0 1},backref=false,colorlinks=false]
{hyperref}
\hypersetup{
	colorlinks,linkcolor=blue,anchorcolor=blue,citecolor=blue}
\usepackage{breakurl}
\usepackage{autobreak}

\usepackage[svgnames]{xcolor} 

\usepackage{graphicx} 

\usepackage{amsmath}

\makeatletter
\numberwithin{equation}{section}
\numberwithin{figure}{section}
\theoremstyle{plain}
\newtheorem{thm}{\protect\theoremname}[section]
\theoremstyle{plain}
\newtheorem{prop}[thm]{\protect\propositionname}
\newtheorem{proposition}[thm]{Proposition}
\newtheorem{corollary}[thm]{Corollary}
\theoremstyle{definition}
\newtheorem{defn}[thm]{\protect\definitionname}
\theoremstyle{plain}
\newtheorem{lem}[thm]{\protect\lemmaname}
\newtheorem{lemma}[thm]{Lemma}
\theoremstyle{remark}
\newtheorem{rem}[thm]{\protect\remarkname}

\theoremstyle{definition}
\newtheorem{definition}[thm]{Definition}
\newtheorem*{claim*}{Claim}
\newtheorem{example}[thm]{Example}

\theoremstyle{remark}
\newtheorem{remark}[thm]{Remark}
\theoremstyle{definition}
\newtheorem*{defn*}{\protect\definitionname}

\usepackage{babel}
\usepackage{amscd}

\makeatother

\providecommand{\definitionname}{Definition}
\providecommand{\lemmaname}{Lemma}
\providecommand{\propositionname}{Proposition}
\providecommand{\remarkname}{Remark}
\providecommand{\theoremname}{Theorem}

\newcommand{\Rmnum}[1]{\expandafter\@slowromancap\romannumeral #1@}
\makeatother

\newcommand{\M}{{\mathcal M}}
\newcommand{\N}{{\mathcal N}}

\newcommand{\8}{\infty}

\newcommand{\la}{\langle}
\newcommand{\ra}{\rangle}

\newcommand{\be}{\begin{eqnarray*}}
	\newcommand{\ee}{\end{eqnarray*}}
\newcommand{\beq}{\begin{equation}}
	\newcommand{\eeq}{\end{equation}}
\newcommand{\beqn}{\begin{equation*}}
	\newcommand{\eeqn}{\end{equation*}}
\newcommand{\bs}{\begin{split}}
	\newcommand{\es}{\end{split}}

\begin{document}

	
	
	\title[noncommutative martingale paraproducts and operator-valued commutators]{On the boundedness and Schatten class property of noncommutative martingale paraproducts and operator-valued commutators}
	
	
	\author[Zhenguo Wei]{Zhenguo Wei}
	\address{Laboratoire de Math{\'e}matiques, Universit{\'e} de Bourgogne Franche-Comt{\'e}, 25030 Besan\c{c}on Cedex, France}
	\email{zhenguo.wei@univ-fcomte.fr}

	\author[Hao Zhang]{Hao Zhang}
	\address{
	Department of Mathematics, University of Illinois Urbana-Champaign, USA}
	\email{hzhang06@illinois.edu}
	\date{}
	\maketitle
	
	\begin{abstract}
	We study the Schatten class membership of semicommutative martingale paraproducts and use the transference method to describe Schatten class membership of purely noncommutative martingale paraproducts, especially for CAR algebras and $\mathop{\otimes}\limits_{k=1}^{\infty}\mathbb{M}_{d}$ in terms of martingale Besov spaces. 
	
	Using Hyt\"{o}nen's dyadic martingale technique, we also obtain sufficient conditions on the Schatten class membership and the boundedness of operator-valued commutators involving general singular integral operators. 
	
	We establish the complex median method, which is applicable to complex-valued functions. We apply it to get the optimal necessary conditions on the Schatten class membership of operator-valued commutators associated with non-degenerate kernels in Hyt\"{o}nen's sense. This resolves the problem on the characterization of Schatten class membership of operator-valued commutators. Our results are new even in the scalar case. 
	
	Our new approach is built on Hyt\"{o}nen's dyadic martingale technique and the complex median method. Compared with all the previous ones, this new one is more powerful in several aspects: $(a)$ it permits us to deal with more general singular integral operators with little smoothness; $(b)$ it allows us to deal with commutators with complex-valued kernels; $(c)$ it goes much further beyond the scalar case and can be applied to the semicommutative setting.
	
	By a weak-factorization type decomposition, we get some necessary but not optimal conditions on the boundedness of operator-valued commutators. In addition, we give a new proof of the boundedness of commutators still involving general singular integral operators concerning $BMO$ spaces in the commutative setting.
	\end{abstract}

		\tableofcontents{}
	
	\section{Introduction}\label{Introduction}
Hankel operators were first studied by Hankel in \cite{Hankel}. Since then they have become an important class of operators. Later, Nehari  characterized the boundedness of Hankel operators on the Hardy space $H_2(\mathbb{T})$ in terms of the $BMO$ space in \cite{Neha}, and Hartman discussed their compactness by the $VMO$ space in \cite{Hart}. In \cite{V1}, Peller obtained the Schatten $p$-class criterion of Hankel operators by Besov space for $1\leq p<\8$, while the case $0<p<1$ was investigated by Peller in \cite{V2} and Semmes in \cite{SS}, respectively.

In harmonic analysis, commutators involving singular integral and multiplication operators are generalizations of Hankel type operators. So it is certainly worthwhile to study their boundedness, compactness and Schatten class membership. In \cite{CRW}, Coifman, Rochberg and Weiss showed the boundedness of commutators with regards to the $BMO$ space on $\mathbb{R}^n$, which yields a new characterization of $BMO$. Soon after their work, Uchiyama sharpened one of their results and showed the compactness of commutators by virtue of the $CMO$ space in \cite{Uchi}. The Schatten class membership of commutators was developed by Janson and Wolff in terms of Besov spaces in \cite{JW}. Afterwards, Janson and Peetre established a fairly general framework to investigate the boundedness and Schatten class of commutators in \cite{JPe}. 

In addition, Petermichl discovered an explicit representation formula for the one-variable Hilbert transform as an average of dyadic shift operators to investigate Hankel operators with matrix symbol in \cite{SPe}. Nazarov, Treil and Volberg in  \cite{NTV} proved the $T(1)$ and the $T(b)$ theorems based on martingale paraproducts. Later, Hyt\"{o}nen refined in an essential way the method of Nazarov, Treil and Volberg, and finally settled the well-known $A_2$ conjecture \cite{TH1}. Meanwhile, as Hankel type operators, the dyadic operators, such as martingale paraproducts, serve as crucial tools in harmonic analysis.

We would like to highlight that the dyadic operators in \cite{TH1} can be considered as a particular case of martingale paraproducts. In addition to its intrinsic connection with various operators in harmonic analysis, martingale paraproducts are of much interest in their own right as they are martingale variants of operators of Hankel type. The boundedness of martingale paraproducts has been studied in \cite{CL1}. In addition, the compactness and Schatten class membership have been discussed in \cite{CP} for $d$-adic martingales.

Motivated by all this, we aim to establish the Schatten class membership of martingale paraproducts in the noncommutative setting. Via the methodology developed by Hyt\"{o}nen, we also obtain the Schatten class and boundedness characterizations for operator-valued commutators involving singular integral operators and noncommutative pointwise multiplication.

At first, we introduce noncommutative martingale paraproducts. Let $\mathcal{M}$ be a von Neumann algebra equipped with a normal semifinite faithful trace $\tau$. Given a semicommutative $d$-adic martingale $b=(b_k)_{k\in\mathbb{Z}}\in L_2(\mathbb{R}, L_2(\M))$, recall that the martingale paraproduct with symbol $b$ is defined as
\begin{equation*}
	\pi_b(f)=\sum_{k=-\infty}^{\infty}d_kb \cdot f_{k-1}, \quad \forall f=(f_k)_{k\in\mathbb{Z}}\in L_2(\mathbb{R},L_2(\mathcal{M})),
\end{equation*}
where $d_kb=b_k-b_{k-1}$ for any $k\in\mathbb{Z}$. See Subsection \ref{sec2.3} for the definition of semicommutative $d$-adic martingales. When $d=2$, $d$-adic martingales are just dyadic martingales.
 
When $\M=\mathbb{C}$, Chao and Peng described the Schatten class membership of $\pi_b$ by virtue of the martingale Besov spaces. They showed the following theorem (see {\cite[Theorem 3.1]{CP}}):
\begin{thm}\label{lem2.1}
	For $0<p<\infty$ and a locally integrable function $b$, $\pi_b\in S_p(L_2(\mathbb{R}))$ if and only if
	$$ \sum_{k=-\infty}^\infty d^{k}\|d_kb\|_{L_p(\mathbb{R})}^p<\8.  $$
\end{thm}
	Here $S_p(\mathcal{H})$ denotes the Schatten $p$-class of operators on a Hilbert space $\mathcal{H}$. Chao and Peng's proof invokes some results about Schatten $p$-norms in \cite{RSe}. In \cite{PS}, Pott and Smith gave another proof of Theorem \ref{lem2.1} based on the $p$-John-Nirenberg inequality for $d=2$. They also obtained an equivalent characterization of the Schatten class membership of $\pi_b$ when $\M=B(\mathcal{H})$, still for $d=2$.
	
	Inspired by all this, we are concerned with the Schatten class membership of $\pi_b$ for semicommutative $d$-adic martingales with arbitrary $d$ and arbitrary semifinite von Neumann algebras $\M$. In addition, with the help of martingale paraproducts, we will consider the Schatten class membership of operator-valued commutators.
	
	\

Our first main theorem concerns the Schatten class membership of $\pi_b$ for semicommutative $d$-adic martingales. More specifically, we use the semicommutative $d$-adic martingale Besov spaces $\pmb{B}_p^d(\mathbb{R},\M)$ (see Definition \ref{mbs1}) to characterize $\|\pi_b\|_{L_p(B(L_2(\mathbb{R}))\otimes \mathcal{M})}$:
 \begin{thm}\label{thm1.2}
	For $0<p<\infty$, $\pi_b\in L_p(B(L_2(\mathbb{R}))\otimes \mathcal{M})$ if and only if $b\in \pmb{B}_p^d(\mathbb{R},\M)$. Moreover,
	\begin{equation}\label{thm2222}
		\|\pi_b\|_{L_p(B(L_2(\mathbb{R}))\otimes \mathcal{M})}\approx_{d, p} \|b\|_{ \pmb{B}_p^d(\mathbb{R},\M)}.
	\end{equation}
\end{thm}

The concrete proof of Theorem \ref{thm1.2} will be presented in Section \ref{Proof of the Necessity of theorem 1.2}. However, we will show a much stronger result than the sufficiency part of Theorem \ref{thm1.2}, i.e. Lemma \ref{nonNWOpre}, by virtue of Theorem \ref{lem2.1}. This new proof of the sufficiency part of Theorem \ref{thm1.2} is much simpler and neater than the original one which is given in \cite{Hao}.
 
 It is much more tempting to study martingale paraproducts for purely noncommutative martingales. Let $b=(b_k)_{k\geq 1}$ be a noncommutative martingale. (see Subsection \ref{sec2.2} for the definition.) The noncommutative martingale paraproduct with symbol $b$ for any noncommutative martingale $f=(f_k)_{k\geq 1}\in L_2(\M)$ is defined by 
$$  \pi_b(f)=\sum_{k=1}^{\infty}d_kb \cdot f_{k-1}.  $$
However, it remains open under which circumstances $\pi_b$ is bounded in $B(L_2(\M))$, which is also deeply related to the operator-valued $T(1)$ problem. The reader is referred to \cite{HM} for more details about this celebrated problem.

 Let $\mathbb{M}_{d}$ be the algebra of $d\times d$ matrices equipped with the normalized trace. In particular, if $\M=L_\8(\mathbb{R})\otimes \mathbb{M}_d$, Katz employed an ingenious stopping time procedure in \cite{KNH} to show
\begin{equation}\label{Katz}
	\|\pi_b\|_{B(L_2(\M))}\lesssim \log (d+1) \|b\|_{BMO(\mathbb{R}, \mathbb{M}_{d})},
\end{equation}
where $BMO(\mathbb{R}, \mathbb{M}_{d})$ is the strong operator $BMO$. We refer the reader to \cite{Mei1} for more information for such $BMO$ spaces. Meanwhile, Nazarov, Treil and Volberg independently obtained  (\ref{Katz}) in \cite{NTV1} by the Bellman method, and they also gave an example to show that for any $d\in \mathbb{N}$ there exists $b$ such that
$$  \|\pi_b\|_{B(L_2(\M))}\gtrsim \sqrt{\log (d+1)} \|b\|_{BMO(\mathbb{R}, \mathbb{M}_{d})}.  $$
This implies that the boundedness of $\pi_b$ cannot be characterized solely by $BMO(\mathbb{R}, \mathbb{M}_{d})$ for infinite-dimensional $\M$.

In \cite{NPTV}, Nazarov, Pisier, Treil and Volberg proved that $\log{(d+1)}$ is the optimal order of the best constant in (\ref{Katz}).  Indeed, Mei showed that in general,  $\|\pi_b\|_{B(L_2(\M))}$ cannot even be dominated by the operator norm $\|b\|_{\M}$ for infinite-dimensional $\M$ in \cite{Mei}.

Even though we do not know how to describe the boundedness of $\pi_b$, surprisingly by the transference method in \cite{PX97} and Theorem \ref{thm1.2}, we get an equivalent characterization of the Schatten class membership of $\pi_b$ for two families of noncommutative martingales, i.e. CAR algebra denoted by $\mathcal{C}$, and $\mathscr{M}=\mathop{\otimes}\limits_{k=1}^{\infty}\mathbb{M}_{d}$.

\

For the CAR algebra, we obtain:
\begin{thm}\label{thm6.1}
	For $0<p<\infty$, $\pi_b\in S_p(L_2(\mathcal{C}))$ if and only if $b\in \pmb{B}_p(\mathcal{C})$. Moreover,
	$$ \|\pi_b\|_{S_p(L_2(\mathcal{C})) }\approx_p \|b\|_{\pmb{B}_p(\mathcal{C})}. $$
\end{thm}

Similarly, for $\mathscr{M}=\mathop{\otimes}\limits_{k=1}^{\infty}\mathbb{M}_{d}$, we also have:
\begin{thm}\label{thm7.1}
	For $0<p<\infty$, $\pi_b\in S_p(L_2(\mathscr{M}))$ if and only if $b\in \pmb{B}_p(\mathscr{M})$. Moreover,
	$$ \|\pi_b\|_{S_p(L_2(\mathscr{M})) }\approx_{d,p} \|b\|_{\pmb{B}_p(\mathscr{M})}. $$
\end{thm}

The martingale Besov spaces $\pmb{B}_p(\mathcal{C})$ and $\pmb{B}_p(\mathscr{M})$ in Theorem \ref{thm6.1} and Theorem \ref{thm7.1} are defined in Definition \ref{mbs2} and Definition \ref{mbs3}, respectively. The proofs of Theorem \ref{thm6.1} and Theorem \ref{thm7.1} rely on Theorem \ref{thm1.2} and the transference method.

\

 Nowadays, the study of Schatten class of commutators attracts much attention and is going through vast development. We refer the interested reader to \cite{FLL2023,FLLX2024,FLMSZ2024,GLW2023,LLW2023,LLW2024,LLWW2024}. Motivated by these remarkable works, we aim to investigate Schatten class of semicommutative commutators. We will employ Theorem \ref{thm1.2} to give a characterization of Schatten class membership for operator-valued commutators involving singular integral operators and noncommutative pointwise multiplication, in terms of operator-valued Besov spaces. Our method is based on the dyadic representation of singular integral operators developed by Hyt\"{o}nen in \cite{TH1} and \cite{TH2}. We first provide the setup for singular integral operators. 

Let $T\in B(L_2(\mathbb{R}^n))$ be a singular integral operator with a kernel $K(x, y)$,  i.e. for any $f\in L_2(\mathbb{R}^n)$
\[T(f)(x)=\int_{\mathbb{R}^n}K(x, y)f(y)dy, \quad x\notin \mathrm{supp}f.\]
We assume that $K(x, y)$ is defined for all $x\neq y$ on $\mathbb{R}^n\times \mathbb{R}^n$ and satisfies the following standard kernel estimates:
\begin{equation}\label{standard}
	\begin{cases}
		\displaystyle|K(x,y)|\le \frac{C}{|x-y|^n},\\
		\displaystyle|K(x,y)-K(x',y)|+|K(y,x)-K(y,x')|\le \frac{C |x-x'|^\alpha}{|x-y|^{n+\alpha}}
			\end{cases}
\end{equation}
for all $x,x',y\in\mathbb{R}^n$ with $|x-y|>2|x-x'|>0$ and some fixed $\alpha\in (0, 1]$ and constant $C>0$. 

In particular, if for any $x\neq y$ 
\begin{equation}\label{phi2333}
	K(x, y)= \phi(x-y),
\end{equation}
where $\phi$ is homogeneous of degree $-n$ with mean value zero on the unit sphere, then $T$ is called a Calder\'{o}n-Zygmund transform. 

In the sequel,  $T:L_2(\mathbb{R}^n)\to L_2(\mathbb{R}^n)$ will always be assumed to satisfy the above standard estimates (\ref{standard}) and to be bounded. The celebrated David-Journé $T(1)$ theorem in \cite{DaJou} asserts that for any singular integral operator $T$ satisfying (\ref{standard}) , $T$ is bounded on $L_2(\mathbb{R}^n)$ if and only if $T(1)$ and $T^*(1)$ both belong to $BMO(\mathbb{R}^n)$ and $T$ satisfies the weak boundedness property. We recall that $BMO(\mathbb{R}^n)$ is the space consisting of all locally integrable functions $b$ such that
$$  \|b\|_{BMO(\mathbb{R}^n)}= \sup _{\substack{Q \subset \mathbb{R}^n \\ Q \operatorname{cube}}} \bigg(\frac{1}{m(Q)} \int_Q\Big|b-\big(\frac{1}{m(Q)} \int_Q b\ d m\big)\Big|^2 d m\bigg)^{1/2}<\infty, $$
where $m$ is Lebesgue measure on $\mathbb{R}^n$.

Assume $b\in L_2(\mathbb{R}^n, L_2(\M))$, and let $M_b$ be the pointwise multiplication by $b$. The operator-valued commutator is defined by $C_{T, b}=[T, M_b]=T M_b - M_b T$, that is for any $f\in L_2(\mathbb{R}^n, L_2(\M))$,
$$  C_{T, b}(f)=T(b\cdot f)-b\cdot T(f).  $$
The operator-valued homogeneous Besov space $\pmb{B}_p(\mathbb{R}^n, L_p(\mathcal{M}))$ is defined as the completion of all $b\in \mathcal{S}(L_\8(\mathbb{R}^n)\otimes \M)$ satisfying
\begin{equation}\label{funeq}
	\|b\|_{\pmb{B}_p(\mathbb{R}^n, L_p(\mathcal{M}))}=	\bigg(\int_{\mathbb{R}^n\times\mathbb{R}^n}\frac{\|b(x)-b(y)\|_{L_p(\mathcal{M})}^p}{|x-y|^{2n}}dxdy\bigg)^{\frac{1}{p}}<\infty,
\end{equation}  
with respect to the semi-norm  $\|\cdot \|_{\pmb{B}_p(\mathbb{R}^n, L_p(\mathcal{M}))}$. If $\mathcal{M}=\mathbb{C}$ and $p>n$, $\pmb{B}_p(\mathbb{R}^n, L_p(\mathbb{C}))=\pmb{B}_p(\mathbb{R}^n)$ coincides with the classical Besov space of parameters $(p,p,n/p)$, namely the space $\Lambda_\alpha^{p,q}(\mathbb{R}^n)$ in \cite[Chapter V, \S 5]{ST}.

In the commutative setting, Janson and Wolff obtained the following theorem (see {\cite[Theorem 1]{JW}}):
 \begin{thm}\label{thm1.5}
	Let $T$ be a Calder\'{o}n-Zygmund transform with a kernel $\phi$ defined in \eqref{phi2333}. Assume $\phi$ is $C^{\infty}(\mathbb{R}^n)$ except at the origin and not identically zero. 
	
	Suppose $n \geq 2$ and $0<p<\infty$. For $0<p\leq n$, $ C_{T, b}\in S_p(L_2(\mathbb{R}^n))$ if and only if $b$ is constant. For $p>n$, $ C_{T, b}\in S_p(L_2(\mathbb{R}^n))$ if and only if $b\in \pmb{B}_p(\mathbb{R}^n)$.
\end{thm}

 In the semicommutative setting, we get an analogous result for the following two cases:  $p>n$ when $n\geq 2$, and $p>\frac{2}{1+\alpha}$ when $n=1$. (here $\alpha$ is the fixed constant in (\ref{standard})). The following theorem describes the Schatten class membership of operator-valued commutators.
 \begin{thm}\label{thm6.4}
 	Suppose $p>n$ when $n\geq 2$, or $p>\frac{2}{1+\alpha}$ when $n=1$. Let $T\in B(L_2(\mathbb{R}^n))$ be a singular integral operator with kernel $K(x,y)$ satisfying \eqref{standard}. Suppose that $b$ is a locally integrable $L_p(\M)$-valued function. If $b\in \pmb{B}_p(\mathbb{R}^n, L_p(\mathcal{M}))$, then $ C_{T, b}\in L_p(B(L_2(\mathbb{R}^n))\otimes \mathcal{M})$ and 
 	$$ \| C_{T, b}\|_{L_p(B(L_2(\mathbb{R}^n))\otimes \mathcal{M})}\lesssim_{n, p,T}\big(1+\|T(1)\|_{BMO(\mathbb{R}^n)}+\|T^*(1)\|_{BMO(\mathbb{R}^n)}\big)\|b\|_{\pmb{B}_p(\mathbb{R}^n, L_p(\mathcal{M}))}. $$
 \end{thm}

 Theorem \ref{thm6.4} directly implies the sufficiency of Theorem \ref{thm1.5} for $p>n$ if we just let $\M=\mathbb{C}$. So we give an alternative proof of the sufficiency of Theorem \ref{thm1.5} for $p>n\geq 2$ based on martingale paraproducts. Note that in particular, if $T$ is the Hilbert transform and $\M=B(\mathcal{H})$, then Theorem \ref{thm6.4} has been shown in \cite{PS}. However, the authors in \cite{PS} used the dyadic representation of the Hilbert transform established by Petermichl in \cite{SPe}, which is much easier than Hyt\"{o}nen's dyadic representation in \cite{TH1} for general singular integral operators that is the key technique in our proof of Theorem \ref{thm6.4}.
 
Our method from real analysis for the proof of Theorem \ref{thm6.4} is different from the original one from complex analysis used in \cite{V1}. To the best of our knowledge, this is the first time that Hyt\"{o}nen's dyadic representation is utilized to estimate Schatten $p$-class of commutators involving general singular integral operators. Compared with the dyadic representation of the Hilbert transform used in \cite{PS}, Hyt\"{o}nen's dyadic representation is much more complicated because of the appearance of dyadic shifts with high complexity and martingale paraproducts. This obviously gives rise to some essential difficulties to overcome.
 
 We will show Theorem \ref{thm6.4} in a universal way for $p\geq 2$ and $n\geq 1$. In particular, as for the case $n=1$, our method can still work well for $\frac{2}{1+\alpha}<p<2$. At the time of this writing, we do not know whether Theorem \ref{thm6.4} holds for $0<p\leq\frac{2}{1+\alpha}$ when $n=1$. However it should be noted that the reason why the estimate of Schatten $p$-class in \cite{V1, FLLX2024} can be achieved for $p<2$ is that they dealt with strong smoothness assumptions on the kernels of singular integral operators. But in our setting, we just require the standard estimates \eqref{standard}. We also would like to stress that the method in \cite{JPe} which heavily relies on Schur multipliers fails  in our current case due to the lack of the smoothness condition of the kernels.
 
  Especially, if $T$ is a Calderón-Zygmund transform with a sufficiently smooth convolution kernel, there is a relatively simple proof of Theorem \ref{thm6.4} by using the dyadic representation developed by Vagharshakyan \cite{Va2010} in the same way. This is because the dyadic shift operators in  Vagharshakyan's dyadic representation are very simple and have little complexity compared with Hyt\"{o}nen's dyadic representation.
  
Indeed, when $n=1$, Theorem \ref{thm6.4} holds for all $1<p<\8$ if the kernel of $T$ satisfies $\alpha=1$. This means in this case that the kernel $K(x,y)$ is very regular in some sense. But this does not imply that $K(x, y)$ is smooth. So Theorem \ref{thm6.4} is also more general than all the previous known results even when $n=1$ except the case $0<p\leq 1$.

\

It is also very interesting to study the converse to Theorem \ref{thm6.4}. To get a lower bound of  $\| C_{T, b}\|_{L_p(B(L_2(\mathbb{R}^n))\otimes \mathcal{M})}$ in terms of the operator-valued Besov norm of $b$, the kernel $K(x, y)$ associated with $T\in B(L_2(\mathbb{R}^n))$ should not be very small. Inspired by Hyt\"{o}nen's recent work \cite{TH3}, we deal with ``non-degenerate'' kernels. (see Definition \ref{nondege} for the definition of non-degenerate kernels).

\begin{thm}\label{nonschatten}
	Let $n\geq 1$, $1<p<\infty$ and $T\in B(L_2(\mathbb{R}^n))$ be a singular integral operator with a non-degenerate kernel $K(x,y)$ satisfying \eqref{standard}. Suppose that $b$ is a locally integrable $L_p(\mathcal{M})$-valued function. If $ C_{T, b}\in L_p(B(L_2(\mathbb{R}^n))\otimes \mathcal{M})$, then $b\in \pmb{B}_p(\mathbb{R}^n, L_p(\mathcal{M}))$. Furthermore, we have
	\begin{equation*}
		\|b\|_{\pmb{B}_p(\mathbb{R}^n, L_p(\mathcal{M}))}\lesssim_{n,p,T} \| C_{T, b}\|_{L_p(B(L_2(\mathbb{R}^n))\otimes \mathcal{M})}.
	\end{equation*}
In particular, when $n\geq 2$ and $1<p\leq n$, if $ C_{T, b}\in L_p(B(L_2(\mathbb{R}^n))\otimes \mathcal{M})$, then $b$ is constant.
\end{thm}

From Theorem \ref{thm6.4} and Theorem \ref{nonschatten}, we derive directly the following corollary.
	\begin{corollary}\label{corollary1.8}
		Suppose $p>n$ when $n\geq 2$, or $p>\frac{2}{1+\alpha}$ when $n=1$. Let $T\in B(L_2(\mathbb{R}^n))$ be a singular integral operator with a non-degenerate kernel $K(x,y)$ satisfying \eqref{standard}. Suppose that $b$ is a locally integrable $L_p(\mathcal{M})$-valued function. Then $ C_{T, b}\in L_p(B(L_2(\mathbb{R}^n))\otimes \mathcal{M})$ if and only if $b\in \pmb{B}_p(\mathbb{R}^n, L_p(\mathcal{M}))$. Moreover, in this case
		$$ \| C_{T, b}\|_{L_p(B(L_2(\mathbb{R}^n))\otimes \mathcal{M})}\approx_{n,p,T} 	\|b\|_{\pmb{B}_p(\mathbb{R}^n, L_p(\mathcal{M}))}.  $$
		In addition, when $n\geq 2$ and $1<p\leq n$, then $ C_{T, b}\in L_p(B(L_2(\mathbb{R}^n))\otimes \mathcal{M})$ if and only if $b$ is constant.
\end{corollary}

We would like to emphasize that Corollary \ref{corollary1.8} is more general than Theorem \ref{thm1.5} because it not only concerns the semicommutative setting, but also deals with commutators involving general singular integral operators, while \cite{JPe} and \cite{JW} focus on Calder\'{o}n-Zygmund transforms. We also remark that our approach is completely different from that of Janson and Wolff. Therefore, Theorem \ref{thm6.4}, Theorem \ref{nonschatten} and Corollary \ref{corollary1.8} almost give a complete picture of Schatten class membership of commutators involving singular integral operators associated with non-degenerate kernels in the semicommutative setting.

Our proof of Theorem \ref{nonschatten} relies on the proof of the commutative case $\M=\mathbb{C}$. So we list the scalar case of Theorem \ref{nonschatten} as Theorem \ref{Converse} in Section \ref{schattenconv}. Our main ingredient of the proof of Theorem \ref{nonschatten} is the following so-called complex median method, which is an extension of the real median method discovered by Lerner, Ombrosi and Rivera-R\'{\i}os in \cite{LOR}.

\begin{thm}\label{divideS}
	Let $(\Omega,\mathcal{F},\mu)$ be a measure space. Let $I\in \mathcal{F}$ be of finite measure, and $b$ be a measurable function on $I$.
	Then there exist two orthogonal lines $L_1$ and $L_2$ in $\mathbb{C}$ which divide $\mathbb{C}$ into four closed quadrants $T_1$, $T_2$, $T_3$, $T_4$ such that
	\begin{equation*}
		\mu(\{x\in I:b(x)\in T_i\})\ge \frac{1}{16}\mu(I),\quad i\in \{1,2,3,4\}.
	\end{equation*}
\end{thm}
 The concept of `median' for real-valued functions was first given by Carleson in \cite{CL1980}. Fujii applied this new concept `median' to the theory of weighted norm inequalities in \cite{FN1991}. Lerner improved their work to establish the celebrated median decomposition in \cite{LAK10}. We would like to stress that Lerner's median decomposition is very powerful and has been used to demonstrate many beautiful and important results on weighted norm inequalities. Besides, Lerner, Ombrosi and Rivera-R\'{\i}os employed `median' to estimate the lower bound of the boundedness of commutators in \cite{LOR}. This is now called the real median method.

However, no matter what the median decomposition or the median method, the concept `median' can be available only for real-valued functions. In addition, it is well-acknowledged that the real median method is only applicable to real-valued functions, and unfortunately cannot be applied to complex-valued ones. We extend the definition of `median' to complex-valued functions with the help of our new complex median method. Thus this allows us to deal with singular integral operators with complex-valued kernels. This is one of the main novelties of this article. We will use this new complex median method to give a new proof of some results in \cite{TH3}, which is about to appear in our subsequent paper \cite{WZ3}.


\

Last but not least, we consider the boundedness of operator-valued commutators involving general singular integral operators on $L_p(\mathbb{R}^n, L_p(\M))$. Denote by $BMO_\mathcal{M}(\mathbb{R}^n)$ the space consisting of all $\mathcal{M}$-valued functions $b$ that are Bochner integrable on any cube such that
\begin{equation*}
	\|b\|_{BMO_\mathcal{M}(\mathbb{R}^n)}=\sup_{Q\subset \mathbb{R}^n \atop Q \operatorname{cube}}\biggl(\frac{1}{m(Q)}\int_Q \Big\|b-\big(\frac{1}{m(Q)} \int_Q b\ d m\big)\Big\|_\mathcal{M}^2 dm\biggr)^{1/2}<\infty.
\end{equation*}
 The next theorem states the boundedness of operator-valued commutators for $p=2$.
\begin{thm}\label{thm1.8}
	Let $T\in B(L_2(\mathbb{R}^n))$ be a singular integral operator with kernel $K(x,y)$ satisfying \eqref{standard}. If $b\in BMO_\mathcal{M}(\mathbb{R}^n)$, then $ C_{T, b}$ is bounded on $L_2(\mathbb{R}^n,L_2(\mathcal{M}))$ and
	$$ \| C_{T, b}\|_{L_2(\mathbb{R}^n,L_2(\mathcal{M}))\rightarrow L_2(\mathbb{R}^n,L_2(\mathcal{M}))}\lesssim_{n,T }\big(1+\|T(1)\|_{BMO(\mathbb{R}^n)}+\|T^*(1)\|_{BMO(\mathbb{R}^n)}\big)\|b\|_{BMO_{\mathcal{M}}(\mathbb{R}^n)}. $$
\end{thm}
Our proof of Theorem \ref{thm1.8} is still based on Hyt\"{o}nen's dyadic representation. When $T$ is a Riesz transform, Theorem \ref{thm1.8} coincides with \cite[Theorem A.1]{HM}. It seems that the proof of \cite[Theorem A.1]{HM} contains a gap that will be fixed in the present paper. Moreover, Theorem \ref{thm1.8} involves general singular integrals, which is new in the semicommutative setting and answers an open question in \cite[Remark A.3]{HM}.

Particularly, when $\mathcal{M}=\mathbb{C}$, Theorem \ref{thm1.8} is known; moreover, in this scalar case, $C_{T,b}$ is bounded on $L_p(\mathbb{R}^n)$ for any $1<p<\infty$. The reader is referred to \cite[Theorem 1.1]{HLW} or \cite[Theorem 3.1]{CPP2012}. We will give a new proof of it. 
Our new approach is based on the boundedness of martingale paraproducts, and needs some new interesting martingale inequalities (Lemma \ref{supk} and Proposition \ref{T1est}). The new proof is presented in Appendix \ref{appendix}.

We also study the converse to Theorem \ref{thm1.8} in terms of operator-valued $BMO$ spaces. Let $b$ be an $\mathcal{M}$-valued function that is Bochner integrable on any cube in $\mathbb{R}^n$. For any cube $Q\subset\mathbb{R}^n$, define
\begin{equation}\label{MOBQ}
	MO(b;Q)=\biggl(\frac{1}{m(Q)} \int_Q \Big|b-\big(\frac{1}{m(Q)} \int_Q b\ d m\big)\Big|^2 dm\biggr)^{1/2}.
\end{equation}
Following a similar argument as in \cite{TH3}, we obtain the following:
\begin{thm}\label{nonbound}
	Let $T\in B(L_2(\mathbb{R}^n))$ be a singular integral operator with a non-degenerate kernel $K(x,y)$ satisfying \eqref{standard}. Let $b$ be an $\mathcal{M}$-valued function that is Bochner integrable on any cube in $\mathbb{R}^n$. Suppose that for any cube $Q\subset\mathbb{R}^n$,
	\begin{equation*}
		MO(b;Q)\in \mathcal{M}\quad \text{and} \quad MO(b^*;Q)\in\mathcal{M}.
	\end{equation*}
	 If $C_{T, b}$ is bounded on $L_2(\mathbb{R}^n,L_2(\mathcal{M}))$, then $b\in BMO_{cr}(\mathbb{R}^n,\mathcal{M})$. (see Subsection \ref{ovbmo} for the definition of $BMO_{cr}(\mathbb{R}^n,\mathcal{M})$.) Furthermore, we have
	\begin{equation*}
		\|b\|_{BMO_{cr}(\mathbb{R}^n,\mathcal{M})}\lesssim_{n,T} \|C_{T, b}\|_{L_2(\mathbb{R}^n,L_2(\mathcal{M}))\to L_2(\mathbb{R}^n,L_2(\mathcal{M}))}.
	\end{equation*}
\end{thm}

We would like to remark that when $\mathcal{M}=\mathbb{C}$, the converse to Theorem \ref{thm1.8}, which is the scalar case of Theorem \ref{nonbound}, seems to be much subtler. Coifman, Rochberg and Weiss obtained a partial result of the ``only if'' part just for Riesz transforms in \cite{CRW}. In \cite{Uchi}, Uchiyama generalized this result and obtained the ``only if'' part for any Calder\'{o}n-Zygmund transforms with non-constant kernel $\phi$ satisfying the smooth estimate:
\begin{equation}\label{smooth}
	|\phi(x)-\phi(y)|\leq |x-y|, \quad  \forall\ |x|=|y|=1.
	\end{equation}
Recently, Hyt\"{o}nen further extended it to general non-degenerate singular integral operators in \cite{TH3}. We refer to \cite{TH3} for more details on the converse. In particular, our Theorem \ref{nonbound} has been shown in \cite{TH3} by Hyt\"{o}nen when $\mathcal{M}=\mathbb{C}$.

Our proof of Theorem \ref{nonbound} is similar to that in \cite{TH3}, and we generalize the weak-factorization type decomposition to the semicommutative setting following a similar argument as in \cite{TH3}. However, as we now deal with the semicommutative case, some noncommutativity issues naturally arise. 

\

A summary of the main techniques and the contents seems to be in order. Section \ref{pre2} is devoted to notation and background, such as noncommutative $L_p$-spaces, noncommutative martingales, martingale Besov spaces, operator-valued $BMO$ spaces and Hardy spaces. We also define the semicommutative martingale paraproducts associated with semicommutative $d$-adic martingales by virtue of operator-valued Haar multipliers.

In Section \ref{Proof of the Necessity of theorem 1.2}, we aim to prove Theorem \ref{thm1.2}. We will prove a stronger result, namely Lemma \ref{nonNWOpre}, to show the sufficiency of Theorem \ref{thm1.2}. Our proof is also different from Pott and Smith's one \cite{PS}. Then we follow the pattern setup in \cite{PS} and \cite{P} to show the necessity of Theorem \ref{thm1.2}. In Section \ref{Application 1:the CAR case} and Section \ref{Application 2}, we will show Theorem \ref{thm6.1} and Theorem \ref{thm7.1} respectively by the transference method with the help of Theorem \ref{thm1.2}. To that end, we transfer purely noncommutative martingale paraproducts to semicommutative martingale paraproducts, which enables us to apply Theorem \ref{thm1.2}.

 In Section \ref{Application 3} and Section \ref{pthm1.8}, by virtue of Hyt\"{o}nen's dyadic representation for general singular integral operators, we show Theorem \ref{thm6.4} and Theorem \ref{thm1.8} respectively. Our approach to Theorem \ref{thm6.4} differs from all the previous ones in similar situations as it is the first time that Hyt\"{o}nen's dyadic representation is utilized to estimate Schatten $p$-norm. In addition, we also need to consider Schatten class and boundedness of commutators involving martingale paraproducts, which are of independent interest (see Proposition \ref{T0est} and Proposition \ref{T2est} for more details). 

 In Section \ref{proofdivide}, we invent the complex median method.   As is well-known, the real median method turns out to be very powerful to deal with the lower bound of the boundedness and Schatten class of commutators (see \cite{LOR,TH3,FLL2023,DGKLWY2021}). However, all previous investigations of Lerner's real median method into lower bounds of commutators $[T,M_b]$ require that the kernel of $T$ and $b$ be real-valued functions. With the help of our new complex median method,  we treat complex-valued kernels. This constitutes perhaps one of the most important ideas of this article. We will also use this complex median method to investigate Schatten class of commutators on spaces of homogeneous type, which is about to appear in a subsequent paper \cite{FWZ2024}.

Then we first show the scalar case of Theorem \ref{nonschatten}, i.e. Theorem \ref{Converse}, in Section \ref{schattenconv} by the complex median method. In Section \ref{noncomschattenconv}, following the idea of the proof of Theorem \ref{Converse} and duality, we show Theorem \ref{nonschatten} by virtue of a semicommutative variant of Rochberg and Semmes' results \cite{RSe}. At the end, we show Theorem \ref{nonbound} in Section \ref{noncomboundconv} by a weak-factorization type decomposition. We mainly use the argument from \cite{TH3}.

Throughout this paper, we will use the following notation: $A\lesssim B$ (resp. $A\lesssim_\varepsilon B$) means that $A\le CB$ (resp. $A\le C_\varepsilon B$) for some absolute positive constant $C$ (resp. a positive constant $C_\varepsilon$ depending only on $\varepsilon$). $A\approx B$ or $A\approx_\varepsilon B$ means that these inequalities as well as their inverses hold.  Denote by $e_{i,j}$ the matrix which has $1$ in the $(i, j)$-th position as its only nonzero entry.

\bigskip

\section{Preliminaries}\label{pre2}

In this section, we provide notation and background that will be used in this paper.

\subsection{Noncommutative $L_p$-spaces} Let $\mathcal{M}$ be a von Neumann algebra equipped with a normal semifinite faithful trace $\tau$. Denote by $\M_+$ the positive part of $\M$. Let $\mathcal{S}_+(\M)$ be the set of all $x\in \M_+$ whose support projection has a finite trace, and $\mathcal{S}(\mathcal{M})$ be the linear span of $\mathcal{S}_+(\mathcal{M})$. Then $\mathcal{S}(\mathcal{M})$ is a $w^*$-dense $*$-subalgebra of $\mathcal{M}$. Let $x\in\mathcal{S}(\mathcal{M})$, then $|x|^p\in\mathcal{S}(\mathcal{M})$ for any $0<p<\infty$, where $|x|:=(x^*x)^{1/2}$. Define
\[\|x\|_p=(\tau(|x|^p))^{1/p}.\]
Thus $\|\cdot\|_p$ is a norm for $p\ge 1$, and a $p$-norm for $0<p<1$. The noncommutative $L_p$-space associated with $(\mathcal{M},\tau)$ is the completion of $(\mathcal{S}(\M),\|\cdot\|_p)$ for $0<p<\8$  denoted by $L_p(\mathcal{M},\tau)$. Let $L_0(\M, \tau)$ be the family of all measurable operators with respect to $(\M, \tau)$. We also write $L_p(\mathcal{M},\tau)$ simply by $L_p(\mathcal{M})$ for short. When $p=\8$, we set $L_\infty(\mathcal{M}):=\mathcal{M}$ equipped with the operator norm. 

In particular, when $p=2$, $L_2(\M)$ is a Hilbert space. We will view $\mathcal{M}$ as a von Neumann algebra on $L_2(\mathcal{M})$ by left multiplication, namely  $\M \hookrightarrow B(L_2(\M))$ via the embedding $x\mapsto L_x\in B(L_2(\M))$, where $x\in \M$ and $L_x(y):= x\cdot y\in L_2(\M)$ for any $y\in L_2(\M)$. Hence in this way, $\mathcal{M}$ is in its standard form. It is well-known that for $1\leq p<\8$ and $p'=\frac{p}{p-1}$
$$  \big(L_p(\M)\big)^*=L_{p'}(\M).  $$
We refer the reader to \cite{PX} for a detailed exposition of noncommutative $L_p$-spaces.

If $\mathcal{H}$ is a Hilbert space and $\mathcal{M}=B(\mathcal{H})$ equipped with the usual trace $\mathrm{Tr}$, then $L_p(\mathcal{M})$ is the Schatten $p$-class on $\mathcal{H}$ and denoted by  $S_p(\mathcal{H})$. Denote by $\eta_1\otimes\eta_2$ the rank $1$ operator on $\mathcal{H}$ given by
\[\eta_1\otimes\eta_2(\eta)= \eta_1\la \eta_2, \eta\ra, \quad \forall\eta\in\mathcal{H},\]
where $\eta_1$ and $\eta_2$ are two vectors in $\mathcal{H}$.
Then $\eta_1\otimes\eta_2\in \mathcal{S}(B(\mathcal{H}))$, and for any $0<p\leq\8$
$$  \|\eta_1\otimes\eta_2\|_{S_p(\mathcal{H})}=\|\eta_1\|_\mathcal{H}\|\eta_2\|_\mathcal{H}.  $$

\

Now we present the tensor product of von Neumann algebras. Assume that each $\mathcal{M}_k$ $(k=1, 2)$ is equipped with a normal semifinite faithful trace $\tau_k$. Then the tensor product of $\M_1$ and $\M_2$ denoted by $\mathcal{M}_1 \otimes \mathcal{M}_2$ is the $w^*$-closure of $\text{span}\{x_1\otimes x_2: x_1\in \M_1, x_2\in \M_2\}$ in $B(L_2(\M_1)\otimes L_2(\M_2))$. Here $L_2(\M_1)\otimes L_2(\M_2)$ is the Hilbert space tensor product of $L_2(\M_1)$ and $ L_2(\M_2)$.

It is well-known that there exists a unique normal semifinite faithful trace $\tau$ on the von Neumann algebra tensor product $\mathcal{M}_1 \otimes \mathcal{M}_2$ such that
$$
\tau\left(x_1 \otimes x_2\right)=\tau_1\left(x_1\right) \tau_2\left(x_2\right), \quad \forall x_1 \in \mathcal{S}(\M_1), \forall x_2 \in \mathcal{S}(\M_2) .
$$
$\tau$ is called the tensor product of $\tau_1$ and $\tau_2$ and denoted by $\tau_1 \otimes \tau_2$. 

Let $\mathbb{M}_d$ be the algebra of $d\times d$ matrices equipped with the usual trace $\mathrm{Tr}$. Denote by $\mathrm{tr}_d:=\frac{1}{d}\mathrm{Tr}$ the normalized trace on $\mathbb{M}_d$. For $k\geq 1$, let
\begin{equation*}
	\big(\mathbb{M}_d^{\otimes k},\text{tr}_d^{\otimes k}\big)=\mathop{\otimes}\limits_{i=1}^k(\mathbb{M}_d,\mathrm{tr}_d)
\end{equation*}
be the tensor products in the sense of von Neumann algebras. We define 	$$\Big(\mathop{\otimes}\limits_{k=1}^{\infty}\mathbb{M}_{d},\mathop{\otimes}\limits_{k=1}^{\infty}\mathrm{tr}_{d}\Big)= \mathop{\otimes}_{i=1}^{\infty}(\mathbb{M}_d,\mathrm{tr}_d)$$
as the inductive limit of $(\mathbb{M}_d^{\otimes k}, \text{tr}_d^{\otimes k})_{k\geq 1}$, also denoted by $\mathop{\otimes}\limits_{k=1}^{\infty}\mathbb{M}_{d}$ for simplicity (see \cite[Lemma 4.5]{HM1} for the inductive limit).

\

In this paper, we are concerned with the von Neumann algebra tensor product of $B(L_2(\mathbb{R}))$ and $\M$, where $B(L_2(\mathbb{R}))$ is endowed with the usual trace $\text{Tr}$, and $\M$ is a semifinite von Neumann algebra equipped with a normal semifinite faithful trace $\tau$.

In the sequel, we will identify any left multiplication $L_x\in B(L_2(\M))$ with $x\in \M$. Then for any $T\in B(L_2(\mathbb{R}))$, $T\otimes L_x \in B(L_2(\mathbb{R}))\otimes \M\hookrightarrow B(L_2(\mathbb{R}))\otimes B(L_2(\M))=B(L_2(\mathbb{R})\otimes L_2(\M))$, and 
$$ \| T\otimes L_x\|_{L_p(B(L_2(\mathbb{R}))\otimes \mathcal{M})}=\|T\|_{S_p(L_2(\mathbb{R}))}\|x\|_{L_p(\mathcal{M})}.$$
In the following, we write $T\otimes L_x$ as $x\cdot T$, and thus
\begin{equation}\label{XT}
	(x\cdot T)(f)=T\otimes L_x(f)=x\cdot T(f),\,\,\,\,\forall f\in L_2(\mathbb{R},L_2(\mathcal{M})).
\end{equation}

\subsection{Noncommutative martingales}\label{sec2.2} This subsection is devoted to noncommutative martingales. The reader is referred to \cite{PX} and \cite{bookXu}. Assume that $\mathcal{M}$ is a von Neumann algebra equipped with a normal faithful semifinite trace $\tau$ and $\mathcal{N}$ is a von Neumann subalgebra of $\mathcal{M}$ such that the restriction of $\tau$ to $\mathcal{N}$ is again semifinite. Then there exists a unique map $\mathcal{E}: \mathcal{M}\to\mathcal{N}$ satisfying the following properties:
\begin{enumerate}
	\item $\mathcal{E}$ is a normal contractive positive projection from $\mathcal{M}$ onto $\mathcal{N}$;
	\item $\mathcal{E}(axb)=a\mathcal{E}(x)b$ for any $x\in\mathcal{M}$ and $a,b\in\mathcal{N}$;
	\item $\tau \circ\mathcal{E}=\tau$.
\end{enumerate}
$\mathcal{E}$ is called the conditional expectation of $\mathcal{M}$ with respect to $\mathcal{N}$. Besides, $\mathcal{E}$ extends to a contractive positive projection from $L_p(\M)$ onto $L_p(\N)$ for any $1\leq p<\8$, still denoted by $\mathcal{E}$.

Recall that a filtration of von Neumann subalgebras of $\mathcal{M}$ is a nondecreasing sequence $(\mathcal{M}_n)_{n\ge 1}$ of von Neumann subalgebras of $\mathcal{M}$ such that $\cup_n\mathcal{M}_n$ is $w^*$-dense in $\mathcal{M}$ and the restriction of $\tau$ to $\mathcal{M}_n$ is also semifinite for every $n$. Let $\mathcal{E}_n$ be the conditional expectation of $\M$ with respect to $\M_n$. A sequence $x=(x_n)\subset L_1(\mathcal{M})$ is called a martingale with respect to $(\mathcal{M}_n)_{n\ge 1}$ if $\mathcal{E}_n(x_{n+1})=x_n$ for every $n\ge 1$. In addition, if $x_n\in L_p(\mathcal{M})$ with $p\ge 1$, $x$ is called an $L_p$-martingale with respect to $(\mathcal{M}_n)_{n\ge 1}$. Denote the martingale differences by $d_nx=x_n-x_{n-1}$ for $n\geq 1$ with the convention $x_0=0$. 

\begin{rem}\label{conmar}
	Let $1<p\leq \8$ and $x=(x_n)$ a noncommutative martingale such that
	$$ \sup_n \|x_n\|_p <\8.  $$
	Then there exists $x_{\infty} \in L_p(\mathcal{M})$ such that $x_n=\mathcal{E}_n\left(x_{\infty}\right)$ for every $n$.
\end{rem}

We are going to introduce two particular noncommutative martingales: the CAR algebras and tensor products of matrix algebras.

 \subsubsection{\textbf{CAR algebra.}}
We consider the following Pauli matrices:
\begin{equation*}
	\sigma_0=\begin{pmatrix}
		1&0\\0&-1
	\end{pmatrix},\quad
	\sigma_1=\begin{pmatrix}
		0&1\\1&0
	\end{pmatrix},\quad
	\sigma_2=\begin{pmatrix}
		0&-i\\i&0
	\end{pmatrix}.
\end{equation*}
For $n\ge 1$, define
\[c_{2n-1}=\sigma_0\otimes\cdots\sigma_0\otimes \sigma_1\otimes 1\otimes 1\cdots,\quad c_{2n}=\sigma_0\otimes\cdots\sigma_0\otimes \sigma_2\otimes 1\otimes 1\cdots,\]
where $\sigma_1$ and $\sigma_2$ occur in the $n$-th position. Then $(c_n)_{n\geq 1}$ are selfadjoint unitary operators and satisfy the following canonical anticommutation relations (CAR):
\begin{equation}\label{equa4.1}
	c_jc_k+c_kc_j=2\delta_{jk},\quad j,k\ge 1.
\end{equation}
The CAR algebra (Clifford algebra) denoted by $\mathcal{C}$ is the von Neumann algebra generated by $(c_n)_{n\ge 1}$. Let us give more details.

Let $\mathcal{I}$ denote the family of all finite subsets of $\mathbb{N}$. For a nonempty $A\in\mathcal{I}$, we arrange the integers of $A$ in an increasing order and write $A=\{k_1<k_2<\cdots<k_n\}$. Define $\max(A)=k_n$ and 
\[c_A=c_{k_1}c_{k_2}\cdots c_{k_n}.\]
If $A=\emptyset$, we set $\max(\emptyset)=1$ and $c_A=1$. Then $c_A$ is unitary for any $A\in\mathcal{I}$. If $A$ is a singleton $\{k\}$, we still use $c_k$ instead of $c_{\{k\}}$. Let $\mathcal{C}_0$ be the family of all finite linear combinations of $(c_A)_{A\in \mathcal{I}}$. Then $\mathcal{C}_0$ is an involutive algebra. Define $\tau$ to be the linear functional on $\mathcal{C}_0$ given by
\begin{equation}\label{tau}
	\tau(x)= \alpha_{\emptyset}
\end{equation} 
for $x=\sum\limits_{A\in \mathcal{I}} \alpha_A c_A$. One can check that $\tau$ is a positive faithful tracial state on $\mathcal{C}_0$. Then the CAR algebra $\mathcal{C}$ is the von Neumann algebra of the GNS representation of $\tau$. Note that $(c_A)_{A\in \mathcal{I}}$ is an orthonormal basis of $L_2(\mathcal{C})$. We refer the reader to \cite{Se} and \cite{PR} for more information on CAR algebra.

Let $\mathcal{C}_n$ be the von Neumann subalgebra generated by $\{c_A: \max(A)\leq n\}$ for any $n\geq 1$. It is clear that $\mathcal{C}_n$ is of dimension $2^n$, and $(\mathcal{C}_n)_{n\geq 1}$ is a filtration of $\mathcal{C}$. Then for any $b\in L_p(\mathcal{C})$ ($1\leq p\leq \8$),
$$  d_nb= \sum_{\max(A)=n}\hat{b}(A)c_A, \quad \forall \ n\geq 1, $$
where $\hat{b}(A)=\tau(c_A^*\cdot b)$.

\subsubsection{\textbf{Tensor product of matrix algebras.}} 
Let $\mathscr{M}_n=\mathbb{M}_d^{\otimes n}$ be endowed with the normalized trace $\text{tr}_d^{\otimes n}$. We embed $\mathscr{M}_n$ into $\mathscr{M}$ via the map $x \in \mathscr{M}_n\longmapsto x\otimes 1\otimes 1\otimes \cdots \in \mathscr{M}$.
Then $(\mathscr{M}_n)_{n\ge 1}$ is a natural filtration of $\mathscr{M}$. We will give an orthonormal basis of $L_2(\mathscr{M})$ in Section \ref{Application 2}.

\

\subsection{$d$-adic martingales and martingale Besov spaces}
In this subsection, we introduce $d$-adic martingales and martingale Besov spaces.
\subsubsection{\textbf{$d$-adic commutative martingales}}

Let $d\ge 2$ be a fixed integer. We are particularly interested in $d$-adic commutative martingales since it is closely related to dyadic martingales on Euclidean spaces. In this subsection, we give a general definition of $d$-adic commutative martingales. Afterwards we will present an orthonormal basis of Haar wavelets for $d$-adic commutative martingales, which will be used to represent martingale paraproducts and to define martingale Besov spaces for semicommutative $d$-adic martingales (to be defined in Section \ref{sec2.3}).

Let $\Omega$ be a measure space endowed with a $\sigma$-finite measure $\mu$. Assume that in $\Omega$, there exists a family of measurable sets $I_{n,k}$ for $n, k\in \mathbb{Z}$ satisfying the following properties:
\begin{enumerate}
	\item  $I_{n,k}$  are pairwise disjoint for any $k$ if $n$ is fixed;
	\item $\cup_{k\in \mathbb{Z}}I_{n,k}=\Omega$ for every $n$;
	\item $I_{n,k}=\cup_{q=1}^d I_{n+1,kd+q-1}$ for any $n, k$, so each $I_{n,k}$ is a union of $d$ disjoint subsets $I_{n+1,kd+q-1}$;
	\item $\mu(I_{n,k})=d^{-n}$ for any $n, k$.
\end{enumerate}
Then $I_{n,k}$ are called $d$-adic intervals, and let $\mathcal{ D}$ be the family of all such $d$-adic intervals. For $I\in \mathcal{D}$, let $\tilde{I}$ be the parent interval of $I$, and $I{(j)}$ the $j$-th subinterval of $I$, namely
\[(I_{n,k}){(j)}=I_{n+1,kd+j-1},\quad \forall n,k\in \mathbb{Z},1\le j\le d.\] 
Denote by $\mathcal{D}_n$ the collection of $d$-adic intervals of length $d^{-n}$ in $\mathcal{D}$. Given $I\in \mathcal{D}$, let $\mathcal{D}(I)$ be the collection of $d$-adic intervals contained in $I$, and $\mathcal{D}_n(I)$ the intersection of $\mathcal{D}_n$ and $\mathcal{D}(I)$.
For each $n\in \mathbb{Z}$, denote by $\mathcal{F}_n$ the $\sigma$-algebra generated by the $d$-adic intervals $I_{n,k}$, $\forall k\in \mathbb{Z}$. Denote by $\mathcal{F}$ the $\sigma$-algebra generated by all $d$-adic intervals for all $I_{n,k}$, $\forall n,k\in\mathbb{Z}$. 

Then $(\mathcal{F}_n)_{n\in \mathbb{Z}}$ is a filtration associated with the measure space $(\Omega, \mathcal{F}, \mu)$. Denote by $L^{\rm{loc}}_1(\Omega)$ the family of all locally integrable functions $g$ on $\Omega$, that is, $g\in L_1(I_{n, k})$ for all $n, k\in \mathbb{Z}$. For a locally integrable function $g\in L^{\rm{loc}}_1(\Omega)$, the sequence $(g_n)_{n\in\mathbb{Z}}$ is called a $d$-adic martingale, where
$$g_n= \mathbb{E}(g|\mathcal{F}_n)=\sum_{k=-\8}^{\8}\frac{\mathbbm{1}_{I_{n,k}}}{\mu(I_{n,k})}\int_{I_{n,k}}g\ d\mu.  $$
The martingale differences are defined as $d_n g=g_n-g_{n-1}$ for any $n\in\mathbb{Z}$. We also denote $g_n$ by $\mathbb{E}_n(g)$ ($n\in\mathbb{Z}$) as usual.

\begin{definition}\label{haar}
	Let $\omega=e^{\frac{2\pi \mathrm{i}}{d}}$ (here $\mathrm{i} $ is the imaginary number). For any $I=I_{n,k}\in \mathcal{D}$, define
	\[h_I^i=d^{n/2}\sum_{j=0}^{d-1}  \omega^{i{(j+1)}}\mathbbm{1}_{I_{n+1,kd+j}}, \quad \forall \ 1\leq i\leq d-1,\]
	and $h_I^0=d^{n/2}\mathbbm{1}_I$.
\end{definition}
Then $\{h_I^i\}_{I\in\mathcal{D},1\leq i\leq d-1}$ is an orthonormal basis on $L_2(\Omega)$ because $\forall g\in L_2(\Omega)$
\begin{equation}
	g=\sum_{k=-\8}^{\8} d_kg=\sum_{k=-\8}^{\8} \biggl(\sum_{|I|=d^{-k+1}}\sum_{i=1}^{d-1}h_I^i\langle h_I^i,g\rangle\biggr).
\end{equation}   
We call $\{h_I^i\}_{I\in\mathcal{D},1\leq i\leq d-1}$  the system of Haar wavelets. Note that for any $1\leq i, j\leq d-1$,
\begin{equation}\label{formula2.2}
	h_I^i\cdot h_I^j=\mu(I)^{-1/2} h_I^{\overline{i+j}},
\end{equation}
where $\overline{i+j}$ is the remainder in $[1, d]$ modulo $d$. The equality (\ref{formula2.2}) is vital in our proofs of Theorem \ref{thm1.2} and Lemma \ref{TLambdab}.

\begin{example}\label{example1.4.3}
	A natural example of $d$-adic martingales is where $\Omega=\mathbb{R}$, $\mu=m$ and $I_{n,k}$ are defined as follows
	\[ I_{n,k}=[kd^{-n},(k+1)d^{-n}), \quad \forall n,k\in\mathbb{Z}.\]
\end{example}
\begin{example}	
	For $d=2^n$ and $\Omega=\mathbb{R}^n$, define
	\[ \mathcal{ D}_k=\{2^{-k}([0,1)^n+q):q\in\mathbb{Z}^n\}, \quad \forall k\in \mathbb{Z}.\]
	Then $\mathcal{D}=\{2^{-k}([0,1)^n+q):k\in\mathbb{Z},q\in\mathbb{Z}^n\} $ is the family of all $2^n$-adic intervals. Indeed, this is the dyadic filtration on $\mathbb{R}^n$.
\end{example}
\noindent$\mathit{\textbf{Convention}}$. In the sequel, for simplicity of notation, we will always assume that $\Omega=\mathbb{R}$ as this does not change the $d$-adic martingale structure. Denote also by $|I|$ the length $m(I)$ of $I\in\mathcal{ D}$. 

\

We define the $d$-adic martingale $BMO$ space as follows:
\begin{defn}
	The martingale $BMO$ space of $d$-adic martingale denoted by $BMO^d(\mathbb{R})$ is the space of all locally integral functions $b$ such that
	\begin{equation}\label{BMOd}
		\|b\|_{BMO^d(\mathbb{R})}= \sup _{n\in\mathbb{Z}}\ \bigg\|\mathbb{E}_n\bigg(\sum_{k=n+1}^{\8}|d_k b|^2\bigg)\bigg\|_\8^{1/2} <\infty. 
	\end{equation} 
\end{defn}

For $h\in L_1^{\rm{loc}}(\mathbb{R})$, we define the $d$-adic martingale square function
\begin{equation*}
	S(h)=\biggl(\sum_{k\in\mathbb{Z}} |d_kh|^2\biggr)^{1/2}.
\end{equation*}

\begin{defn}
	The $d$-adic martingale Hardy space is defined by
	\begin{equation}\label{defnh1d}
		H^d_1(\mathbb{R})=\{h\in L_1(\mathbb{R}):\|h\|_{H^d_1(\mathbb{R})}=\|S(h)\|_{L_1(\mathbb{R})}<\infty\}.
	\end{equation}
\end{defn}
It is well-known that $(H_1^d(\mathbb{R}))^*=BMO^d(\mathbb{R})$. We refer the reader to \cite{GA} for more details on martingale Hardy spaces.

\subsubsection{\textbf{Semicommutative $d$-adic martingales and martingale Besov spaces}}\label{sec2.3}

In this subsection, we are concerned with semicommutative $d$-adic martingales and martingale Besov spaces. Firstly, we introduce the definition of semicommutative $d$-adic martingales. Then we give the definitions of martingale Besov spaces for semicommutative $d$-adic martingales, CAR algebra and $\mathscr{M}=\mathop{\otimes}\limits_{k=1}^{\infty}\mathbb{M}_{d}$.

We define the semicommutative $d$-adic martingales in the same way as in the commutative setting. Similarly, denote by $L^{\rm{loc}}_1({\mathbb{R}}, L_1(\M))$ the family of all $f$ such that $\mathbbm{1}_{I_{n, k}}\cdot f \in L_1({\mathbb{R}}, L_1(\M))$ for any $n, k\in\mathbb{Z}$. Then $\forall f\in L_1^{\rm{loc}}(\mathbb{R}, L_1(\M))$, the sequence $(f_n)_{n\in\mathbb{Z}}$ is called a semicommutative $d$-adic martingale, where
\begin{equation}\label{fn}
	f_n= \mathbb{E}(f|\mathcal{F}_n)=\sum_{k=-\8}^{\8}\frac{\mathbbm{1}_{I_{n,k}}}{\mu(I_{n,k})}\int_{I_{n,k}}f\ d\mu.  
\end{equation}
For any $f\in L_1(\mathbb{R},L_1(\mathcal{M}))$ and $g\in L_\8(\mathbb{R})$, define
\[\langle g, f\rangle= \int_{\mathbb{R}}\overline{g} \cdot f \ dm. \]
One easily checks that $\la g, f\ra \in L_1(\M)$. By a slight abuse of notation, we use the same notation $\langle$$\cdot$,$\cdot$$\rangle$ to denote the inner product in any given Hilbert space. Besides, by \eqref{fn}, the martingale differences are given by 
$$ d_n f=\sum_{|I|=d^{-n+1}}\sum_{i=1}^{d-1}h_I^i\otimes\langle h_I^i,f\rangle, \quad \forall f\in L_1^{\rm{loc}}(\mathbb{R}, L_1(\M)) \ \text{and} \ n\in \mathbb{Z}. $$
Note that $\langle h_I^i,f\rangle\in L_1(\M)$. Later, we will give a more general definition of martingale differences.

\begin{rem}
	By Remark \ref{conmar}, if $f\in L_p(\mathbb{R}, L_p(\M))$ for $1<p\leq \8$, then
	$$   \sum_{n=-k}^k d_nf \longrightarrow f  \quad \text{as} \quad k\rightarrow\8  $$
	in $L_p(\mathbb{R}, L_p(\M))$ (in $w^*$-topology for $p=\8$).
\end{rem}

We will utilize $h_I^i$ to give a direct representation of $\pi_b$, which is easier to handle. It is well-known that $L_2(\mathbb{R},L_2(\mathcal{M}))=L_2(\mathbb{R})\otimes L_2(\mathcal{M})$. In the sequel, for any $f\in L_2(\mathbb{R})$ and $x\in L_2(\mathcal{M})$, we use ``$x\cdot f$ ''  (or ``$f\cdot x$ '') to denote $f\otimes x\in L_2(\mathbb{R},L_2(\mathcal{M}))$ for sake of simplicity.

Now we calculate $\pi_b$. Let $b\in  L_1^{\rm{loc}}(\mathbb{R}, L_1(\M))$. For $f\in L_2(\mathbb{R},L_2(\mathcal{M}))$, we have 
\begin{equation}\label{calcu}
	\begin{aligned}
		\pi_b(f){}&=\sum_{k=-\infty}^{\infty}d_kb\cdot f_{k-1}\\&=\sum_{k=-\infty}^{\infty}\biggl(\sum_{|I|=d^{-k+1}}\sum_{i=1}^{d-1}h_I^i\otimes\langle h_I^i,b\rangle\biggr)\biggl(\sum_{|I|=d^{-k+1}}\mathbbm{1}_I\otimes\biggl\langle \frac{\mathbbm{1}_I}{|I|},f\biggr\rangle\biggr)\\&=\sum_{I\in \mathcal{D}}\sum_{i=1}^{d-1}h^i_I\otimes\langle h_I^i,b\rangle \biggl\langle \frac{\mathbbm{1}_I}{|I|},f\biggr\rangle,
	\end{aligned}
\end{equation}
which, by (\ref{XT}), can be rewritten as
\begin{equation}\label{pib}
	\begin{aligned}
		\pi_b(f)=\sum_{I\in \mathcal{D}}\sum_{i=1}^{d-1}h^i_I\langle h_I^i,b\rangle \biggl\langle \frac{\mathbbm{1}_I}{|I|},f\biggr\rangle.
	\end{aligned}
\end{equation}
The adjoint operator of $\pi_b$ is given by $\forall f\in L_2(\mathbb{R},L_2(\mathcal{M}))$
\begin{equation}\label{pistar}
	\begin{aligned}
		\pi_b^*(f)&=\sum_{k\in\mathbb{Z}} \mathbb{E}_{k-1}(d_k b^* d_kf) \\
		&=\sum_{I\in \mathcal{D}}\sum_{i=1}^{d-1}\frac{\mathbbm{1}_I}{|I|}\langle h_I^i, b\rangle^* \langle h_I^i,f\rangle\\
		&=\sum_{I\in \mathcal{D}}\sum_{i=1}^{d-1}\frac{\mathbbm{1}_I}{|I|}\langle b,h_I^i\rangle \langle h_I^i,f\rangle. 
	\end{aligned}
\end{equation}

From \eqref{pib}, we see that the martingale paraproduct $\pi_b$ with symbol $b$ is induced by the operator-valued Haar multiplier $(b_I^i)_{I\in \mathcal{D}, 1\leq i\leq d-1}$ where
$$  b_I^i= \langle h_I^i,b\rangle. $$
Hence, in general, we define $\pi_b$ in the following way: 
\begin{definition}
	For any operator-valued Haar multiplier $b=(b_I^i)_{I\in \mathcal{D}, 1\leq i\leq d-1}\subset L_0(\M)$, $\pi_b$ with symbol $b$ is defined as follows
	\begin{equation}\label{defn1.7.4}
		\begin{aligned}
			\pi_{b}(f)=\sum_{I\in \mathcal{D}}\sum_{i=1}^{d-1}h^i_I b_I^i \biggl\langle \frac{\mathbbm{1}_I}{|I|},f\biggr\rangle, \quad \forall f\in L_2(\mathbb{R},L_2(\mathcal{M})).
		\end{aligned}
	\end{equation}
	Besides, the corresponding sequence of martingale differences $(d_n b)_{n\in\mathbb{Z}}$ with symbol $b$ is given by
	\begin{equation}\label{dnb}
		d_n b= \sum_{|I|=d^{-n+1}}\sum_{i=1}^{d-1}h_I^i \cdot b_I^i, \quad \forall n\in \mathbb{Z}. 
	\end{equation}
\end{definition}
\begin{rem}
	In \eqref{defn1.7.4}, if all but finitely many $b_I^i$ are $0$, then $\pi_b$ is densely defined.
\end{rem}
Therefore, each operator-valued Haar multiplier in $L_0(\M)$ corresponds to a sequence of martingale differences and vice versa. In the sequel, $\pi_b$ is defined as in \eqref{defn1.7.4}, and for consistency of notation, we identify $b_I^i$ and $\langle h_I^i,b\rangle$ by a slight abuse of notation.

As mentioned before, we use Haar wavelets to define martingale Besov spaces for semicommutative $d$-adic martingales.
\begin{definition}\label{mbs1}
	The martingale Besov space $\pmb{B}_p^d(\mathbb{R},\M)$ $(0<p<\8)$ of semicommutative $d$-adic martingales is the space of all operator-valued Haar multipliers $b=(\langle h_I^i,b\rangle)_{I\in \mathcal{D}, 1\leq i\leq d-1}\subset L_0(\M)$ such that
	\begin{equation}\label{e1.2}
		\|b\|_{\pmb{B}_p^d(\mathbb{R},\M)}=\biggl(\sum_{I\in \mathcal{D}}\sum_{i=1}^{d-1}\biggl(\frac{\|\langle h_I^i,b\rangle\|_{L_p{(\mathcal{M})}}}{|I|^{1/2}}\biggr)^p\biggr)^{1/p}<\infty.
	\end{equation}
\end{definition}

	\begin{rem}
	In particular, if $\M=\mathbb{C}$, the martingale Besov space $\pmb{B}_p^d(\mathbb{R},\mathbb{C})$ is as same as that in \cite{CP}. 
\end{rem}

\begin{proposition}\label{prop2.100}
	When $1\leq p<\8$, $\pmb{B}_p^d(\mathbb{R},\M)$ is a Banach space. When $0<p<1$, $\pmb{B}_p^d(\mathbb{R},\M)$ is a quasi-Banach space.
\end{proposition}
\begin{proof}
	We just consider $1\leq p<\8$ and $d=2$ for simplicity as it is similar for $0<p<1$ and $d>2$. Note that $ \mathcal{D}$ is a denumerable set. Let $\mathcal{P}( \mathcal{D})$ be the power set of $ \mathcal{D}$. Consider the following function from all singletons of $\mathcal{P}( \mathcal{D})$ to $[0, \8)$:
	$$  \mu(\{I\})=\dfrac{1}{|I|^{p/2}}, \ \forall \ I\in  \mathcal{D}. $$
	Then $\mu$ can be extended to be a measure on $\mathcal{P}( \mathcal{D})$, and $(\mathcal{ D}, \mathcal{P}( \mathcal{D}), \mu)$ is a measure space. Hence ${\pmb{B}_p^d(\mathbb{R},\M)}$ is equivalent to the Banach space $L_p( \mathcal{ D}, L_p(\M))$ consisting of all $L_p(\M)$-valued Bochner $L_p$-integrable functions.
\end{proof}
\begin{proposition}
	 The subspace $\mathcal{S}(L_\8(\mathbb{R})\otimes \M) \cap \pmb{B}_p^d(\mathbb{R},\M)$ is dense in $\pmb{B}_p^d(\mathbb{R},\M)$ for $0<p<\8$.
\end{proposition}
\begin{proof}
	It is clear that for $x\in \mathcal{S}(\M)$, $b:=h_I^i\otimes x \in \mathcal{S}(L_\8(\mathbb{R})\otimes \M)$ belongs to $\pmb{B}_p^d(\mathbb{R},\M)$, which implies the desired result.
\end{proof}

The following two propositions give equivalent descriptions of $\pmb{B}_p^d(\mathbb{R},\M)$.
\begin{proposition}\label{bbpdrmdp}
For any $0<p<\8$ and $b\in \pmb{B}_p^d(\mathbb{R},\M)$,
	\[\|b\|_{\pmb{B}_p^d(\mathbb{R},\M)}\approx_{d, p}\biggl(\sum_{k=-\infty}^\infty d^{k}\|d_kb\|_{L_p(\mathbb{R},L_p({\mathcal{M}}))}^p\biggr)^{1/p}.\]
\end{proposition}
\begin{proof}
		Recall that martingale differences are defined in \eqref{dnb}. 
		Then $\forall k\in\mathbb{Z}$, one has
		\begin{equation*}
			\begin{aligned}
				\|d_kb\|_{L_p(\mathbb{R},L_p({\mathcal{M}}))}^p{}&=\sum_{I\in\mathcal{D}_{k-1}}\bigg\|\sum_{i=1}^{d-1}h_I^i \cdot b_I^i\bigg\|_{L_p(\mathbb{R},L_p({\mathcal{M}}))}^p
				=d^{-k}\cdot\sum_{I\in\mathcal{D}_{k-1}}\frac{1}{|I|^{p/2}}\sum_{j=0}^{d-1}\bigg\|\sum_{i=1}^{d-1}\omega^{i(j+1)} b_I^i\bigg\|_{L_p({\mathcal{M}})}^p.
			\end{aligned}
		\end{equation*}
	Let $\mathcal{Z}=\mathop{\oplus}\limits_{k=1}^d \mathcal{M}$ be the von Neumann algebra direct sum of $d$ copies of $\mathcal{M}$. One can check that for any $x:=(x_1,x_2,\cdots, x_d)^\top\in\mathcal{Z}$,
	\begin{equation*}
		\|x\|_{L_p(\mathcal{Z})}=\biggl(\sum_{i=1}^d\|x_i\|^p_{L_p(\mathcal{M})}\biggr)^{1/p}.
	\end{equation*}
In other words, $L_p(\mathcal{Z})=\oplus_p^d L_p(\mathcal{M})$, where $\oplus_p^d$ denotes the direct sum in the $\ell_p^d$-sense. Let
\begin{equation*}
	\xi_I=(b_I^1, b_I^2, \cdots, b_I^{d-1})^\top,\quad \tilde{\xi}_I=(b_I^1, b_I^2, \cdots, b_I^{d-1}, 0)^\top
\end{equation*}
and
\begin{equation*}
	B_I=\sum\limits_{i=1}^{d-1}\sum\limits_{j=0}^{d-1}\omega^{i(j+1)}e_{j+1,i},\quad \tilde{B}_I=\sum\limits_{i=1}^{d}\sum\limits_{j=0}^{d-1}\omega^{i(j+1)}e_{j+1,i}.
\end{equation*}
Then
\begin{equation*}
	\begin{aligned}
		\sum_{j=0}^{d-1}\bigg\|\sum_{i=1}^{d-1}\omega^{i(j+1)} b_I^i\bigg\|_{L_p({\mathcal{M}})}^p=\|B_I\xi\|^p_{L_p(\mathcal{Z})} =\|\tilde{B}_I\tilde{\xi}\|^p_{L_p(\mathcal{Z})}\approx_{d, p} \|\tilde{\xi}\|^p_{L_p(\mathcal{Z})}=\sum_{i=1}^{d-1}\|b_I^i\|_{L_p({\mathcal{M}})}^p.
	\end{aligned}
\end{equation*}
It implies that
\begin{equation*}
	\begin{aligned}
		\biggl(\sum_{k=-\infty}^\infty d^{k}\|d_kb\|_{L_p(\mathbb{R},L_p({\mathcal{M}}))}^p\biggr)^{1/p}\approx_{d, p}  \biggl(\sum_{I\in \mathcal{D}}\sum_{i=1}^{d-1}\biggl(\frac{\|b_I^i\|_{L_p{(\mathcal{M})}}}{|I|^{1/2}}\biggr)^p\biggr)^{1/p}=\|b\|_{\pmb{B}_p^d(\mathbb{R},\M)}.
	\end{aligned}
\end{equation*}
This finishes the proof.
\end{proof}	

	
\begin{prop}\label{equbbk}
	Let $1\le p<\infty$. Suppose $b$ is a locally integrable $L_p(\M)$-valued function and $b\in \pmb{B}_p^{d}(\mathbb{R},\mathcal{M})$. Then
	\begin{equation*}
			\|b\|_{\pmb{B}_p^d(\mathbb{R},\M)} \approx_{d,p} \left(\sum_{k\in\mathbb{Z}} d^k \|b-b_k\|^p_{L_p(\mathbb{R},L_p(\mathcal{M}))}\right)^{1/p}.
	\end{equation*}
\end{prop}

\begin{proof}
	On the one hand, since $b\in \pmb{B}_p^{d}(\mathbb{R},\mathcal{M})$, from Proposition \ref{bbpdrmdp} we know that
	$$\biggl(\sum_{k\in\mathbb{Z}} d^{k}\|d_kb\|_{L_p(\mathbb{R},L_p({\mathcal{M}}))}^p\biggr)^{1/p}<\infty.$$
	Then due to the Minkowski inequality we obtain
	\begin{equation*}
		\begin{aligned}
			{}&\biggl(\sum_{k\in\mathbb{Z}}  \Big(\sum_{j=1}^\infty d^{-j/p}\big\| d^{(j+k)/p}\cdot d_{j+k}b\big\|_{L_p(\mathbb{R},L_p(\mathcal{M}))}\Big)^p\biggr)^{1/p}\\
			&\le \sum_{j=1}^\infty  d^{-j/p} \biggl(\sum_{k\in\mathbb{Z}} \big\| d^{(j+k)/p}\cdot d_{j+k}b\big\|^p_{L_p(\mathbb{R},L_p(\mathcal{M}))}\biggr)^{1/p}\\
			&=\sum_{j=1}^\infty  d^{-j/p} \biggl(\sum_{k\in\mathbb{Z}} d^k \|d_{k}b\|^p_{L_p(\mathbb{R},L_p(\mathcal{M}))}\biggr)^{1/p}\\
			&=\frac{1}{d^{1/p}-1}\biggl(\sum_{k\in\mathbb{Z}} d^k \|d_{k}b\|^p_{L_p(\mathbb{R},L_p(\mathcal{M}))}\biggr)^{1/p}<\infty.
		\end{aligned}
	\end{equation*}
	Besides, for any $k\in\mathbb{Z}$, notice that
	\begin{equation*}
		\begin{aligned}
			\|b-b_k\|_{L_p(\mathbb{R},L_p(\mathcal{M}))}{}&\le  \sum_{j=k+1}^\infty \| d_jb\|_{L_p(\mathbb{R},L_p(\mathcal{M}))}\\
			&=\sum_{j=1}^\infty\| d_{j+k}b\|_{L_p(\mathbb{R},L_p(\mathcal{M}))}\\
			&=d^{-k/p}\cdot\sum_{j=1}^\infty d^{-j/p} \big\| d^{(j+k)/p}\cdot d_{j+k}b\big\|_{L_p(\mathbb{R},L_p(\mathcal{M}))}.
		\end{aligned}
	\end{equation*}
	Thus 
	\begin{equation*}
		\begin{aligned}
			\biggl(\sum_{k\in\mathbb{Z}} d^k \|b-b_k\|^p_{L_p(\mathbb{R},L_p(\mathcal{M}))}\biggr)^{1/p}{}&\le \biggl(\sum_{k\in\mathbb{Z}}  \Big(\sum_{j=1}^\infty d^{-j/p}\big\| d^{(j+k)/p}\cdot d_{j+k}b\big\|_{L_p(\mathbb{R},L_p(\mathcal{M}))}\Big)^p\biggr)^{1/p}\\
			&\le \frac{1}{d^{1/p}-1}\biggl(\sum_{k\in\mathbb{Z}} d^k \|d_{k}b\|^p_{L_p(\mathbb{R},L_p(\mathcal{M}))}\biggr)^{1/p}.
		\end{aligned}
	\end{equation*}
	On the other hand, notice that for any $k\in\mathbb{Z}$,
	\begin{equation*}
		\|d_kb\|_{L_p(\mathbb{R},L_p(\mathcal{M}))}=\|\mathbb{E}_k(b-b_{k-1})\|_{L_p(\mathbb{R},L_p(\mathcal{M}))}\le \|b-b_{k-1}\|_{L_p(\mathbb{R},L_p(\mathcal{M}))},
	\end{equation*}
	which implies that
	\begin{equation*}
		\sum_{k\in\mathbb{Z}} d^k \|d_kb\|^p_{L_p(\mathbb{R},L_p(\mathcal{M}))} \le d\cdot\sum_{k\in\mathbb{Z}} d^k \|b-b_k\|^p_{L_p(\mathbb{R},L_p(\mathcal{M}))}.
	\end{equation*}
	The proof is completed.
\end{proof}
	
\

As for the martingale Besov spaces concerning the CAR algebra and $\mathscr{M}=\mathop{\otimes}\limits_{k=1}^{\infty}\mathbb{M}_{d}$, we use the martingale differences to formulate their definitions.

\begin{definition}\label{mbs2}
	The martingale Besov space $\pmb{B}_p(\mathcal{C})$ $(0<p<\8)$ for the CAR algebra is the completion of the set consisting of all $b\in \mathcal{S}(\mathcal{C})$ such that
	\begin{equation*}
		\|b\|_{\pmb{B}_p(\mathcal{C})}=\biggl(\sum_{k=1}^\infty2^{k}\|d_kb\|^p_{L_p(\mathcal{C})}\biggr)^{1/p}<\infty,
	\end{equation*}
	with respect to $ \|\cdot \|_{\pmb{B}_p(\mathcal{C})} $.
\end{definition}
\begin{definition}\label{mbs3}
	The martingale Besov space $\pmb{B}_p(\mathscr{M})$ $(0<p<\8)$ for $\mathscr{M}=\mathop{\otimes}\limits_{k=1}^{\infty}\mathbb{M}_{d}$ is the completion of the set consisting of all $b\in \mathcal{S}(\mathscr{M})$ such that
	\begin{equation*}
		\|b\|_{\pmb{B}_p(\mathscr{M})}=\biggl(\sum_{k=1}^\infty d^{2k}\|d_kb\|^p_{L_p(\mathscr{M})}\biggr)^{1/p}<\infty,
	\end{equation*}
	with respect to $ \|\cdot \|_{\pmb{B}_p(\mathscr{M})} $.
\end{definition}
\begin{proposition}\label{prop2.180}
	When $1\leq p<\8$, $\pmb{B}_p(\mathcal{C})$ and $\pmb{B}_p(\mathscr{M})$ are Banach spaces. When $0<p<1$,  they are quasi-Banach spaces.
\end{proposition}

The proof of Proposition \ref{prop2.180} is similar to that of Proposition \ref{prop2.100}, and we omit its proof.

\subsection{Operator-valued $BMO$ spaces and Hardy spaces}\label{ovbmo}

Let $b$ be an $\mathcal{M}$-valued function that is Bochner integrable on any cube in $\mathbb{R}^n$, and define the following operator-valued $BMO$ spaces:
\begin{equation*}
	\begin{aligned}
		BMO_c(\mathbb{R}^n,\mathcal{M}){}&=\biggl\{b:\|b\|_{BMO_c(\mathbb{R}^n,\mathcal{M})}=\sup_{Q\subset \mathbb{R}^n \atop Q \operatorname{cube}}\big\|MO(b;Q)\big\|_{\mathcal{M}}<\infty \biggr\},\\
		BMO_r(\mathbb{R}^n,\mathcal{M})	&=\biggl\{b:\|b\|_{BMO_r(\mathbb{R}^n,\mathcal{M})}=\|b^*\|_{BMO_c(\mathbb{R}^n,\mathcal{M})}<\infty\biggr\},\\
		BMO_{cr}(\mathbb{R}^n,\mathcal{M})&=BMO_c(\mathbb{R}^n,\mathcal{M})\cap	BMO_r(\mathbb{R}^n,\mathcal{M}),
	\end{aligned}
\end{equation*}
where $MO(b;Q)$ is defined in \eqref{MOBQ}.

Now we introduce the preduals of $BMO_c(\mathbb{R}^n,\mathcal{M})$ and $BMO_r(\mathbb{R}^n,\mathcal{M})$, which are called operator-valued Hardy spaces. Let $1\le p<\infty$. For any $f\in L_1(\mathbb{R}^n,L_1(\mathcal{M}))$, let $\tilde{f}(x,y)=P_y\ast f(x)$ be the Poisson integral of $f$ on the upper half plane $\mathbb{R}_+^{n+1}=\{(x,y):x\in \mathbb{R}^n,y>0\}$, where 
$$P_y(x)=\frac{\Gamma(\frac{n+1}{2})}{\pi^{\frac{n+1}{2}}}\frac{y}{(y^2+|x|^2)^{\frac{n+1}{2}}}$$
is the Poisson kernel, and $\Gamma$ is the gamma function. Let $\gamma=\{(x,y)\in \mathbb{R}_+^{n+1}:|x|^2<y^2\}$. The operator-valued column Hardy space $H_{p,c}(\mathbb{R}^n,\mathcal{M})$ is the space of all $f\in L_1(\mathbb{R}^n,L_1(\mathcal{M}))$ such that 
\begin{equation*}
	\|f\|_{H_{p,c}(\mathbb{R}^n,\mathcal{M})}=\bigg\|\bigg(\int_{\gamma}  \Big|\frac{\partial \tilde{f}}{\partial x}(x+\cdot,y)\Big|^2+  \Big|\frac{\partial \tilde{f}}{\partial y}(x+\cdot,y)\Big|^2 \frac{dxdy}{y^{n-1}}\bigg)^{1/2}\bigg\|_{L_p(\mathbb{R}^n,L_p(\mathcal{M}))}<\infty.
\end{equation*}
Similarly, define the operator-valued row Hardy space as follows
\begin{equation*}
	H_{p,r}(\mathbb{R}^n,\mathcal{M})=\bigg\{f\in L_1(\mathbb{R}^n,L_1(\mathcal{M})):\|f\|_{H_{p,r}(\mathbb{R}^n,\mathcal{M})}=\|f^*\|_{H_{p,c}(\mathbb{R}^n,\mathcal{M})}<\infty  \bigg\}.
\end{equation*}

The following theorem in \cite{Mei1} is on the duality between operator-valued Hardy spaces and $BMO$ spaces.
\begin{thm}\label{HBMOdual}
   $\big(H_{1,c}(\mathbb{R}^n,\mathcal{M})\big)^*=BMO_c(\mathbb{R}^n,\mathcal{M})$. More precisely, every $g\in BMO_c(\mathbb{R}^n,\mathcal{M})$ defines a continuous linear functional $L_g$ on $H_{1,c}(\mathbb{R}^n,\mathcal{M})$ by
    \begin{equation*}
    	L_g(f)=\tau\biggl(\int_{\mathbb{R}^n} g(x)^*f(x)dx\biggr),\quad \forall f\in H_{1,c}(\mathbb{R}^n,\mathcal{M}).
    \end{equation*}
    Conversely, if $L\in H_{1,c}(\mathbb{R}^n,\mathcal{M})^*$, then there exists some $g\in BMO_c(\mathbb{R}^n,\mathcal{M})$ such that $L=L_g$ as above.
    Similarly,
    \begin{equation*}
       \big(H_{1,r}(\mathbb{R}^n,\mathcal{M})\big)^*=BMO_r(\mathbb{R}^n,\mathcal{M}).
    \end{equation*}
\end{thm}

In addition, we also need another two kinds of $L_2(\mathbb{R}^n)$-valued noncommutative $L_p$-spaces. For any $1\le p\le\infty$, the column noncommutative $L_p$-space $L_p(\mathcal{M},L_2^c(\mathbb{R}^n))$ (resp. the row noncommutative $L_p$-space $L_p(\mathcal{M},L_2^r(\mathbb{R}^n))$) is defined to be the space of all $f\in L_p(B(L_2(\mathbb{R}^n))\otimes \mathcal{M})$ such that
\begin{equation*}
	\|f\|_{L_p(\mathcal{M},L_2^c(\mathbb{R}^n))}=\bigg\|\biggl(\int_{\mathbb{R}^n} |f(x)|^2dx\biggr)^{1/2}\bigg\|_{L_p(\mathcal{M})}<\infty
\end{equation*}

\begin{equation*}
   \bigg(\,\, \mathrm{resp}.\quad	\|f\|_{L_p(\mathcal{M},L_2^r(\mathbb{R}^n))}=\|f^*\|_{L_p(\mathcal{M},L_2^c(\mathbb{R}^n))} <\infty \,\,\bigg).
\end{equation*}
As in the commutative setting, operator-valued Hardy spaces admit atomic decomposition as well. In what follows, we present the atomic decomposition for operator-valued Hardy spaces. At first, we give the definition of $\mathcal{M}^c$-atoms and $\mathcal{M}^r$-atoms.
\begin{definition}
	A function $a\in L_1(\mathcal{M},L_2^c(\mathbb{R}^n))$ (resp. $a\in L_1(\mathcal{M},L_2^r(\mathbb{R}^n))$) is an $\mathcal{M}^c$-atom  (resp. $\mathcal{M}^r$-atom) if there exists a cube $Q$ such that
	\begin{enumerate}
		\item $\mathrm{supp}\,a\subseteq Q$,
		\item $\int_Q a(x)dx=0$,
		\item $\|a\|_{L_1(\mathcal{M},L_2^c(\mathbb{R}^n))}\le |Q|^{-1/2}$\,\, (\,\,resp.\,\, $\|a\|_{L_1(\mathcal{M},L_2^r(\mathbb{R}^n))}\le |Q|^{-1/2}$\,\,).
	\end{enumerate}
\end{definition}

In \cite{Mei1}, Mei proved that $H_{1,c}(\mathbb{R}^n,\mathcal{M})$ and $H_{1,r}(\mathbb{R}^n,\mathcal{M})$ have the following atomic decomposition.
\begin{thm}\label{atomH1}
	For any $f\in H_{1,c}(\mathbb{R}^n,\mathcal{M})$,
	\begin{equation*}
		\|f\|_{H_{1,c}(\mathbb{R}^n,\mathcal{M})}\approx \inf \bigg\{\sum_{i\in\mathbb{N}} |\lambda_i|: f=\sum_{i\in\mathbb{N}} \lambda_i a_i,\,\,\text{where}\,\, a_i \,\,\text{is an}\,\, \mathcal{M}^c\text{-atom}, \,\, \text{and}\,\,\lambda_i\in\mathbb{C}    \bigg\}.
	\end{equation*}
    Similarly, for any $f\in H_{1,r}(\mathbb{R}^n,\mathcal{M})$,
    \begin{equation*}
    	\|f\|_{H_{1,r}(\mathbb{R}^n,\mathcal{M})}\approx \inf \bigg\{\sum_{i\in\mathbb{N}} |\lambda_i|: f=\sum_{i\in\mathbb{N}} \lambda_i a_i,\,\,\text{where}\,\, a_i \,\,\text{is an}\,\, \mathcal{M}^r\text{-atom}, \,\, \text{and}\,\,\lambda_i\in\mathbb{C}    \bigg\}
    \end{equation*}
\end{thm}

\bigskip

\section{Proof of Theorem \ref{thm1.2}}\label{Proof of the Necessity of theorem 1.2}
   This section is devoted to the proof of Theorem \ref{thm1.2}. We divide this proof into two parts, namely the sufficiency and necessity respectively. In this section, we will denote $B(L_2(\mathbb{R}))\otimes \mathcal{M}$ by $\mathcal{N}$ for convenience. We first proceed with the sufficiency part.
   \subsection{The Sufficiency of Theorem \ref{thm1.2}} In this subsection, we are about to show that
   $$ b\in \pmb{B}_p^d(\mathbb{R},\M)\ \Longrightarrow \ \pi_b\in L_p(\mathcal{N}) \quad \forall \ 0<p<\8. \eqno{(*)} $$

We will prove a stronger result, namely Lemma \ref{nonNWOpre}. We begin with the following lemma, which was already shown by Rochberg and Semmes in \cite{RSe}. However, we give a different and self-contained proof by virtue of Theorem \ref{lem2.1} with the help of martingale paraproducts.
\begin{lemma}\label{nonNWOpre1}
	Let $0<p<\infty$. Assume that $\{e_{I,i}\}_{I\in\mathcal{D},1\le i\le d-1}$ and $\{f_{I,i}\}_{I\in\mathcal{D},1\le i\le d-1}$ are function sequences in $L_2(\mathbb{R})$ satisfying $\mathrm{supp}e_{I,i}, \,\mathrm{supp}f_{I,i}\subseteq I$ and
	$\|e_{I,i}\|_{\infty},\,\|f_{I,i}\|_{\infty}\le |I|^{-\frac{1}{2}}$. Define
	\begin{equation*}
		A(f)=\sum_{I\in \mathcal{D}}\sum_{i=1}^{d-1} \lambda_{I,i} \langle e_{I,i},f\rangle f_{I,i},\quad \forall f\in L_2(\mathbb{R}),
	\end{equation*}
	where $\{\lambda_{I,i}\}_{I\in\mathcal{D},1\le i\le d-1}\subset \mathbb{C}$.
	Then
	\begin{equation*}
		\begin{aligned}
			\|A\|^p_{S_p(L_2(\mathbb{R}))}\lesssim_{d,p}	\sum_{I\in\mathcal{D}}\sum_{i=1}^{d-1}|\lambda_{I,i}|^p.
		\end{aligned}
	\end{equation*}
\end{lemma}
\begin{proof}
	Notice that $A$ can be rewritten as 
	$$A=\sum_{I\in \mathcal{D}}\sum_{i=1}^{d-1}\lambda_{I,i}\cdot f_{I,i}\otimes e_{I,i},$$
	where $f_{I,i}\otimes e_{I,i}$ is defined in Section \ref{pre2}. When $0<p\le 1$, by the triangle inequality, one has
	\begin{equation*}
		\begin{aligned}
			\|A\|^p_{S_p(L_2(\mathbb{R}))}{}&\le \sum_{I\in\mathcal{D}}\sum_{i=1}^{d-1} |\lambda_{I,i}|^p\|f_{I,i}\otimes e_{I,i}\|^p_{S_p(L_2(\mathbb{R}))}\\
			&\leq \sum_{I\in\mathcal{D}}\sum_{i=1}^{d-1} |\lambda_{I,i}|^p\|f_{I,i}\|^p_{L_2(\mathbb{R})}\|e_{I,i}\|^p_{L_2(\mathbb{R})}\le\sum_{I\in\mathcal{D}}\sum_{i=1}^{d-1} |\lambda_{I,i}|^p.
		\end{aligned}
	\end{equation*}
	When $p=2k$ and $k\in\mathbb{N}_+$, we calculate that
	\begin{equation}\label{mp}
		\begin{aligned}
			\|A\|^p_{S_p(L_2(\mathbb{R}))}
			&=\mathrm{Tr} \biggl(\Big|\sum_{I\in\mathcal{D}}\sum_{i=1}^{d-1}\lambda_{I,i}\cdot f_{I,i}\otimes e_{I,i}\Big|^{2k}\biggr)\\
			&=\mathrm{Tr} \biggl(\Big(\sum_{I,J\in\mathcal{D}}\sum_{i,j=1}^{d-1}\lambda_{I,i}^*\lambda_{J,j}\cdot \langle f_{I,i},f_{J,j}\rangle e_{I,i}\otimes e_{J,j}\Big)^{k}\biggr)\\
			&=\sum_{I_1,J_1,\cdots,I_k,J_k\in\mathcal{D}}\sum_{ i_1,j_1,\cdots,i_k,j_k=1}^{d-1} \lambda_{I_1,i_1}^*\lambda_{J_1,j_1}\cdots \lambda_{I_k,i_k}^*\lambda_{J_k,j_k} \\
			&\quad\quad\quad\cdot\prod_{s=1}^k \langle f_{I_s,i_s},f_{J_s,j_s}\rangle \langle e_{J_s,j_s},e_{I_{s+1},i _{s+1}}\rangle,
		\end{aligned}
	\end{equation}
	where we set $I_{k+1}=I_1$ and $i_{k+1}=i_1$. Note that 
	\begin{equation*}
		\begin{aligned}
			\bigg|\prod_{s=1}^k \langle f_{I_s,i_s},f_{J_s,j_s}\rangle \langle e_{J_s,j_s},e_{I_{s+1},i_{s+1}}\rangle\bigg|\le \prod_{s=1}^k \bigg\langle \frac{\mathbbm{1}_{I_s}}{|I_s|^{1/2}},\frac{\mathbbm{1}_{J_s}}{|J_s|^{1/2}}\bigg\rangle \bigg\langle \frac{\mathbbm{1}_{J_s}}{|J_s|^{1/2}},\frac{\mathbbm{1}_{I_{s+1}}}{|I_{s+1}|^{1/2}}\bigg\rangle.
		\end{aligned}
	\end{equation*}
	Then we have
	\begin{equation}\label{mp1}
		\begin{aligned}
			\|A\|^p_{S_p(L_2(\mathbb{R}))}
			&\leq \sum_{I_1,J_1,\cdots,I_k,J_k\in\mathcal{D}}\sum_{ i_1,j_1,\cdots,i_k,j_k=1}^{d-1} |\lambda_{I_1,i_1}^*| |\lambda_{J_1,j_1}|\cdots |\lambda_{I_k,i_k}^*| |\lambda_{J_k,j_k}|\\
			&\quad\quad\quad \cdot\prod_{s=1}^k \bigg\langle \frac{\mathbbm{1}_{I_s}}{|I_s|^{1/2}},\frac{\mathbbm{1}_{J_s}}{|J_s|^{1/2}}\bigg\rangle \bigg\langle \frac{\mathbbm{1}_{J_s}}{|J_s|^{1/2}},\frac{\mathbbm{1}_{I_{s+1}}}{|I_{s+1}|^{1/2}}\bigg\rangle\\
			&=\sum_{I_1,I_2,\cdots,I_{2k-1},I_{2k}\in\mathcal{D}}\sum_{i_1,i_2,\cdots,i_{2k-1},i_{2k}=1}^{d-1} |\lambda_{I_1,i_1}| |\lambda_{I_2,i_2}|\cdots |\lambda_{I_{2k-1},i_{2k-1}}| |\lambda_{I_{2k},i_{2k}}| \\
			&\quad\quad\quad\,\,\,\cdot\prod_{s=1}^{2k} \bigg\langle \frac{\mathbbm{1}_{I_s}}{|I_s|^{1/2}},\frac{\mathbbm{1}_{I_{s+1}}}{|I_{s+1}|^{1/2}}\bigg\rangle, 
		\end{aligned}
	\end{equation}
	where we set $I_{2k+1}=I_1$. Consider the following martingale paraproduct
	$$  C= \sum_{I\in \mathcal{D}}\sum_{i=1}^{d-1}|\lambda_{I,i}|^{1/2}\cdot h_I^i\otimes\frac{\mathbbm{1}_{I}}{|I|^{1/2}}. $$
	By the same calculation as in \eqref{mp}, one has
	\begin{equation*}
		\begin{aligned}
			&\|C\|^{2p}_{S_{2p}(L_2(\mathbb{R}))}\\
			&=\sum_{I_1,I_2,\cdots,I_{2k-1},I_{2k}\in\mathcal{D}}\sum_{i_1,i_2,\cdots,i_{2k-1},i_{2k}=1}^{d-1} (|\lambda_{I_1,i_1}|^{1/2})^2 (|\lambda_{I_2,i_2}|^{1/2})^2\cdots (|\lambda_{I_{2k-1},i_{2k-1}}|^{1/2})^2 (|\lambda_{I_{2k},i_{2k}}|^{1/2})^2\\
			&\quad\quad\quad\,\,\, \cdot \prod_{s=1}^{2k} \bigg\langle \frac{\mathbbm{1}_{I_s}}{|I_s|^{1/2}},\frac{\mathbbm{1}_{I_{s+1}}}{|I_{s+1}|^{1/2}}\bigg\rangle.
		\end{aligned}
	\end{equation*}
	Then by \eqref{mp1} and Theorem \ref{lem2.1},
	$$ \|A\|^p_{S_p(L_2(\mathbb{R}))}\leq \|C\|^{2p}_{S_{2p}(L_2(\mathbb{R}))} \approx_{d, p} \sum_{I\in\mathcal{D}}\sum_{i=1}^{d-1}(|\lambda_{I,i}|^{1/2})^{2p}=\sum_{I\in\mathcal{D}}\sum_{i=1}^{d-1}|\lambda_{I,i}|^p.  $$
	This proves Lemma \ref{nonNWOpre1} for all even numbers $p$. Therefore by interpolation, the proof of Lemma \ref{nonNWOpre1} is finished.
\end{proof}

Now we extend Lemma \ref{nonNWOpre1} to the semicommutative setting.
\begin{lemma}\label{nonNWOpre}
	Let $0<p<\infty$. Assume that $\{e_{I,i}\}_{I\in\mathcal{D},1\le i\le d-1}$ and $\{f_{I,i}\}_{I\in\mathcal{D},1\le i\le d-1}$ are function sequences in $L_2(\mathbb{R})$ satisfying $\mathrm{supp}e_{I,i}, \,\mathrm{supp}f_{I,i}\subseteq I$ and
	$\|e_{I,i}\|_{\infty},\,\|f_{I,i}\|_{\infty}\le |I|^{-\frac{1}{2}}$. Define
	\begin{equation*}
		A(f)=\sum_{I\in \mathcal{D}}\sum_{i=1}^{d-1} \lambda_{I,i} \langle e_{I,i},f\rangle f_{I,i},\quad \forall f\in L_2(\mathbb{R},L_2(\mathcal{M})),
	\end{equation*}
	where $\{\lambda_{I,i}\}_{I\in\mathcal{D},1\le i\le d-1}\subset L_p(\mathcal{M})$.
	Then
	\begin{equation*}
		\begin{aligned}
			\|A\|^p_{L_p(\mathcal{N})}\lesssim_{d,p}	\sum_{I\in\mathcal{D}}\sum_{i=1}^{d-1}\|\lambda_{I,i}\|^p_{L_p(\mathcal{M})}.
		\end{aligned}
	\end{equation*}
\end{lemma}
\begin{proof}
	Note that $A$ can be rewritten as 
	$$A=\sum_{I\in \mathcal{D}}\sum_{i=1}^{d-1}\lambda_{I,i}\cdot f_{I,i}\otimes e_{I,i}. $$
	When $0<p\le 1$, by the triangle inequality, one has
	\begin{equation*}
		\begin{aligned}
			\|A\|^p_{L_p(\mathcal{N})}{}&\le \sum_{I\in\mathcal{D}}\sum_{i=1}^{d-1} \|\lambda_{I,i}\|^p_{L_p(\mathcal{M})}\|f_{I,i}\otimes e_{I,i}\|^p_{S_p(L_2(\mathbb{R}))}\\
			&\leq \sum_{I\in\mathcal{D}}\sum_{i=1}^{d-1} \|\lambda_{I,i}\|^p_{L_p(\mathcal{M})}\|f_{I,i}\|^p_{L_2(\mathbb{R})}\|e_{I,i}\|^p_{L_2(\mathbb{R})}\le\sum_{I\in\mathcal{D}}\sum_{i=1}^{d-1} \|\lambda_{I,i}\|^p_{L_p(\mathcal{M})}.
		\end{aligned}
	\end{equation*}
	When $p=2k$, where $k\in\mathbb{N}_+$, by the Hölder inequality we calculate that
	\begin{equation}\label{mp2}
		\begin{aligned}
			{}&\|A\|^p_{L_p(\mathcal{N})}\\
			&=(\mathrm{Tr}\otimes \tau) \biggl(\Big|\sum_{I\in\mathcal{D}}\sum_{i=1}^{d-1}\lambda_{I,i}\cdot f_{I,i}\otimes e_{I,i}\Big|^{2k}\biggr)\\
			&=(\mathrm{Tr}\otimes \tau) \biggl(\Big(\sum_{I,J\in\mathcal{D}}\sum_{i,j=1}^{d-1}\lambda_{I,i}^*\lambda_{J,j}\cdot \langle f_{I,i},f_{J,j}\rangle e_{I,i}\otimes e_{J,j}\Big)^{k}\biggr)\\
			&=\sum_{I_1,J_1,\cdots,I_k,J_k\in\mathcal{D}}\sum_{ i_1,j_1,\cdots,i_k,j_k=1}^{d-1} \tau(\lambda_{I_1,i_1}^*\lambda_{J_1,j_1}\cdots \lambda_{I_k,i_k}^*\lambda_{J_k,j_k}) \prod_{s=1}^k \langle f_{I_s,i_s},f_{J_s,j_s}\rangle \langle e_{J_s,j_s},e_{I_{s+1},i_{s+1}}\rangle\\
			&\le \sum_{I_1,J_1,\cdots,I_k,J_k\in\mathcal{D}} \sum_{ i_1,j_1,\cdots,i_k,j_k=1}^{d-1} \prod_{s=1}^k \|\lambda_{I_s,i_s}\|_{L_{2k}(\mathcal{M})}\|\lambda_{J_s,j_s}\|_{L_{2k}(\mathcal{M})}\left|\langle f_{I_s,i_s},f_{J_s,j_s}\rangle \langle e_{J_s,j_s},e_{I_{s+1},i_{s+1}}\rangle\right|\\
			&\leq \sum_{I_1,J_1,\cdots,I_k,J_k\in\mathcal{D}} \sum_{ i_1,j_1,\cdots,i_k,j_k=1}^{d-1} \prod_{s=1}^k \|\lambda_{I_s,i_s}\|_{L_{2k}(\mathcal{M})}\|\lambda_{J_s,j_s}\|_{L_{2k}(\mathcal{M})} \\
			&\quad\quad\quad\cdot\prod_{s=1}^k \bigg\langle \frac{\mathbbm{1}_{I_s}}{|I_s|^{1/2}},\frac{\mathbbm{1}_{J_s}}{|J_s|^{1/2}}\bigg\rangle \bigg\langle \frac{\mathbbm{1}_{J_s}}{|J_s|^{1/2}},\frac{\mathbbm{1}_{I_{s+1}}}{|I_{s+1}|^{1/2}}\bigg\rangle,
		\end{aligned}
	\end{equation}
	where we set $I_{k+1}=I_1$ and $i_{k+1}=i_1$. Now we define
	\begin{equation*}
		B=\sum_{I\in \mathcal{D}}\sum_{i=1}^{d-1}\|\lambda_{I,i}\|_{L_{p}(\mathcal{M})}\cdot \frac{\mathbbm{1}_I}{|I|^{1/2}}\otimes \frac{\mathbbm{1}_I}{|I|^{1/2}}.
	\end{equation*}
	Thus by \eqref{mp2} and Lemma \ref{nonNWOpre1}, one has
	\begin{equation}\label{BlambdaI}
		\|A\|^p_{L_p(\mathcal{N})}\leq \|B\|^p_{S_p(L_2(\mathbb{R}))}\lesssim_{d,p} \sum_{I\in\mathcal{D}}\sum_{i=1}^{d-1}\|\lambda_{I,i}\|^p_{L_p(\mathcal{M})}.
	\end{equation}
		Therefore by using interpolation, for any $1< p<\infty$, we have
		\begin{equation*}
			\begin{aligned}
				\|A\|^p_{L_p(\mathcal{N})}\lesssim_{d,p}	\sum_{I\in\mathcal{D}}\sum_{i=1}^{d-1}\|\lambda_{I,i}\|^p_{L_p(\mathcal{M})}.
			\end{aligned}
		\end{equation*}
		This completes the proof.
	\end{proof}

\begin{proof}[Proof of the Sufficiency of Theorem \ref{thm1.2}]
	For any $I\in\mathcal{D}$ and $1\le i\le d-1$, let 
	\begin{equation*}
		e_{I,i}=\frac{\mathbbm{1}_{I}}{|I|^{1/2}},\quad f_{I,i}=h_I^i,\quad \text{and}\quad \lambda_{I,i}=\frac{\langle h_I^i,b\rangle}{|I|^{1/2}}.
	\end{equation*}
One checks that $\mathrm{supp}e_{I,i}, \,\mathrm{supp}f_{I,i}\subseteq I$ and
$\|e_{I,i}\|_{\infty},\,\|f_{I,i}\|_{\infty}\le |I|^{-\frac{1}{2}}$. Hence from Lemma \ref{nonNWOpre} we have
\begin{equation*}
	\|\pi_b\|^p_{L_p(\mathcal{N})}\lesssim_{d,p}	\sum_{I\in\mathcal{D}}\sum_{i=1}^{d-1}\|\lambda_{I,i}\|^p_{L_p(\mathcal{M})}=\|b\|^p_{\pmb{B}_p^d(\mathbb{R},\M)}.
\end{equation*}
The proof is finished.
\end{proof}

\subsection{The Necessity of Theorem \ref{thm1.2}}\label{Proof of the Sufficiency of theorem 1.2}
We divide the proof into two cases: $p\ge 1$ and $0<p<1$. Each one will be stated and proved in Propositions \ref{Case 4} and \ref{Case 5}, respectively. For the first one, the proof is easier and relies on the following elementary lemma.

\begin{lemma}\label{lem3.1}
	Let $1\le p<\infty$ and $T\in L_p(\mathcal{N})$. $E=(E_{I,i})_{I\in\mathcal{D},1\le i\le d-1}$ is defined as the block diagonal of $T$, where for $I\in\mathcal{D}$ and $1\le i\le d-1$, $E_{I,i}:L_2(\mathcal{M})\to L_2(\mathcal{M})$ is given by
	\[\langle E_{I,i}x,y\rangle=\langle Th_I^i\otimes x,h_I^i\otimes y\rangle,\,\,\,\,\forall x,y\in L_2(\mathcal{M}).\] Then 
	\begin{equation*}
		\|T\|^p_{L_p(\mathcal{N})}\ge \sum_{I\in \mathcal{D}}\sum_{i=1}^{d-1}\|E_{I,i}\|^p_{L_p(\mathcal{M})}.
	\end{equation*}
\end{lemma}
\begin{proof}
	Note that $E$ is a trace preserving conditional expectation, and thereby contractive.
\end{proof}

\begin{proposition}\label{Case 4}
	If $p\ge 1$ and $\pi_b\in L_p(\mathcal{N})$, then $b\in \pmb{B}_p^d(\mathbb{R},\M)$.
\end{proposition}
\begin{proof}
	First, define for any $I\in \mathcal{D}$, $1\le i\le d-1$, $x\in L_2(\mathcal{M})$,
	\begin{equation*}
		\begin{aligned}
			R:L_2(\mathbb{R},L{}&_2(\mathcal{M}))\to L_2(\mathbb{R},L_2(\mathcal{M}))\\&
			h_I^i\otimes x\mapsto h_{\tilde{I}}^i\otimes x,
		\end{aligned}
	\end{equation*}
	where $\tilde{I}$ is the parent interval of $I$. Then $R$ is well-defined and bounded. Indeed, for any $ f\in L_2(\mathbb{R},L_2(\mathcal{M}))$,
	\begin{equation}\label{Rnorm}
		\begin{aligned}
			\|Rf\|_{L_2(\mathbb{R},L_2(\mathcal{M}))}{}&=\biggl\|R\Bigl(\sum_{I\in \mathcal{D}}\sum_{i=1}^{d-1}\langle h_I^i,f\rangle h_I^i\Bigr)\biggr\|_{L_2(\mathbb{R},L_2(\mathcal{M}))}=\biggl\|\sum_{I\in \mathcal{D}}\sum_{i=1}^{d-1}\langle h_I^i,f\rangle h_{\tilde{I}}^i\biggr\|_{L_2(\mathbb{R},L_2(\mathcal{M}))}\\&=\biggl\|\sum_{\tilde{I}\in \mathcal{D}}\sum_{i=1}^{d-1}\Bigl(\sum_{I\subset \tilde{I}\atop |I|=d^{-1}|\tilde{I}|}\langle h_I^i,f\rangle\Bigr) h_{\tilde{I}}^i\biggr\|_{L_2(\mathbb{R},L_2(\mathcal{M}))}\\&=\biggl(\sum_{\tilde{I}\in \mathcal{D}}\sum_{i=1}^{d-1}\Bigl\|\sum_{I\subset \tilde{I}\atop |I|=d^{-1}|\tilde{I}|}\langle h_I^i,f\rangle \Bigr\|^2_{L_2(\mathcal{M})}\biggr)^{1/2}\\
			&\le \biggl(d\sum_{\tilde{I}\in \mathcal{D}}\sum_{i=1}^{d-1}\sum_{I\subset \tilde{I}\atop |I|=d^{-1}|\tilde{I}|}\|\langle h_I^i,f\rangle \|^2_{L_2(\mathcal{M})}\biggr)^{1/2}=\sqrt{d}\|f\|_{L_2(\mathbb{R},L_2(\mathcal{M}))}\,.
		\end{aligned}
	\end{equation}
	Now let $E=(E_{I,i})_{I\in\mathcal{D},1\le i\le d-1}$ be the block diagonal of $\pi_bR$ defined in Lemma \ref{lem3.1}. Then for $ x,y\in L_2(\mathcal{M})$, we have
	\begin{equation*}
		\begin{aligned}
			\langle E_{I,i}x,y\rangle{}&=\langle\pi_bR(h_I^i\otimes x),h_I^i\otimes y\rangle=\langle\pi_b(h_{\tilde{I}}^i\otimes x),h_I^i\otimes y\rangle\\&=\frac{\omega^{-iq}}{|\tilde{I}|^{1/2}}\langle \langle h_I^i,b\rangle x,y\rangle=\frac{\omega^{-iq}}{\sqrt{d}|I|^{1/2}}\langle \langle h_I^i,b\rangle x,y\rangle,
		\end{aligned}
	\end{equation*}
	where $\tilde{I}(q)=I, 1\le q\le d$. Thus $E_{I,i}=\frac{\omega^{-iq}}{\sqrt{d}|I|^{1/2}} \langle h_I^i,b\rangle$. Therefore, from \eqref{Rnorm} and Lemma \ref{lem3.1}, we get
	\begin{equation*}
		\begin{aligned}
			\|\pi_b\|^p_{L_p(\mathcal{N})}\ge \frac{1}{d^{p/2}} \|\pi_bR\|^p_{L_p(\mathcal{N})}\ge \frac{1}{d^{p/2}} \sum_{I\in \mathcal{D}}\sum_{i=1}^{d-1}\|E_{I,i}\|^p_{L_p(\mathcal{M}) )}{}&= \frac{1}{d^{p}}\sum_{I\in \mathcal{D}}\sum_{i=1}^{d-1}\biggl(\frac{\|\langle h_I^i,b\rangle\|_{L_p(\mathcal{M}) )}}{|I|^{1/2}}\biggr)^p\\&\gtrsim_{d,p} \|b\|^p_{\pmb{B}^d_p(\mathbb{R},\M)}.
		\end{aligned}
	\end{equation*}
	This yields the desired result.
\end{proof}

\begin{proposition}\label{Case 5}
	If $0<p<1$ and $\pi_b\in L_p(\mathcal{N})$, then $b\in \pmb{B}_p^d(\mathbb{R},\M)$.
\end{proposition}

The remaining part of this subsection is devoted to the proof of Proposition \ref{Case 5}. We will follow the arguments in \cite{P} or \cite{PS}. To this end, define for any $m,n \in \mathbb{Z}$
\begin{equation}\label{pibnm}
	\pi_b^{n,m}=d_{m+1}\pi_bd_{n+1}.
\end{equation}
Recall that $d_kb=\sum\limits_{|I|=d^{-k+1}}\sum\limits_{i=1}^{d-1}h_I^i\langle h_I^i,b\rangle$. Thus for any $ f\in L_2(\mathbb{R},L_2(\mathcal{M}))$,
\begin{equation*}
	\pi_b^{n,m}(f)=\sum_{I\in \mathcal{D}_m\atop J\in \mathcal{D}_n}\sum_{i,j=1}^{d-1}\biggl\langle \frac{\mathbbm{1}_I}{|I|},h_J^j\biggr\rangle h_I^i \langle h_I^i,b\rangle \langle h_J^j,f\rangle.
\end{equation*}
If $m\le n$, then for $I\in \mathcal{D}_m$ and $ J\in \mathcal{D}_n$,
$$ \biggl\langle \frac{\mathbbm{1}_I}{|I|},h_J^j\biggr\rangle=0.  $$
It thus follows that $\pi_b^{n,m}=0$.

Let $b\in \pmb{B}^d_p(\mathbb{R},\M)$ and $N\ge 2$ be a fixed positive integer (to be chosen later). For $k=0,1,\cdots, N-1$, define
\begin{equation}\label{pibk}
	\pi_{b,k}=\sum_{n=-\infty}^{\infty}\sum_{m=-\infty}^{\infty}\pi_b^{Nn+k,Nm+k+1}=\sum_{m=-\infty}^{\infty}\sum_{n=-\infty}^{m}\pi_b^{Nn+k,Nm+k+1}.
\end{equation}
In addition, we define
\begin{equation*}\label{pibk0}
	\pi_{b,k}^{(0)}=\sum_{n=-\infty}^{\infty}\pi_b^{Nn+k,Nn+k+1},\quad\quad \pi_{b,k}^{(1)}=\sum_{m=-\infty}^{\infty}\sum_{n=-\infty}^{m-1}\pi_b^{Nn+k,Nm+k+1}.
\end{equation*}
Then
\begin{equation*}\label{pibksum}
	\pi_{b,k}=\pi_{b,k}^{(0)}+\pi_{b,k}^{(1)}.
\end{equation*}
$\pi_{b, k}^{(0)}$ is defined as the minor diagonal and will play an important role later. In the following, we are about to obtain the lower bound of $\|\pi_b\|_{L_p(\mathcal{N})}$ by $\|\pi_{b, k}^{(0)}\|_{L_p(\mathcal{N})}$, which will be the dominant term. The following lemma implies that $\|\pi_{b,k}^{(1)}\|_{L_p(\mathcal{N})}$ is the minor term  since $	\|\pi_b^{n,m}\|_{L_p(\mathcal{N})}$ shrinks rapidly when $m>n$.

\begin{lemma}\label{lem3.2}
	Let $b\in\pmb{B}^d_p(\mathbb{R},\M)$. If $m>n$ and $0<p<1$, then
	\begin{equation*}
		\|\pi_b^{n,m}\|^p_{L_p(\mathcal{N})}\le (d-1)d^{(n-m)p/2}\sum_{I\in \mathcal{D}_m}\sum_{i=1}^{d-1}\biggl(\frac{\|\langle h_I^i,b\rangle\|_{L_p(\mathcal{M})}}{|I|^{1/2}}\biggr)^p.
	\end{equation*}
\end{lemma}
\begin{proof}
	Write $\pi_b^{n,m}$ in the following concrete form:
	\begin{equation}\label{pibnmsum}
		\begin{aligned}
			\pi_b^{n,m}{}&=\sum_{I\in \mathcal{D}_m\atop J\in \mathcal{D}_n}\sum_{i,j=1}^{d-1}\biggl\langle \frac{\mathbbm{1}_I}{|I|},h_J^j\biggr\rangle(\pi_b^{n,m})_{I,J}^{i,j}\\&=\sum_{J\in \mathcal{D}_n}\sum_{q=1}^d\sum_{I\in \mathcal{D}_m(J(q))}\sum_{i,j=1}^{d-1}\omega^{qj}\frac{1}{|J|^{1/2}}(\pi_b^{n,m})_{I,J}^{i,j},
		\end{aligned}
	\end{equation}
	where 
	\begin{equation*}
		(\pi_b^{n,m})_{I,J}^{i,j}=\langle h_I^i,b\rangle\cdot h_I^i\otimes h_J^j.
	\end{equation*}
	Hence
	\begin{equation*}
		\begin{aligned}
			\|\pi_b^{n,m}\|^p_{L_p(\mathcal{N})}{}&\le \sum_{J\in \mathcal{D}_n}\sum_{I\in \mathcal{D}_m(J)}\sum_{i,j=1}^{d-1}\biggl(\frac{\|(\pi_b^{n,m})_{I,J}^{i,j}\|_{L_p(\mathcal{N})}}{|J|^{1/2}}\biggr)^p\\&=\sum_{J\in \mathcal{D}_n}\sum_{I\in \mathcal{D}_m(J)}\sum_{i,j=1}^{d-1}\biggl(\frac{\|h_I^i\otimes h_J^j\|_{S_p(L_2(\mathbb{R}))}\|\langle h_I^i,b\rangle\|_{L_p(\mathcal{M})}}{|J|^{1/2}}\biggr)^p\\&=\sum_{J\in \mathcal{D}_n}\sum_{I\in \mathcal{D}_m(J)}\sum_{i,j=1}^{d-1}\biggl(\frac{\|\langle h_I^i,b\rangle\|_{L_p(\mathcal{M})}}{|J|^{1/2}}\biggr)^p\\&=(d-1)d^{(n-m)p/2}\sum_{I\in \mathcal{D}_m}\sum_{i=1}^{d-1}\biggl(\frac{\|\langle h_I^i,b\rangle\|_{L_p(\mathcal{M})}}{|I|^{1/2}}\biggr)^p.
		\end{aligned}
	\end{equation*}
	This finishes the proof.
\end{proof}

By virtue of Lemma \ref{lem3.2}, we estimate $\|\pi_{b,k}^{(1)}\|_{L_p(\mathcal{N})}$ as follows:
\begin{lemma}\label{cor3.12}
		We have $$\sum\limits_{k=0}^{N-1}\|\pi^{(1)}_{b,k}\|^p_{{L_p(\mathcal{N})}}\leq \frac{(d-1)d^{-p/2}}{d^{Np/2}-1}\|b\|^p_{\pmb{B}^d_p(\mathbb{R},\M)}.$$
	\end{lemma}
\begin{proof}
	By Lemma \ref{lem3.2}, one has
	\[\begin{aligned}
		\|\pi_{b,k}^{(1)}\|^p_{L_p(\mathcal{N})}{}&\le \sum_{m=-\infty}^{\infty}\sum_{n=-\infty}^{m-1} \|\pi_b^{Nn+k,Nm+k+1}\|^p_{L_p(\mathcal{N})}\\&\le\sum_{m=-\infty}^{\infty}\sum_{n=-\infty}^{m-1}(d-1)d^{{(Nn-Nm-1)p}/{2}}\sum_{I\in \mathcal{D}_{Nm+k+1}}\sum_{i=1}^{d-1}\biggl(\frac{\|\langle h_I^i,b\rangle\|_{L_p(\mathcal{M})}}{|I|^{1/2}}\biggr)^p\\&=\sum_{m=-\infty}^{\infty}(d-1)d^{-(Nm+1)p/2}\sum_{I\in \mathcal{D}_{Nm+k+1}}\sum_{i=1}^{d-1}\biggl(\frac{\|\langle h_I^i,b\rangle\|_{L_p(\mathcal{M})}}{|I|^{1/2}}\biggr)^p\sum_{n=-\infty}^{m-1}d^{Nnp/2}\\&=\frac{(d-1)d^{-p/2}}{d^{Np/2}-1}\sum_{m=-\infty}^{\infty}\sum_{I\in \mathcal{D}_{Nm+k+1}}\sum_{i=1}^{d-1}\biggl(\frac{\|\langle h_I^i,b\rangle\|_{L_p(\mathcal{M})}}{|I|^{1/2}}\biggr)^p.
	\end{aligned}\]
	We then deduce
	\begin{equation*}
		\begin{aligned}
			\sum_{k=0}^{N-1}\|\pi^{(1)}_{b,k}\|^p_{{L_p(\mathcal{N})}}&\le \frac{(d-1)d^{-p/2}}{d^{Np/2}-1}\sum_{I\in \mathcal{D}}\sum_{i=1}^{d-1}\biggl(\frac{\|\langle h_I^i,b\rangle\|_{L_p(\mathcal{M})}}{|I|^{1/2}}\biggr)^p
			= \frac{(d-1)d^{-p/2}}{d^{Np/2}-1}\|b\|^p_{\pmb{B}^d_p(\mathbb{R},\M)}.
		\end{aligned}
	\end{equation*}
\end{proof}

Now we come to the estimate of $\|\pi^{(0)}_{b,k}\|^p_{{L_p(\mathcal{N})}}$. The following well-known lemma is straightforward but very helpful for us.
\begin{lemma}\label{lem3.3}
	Let $0<p<\infty$. If $\{R_i\}_{1\le i\le n}$  are operators in $L_p(\mathcal{N})$ satisfying $R_i^*R_j=0,\,\,\,\,\forall 1\le i,j\le n, i\neq j$ and $T=\sum\limits_{i=1}^nR_i$, then
	\begin{equation}\label{ine3.13}
		\|T\|^p_{{L_p(\mathcal{N})}}\ge \frac{1}{n}\sum_{i=1}^n\|R_i\|^p_{{L_p(\mathcal{N})}}.
	\end{equation}
\end{lemma}
\begin{proof}
	This follows from the fact that
	$$ T^*T=\sum_{i=1}^{n}R_i^*R_i\geq R_i^*R_i.  $$
\end{proof}

\begin{remark}
	It is obvious that our estimate is far from being optimal in (\ref{ine3.13}), but it does not affect our later proof. See \cite[Theorem 1.3]{JXu} or \cite[Lemma 2.1]{CRX} for better constants in (\ref{ine3.13}).
\end{remark}
\begin{lemma}\label{lem3.4}
	Let $b\in\pmb{B}^d_p(\mathbb{R},\M)$ and $0<p<1$. Then
	\[\sum_{k=0}^{N-1}\|\pi_{b,k}^{(0)}\|_{{L_p(\mathcal{N})}}\ge \frac{(d-1)^{p/2-1}}{d^{p/2+1}}\|b\|^p_{\pmb{B}^d_p(\mathbb{R},\M)}.\]
\end{lemma}
\begin{proof}
	From \eqref{pibnmsum} we deduce
	\begin{equation*}
		\begin{aligned}
			\pi_b^{Nn+k,Nn+k+1}{}&=\sum_{J\in \mathcal{D}_{Nn+k}}\sum_{q=1}^{d}\sum_{I\in \mathcal{D}_{Nn+k+1}(J(q))}\sum_{i,j=1}^{d-1}\frac{\omega^{qj}}{|J|^{1/2}}(\pi_b^{Nn+k,Nn+k+1})_{I,J}^{i,j}\\&=\sum_{J\in \mathcal{D}_{Nn+k}}\sum_{q=1}^d\sum_{i,j=1}^{d-1}\frac{\omega^{qj}}{|J|^{1/2}}(\pi_b^{Nn+k,Nn+k+1})_{J(q),J}^{i,j}.
		\end{aligned}
	\end{equation*}
	Then 
	\begin{equation*}
		\begin{aligned}
			\pi^{(0)}_{b,k}{}&=\sum_{n=-\infty}^{\infty}\sum_{J\in \mathcal{D}_{Nn+k}}\sum_{q=1}^d\sum_{i,j=1}^{d-1}\frac{\omega^{qj}}{|J|^{1/2}}(\pi_b^{Nn+k,Nn+k+1})_{J(q),J}^{i,j}\\
			&=\sum_{q=1}^d\sum_{i=1}^{d-1} \sum_{n=-\infty}^{\infty}\sum_{J\in \mathcal{D}_{Nn+k}}\sum_{j=1}^{d-1}\frac{\omega^{qj}}{|J|^{1/2}}(\pi_b^{Nn+k,Nn+k+1})_{J(q),J}^{i,j}\\
			&=:\sum_{q=1}^d\sum_{i=1}^{d-1}A_{q,i}.
		\end{aligned}
	\end{equation*}
	Since the ranges of $\{(\pi_b^{Nn+k,Nn+k+1})_{J(q),J}^{i,j}\}_{1\leq q\leq d, 1\leq i\leq d-1}$ are mutually orthogonal, by Lemma \ref{lem3.3} we have 
	\begin{equation*}\label{3.4.1}
		\begin{aligned}
			\|\pi^{(0)}_{b,k}\|^p_{{L_p(\mathcal{N})}}\ge \frac{1}{d(d-1)}\sum_{q=1}^d\sum_{i=1}^{d-1}\| A_{q,i}\|^p_{L_p(\mathcal{N})}.
		\end{aligned}
	\end{equation*}
	When $q$ and $i$ are fixed, the operator $A_{q, i}$ is a block diagonal matrix with respect to the basis $\{h_J^j,h_{J(q)}^i\}_{J\in\mathcal{D}_{Nn+k}}$. Consequently, one has
	\begin{equation*}\label{3.4.2}
		\begin{aligned}
			{}&\biggl\|\sum_{n=-\infty}^{\infty}\sum_{J\in \mathcal{D}_{Nn+k}}\sum_{j=1}^{d-1}\frac{\omega^{qj}}{|J|^{1/2}}(\pi_b^{Nn+k,Nn+k+1})_{J(q),J}^{i,j}\biggr\|^p_{L_p(\mathcal{N})}\\= &\sum_{n=-\infty}^{\infty}\sum_{J\in \mathcal{D}_{Nn+k}}\biggl\|\sum_{j=1}^{d-1}\frac{\omega^{qj}}{|J|^{1/2}}(\pi_b^{Nn+k,Nn+k+1})_{J(q),J}^{i,j}\biggr\|^p_{L_p(\mathcal{N})}\\=&\sum_{n=-\infty}^{\infty}\sum_{J\in \mathcal{D}_{Nn+k}}\biggl\|\sum_{j=1}^{d-1}\frac{\omega^{qj}}{|J|^{1/2}}h_{J(q)}^i\otimes h_J^j\biggr\|_{S_p(L_2(\mathbb{R}))}^p\|\langle h_{J(q)}^i,b\rangle\|_{L_p(\mathcal{M})}^p.
		\end{aligned}
	\end{equation*}
	It is clear that 
	\begin{equation*}\label{3.4.3}
		\begin{aligned}
			\biggl\|\sum_{j=1}^{d-1}\frac{\omega^{qj}}{|J|^{1/2}}h_{J(q)}^i\otimes h_J^j\biggr\|_{S_p(L_2(\mathbb{R}))}=\frac{(d-1)^{1/2}}{|J|^{1/2}}.
		\end{aligned}
	\end{equation*}
	Combining the preceding inequalities, we obtain
	\begin{equation*}
		\begin{aligned}
			\|\pi^{(0)}_{b,k}\|^p_{{L_p(\mathcal{N})}}{}&\ge \frac{(d-1)^{p/2}}{d(d-1)}\sum_{q=1}^d\sum_{i=1}^{d-1}\sum_{n=-\infty}^{\infty}\sum_{J\in \mathcal{D}_{Nn+k}}\frac{1}{|J|^{p/2}}\|\langle h_{J(q)}^i,b\rangle\|_{L_p(\mathcal{M})}^p\\&= \frac{(d-1)^{p/2-1}}{d^{p/2+1}}\sum_{i=1}^{d-1}\sum_{n=-\infty}^{\infty}\sum_{J\in \mathcal{D}_{Nn+k+1}}\frac{1}{|J|^{p/2}}\|\langle h_{J}^i,b\rangle\|_{L_p(\mathcal{M})}^p.
		\end{aligned}	
	\end{equation*}
	Hence
	\begin{equation*}
		\begin{aligned}
			\sum_{k=0}^{N-1}\|\pi^{(0)}_{b,k}\|^p_{{L_p(\mathcal{N})}}{}&\ge \frac{(d-1)^{p/2-1}}{d^{p/2+1}}\sum_{J\in \mathcal{D}}\sum_{i=1}^{d-1}\biggl(\frac{\|\langle h_J^i,b\rangle\|_{L_p(\mathcal{M})}}{|J|^{1/2}}\biggr)^p
			=\frac{(d-1)^{p/2-1}}{d^{p/2+1}}\|b\|^p_{\pmb{B}^d_p(\mathbb{R},\M)}.
		\end{aligned}	
	\end{equation*}
\end{proof}

\begin{proposition}\label{pro3.4}
	Let $b\in\pmb{B}^d_p(\mathbb{R},\M)$ and $0<p<1$. Then $$ \|b\|_{\pmb{B}^d_p(\mathbb{R},\M)}\lesssim_{d,p}\|\pi_b\|_{L_p(\mathcal{N})}.$$
\end{proposition}
\begin{proof}
	From \eqref{pibnm} and \eqref{pibk} we observe that
	\begin{equation*}
		\pi_{b,k}=\biggl(\sum_{m=-\infty}^{\infty}d_{Nm+k+2}\biggr)\pi_b\biggl(\sum_{n=-\infty}^{\infty}d_{Nn+k+1}\biggr).
	\end{equation*}
	Note that $\sum\limits_{n=-\infty}^{\infty}d_{Nn+k+1}$ and $\sum\limits_{m=-\infty}^{\infty}d_{Nm+k+2}$ are projections with norm $1$. Thus
	\begin{equation*}
		\|\pi_{b,k}\|_{L_p(\mathcal{N})}\le \|\pi_{b}\|_{L_p(\mathcal{N})}.
	\end{equation*}
	By Lemmas \ref{cor3.12} and \ref{lem3.4}, we have
	\begin{equation*}
		\begin{aligned}
			\|\pi_{b}\|^p_{L_p(\mathcal{N})}\ge \frac{1}{N}\sum_{k=0}^{N-1}\|\pi_{b,k}\|^p_{L_p(\mathcal{N})}{}&\ge \frac{1}{N}\sum_{k=0}^{N-1}\big(\|\pi^{(0)}_{b,k}\|^p_{L_p(\mathcal{N})}-\|\pi^{(1)}_{b,k}\|^p_{L_p(\mathcal{N})}\big)\\&\ge \frac{1}{N}\biggl(\frac{(d-1)^{p/2-1}}{d^{p/2+1}}-\frac{(d-1)d^{-p/2}}{d^{Np/2}-1}\biggr)\|b\|^p_{\pmb{B}^d_p(\mathbb{R},\M)},
		\end{aligned}
	\end{equation*}	
	which yields the desired result as long as we choose $N$ sufficiently large.
\end{proof}

Now we give the proof of Proposition \ref{Case 5}.
\begin{proof}[Proof of Proposition \ref{Case 5}]
	When $0<p<1$, Proposition \ref{Case 5} follows from Proposition \ref{pro3.4} and the standard limit argument. Indeed, for any positive integer $a$, we define
	\begin{equation*}
		b^{(a)}=\sum_{I\in \mathcal{D}^{(a)}}\sum_{i=1}^d\langle h_I^i,b\rangle h_I^i,
	\end{equation*}
	where
	\[\mathcal{D}^{(a)}=\{I_{n,k}\in\mathcal{D}:|n|\le a, |k|\le a\}.\]
	To implement the limit argument, we need to show $b^{(a)}\in {\pmb{B}^d_p(\mathbb{R},\mathcal{M})}$. For any $I\in \mathcal{ D}$ and $1\leq i\leq d-1$, define 
	\begin{equation}\label{pibii}
		\pi_b^{I,i}=\langle h_I^i,b\rangle\cdot B^{I,i},
	\end{equation}
	where $B^{I,i}\in B(L_2(\mathbb{R}))$ is defined by
	\begin{equation}\label{Bii}
		B^{I,i}=h_I^i\otimes\frac{\mathbbm{1}_I}{|I|}.
	\end{equation}	
	Then we have
	\begin{equation}\label{pibsum}
		\pi_b=\sum_{I\in \mathcal{D}}\sum_{i=1}^{d-1}\pi_b^{I,i}.
	\end{equation}	
	If $I\neq J$ or $i\neq j$, then $\forall g, h\in L_2(\mathbb{R})$,
	\begin{equation*}
		\begin{aligned}
			\langle (B^{I,i})^*B^{J,j}(g), h\rangle{}&=\langle B^{J,j}(g), B^{I,i}(h)\rangle=\overline{\biggl\langle \frac{\mathbbm{1}_J}{|J|},g\biggr\rangle}\biggl\langle \frac{\mathbbm{1}_I}{|I|},h\biggr\rangle\langle h_J^j,h_I^i\rangle=0,
		\end{aligned}
	\end{equation*}	
	which implies that  
	\begin{equation}\label{pibiijj}
		(\pi_b^{I,i})^*(\pi_b^{J,j})=0.
	\end{equation}
	So from \eqref{pibsum} we get
	\begin{equation}\label{pibpib}
		\pi_b^*\pi_b=\sum_{I\in\mathcal{ D}}\sum_{i=1}^{d-1}(\pi_b^{I,i})^*(\pi_b^{I,i}).
	\end{equation}	 
	By (\ref{pibpib}), one has for any $0<p<\8$
	$$  	\|\pi_{b}\|_{L_p(\mathcal{N})} \geq 	\|\pi_{b}^{I, i}\|_{L_p(\mathcal{N})}=\frac{\|\langle h^i_I,b\rangle\|_{L_p(\mathcal{M})}}{|I|^{1/2}}. $$
	This implies that $b^{(a)}\in {\pmb{B}^d_p(\mathbb{R},\mathcal{M})}$. Therefore by Proposition \ref{pro3.4},
	\[\|b\|_{\pmb{B}^d_p(\mathbb{R},\mathcal{M})}=\lim_{a\to \infty}\|b^{(a)}\|_{\pmb{B}^d_p(\mathbb{R},\mathcal{M})}\lesssim_{d,p} \|\pi_b\|_{L_p(\mathcal{N})}.\]
	This finishes the proof.
\end{proof}

\begin{proof}[Proof of the Necessity of Theorem \ref{thm1.2}]
	The desired result follows from Proposition \ref{Case 4} for $p\geq 1$, and from Proposition \ref{Case 5} for $0<p<1$.
\end{proof}

\bigskip

\section{Proof of Theorem \ref{thm6.1}}\label{Application 1:the CAR case}
	First recall the Walsh system. Let $\mathcal{G}=\{1,-1\}^{\mathbb{N}}$ be equipped with the uniform distribution $P$. Recall that for any $n\geq 1$, $\varepsilon_n((\theta_k)_{k\in \mathbb{N}}):= \theta_n $, $\forall \theta=(\theta_k)_{k\in \mathbb{N}}\in \mathcal{G}$. Then $(\varepsilon_n)_{n\geq 1}$ is the Rademacher sequence on $\mathcal{G}$, namely a sequence of independent identically distributed random variables on $(\mathcal{G},P)$ such that $P(\varepsilon_n=1)=P(\varepsilon_n=-1)=1/2$ for all $n\in\mathbb{N}$. 
	
	Recall that $\mathcal{I}$ denotes the family of all finite subsets of $\mathbb{N}$. For a nonempty set $A\in\mathcal{I}$, we write $A=\{k_1<k_2<\cdots<k_n\}$ in an increasing order. Define
	\begin{equation*}
		\omega_A=\varepsilon_{k_1}\varepsilon_{k_2}\cdots\varepsilon_{k_n}.
	\end{equation*}
	If $A=\emptyset$, we set $\varepsilon_A=1$. If $A$ is a singleton $\{k\}$, we still use $\omega_k$ instead of $\omega_{\{k\}}$. Thus $(\omega_A)_{A\in \mathcal{I}}$, called the Walsh system, is an orthonormal basis of $L_2(\mathcal{G})$. Denote by $\mathcal{G}_n$ the $\sigma$-algebra generated by $\{\omega_A : {\max (A)\le n}\}$. Then $(\mathcal{G}_n)_{n\geq 1}$ is the filtration of $(\mathcal{G}, P)$ for the Walsh system.
	
	We define for any $\theta\in \mathcal{G}$,
	\begin{equation*}
		\begin{aligned}
			\sigma_\theta:{}&\mathcal{C}\to \mathcal{C}\\&
			c_i\mapsto \varepsilon_i(\theta)c_i,\quad \forall i\in\mathbb{N}.
		\end{aligned}
	\end{equation*}
	Then $\sigma_\theta$ extends to a trace preserving automorphism of the CAR algebra $\mathcal{C}$, and consequently extends to an isometry on $L_p(\mathcal{C})$ for all $0< p< \infty$. By virtue of $	\sigma_\theta$, the CAR algebra can be transfered to the operator-valued Walsh system. For any given $b=\sum\limits_{A\in\mathcal{I}}\hat{b}(A)c_A\in L_p(\mathcal{C})$ with $0<p<\8$, we define 
	\[\tilde b(\theta)=\sigma_\theta(b)=\sum_{A\in\mathcal{I}}\hat{b}(A)c_A\cdot\omega_A(\theta).\]
	Then $\tilde{b}\in L_p(\mathcal{G},L_p(\mathcal{C}))$. Hence, for any given $b$,  define the martingale paraproduct $\pi_{\tilde{b}}$ of symbol $\tilde{b}$ associated with the Walsh system  on $L_2(\mathcal{G},L_2(\mathcal{C}))$ by
	\begin{equation*}
		\begin{aligned}
			\pi_{\tilde{b}}:L_2(\mathcal{G},L_2(\mathcal{C}){}&)\to L_2(\mathcal{G},L_2(\mathcal{C}))\\&
			g\mapsto \sum_{k=1}^\infty d_k\tilde{b}\cdot g_{k-1}.
		\end{aligned}
	\end{equation*}
	In fact, $	\pi_{\tilde{b}}$ is a martingale paraproduct for semicommutative dyadic martingales.
	
	Now we come to the proof of Theorem \ref{thm6.1}.
	\begin{proof}[Proof of Theorem \ref{thm6.1}]
		Since $L_2(\mathcal{C})\cong \ell_2(\mathcal{I})$ and $L_2(\mathcal{G},L_2(\mathcal{C}))\cong \ell_2(\mathcal{I}, L_2(\mathcal{C}))$, we represent $\pi_b$ and $\pi_{\tilde{b}}$ in the matrix form. For any $A,B\in\mathcal{I}$, note that for $k\geq 1$
		\begin{equation*}
			(c_B)_{k-1}=\sum_{\max(D)\le k-1}\tau(c_D^*\cdot c_B)c_D=
			\begin{cases}
				c_B,& \text{if $k-1\ge \max(B)$};\\
				0,& \text{otherwise}.
			\end{cases}
		\end{equation*}
		Then
		\begin{equation*}
			\begin{aligned}
				\langle c_A,\pi_b (c_B)\rangle{}&=\langle c_A,\sum_{k=1}^\infty d_kb\cdot(c_B)_{k-1}\rangle=\langle c_A,\sum_{k-1\ge \max(B)}d_kb\cdot c_B\rangle\\&=\langle c_A,\sum_{\max (E)\ge \max(B)+1}\hat{b}(E)c_E\cdot c_B\rangle\\&=\langle c_Ac_B^*,\sum_{\max (E)\ge \max(B)+1}\hat{b}(E)c_E\rangle.
			\end{aligned}
		\end{equation*}
		From the CAR \eqref{equa4.1}, we have 
		\[c_A^*=\pm c_A\,\,\,\,\mathrm{and}\,\,\,\,c_Ac_B=\pm c_{A\Delta B},\,\,\,\, \forall A,B\in\mathcal{I},\]
		where $A\Delta B=(A\cup B)\backslash(A\cap B)$. Then
		\begin{equation*}
			\begin{aligned}
				\langle c_A,\pi_b (c_B)\rangle{}&=\langle \pm c_{A\Delta B},\sum_{\max (E)\ge \max{B}+1}\hat{b}(E)c_E\rangle\\&=
				\begin{cases}
					\pm \hat{b}(A\Delta B),& \text{if $\max(A\Delta B)>\max (B)$};\\
					0,& \text{if $\max(A\Delta B)\le\max (B)$},
				\end{cases}\\&=
				\begin{cases}
					\hat{b}(A\Delta B),& \text{if $\max (A)>\max (B)$ and $c_Ac_B^*=c_{A\Delta B}$};\\
					- \hat{b}(A\Delta B),& \text{if $\max (A)>\max (B)$ and $c_Ac_B^*=-c_{A\Delta B}$};\\
					0,& \text{if $\max (A)\le\max (B)$}.
				\end{cases}
			\end{aligned}
		\end{equation*}
		In the same way, one has
		\begin{equation*}
			\begin{aligned}
				\langle \omega_A,\pi_{\tilde{b}} (\omega_B)\rangle=
				\begin{cases}
					\hat{b}(A\Delta B)c_{A\Delta B},& \text{if $\max (A)>\max (B)$};\\
					0,& \text{if $\max (A)\le\max (B)$}.
				\end{cases}
			\end{aligned}
		\end{equation*}
		Denote by
		\[[\pi_b]=\Bigl((\pi_b)_{A,B}\Bigr)_{A,B\in\mathcal{I}}\]
		the matrix form of $\pi_b$ with respect to the basis $(c_A)_{A\in\mathcal{I}}$, where $(\pi_b)_{A,B}=\langle c_A,\pi_b (c_B)\rangle$. Analogously, let \[[\pi_{\tilde{b}}]=\Bigl((\pi_{\tilde{b}})_{A,B}\Bigr)_{A,B\in\mathcal{I}}\]
		be the matrix form of $\pi_{\tilde{b}}$ with respect to the basis $(\omega_A)_{A\in\mathcal{I}}$, where
		$(\pi_{\tilde{b}})_{A,B}=\langle \omega_A,\pi_{\tilde{b}} (\omega_B)\rangle$. By the above discussion, we see that
		\begin{equation*}
			\begin{aligned}
				[\pi_{\tilde{b}}]{}&=\biggl((\pi_b)_{A,B}c_Ac_B^*\biggr)_{A,B\in\mathcal{I}}\\&=
				\begin{pmatrix}
					\ddots & \,  &   0 \\
					\,     & c_A &   \,\\
					0      & \,  &   \ddots
				\end{pmatrix}_{A\in\mathcal{I}}([\pi_b]\otimes 1_{\mathcal{C}})
				\begin{pmatrix}
					\ddots & \,  &   0 \\
					\,     & c_B^* &   \,\\
					0      & \,  &   \ddots
				\end{pmatrix}_{B\in\mathcal{I}},
			\end{aligned}
		\end{equation*}
		where $1_{\mathcal{C}}$ is the identity of $\mathcal{C}$. So for any $0<p<\8$ this leads to 
		\begin{equation*}
			\begin{aligned}
				\|\pi_{\tilde{b}}\|_{L_p(B(L_2(\mathcal{G}))\otimes \mathcal{C})}{}&=\|[\pi_{\tilde{b}}]\|_{L_p(B(\ell_2(\mathcal{I}))\otimes \mathcal{C})}= \|[\pi_{b}]\|_{S_p(\ell_2(\mathcal{I}))}\cdot \| 1_\mathcal{C}\|_{L_p(\mathcal{C})}
				=\|\pi_{b}\|_{S_p(L_2(\mathcal{C}))}.
			\end{aligned}
		\end{equation*}
		By Theorem \ref{thm1.2}, we have $\pi_{\tilde{b}}\in {L_p(B(L_2(\mathcal{G}))\otimes \mathcal{C})}$ if and only if $\tilde{b}\in \pmb{B}_p^2(\mathbb{R},\mathcal{C})$, where 
		\begin{equation*}
			\|\tilde{b}\|_{\pmb{B}_p^2(\mathbb{R},\mathcal{C})}\approx_p\biggl(\sum_{k=1}^\infty2^{k}\|d_k\tilde{b}\|^p_{L_p(\mathcal{G},L_p(\mathcal{C}))}\biggr)^{1/p}. 
		\end{equation*}
		However, note that for any $\theta\in\mathcal{G}$ and $k\geq 1$,
		\begin{equation*}
			\begin{aligned}
				(d_k \tilde{b})(\theta)&=\big(\sum_{\max(A)=k}\hat{b}(A)c_A\omega_A\big)(\theta)=\sigma_\theta\big(\sum_{\max(A)=k}\hat{b}(A)c_A\big)=\sigma_\theta(d_k b),
			\end{aligned}
		\end{equation*}
		which yields
		\[\|d_k\tilde{b}\|^p_{L_p(\mathcal{G},L_p(\mathcal{C}))}=\int_{\mathcal{G}}\|(d_k \tilde{b})(\theta)\|_{L_p((\mathcal{C})}^p dP(\theta) =\int_{\mathcal{G}}\|\sigma_\theta(d_kb)\|_{L_p((\mathcal{C})}^p dP(\theta) =\|d_kb\|^p_{L_p(\mathcal{C})}.\]
		Therefore, we get $\pi_b\in S_p(L_2(\mathcal{C}))$ if and only if $b\in \pmb{B}_p(\mathcal{C})$ with relevant constants depending only on $p$. Thus Theorem \ref{thm6.1} is proved.
	\end{proof}

\bigskip

\section{Proof of Theorem \ref{thm7.1}}\label{Application 2}

First we construct an orthonormal basis of $\mathbb{M}_d$, which will induce an orthonormal basis in $\mathscr{M}=\mathop{\otimes}\limits_{k=1}^{\infty}\mathbb{M}_{d}$.  Let 
$\sigma=(1\,2\,\cdots\,d)$ be the $d$-cycle, and recall $\omega=e^{2\pi \mathrm{i}/d}$. Define
\[\Omega=\biggl\{U_{(i,j)}=\sum_{l=1}^d \omega^{i\cdot l}e_{l,\sigma^{j}(l)}:\quad 1\le i,j\le d\biggr\}.\]
Then $\Omega$ is an orthonormal basis of $L_2(\mathbb{M}_d,\mathrm{tr}_d)$, and every element of $\Omega$ is unitary. In particular, $U_{(d,d)}=1$. Moreover, such matrices $U_{(i,j)}$ satisfy the following properties.

\begin{lemma}\label{uuu} 
For any $ 1\le i,j, k, l\le d$, we have
\begin{equation}
	\begin{cases}
		U_{(i, j)}^*=\omega^{i\cdot j}U_{(\overline{-i}, \overline{-j})},\\
		U_{(i, j)}U_{(k, l)}=\omega^{j\cdot k}U_{(\overline{i+k}, \overline{j+l})},
	\end{cases}
\end{equation}
where $\overline{i}$ and $\overline{i+j}$ are the remainder in $[1, d]$ modulo $d$.
\end{lemma}
\begin{proof}
	For any $1\le i,j\le d$, we calculate
	\begin{equation*}
		\begin{aligned}
			U_{(i, j)}^*=\sum_{l=1}^d \omega^{-i\cdot l}e_{\overline{j+l},l}=\sum_{l=1}^d \omega^{-i\cdot (l-j)}e_{l,\overline{l-j}}=\omega^{i\cdot j}\sum_{l=1}^d \omega^{-i\cdot l}e_{l,\sigma^{-j}(l)}=\omega^{i\cdot j}U_{(\overline{-i}, \overline{-j})}.
		\end{aligned}
	\end{equation*}
Besides, for any $ 1\le i,j, k, l\le d$, one has
\begin{equation*}
	\begin{aligned}
		U_{(i, j)}U_{(k, l)}=\sum_{s=1}^d \omega^{i\cdot s}e_{s,\overline{j+s}}\cdot \sum_{t=1}^d \omega^{k\cdot t}e_{t,\overline{t+l}}=\sum_{s=1}^d \omega^{i\cdot s}\omega^{k\cdot (j+s)}e_{s,\overline{j+s+l}}{}&=\omega^{j\cdot k}\sum_{s=1}^d \omega^{(i+k)\cdot s}e_{s,\sigma^{j+l}(s)}\\
		&=\omega^{j\cdot k}U_{(\overline{i+k}, \overline{j+l})},
	\end{aligned}
\end{equation*}
as desired.
\end{proof}

Denote $\mathscr{A}=\{(k,i_k,j_k): k\in\mathbb{N},1\le i_k,j_k\le d\}$. For any nonempty finite subset $\alpha=\{(1,i_{1},j_{1}),(2,i_{2},j_{2}),\cdots,(n,i_{n},j_{n})\}\subset \mathscr{A}$, define $\max(\alpha)=n$. Besides, define $\max(\emptyset)=1$. Let $\mathcal{J}$ be the family of all finite subsets $\alpha\subset\mathscr{A}$ with $(i_{\max(\alpha)},j_{\max(\alpha)})\neq (d,d)$. For any given $\alpha=\{(1,i_{1},j_{1}),(2,i_{2},j_{2}),\cdots,(n,i_{n},j_{n})\}\in\mathcal{J}$, define
\[U_\alpha=U_{(i_1,j_1)}\otimes U_{(i_2,j_2)}\otimes\cdots\otimes U_{(i_n,j_n)}\otimes 1\otimes 1 \cdots\in \mathscr{M} .\] 
In addition, we set $U_\emptyset=1$. Then $(U_\alpha)_{\alpha\in\mathcal{J}}$ is an orthonormal basis of $L_2(\mathscr{M})$. Next, we calculate  $U_\alpha U_\beta^*$. For any given $\alpha,\beta\in\mathcal{J}$, write 
\[\alpha=\{(1,\tilde{i}_1,\tilde{j}_1),\cdots,(\max(\alpha),\tilde{i}_{\max(\alpha)},\tilde{j}_{\max(\alpha)})\}\]
where $1\le \tilde{i}_1,\tilde{j}_1,\cdots,\tilde{i}_{\max(\alpha)},\tilde{j}_{\max(\alpha)}\le d$ and
\[\beta=\{(1,{i}_1,{j}_1),\cdots,(\max(\beta),{i}_{\max(\beta)},{j}_{\max(\beta)})\}\]
where $1\le {i}_1,{j}_1,\cdots,{i}_{\max(\beta)},{j}_{\max(\beta)}\le d$. To calculate $U_\alpha U_\beta^*$, we define $\eta_{\alpha,\beta}\in \mathcal{J}$ associated with $\alpha$ and $\beta$ as follows:
\begin{enumerate}
	\item if $\max(\alpha)=\max (\beta)$,
	\begin{equation}\label{5.1}
		\eta_{\alpha,\beta}=\{(1,\tilde{i}_1-i_1,\tilde{j}_1-j_1),\cdots,(\max(\alpha),\tilde{i}_{\max(\alpha)}-i_{\max(\alpha)},\tilde{j}_{\max(\alpha)}-j_{\max(\alpha)})\};
	\end{equation}
	\item if $\max(\alpha)<\max (\beta)$,
	\begin{equation}\label{5.2}
		\begin{aligned}
			\eta_{\alpha,\beta}={}&\{(1,\tilde{i}_1-i_1,\tilde{j}_1-j_1),\cdots,(\max(\alpha),\tilde{i}_{\max(\alpha)}-i_{\max(\alpha)},\tilde{j}_{\max(\alpha)}-j_{\max(\alpha)}),\\&(\max(\alpha)+1,i_{\max(\alpha)+1},j_{\max(\alpha)+1}),\cdots,(\max(\beta),i_{\max(\beta)},j_{\max(\beta)})\};
		\end{aligned}
	\end{equation}
	\item if $\max(\alpha)>\max (\beta)$,
	\begin{equation}\label{5.3}
		\begin{aligned}
			\eta_{\alpha,\beta}={}&\{(1,\tilde{i}_1-i_1,\tilde{j}_1-j_1),\cdots,(\max(\beta),\tilde{i}_{\max(\beta)}-i_{\max(\beta)},\tilde{j}_{\max(\beta)}-j_{\max(\beta)}),\\&(\max(\beta)+1,\tilde{i}_{\max(\beta)+1},\tilde{j}_{\max(\beta)+1}),\cdots,(\max(\alpha),\tilde{i}_{\max(\alpha)},\tilde{j}_{\max(\alpha)})\}.
		\end{aligned}
	\end{equation}
\end{enumerate}
Notice that the case where $\alpha=\emptyset$ or $\beta=\emptyset$ has been included in the construction of $\eta_{\alpha,\beta}$. In (\ref{5.1}), (\ref{5.2}) and (\ref{5.3}), if $\tilde{i}_k-i_k\leq 0$ (respectively $\tilde{j}_k-j_k\leq 0$), then we can substitute $\tilde{i}_k-i_k+d$ (respectively $\tilde{j}_k-j_k+d$) for $\tilde{i}_k-i_k$ (respectively $\tilde{j}_k-j_k$).


By Lemma \ref{uuu}, one verifies that 
\begin{equation}\label{ual}
	U_\alpha U_\beta^*=\lambda_{\alpha,\beta}U_{\eta_{\alpha,\beta}},
\end{equation}
where
\begin{equation*}
	\lambda_{\alpha,\beta}=
	\begin{cases}
		\omega^{-i_{1}(\tilde{j}_{1}-j_{1})}\cdots \omega^{-i_{\max(\alpha)}(\tilde{j}_{\max(\alpha)}-j_{\max(\alpha)})},& \text{if $\max(\alpha)\le \max(\beta)$};\\
		\omega^{-i_{1}(\tilde{j}_{1}-j_{1})}\cdots \omega^{-i_{\max(\beta)}(\tilde{j}_{\max(\beta)}-j_{\max(\beta)})},& \text{if $\max(\alpha)> \max(\beta)$}.
	\end{cases}
\end{equation*}
This implies that $|\lambda_{\alpha,\beta}|=1$.

\

Let $\upsilon=e^{2\pi \mathrm{i}/d^2}$. Let $\mathcal{R}=\{\upsilon^1,\upsilon^2,\cdots,\upsilon^{d^2}\}^{\mathbb{N}}$ be equipped with the uniform distribution. For $1\le i, j\le d$, we define
\[h_{(i, j)}=\sum_{l=1}^{d^2}\upsilon^{(di+j)l}\mathbbm{1}_{{\{v^l\}}}.\]
Similarly, for any given $\alpha=\{(1,i_{1},j_{1}),(2,i_{2},j_{2}),\cdots,(n,i_{n},j_{n})\}\in\mathcal{J}$, define
\[h_\alpha=h_{(i_1,j_1)}\otimes h_{(i_2,j_2)}\otimes\cdots\otimes h_{(i_n,j_n)}\otimes 1\otimes 1\cdots\in L_2(\mathcal{R}),\]
namely, for every $t=(t_m)_{m\in\mathbb{N}}\in\mathcal{R}$,
$$ h_\alpha(t)=\prod_{k=1}^nh_{(i_k,j_k)}(t_k). $$
We also set $h_\emptyset=1$. Then $(h_\alpha)_{\alpha\in\mathcal{J}}$ is an orthonormal basis of $L_2(\mathcal{R})$. Let $\mathscr{R}_n$ be the $\sigma$-algebra generated by $\{h_\alpha: \max(\alpha)\leq n\}$, and then $(\mathscr{R}_n)_{n\geq 1}$ is a filtration for $\mathcal{R}$. Indeed, a martingale in $L_2(\mathcal{R})$ with respect to the filtration $(\mathscr{R}_n)_{n\geq 1}$ is a $d^2$-adic martingale.

\

Define for any $t=(t_m)_{m\in\mathbb{N}}\in\mathcal{R}$, and $\forall k\in\mathbb{N}, 1\le i_k,j_k\le d,$
\begin{equation*}
	\begin{aligned}
		{}&\,\,\sigma_{h(t)}:\mathscr{M}\to \mathscr{M}\\&
		U_{\{(k,i_k,j_k)\}}\mapsto h_{\{(k,i_k,j_k)\}}(t_k)U_{\{(k,i_k,j_k)\}}.
	\end{aligned}
\end{equation*}
Then $\sigma_{h(t)}$ extends to a trace preserving automorphism of $\mathscr{M}$, and hence extends to an isometry on $L_p(\mathscr{M})$ for all $0< p< \infty$.

Now for any given $b=\sum\limits_{\alpha\in\mathcal{J}}\hat{b}(\alpha)U_\alpha\in L_p(\mathscr{M})$ with $\hat{b}(\alpha)=\tau(U_\alpha^*\cdot b)$, we define 
\[\tilde b(t)=\sigma_{h(t)}(b)=\sum_{\alpha\in\mathcal{J}}\hat{b}(\alpha)U_\alpha\cdot h_\alpha(t).\]
Then $\tilde{b}\in L_p(\mathcal{R},L_p(\mathscr{M}))$. Therefore, for any given $b$, define the martingale paraproduct $\pi_{\tilde{b}}$ of symbol $\tilde{b}$ on $L_2(\mathcal{R},L_2(\mathscr{M}))$ by
\begin{equation*}
	\begin{aligned}
		\pi_{\tilde{b}}:L_2(\mathcal{R},L_2(\mathscr{M}){}&)\to L_2(\mathcal{R},L_2(\mathscr{M}))\\&
		g\mapsto \sum_{k=1}^\infty d_k\tilde{b}\cdot g_{k-1}.
	\end{aligned}
\end{equation*}
In fact, $\pi_{\tilde{b}}$ is a martingale paraproduct for semicommutative $d^2$-adic martingales.

Now we come to the proof of Theorem \ref{thm7.1}.

\begin{proof}[Proof of Theorem \ref{thm7.1}]

	Since $L_2(\mathscr{M})\cong \ell_2(\mathcal{J})$ and $L_2(\mathcal{R},L_2(\mathscr{M}))\cong \ell_2(\mathcal{J},L_2(\mathscr{M}))$, we represent $\pi_b$ and $\pi_{\tilde{b}}$ in the matrix form. Note that for $k\ge 1$,
	\begin{equation*}
		(U_\beta)_{k-1}=
		\begin{cases}
			U_\beta,& \text{if $k-1\ge \max(\beta)$};\\
			0,& \text{otherwise}.
		\end{cases}
	\end{equation*}
	This implies
	\begin{equation*}
		\begin{aligned}
			\langle U_\alpha,\pi_b (U_\beta)\rangle{}&=\langle U_\alpha,\sum_{k=1}^\infty d_kb\cdot(U_\beta)_{k-1}\rangle=\langle U_\alpha,\sum_{k-1\ge \max(\beta)}d_kb\cdot U_\beta\rangle\\&=\langle U_\alpha,\sum_{\max (\gamma)\ge \max(\beta)+1}\hat{b}(\gamma)U_\gamma U_\beta\rangle\\&=\langle U_\alpha U_\beta^*,\sum_{\max \gamma\ge \max(\beta)+1}\hat{b}(\gamma)U_\gamma\rangle.
		\end{aligned}
	\end{equation*}
	Then by \eqref{ual}
	
	\begin{equation*}
		\begin{aligned}
			\langle U_\alpha,\pi_b (U_\beta)\rangle{}&=\langle  \lambda_{\alpha,\beta}U_{\eta_{\alpha,\beta}},\sum_{\max \gamma\ge \max(\beta)+1}\hat{b}(\gamma)U_\gamma\rangle
		\\&=
		\begin{cases}
			\overline{\lambda_{\alpha,\beta}}\cdot \hat{b}(\eta_{\alpha,\beta}),& \text{if $\max(\alpha)>\max (\beta)$};\\
			0,& \text{if $\max(\alpha)\le\max (\beta)$}.
		\end{cases}
	\end{aligned}
\end{equation*}
In the same way, one has
\begin{equation*}
	\begin{aligned}
		\langle h_\alpha,\pi_{\tilde{b}} (h_\beta)\rangle=
		\begin{cases}
			\hat{b}(\eta_{\alpha,\beta})U_{\eta_{\alpha,\beta}},& \text{if $\max (\alpha)>\max (\beta)$};\\
			0,& \text{if $\max (\alpha)\le\max (\beta)$}.
		\end{cases}
	\end{aligned}
\end{equation*}
Denote by \[[\pi_b]=\Bigl((\pi_b)_{\alpha,\beta}\Bigr)_{\alpha,\beta\in\mathcal{J}}\] 
the matrix form of $\pi_b$ with respect to the basis $(U_\alpha)_{\alpha\in\mathcal{J}}$, where
$(\pi_b)_{\alpha,\beta}=\langle U_\alpha,\pi_b (U_\beta)\rangle$. Analogously, let
\[[\pi_{\tilde{b}}]=\Bigl((\pi_{\tilde{b}})_{\alpha,\beta}\Bigr)_{\alpha,\beta\in\mathcal{J}}\]
be the matrix form of $\pi_{\tilde{b}}$ with respect to the basis $(h_\alpha)_{\alpha\in\mathcal{J}}$, where
$(\pi_{\tilde{b}})_{\alpha,\beta}=\langle h_\alpha,\pi_{\tilde{b}} (h_\beta)\rangle$.

Observing that $(\pi_{\tilde{b}})_{\alpha,\beta}=(\pi_b)_{\alpha,\beta}U_\alpha U_\beta^*$ for any $\alpha,\beta\in\mathcal{J}$, one has
\begin{equation*}
	\begin{aligned}
		[\pi_{\tilde{b}}]{}&=\biggl((\pi_b)_{\alpha,\beta}U_\alpha U_\beta^*\biggr)_{\alpha,\beta\in\mathcal{J}}\\&=
		\begin{pmatrix}
			\ddots & \,  &   0 \\
			\,     &  U_\alpha &   \,\\
			0      & \,  &   \ddots
		\end{pmatrix}_{\alpha\in\mathcal{J}}([\pi_b]\otimes 1_{\mathscr{M}})
		\begin{pmatrix}
			\ddots & \,  &   0 \\
			\,     &  U_\beta^* &   \,\\
			0      & \,  &   \ddots
		\end{pmatrix}_{\beta\in\mathcal{J}},
	\end{aligned}
\end{equation*}
where $1_{\mathscr{M}}$ is the identity of $\mathscr{M}$. So this implies that for any $0<p<\infty$
\begin{equation*}
	\begin{aligned}
		\|\pi_{\tilde{b}}\|_{L_p(B(L_2(\mathcal{R}))\otimes \mathscr{M})}{}&=\|[\pi_{\tilde{b}}]\|_{L_p(B(\ell_2(\mathcal{J}))\otimes \mathscr{M})}=\|[\pi_b]\|_{S_p(\ell_2(\mathcal{J}))}\cdot\|1_{\mathscr{M}}\|_{L_p(\mathscr{M})}=\|\pi_{b}\|_{S_p(L_2(\mathscr{M}))}.
	\end{aligned}
\end{equation*}
By Theorem \ref{thm1.2}, we have $\pi_{\tilde{b}}\in {L_p(B(L_2(\mathcal{R}))\otimes \mathscr{M})}$ if and only if $\tilde{b}\in \pmb{B}_p^{d^2}(\mathbb{R},\mathscr{M})$, where 
\begin{equation*}
	\|\tilde{b}\|_{\pmb{B}_p^{d^2}(\mathbb{R},\mathscr{M})}\approx_{d,p}\biggl(\sum_{k=1}^\infty d^{2k}\|d_k\tilde{b}\|^p_{L_p(\mathcal{R},L_p(\mathscr{M}))}\biggr)^{1/p}. 
\end{equation*}
However, note that for any $t\in\mathcal{R}$ and $k\ge 1$,
\begin{equation*}
	\begin{aligned}
		(d_k\tilde{b})(t){}&=\big(\sum_{\max(\alpha)=k}\hat{b}(\alpha)U_\alpha h_\alpha\big)(t)=\sigma_{h(t)}(\sum_{\max(\alpha)=k}\hat{b}(\alpha)U_\alpha)=\sigma_{h(t)}(d_k b).
	\end{aligned}
\end{equation*}
This yields
\[\|d_k\tilde{b}\|^p_{L_p(\mathcal{R},L_p(\mathscr{M}))}=\int_{\mathcal{R}}\|(d_k \tilde{b})(t)\|_{L_p(\mathscr{M})}^p dt =\int_{\mathcal{R}}\|\sigma_{h(t)}(d_kb)\|_{L_p(\mathscr{M})}^p dt =\|d_kb\|^p_{L_p(\mathscr{M})},\]
Therefore, we conclude that $\pi_b\in S_p(L_2(\mathscr{M}))$ if and only if $b\in \pmb{B}_p(\mathscr{M})$ with relevant constants depending only on $d$ and $p$. This completes the proof of Theorem \ref{thm7.1}.
\end{proof}

\bigskip

\section{Proof of Theorem \ref{thm6.4}}\label{Application 3}
We first start with preparations concerning martingale paraproducts and Schatten classes, namely Lemma \ref{TLambdab} and Proposition \ref{T0est}, which will be helpful in the proof of Theorem \ref{thm6.4}. Then we will introduce the key ingredient: the dyadic representation of singular integral operators by Hyt\"{o}nen in \cite{TH1} and \cite{TH2}. This representation enables the reduction to the $d$-adic martingale setting. Finally, we will give a proof of Theorem \ref{thm6.4} using the result about martingale paraproducts stated in Theorem \ref{thm1.2}. In the remainder of this section, we will still denote $B(L_2(\mathbb{R}))\otimes \mathcal{M}$ by $\mathcal{N}$.

\subsection{Schatten class of operator-valued commutators involving martingale paraproducts}
\begin{lem}\label{TLambdab}
	Assume that $1\le p<\infty$. For any semicommutative $d$-adic martingale $ f=(f_k)_{k\in\mathbb{Z} }\in L_2(\mathbb{R},L_2(\mathcal{M}))$, we define
	\[\varLambda_b(f)=\sum_{k\in\mathbb{Z}}d_kb\cdot d_kf.\]
	If $b\in\pmb{B}_p^d(\mathbb{R},\mathcal{M})$, then $\varLambda_b\in L_p(\mathcal{N})$ and 
	$$ \|\varLambda_b\|_{L_p(\mathcal{N})} \lesssim_{d, p} \|b\|_{\pmb{B}_p^d(\mathbb{R},\mathcal{M})}. $$
\end{lem}	

\begin{proof}
	We write $	\varLambda_b$ as follows:
	\begin{equation}\label{Lambdasum}
		\begin{aligned}
			\varLambda_b(f){}&=\sum_{k\in\mathbb{Z}}d_kb\cdot d_kf\\
			&=\sum_{k\in\mathbb{Z}}\Bigl(\sum_{I\in\mathcal{D}_{k-1}}\sum_{i=1}^{d-1}\langle h_I^i,b\rangle h_I^i\Bigr)\Bigl(\sum_{J\in\mathcal{D}_{k-1}}\sum_{j=1}^{d-1}\langle h_J^j,f\rangle h_J^j\Bigr)\\
			&=\sum_{k\in\mathbb{Z}}\sum_{I\in\mathcal{D}_{k-1}}\Bigl(\sum_{i=1}^{d-1}\langle h_I^i,b\rangle h_I^i\Bigr)\Bigl(\sum_{j=1}^{d-1}\langle h_I^j,f\rangle h_I^j\Bigr)\\
			&=\sum_{k\in\mathbb{Z}}\sum_{I\in\mathcal{D}_{k-1}}\Bigl(\sum_{\overline{i+j}=d}\langle h_I^i,b\rangle\langle h_I^j,f\rangle \frac{\mathbbm{1}_I}{|I|}+\sum_{l=1}^{d-1}\sum_{\overline{i+j}=l}\langle h_I^i,b\rangle\langle h_I^j,f\rangle \frac{h_I^l}{|I|^{1/2}}\Bigr)\\
			&=\sum_{k\in\mathbb{Z}}\sum_{I\in\mathcal{D}_{k-1}}\Bigl(\sum_{\overline{i+j}=d}\langle b^*,h_I^{d-i}\rangle\langle h_I^j,f\rangle \frac{\mathbbm{1}_I}{|I|}+\sum_{l=1}^{d-1}\sum_{\overline{i+j}=l}\langle h_I^i,b\rangle\langle h_I^j,f\rangle \frac{h_I^l}{|I|^{1/2}}\Bigr)\\
			&=(\pi_{b^*})^*(f)+\tilde{\varLambda}_b(f),
		\end{aligned}
	\end{equation}
	where we have used \eqref{pistar}, and where
	\begin{equation}\label{tildevar}
		\tilde{\varLambda}_b(f)=\sum_{k\in\mathbb{Z}}\sum_{I\in\mathcal{D}_{k-1}}\sum_{l=1}^{d-1}\sum_{\overline{i+j}=l}\langle h_I^i,b\rangle\langle h_I^j,f\rangle \frac{h_I^l}{|I|^{1/2}}.
	\end{equation}
	By Theorem \ref{thm1.2}, we know
	\begin{equation}\label{bbbp}
		\|(\pi_{b^*})^*\|_{L_p(\mathcal{N})}\approx_{d,p} \|b\|_{\pmb{B}_p^{d}(\mathbb{R},\mathcal{M})}.
	\end{equation}
	It remains to estimate $\|\tilde{\varLambda}_b\|_{L_p(\mathcal{N})}$. We represent it into the matrix form. Note that for any $S,T\in \mathcal{D}$,  $1\le s,t\le d-1$, and $x,y\in L_2(\mathcal{M})$,
	\begin{equation*}
		\begin{aligned}
			\langle h_S^s\otimes x,\tilde{\varLambda}_b(h_T^t\otimes y)\rangle{}&=
			\begin{cases}
				\langle x,|S|^{-1/2}\langle h_S^{s-t},b\rangle y\rangle ,& \text{if $S=T$ and $s\neq t$},\\
				0,& \text{otherwise}.
			\end{cases}
		\end{aligned}
	\end{equation*}
	This yields that $\tilde{\varLambda}_b$ is a block diagonal matrix with respect to the basis $\{h_I^i\}_{I\in\mathcal{D}, 1\leq i\leq d-1}$. (If $s-t< 0$, replace $s-t$ with $s-t+d$, and still denote it by $s-t$.)
	For any $I\in \mathcal{ D}$, $1\leq s\neq t \leq d-1$, denote $|I|^{-1/2}\langle h_I^{s-t},b\rangle$ by $a^I_{s-t}$, and define $a_0^I=0$.
	Hence one has
	\begin{equation*}
		\begin{aligned}
			\|\tilde{\varLambda}_b\|_{L_p(\mathcal{N})}^p{}&=\sum_{I\in\mathcal{D}}\biggl\|\biggl(a_{s-t}^I\biggr)_{1\le s,t\le d-1}\biggr\|_{L_p(\mathbb{M}_{d-1}\otimes\mathcal{M})}^p,
		\end{aligned}
	\end{equation*}
	where $\mathbb{M}_{d-1}$ is equipped with the usual trace.
	Let 
	\begin{equation}\label{BI}
		B^I=\begin{pmatrix}
			a_0^I       & a_{d-1}^I & \cdots    & a_2^I    & a_1^I \\
			a_1^I       & a_0^I       & a_{d-1}^I & \cdots & a_2^I \\
			\vdots    & \ddots    & \ddots    & \ddots & \vdots \\
			a_{d-2}^I & \cdots    & a_1^I       & a_0^I    & a_{d-1}^I  \\
			a_{d-1}^I & \cdots    & a_2^I      & a_1^I    & a_0^I
		\end{pmatrix}
		=\begin{pmatrix}
			&  &   & a_1^I \\
			&  \biggl(a_{s-t}^I\biggr)_{1\le s,t\le d-1}    & &\vdots  \\
			&  & & a_{d-1}^I \\
			a_{d-1}^I      & \cdots  & a_1^I  & a_0^I
		\end{pmatrix}.
	\end{equation}
	By Lemma \ref{lem3.1}, we have 
	\[\biggl\|\biggl(a_{s-t}^I\biggr)_{1\le s,t\le d-1}\biggr\|_{L_p(\mathbb{M}_{d-1}\otimes \mathcal{M})}\le \|B^I\|_{L_p(\mathbb{M}_{d}\otimes\mathcal{M})},\]
	which implies
	\begin{equation*}
		\begin{aligned}
			\|\tilde{\varLambda}_b\|_{L_p(\mathcal{N})}^p{}&\le \sum_{I\in\mathcal{D}}\|B^I\|_{L_p(\mathbb{M}_{d}\otimes\mathcal{M})}^p.
		\end{aligned}
	\end{equation*}
	Note that we can write $B^I$ as
	\begin{equation*}
		\begin{aligned}
			B^I=a_1^IA+a_2^IA^2+\cdots+a_{d-1}^IA^{d-1}
		\end{aligned}
	\end{equation*}
	with
	\begin{equation*}
		A=e_{1,d}+\sum_{j=1}^{d-1}e_{j+1,j}.
	\end{equation*}
	Using the triangle inequality, one has
	\begin{equation*}\label{bbbbp}
		\begin{aligned}
			\quad\|\tilde{\varLambda}_b\|_{L_p(\mathcal{N})}^p{}&\le \sum_{I\in\mathcal{D}}\|a_1^IA+a_2^IA^2+\cdots+a_{d-1}^IA^{d-1}\|_{L_p(\mathbb{M}_{d}\otimes\mathcal{M})}^p\\&\lesssim_{d,p}\sum_{I\in\mathcal{D}}\sum_{i=1}^{d-1}\|a_i^IA^i\|^p_{L_p(\mathbb{M}_{d}\otimes\mathcal{M})}\le\sum_{I\in\mathcal{D}}\sum_{i=1}^{d-1}\|a_i^I\|_{L_p(\mathcal{M})}^p\|A\|^p_{S_p(\mathbb{M}_{d})}\\&\lesssim_{d,p}\sum_{I\in\mathcal{D}}\sum_{i=1}^{d-1}\|a_i^I\|_{L_p(\mathcal{M})}^p=\sum_{I\in \mathcal{D}}\sum_{i=1}^{d-1}\biggl(\frac{\|\langle h_I^i,b\rangle\|_{L_p(\mathcal{M})}}{|I|^{1/2}}\biggr)^p=\|b\|_{\pmb{B}_p^{d}(\mathbb{R},\mathcal{M})}^p.
		\end{aligned}
	\end{equation*}
	Combining this with (\ref{bbbp}), we obtain the desired result.
\end{proof}
\begin{rem}
	Lemma \ref{TLambdab} also holds for $0<p<1$ with the same proof,  we leave the details to the interested reader.
\end{rem}

In what follows, we need to use the boundedness of  the triangular projection on Schatten classes. The triangular projection is defined as follows
\begin{equation*}
	\begin{aligned}
		\mathcal{P}: B(\ell_2{}) &\longrightarrow B(\ell_2) \\
		(m_{ij})_{i,j}&\longmapsto (\delta_{i>j}\cdot m_{ij})_{i,j},
	\end{aligned}
\end{equation*}
where $\delta_{i>j}=1$ if $i>j$, and $\delta_{i>j}=0$ if $i\leq j$. It is well-known that $\mathcal{P}$ is bounded from $S_p(\ell_2)$ to $S_p(\ell_2)$ when $1<p<\8$. We refer the reader to \cite{GK} for more information.

Then for $1<p<\8$, we can define $\mathcal{P}\otimes Id_{L_p(\M)}$ on the algebraic tensor product $S_p(\ell_2)\otimes L_p(\mathcal{M})$. The next lemma is well-known, and the interested reader is referred to \cite{Hao} for a detailed proof.
\begin{lemma}\label{triangle}
	Let $1<p<\infty$. Then $\mathcal{P}\otimes Id_{L_p(\mathcal{M})}$ extends to a bounded map on $L_p(B(\ell_2)\otimes \mathcal{M})$. Moreover, $$ \|\mathcal{P}\otimes Id_{L_p(\mathcal{M})}\|_{L_p(B(\ell_2)\otimes \mathcal{M}) \to L_p(B(\ell_2)\otimes \mathcal{M})}\lesssim \max\{p', p\}. $$
\end{lemma}

Before proving Theorem \ref{thm6.4}, we give the following proposition, which concerns the $p$-Schatten class of operator-valued commutators involving martingale paraproducts and the left multiplication operator $M_b$.

\begin{proposition}\label{T0est}
	Let $1<p<\infty$. If $a\in BMO^d(\mathbb{R})$ and $b\in \pmb{B}_p^{d}(\mathbb{R},\mathcal{M})$, then $[\pi_a,M_b]$ and $[\pi^*_a,M_b]$ both belong to $L_p(\mathcal{N})$. Moreover,
	\begin{equation*}
		\|[\pi_a,M_b]\|_{L_p(\mathcal{N})}\lesssim_{d,p} \|a\|_{BMO^d(\mathbb{R})} \|b\|_{\pmb{B}_p^{d}(\mathbb{R},\mathcal{M})}
	\end{equation*}
and
\begin{equation*}
	\|[\pi_a^*,M_b]\|_{L_p(\mathcal{N})}\lesssim_{d,p} \|a\|_{BMO^d(\mathbb{R})} \|b\|_{\pmb{B}_p^{d}(\mathbb{R},\mathcal{M})}.
\end{equation*}
\end{proposition} 
\begin{proof}
	Let
	\begin{equation}\label{R_b}
		R_b(f)=\sum_{k\in\mathbb{Z}}b_{k-1}\cdot d_kf, \quad \forall f\in L_2(\mathbb{R},L_2(\mathcal{M})).
	\end{equation}
	Note that for $b, f\in L_2(\mathbb{R},L_2(\mathcal{M}))$,  $M_b(f)=\pi_b(f)+\varLambda_b(f)+R_b(f)$. Thus
	$$ 	[\pi_{{a}},M_b]=[\pi_{{a}},\pi_b]+[\pi_{{a}},\varLambda_b]+[\pi_{{a}},R_b].$$
	We first estimate $\|[\pi_{a}, R_b]\|_{L_p(\mathcal{N})}$. For any $f\in L_2(\mathbb{R},L_2(\mathcal{M}))$,
	\begin{equation*}
		\begin{aligned}
			[\pi_{a}, R_b](f){}&=\pi_{a} (R_b(f))-R_b (\pi_{a}(f))\\
			&=\sum_{k\in\mathbb{Z}} d_ka\cdot\mathbb{E}_{k-1} \biggl(\sum_{j\in\mathbb{Z}} b_{j-1}\cdot d_jf \biggr)- \sum_{k\in\mathbb{Z}} b_{k-1} \cdot d_k \biggl(\sum_{j\in\mathbb{Z}} d_ja\cdot f_{j-1}\biggr)\\
			&=\sum_{k\in\mathbb{Z}} d_ka\cdot \biggl(\sum_{j\le k-1}b_{j-1}\cdot d_jf\biggr)-\sum_{k\in\mathbb{Z}} b_{k-1}\cdot d_ka\cdot f_{k-1}\\
			&=\sum_{k\in\mathbb{Z}} d_ka\cdot \biggl(\sum_{j\le k-1}b_{j-1}\cdot d_jf- b_{k-1}\cdot f_{k-1}\biggr)\\
			&=-\sum_{k\in\mathbb{Z}} d_ka\cdot \biggl( \sum_{j\le k-1}d_jb\cdot d_jf\biggr)-\sum_{k\in\mathbb{Z}} d_ka\cdot \biggl(\sum_{j\le k-1}d_jb\cdot f_{j-1}\biggr)\\
			&=-\sum_{k\in\mathbb{Z}} d_ka\cdot \biggl( \sum_{j\le k-1}d_jb\cdot d_jf\biggr)-\pi_{a}(\pi_b(f))\\
			&=: -\varPsi_{a,b}(f)-\pi_{a}(\pi_b(f)).
		\end{aligned}
	\end{equation*}
	Thus
	\begin{equation}\label{piarb}
		[\pi_{a}, R_b]=-\varPsi_{a,b}-\pi_{a}\pi_b.
	\end{equation}
	From Theorem \ref{thm1.2}, we know that
	\begin{equation*}
		\begin{aligned}
			\|\pi_b\|_{L_p(\mathcal{N})}\lesssim_{d, p}\|b\|_{\pmb{B}_p^{d}(\mathbb{R},\mathcal{M})}.
		\end{aligned}
	\end{equation*}
	Since $\pi_a$ is bounded on $L_2(\mathbb{R}, L_2(\M))$, one has
	\begin{equation}\label{piarb1}
		\|\pi_a\pi_b\|_{L_p(\mathcal{N})}\lesssim_{d, p}\|a\|_{BMO^d(\mathbb{R})}\|b\|_{\pmb{B}_p^{d}(\mathbb{R},\mathcal{M})}.
	\end{equation}
	
	To deal with $\varPsi_{a,b}$, let $\mathcal{Q}_n=\big\{(S,s):S\in\mathcal{D}_n,1\le s\le d-1\big\}$ and 
	\begin{equation*}
		\mathcal{Q}=\cdots \mathcal{Q}_{-2}\cup \mathcal{Q}_{-1}\cup \mathcal{Q}_{0}\cup \mathcal{Q}_{1}\cup \mathcal{Q}_{2}\cdots.
	\end{equation*}
    One needs only order the elements of $\mathcal{Q}$ according to $n$ in $\mathcal{Q}_n$. The elements inside each $\mathcal{Q}_n$ need not be ordered. Note that $L_2(\mathbb{R},L_2(\mathcal{M}))\cong \ell_2(\mathcal{Q}, L_2(\mathcal{M}))$. Besides, we denote by
	\begin{equation*}
		[\varPsi_{a,b}]=\biggl((\varPsi_{a,b})_{(S,s),(T,t)}\biggr)_{(S,s),(T,t)\in \mathcal{Q}}
	\end{equation*}
	the matrix form of $\varPsi_{a,b}$ with respect to the basis $(h_S^s)_{(S,s)\in \mathcal{Q}}$, where
	$(\varPsi_{a,b})_{(S,s),(T,t)}=\langle h_S^s,\varPsi_{a,b} (h_T^t)\rangle$. Analogously, let
	\begin{equation*}
		[\pi_a\varLambda_b]=\biggl((\pi_a\varLambda_b)_{(S,s),(T,t)}\biggr)_{(S,s),(T,t)\in \mathcal{Q}}
	\end{equation*}
	be the matrix form of $\pi_a\varLambda_b$.
	
	Our aim is to prove that
	\begin{equation}\label{varpsi}
		[\varPsi_{a,b}]=(\mathcal{P}\otimes Id_{L_p(\mathcal{M})})([\pi_{a}\varLambda_{b}]).
	\end{equation}
	On the one hand, note that
	\begin{equation}\label{varpsif}
		\begin{aligned}
			\varPsi_{a,b}(f){}&=\sum_{k\in\mathbb{Z}} d_ka\cdot\biggl(\mathbb{E}_{k-1}\Bigl(\sum_{j\in\mathbb{Z}} d_j{b}\cdot d_j{f}\Bigr)-\mathbb{E}_{k-1}\Bigl(\sum_{j\ge k} d_j{b}\cdot d_j{f}\Bigr)\biggr)\\
			&=\sum_{k\in\mathbb{Z}} d_ka\cdot\biggl(\mathbb{E}_{k-1}(\varLambda_{{b}}({f}))-\mathbb{E}_{k-1}\Bigl(\sum_{j\ge k} d_j{b}\cdot d_j{f}\Bigr)\biggr)\\
			&=\pi_a(\varLambda_{{b}}({f}))-\sum_{k\in\mathbb{Z}} d_ka\cdot \mathbb{E}_{k-1}\Bigl(\sum_{j\ge k} d_j{b}\cdot d_j{f}\Bigr).
		\end{aligned}
	\end{equation}
	Suppose that $S\in \mathcal{D}_m$, $T\in\mathcal{D}_n$, and $m>n$. For any $1\le s,t\le d-1$, note that $d_{m+1}h_S^s=h_S^s$ and $d_{n+1}h_T^t=h_T^t$, hence
	\begin{equation*}
		\begin{aligned}
			\langle h_S^s,\varPsi_{a,b}(h_T^t)\rangle-\langle h_S^s,\pi_a(\varLambda_b(h_T^t))\rangle{}&=-\biggl\langle h_S^s, \sum_{k\in\mathbb{Z}} d_ka\cdot \mathbb{E}_{k-1}\Bigl(\sum_{j\ge k} d_j{b}\cdot d_j{h_T^t}\Bigr)\biggr\rangle \\
			&=-\biggl\langle d_{m+1}h_S^s, \sum_{k\leq {n+1}} d_ka\cdot \mathbb{E}_{k-1}\Bigl(d_{n+1}{b}\cdot d_{n+1}{h_T^t}\Bigr)\biggr\rangle\\
			&=-\biggl\langle h_S^s, d_{m+1}\Bigl(\sum_{k\le n+1} d_ka\cdot \mathbb{E}_{k-1}\bigl(d_{n+1}{b}\cdot {h_T^t}\bigr)\Bigr)\biggr\rangle\\
			&=\langle h_S^s,0\rangle=0.
		\end{aligned}
	\end{equation*}
	This implies that when $S\in \mathcal{D}_m$, $T\in\mathcal{D}_n$, and $m>n$
	\begin{equation}\label{pusiab}
		\begin{aligned}
			(\varPsi_{a,b})_{(S,s),(T,t)}=(\pi_a\varLambda_b)_{(S,s),(T,t)}.
		\end{aligned}
	\end{equation}
	On the other hand, 
	\begin{equation*}
		\begin{aligned}
			\varPsi_{a,b}(f){}&=\sum_{k\in\mathbb{Z}}\sum_{I\in\mathcal{D}_{k-1}}\sum_{i=1}^{d-1}\langle h_I^i,a\rangle h_I^i \cdot \biggl(\sum_{j\le k-1}\sum_{L\in \mathcal{D}_{j-1}}\sum_{l=1}^{d-1}\langle h_L^l,b\rangle h_L^l \cdot \sum_{Q\in \mathcal{D}_{j-1}}\sum_{q=1}^{d-1}\langle h_Q^q,f\rangle h_Q^q\biggr)\\
			&=\sum_{k\in\mathbb{Z}}\sum_{I\in\mathcal{D}_{k-1}}\sum_{i=1}^{d-1}\langle h_I^i,a\rangle h_I^i \cdot \biggl(\sum_{j\le k-1}\sum_{L\in \mathcal{D}_{j-1}}\sum_{l=1}^{d-1}\sum_{q=1}^{d-1}\langle h_L^l,b\rangle \langle h_L^q,f\rangle h_L^{\overline{l+q}}\biggr)\\
			&=\sum_{I\in\mathcal{D}}\sum_{i=1}^{d-1}\langle h_I^i,a\rangle\cdot \biggl(\sum_{I\subsetneqq L}\sum_{l=1}^{d-1}\sum_{q=1}^{d-1}\langle h_L^l,b\rangle \langle h_L^q,f\rangle h_I^i h_L^{\overline{l+q}}\biggr)\\
			&=\sum_{I\in\mathcal{D}}\sum_{i=1}^{d-1}\langle h_I^i,a\rangle\cdot \biggl(\sum_{I\subsetneqq L}\sum_{l=1}^{d-1}\sum_{q=1}^{d-1}\langle h_L^l,b\rangle \langle h_L^q,f\rangle  \bigg\langle \frac{\mathbbm{1}_I}{|I|},h_L^{\overline{l+q}}\bigg\rangle h_I^i\biggr),
		\end{aligned}
	\end{equation*}
	where the last equality follows from 
	$$  h_I^i h_L^{\overline{l+q}}=h_I^i (\mathbbm{1}_I\cdot h_L^{\overline{l+q}})= \bigg\langle \frac{\mathbbm{1}_I}{|I|},h_L^{\overline{l+q}}\bigg\rangle h_I^i, \quad \forall  I\subsetneqq L.$$
	Combining the preceding equalities, we get 
	\begin{equation}\label{pusiab2}
		\begin{aligned}
			(\varPsi_{a,b})_{(S,s),(T,t)}=
			\begin{cases}
				\langle h_S^s,a\rangle \sum\limits_{l=1}^{d-1}\langle h_T^l,b\rangle \bigg\langle \frac{\mathbbm{1}_S}{|S|},h_T^{\overline{l+t}}\bigg\rangle=(\pi_a\varLambda_b)_{(S,s),(T,t)} ,& \text{if $S\subsetneqq T$ };\\
				0,& \text{otherwise}.
			\end{cases}
		\end{aligned}
	\end{equation}
	Hence from \eqref{pusiab} and \eqref{pusiab2}, we conclude \eqref{varpsi}.
	
	From Lemma \ref{TLambdab}, we know that
	\begin{equation*}
		\begin{aligned}
			\|\varLambda_b\|_{L_p( \mathcal{N})}\lesssim_{d, p}\|b\|_{\pmb{B}_p^{d}(\mathbb{R},\mathcal{M})}.
		\end{aligned}
	\end{equation*}
	Hence from \eqref{varpsi} and Lemma \ref{triangle} one has
	\begin{equation*}\label{piarb2}
		\begin{aligned}
			\|\varPsi_{a,b}\|_{L_p(\mathcal{N})}{}&=\|[\varPsi_{a,b}]\|_{L_p(B(\ell_2(\mathcal{Q}))\otimes \mathcal{M})}\\
			&=\|(\mathcal{P}\otimes Id_{L_p(\mathcal{M})})([\pi_{a}\varLambda_{b}])\|_{L_p(B(\ell_2(\mathcal{Q}))\otimes \mathcal{M})}\\
			&\lesssim_p \|[\pi_{a}\varLambda_{b}]\|_{L_p(B(\ell_2(\mathcal{Q}))\otimes \mathcal{M})}\\
			&=\|[\pi_{a}\varLambda_{b}]\|_{L_p(\mathcal{N})}
			\lesssim_{d, p}\|a\|_{BMO^d(\mathbb{R})}\|b\|_{\pmb{B}_p^{d}(\mathbb{R},\mathcal{M})}.
		\end{aligned}
	\end{equation*}
	Combining the preceding inequalities, we arrive at
	\begin{equation*}
		\|[\pi_a,R_b]\|_{L_p(\mathcal{N})}\lesssim_{d,p}\|a\|_{BMO^d(\mathbb{R})} \|b\|_{\pmb{B}_p^{d}(\mathbb{R},\mathcal{M})}.
	\end{equation*}
		Therefore, by the triangle inequality we deduce that
		\begin{equation*}
			\begin{aligned}
				\|[\pi_a,M_b]\|_{L_p(\mathcal{N})}{}&\le \|[\pi_a,\pi_b]\|_{L_p(\mathcal{N})}+\|[\pi_a,\varLambda_b]\|_{L_p(\mathcal{N})}\\
				&\quad+\|[\pi_a,R_b]\|_{L_p(\mathcal{N})}\\
				&\lesssim_{d,p}\|a\|_{BMO^d(\mathbb{R})} \|b\|_{\pmb{B}_p^{d}(\mathbb{R},\mathcal{M})}.
			\end{aligned}
		\end{equation*}
			It is easy to verify that 
		\begin{equation*}
			[\pi_a^*,M_b]^*=-[\pi_a,M_{b^*}].
		\end{equation*}
Hence
\begin{equation*}
\begin{aligned}
	\|[\pi_a^*,M_b]\|_{L_p(\mathcal{N})}=\|[\pi_a,M_{b^*}]\|_{L_p(\mathcal{N})}\lesssim_{d,p} \|a\|_{BMO^d(\mathbb{R})}\|b\|_{\pmb{B}_p^{d}(\mathbb{R},\mathcal{M})}.
\end{aligned}
\end{equation*}
	\end{proof}

\subsection{Hyt\"{o}nen's dyadic representation}
Now we introduce the dyadic system on $\mathbb{R}^n$. Recall that the standard system of dyadic cubes is 
\[\mathcal{D}^0=\big\{2^{-k}([0,1)^n+q):k\in\mathbb{Z},q\in\mathbb{Z}^n\big\}.\]
Let $\mathcal{ D}^0_k=\big\{2^{-k}([0,1)^n+q):q\in\mathbb{Z}^n\big\}$ for any $k\in \mathbb{Z}$. Define $\ell(I)=2^{-k}$ and $|I|= 2^{-nk}$ if $I\in \mathcal{ D}^0_k$. Let $\omega=(\omega_j)_{j\in\mathbb{Z}}\in(\{0,1\}^n)^\mathbb{Z}$ and define
\begin{equation}\label{Idot}
	I\dot{+}\omega=I+\sum_{j:2^{-j}<\ell(I)}2^{-j}\omega_j.
\end{equation}
Note that if $\ell(I)=2^{-k}$, $I\cap I\dot{+}\omega\neq \emptyset$ unless a coordinate of $\sum_{j:2^{-j}<\ell(I)}2^{-j}\omega_j$ is exactly $2^{-k}$. Then set
\[\mathcal{D}^\omega=\big\{I\dot{+}\omega:I\in\mathcal{D}^0\big\},\]
which is obtained by translating the standard system. Indeed, in particular, if $\omega=(\omega_j)_{j\in\mathbb{Z}}\in(\{0,1\}^n)^\mathbb{Z}$ such that $\exists j_0\in \mathbb{Z}$, $\forall j\leq j_0$, $\omega_j=0$, then
$$  \mathcal{D}^\omega=\big\{I+\sum_{j}2^{-j}\omega_j:I\in\mathcal{D}^0\big\}. $$
Now for any $k\in\mathbb{Z}$, $\mathcal{ D}_k^\omega=\big\{I\dot{+}\omega:I\in\mathcal{D}_k^0\big\}$ is the family of all dyadic cubes with volume $2^{-nk}$. The following lemma implies that $\mathcal{D}^\omega$ is still a dyadic system.
\begin{lem}
	For any $\omega=(\omega_j)_{j\in\mathbb{Z}}\in(\{0,1\}^n)^\mathbb{Z}$, $\mathcal{D}^\omega$ on $\mathbb{R}^n$ is a dyadic system. More precisely, for any $k\in\mathbb{Z}$,
    \begin{equation*}
    	\mathcal{ D}_k^\omega=\mathcal{D}^0_k+\sum\limits_{j>k}2^{-j}\omega_j.
    \end{equation*}
\end{lem}
\begin{proof}
By direct calculation,
	\begin{align*}
	\mathcal{ D}_k^\omega&=\big\{I\dot{+}\omega:I\in\mathcal{D}_k^0\big\}\\
	&=\left\{I+\sum\limits_{j>k}2^{-j}\omega_j:I\in\mathcal{D}_k^0\right\}.
	\end{align*}
Note that 
$$	\mathcal{ D}_k^\omega= \mathcal{D}^0_k+\sum\limits_{j>k}2^{-j}\omega_j=\mathcal{D}^0_k+\sum\limits_{j\geq k}2^{-j}\omega_j=(\mathcal{D}^0+\sum\limits_{j\geq k}2^{-j}\omega_j)_k, $$
and
$$ \mathcal{ D}_{k-1}^\omega=\mathcal{D}^0_{k-1}+\sum\limits_{j>k-1}2^{-j}\omega_j=\mathcal{D}^0_{k-1}+\sum\limits_{j\geq k}2^{-j}\omega_j=(\mathcal{D}^0+\sum\limits_{j\geq k}2^{-j}\omega_j)_{k-1}, $$
where $(\mathcal{D}^0+\sum\limits_{j\geq k}2^{-j}\omega_j)_k$ and $(\mathcal{D}^0+\sum\limits_{j\geq k}2^{-j}\omega_j)_{k-1}$ are the collections of cubes in $\mathcal{D}^0+\sum\limits_{j\geq k}2^{-j}\omega_j$ with the corresponding volumes $2^{-nk}$ and $2^{-n(k-1)}$ respectively. This implies that every cube in $ \mathcal{ D}_{k-1}^\omega$ is a union of $2^n$ disjoint cubes in $ \mathcal{ D}_{k}^\omega$. Hence, $\mathcal{D}^\omega$ is a dyadic system.
\end{proof}

We refer the reader to \cite{TH2008} for more details on $\mathcal{D}^\omega$. For any $I\in\mathcal{ D}^\omega$, let $\mathcal{D}^\omega(I)$ be the collection of cubes in $\mathcal{D}^\omega$ contained in $I$, and $\mathcal{D}_k^\omega(I)$ the intersection of $\mathcal{D}_k^\omega$ and $\mathcal{D}^\omega(I)$ for any $k\in\mathbb{Z}$. In addition, we assign to the parameter set $(\{0,1\}^n)^\mathbb{Z}$ the natural probability measure, that is the infinite tensor product of the uniform probability measure $\sum\limits_{\omega\in \{0,1\}^n}2^{-n}\delta_{\omega}$. Here $\delta_{\omega}$ is the Dirac measure. Denote by $\mathbb{E}_\omega$ the expectation on $(\{0,1\}^n)^\mathbb{Z}$.

For convenience, denote the set $\{0,1\}^n\backslash \{0\}$ by $\{0,1\}^n_0$. Now for any given cube $I=I_1\times\cdots\times I_n\in\mathcal{D}^\omega$, let $H_{I_i}^0=|I_i|^{-1/2}\mathbbm{1}_{I_i}$ and $H_{I_i}^1=|I_i|^{-1/2}(\mathbbm{1}_{I_{i\ell}}-\mathbbm{1}_{I_{ir}})$, where $\mathbbm{1}_{I_{i\ell}}$ and $\mathbbm{1}_{I_{ir}}$ are the left and right halves of $I_i$ for $1\le i\le n$. For any $ \eta\in\{0,1\}^n_0$, we denote by $H_I^\eta$ the function on the cube $I=I_1\times\cdots\times I_n$ which is the product of the one-variable functions:
\[H_I^\eta(x)=H_{I_1\times\cdots\times I_n}^{(\eta_1,\cdots,\eta_n)}(x_1,\cdots,x_n)=\prod_{i=1}^nH_{I_i}^{\eta_i}(x_i).\]
Hence $\{H_I^\eta\}_{I\in\mathcal{D}^\omega,\eta\in\{0,1\}^n_0}$ form an orthonormal basis of $L_2(\mathbb{R}^n)$.

\

Let $i,j\in \mathbb{N}\cup\{0\}$. For a fixed dyadic system $\mathcal{D}^\omega$, the dyadic shift with parameters $i,j$ is an operator of the form
\[S_\omega^{ij}(f)=\sum_{K\in\mathcal{D}^\omega}A_K^{ij}(f),\quad\quad A_K^{ij}(f)=\sum_{\substack{I,J\in\mathcal{D}^\omega;I,J\subseteq K\\ \ell(I)=2^{-i}\ell(K)\\\ell(J)=2^{-j}\ell(K)}}\sum_{\xi,\eta\in\{0,1\}^n_0}a_{IJK}^{\xi\eta}\langle H_I^\xi,f\rangle H_J^\eta,\]
with coefficients $a_{IJK}^{\xi\eta}$ satisfying
\begin{equation}\label{aijkxieta}
	|a_{IJK}^{\xi\eta}|\le\frac{\sqrt{|I||J|}}{|K|}.
\end{equation}
From the definition of $S_\omega^{ij}$, we see that $I$ and $J$ are the $i$-th and $j$-th generation of $K$ respectively. So when $i$ or $j$ is very large, the dyadic shift $S_\omega^{ij}$ has very high complexity. This is the main difficulty when dealing with $S_\omega^{ij}$. We  also have the following properties:
\begin{enumerate}
	\item The map $S_\omega^{ij}:L_2(\mathbb{R}^n)\to L_2(\mathbb{R}^n)$  is bounded with norm at most one,
	\item $S_\omega^{ij}$ is of weak type $(1,1)$ with norm $O(i)$,
	\item For $1<p<\8$
	$$ \|S_\omega^{ij}\|_{L_p(\mathbb{R}^n)\rightarrow L_p(\mathbb{R}^n)}\lesssim_{n, p} i+j.   $$
\end{enumerate} 
The reader is referred to \cite{TH1} and \cite{TH2} for more information.

Recall that in this paper, $T:L_2(\mathbb{R}^n)\to L_2(\mathbb{R}^n)$ is always assumed to be bounded and its kernel satisfies the estimates (\ref{standard}). The following is the dyadic representation of singular integral operators discovered by Hyt\"{o}nen in \cite{TH1} and \cite{TH2} (see {\cite[Theorem 3.3]{TH2}}).
\begin{thm}\label{CZdec}
	Let $T$ be a bounded singular integral operator.
	Then $T$ has a dyadic expansion, say for $f,g\in L_2(\mathbb{R}^n)$,
	\begin{equation}\label{T}
		\begin{aligned}
			\langle g,T(f)\rangle{}&=C_1(T)\mathbb{E}_\omega\sum_{i,j=0\atop \max\{i,j\}>0}^\infty \tau(i,j)\langle g,S_\omega^{ij}(f)\rangle+C_2(T)\mathbb{E}_\omega\langle g,S_\omega^{00}(f)\rangle\\&\quad+\mathbb{E}_\omega \langle g,\pi_{T(1)}^\omega (f)\rangle+\mathbb{E}_\omega \langle g,(\pi_{T^*(1)}^\omega)^* (f)\rangle,
		\end{aligned}
	\end{equation}
	where $S_\omega^{ij}$ is the dyadic shift of parameters $(i,j)$ on the dyadic system $\mathcal{D}^\omega$, $\pi_b^\omega$ is the dyadic martingale paraproduct on the dyadic system $\mathcal{D}^\omega$ associated with the $BMO$-function $b\in\{T(1),T^*(1)\}$, $C_1(T)$, $C_2(T)$ are positive constants depending on $T$, and $\tau(i,j)$ satisfies 
	\begin{equation*}
		0\leq \tau(i,j)\lesssim (1+\max\{i,j\})^{2(n+\alpha)}2^{-\alpha \max\{i,j\}}.
	\end{equation*}
\end{thm}

\begin{rem}
	Note that $S_\omega^{ij}$ is always contractive on $L_2(\mathbb{R}^n, L_2(\M))$ for all $w$ and $i, j$. Besides, the assumption that $T(1),T^*(1)\in BMO(\mathbb{R}^n)$ implies that $\pi_{T(1)}^\omega$ and $(\pi_{T^*(1)}^\omega)^* $ are still bounded on $L_2(\mathbb{R}^n, L_2(\M))$. This yields that $T$ is bounded on $L_2(\mathbb{R}^n, L_2(\M))$, and (\ref{T}) also holds for any $f,g\in L_2(\mathbb{R}^n, L_2(\M))$.
\end{rem}

The dyadic system $\mathcal{ D}^\omega$ on $\mathbb{R}^n$ can be regarded as the $2^n$-adic system by our definition of $d$-adic martingales (see Subsection \ref{sec2.3}). So we can define the martingale Besov space $\pmb{B}_p^{\omega, 2^n}(\mathbb{R}^n,\mathcal{M})$ on $\mathbb{R}^n$ by virtue of $H_I^\eta$ similarly as in Definition \ref{mbs1}. More precisely, $\pmb{B}_p^{\omega, 2^n}(\mathbb{R}^n,\M)$ $(0<p<\8)$ associated with semicommutative dyadic martingales on $\mathbb{R}^n$ is the space of all operator-valued Haar multipliers $b=(\langle H_I^\eta,b\rangle)_{I\in \mathcal{D}^\omega, \eta\in\{0,1\}^n_0}\subset L_0(\M)$ such that
\begin{equation}\label{Bpw2nrnm}
	\|b\|_{\pmb{B}_p^{\omega, 2^n}(\mathbb{R}^n,\M)}=\biggl(\sum_{I\in \mathcal{D}^\omega}\sum_{\eta\in\{0,1\}^n_0}\frac{\|\langle H_I^\eta,b\rangle\|^p_{L_p{(\mathcal{M})}}}{|I|^{p/2}}\biggr)^{1/p}<\infty.
\end{equation}

In addition, Theorem \ref{thm1.2}, Lemma \ref{TLambdab}, Proposition \ref{T0est} also hold for the dyadic system on $\mathbb{R}^n$ with $d=2^n$ since our proof only depends on the martingale structure, and does not depend on the dimension of the Euclidean space. 

The following lemma shows that $\|b\|_{\pmb{B}_p^{\omega, 2^n}(\mathbb{R}^n,\mathcal{M})}$ can be dominated by $\|b\|_{\pmb{B}_p(\mathbb{R}^n,L_p(\mathcal{M}))}$. The converse to this lemma can be found in Proposition \ref{pdayun}.

\begin{lemma}\label{Comparison}
	Let $1\leq p<\infty$. If $b\in\pmb{B}_p(\mathbb{R}^n,L_p(\mathcal{M}))$, then $b\in\pmb{B}_p^{\omega, 2^n}(\mathbb{R}^n,\mathcal{M})$. Moreover, in this case, $\|b\|_{\pmb{B}_p^{\omega, 2^n}(\mathbb{R}^n,\mathcal{M})}\lesssim_{n} \|b\|_{\pmb{B}_p(\mathbb{R}^n,L_p(\mathcal{M}))}$.
\end{lemma}	

\begin{proof}
	Without loss of generality, assume $\omega=0$. For any given $J\in\mathcal{D}^0$ and $\eta\in\{0,1\}^n_0$,
	\begin{equation*}
		\begin{aligned}
			\frac{\|\langle H_J^\eta,b\rangle\|_{L_p(\mathcal{M})}}{|J|^{1/2}}
			{}&=\frac{1}{|J|^{1/2}}\bigg\|\biggl\langle H_J^\eta,b-\biggl\langle \frac{\mathbbm{1}_{J}}{|J|},b\biggr\rangle\biggr\rangle\bigg\|_{L_p(\mathcal{M})}\\
			&\le \frac{1}{|J|}\int_J\biggl\|b(x)-\biggl\langle \frac{\mathbbm{1}_J}{|J|},b\biggr\rangle\biggr\|_{L_p(\mathcal{M})} dx\\
			&\le \frac{1}{|J|^2}\int_{J\times J}\|b(x)-b(y)\|_{L_p(\mathcal{M})}dxdy.
		\end{aligned}
	\end{equation*}
	Given $t\in\mathbb{Z}$, $I\in \mathcal{D}^0_t$ and $\eta\in\{0,1\}^n_0$, by the H\"{o}lder inequality, one has
	\begin{equation*}
		\begin{aligned}
			\sum_{J\in\mathcal{D}^0(I)}\frac{\|\langle H_J^\eta,b\rangle\|_{L_p(\mathcal{M})}^p}{|J|^{p/2}}{}&\le  \sum_{J\in\mathcal{D}^0(I)} \frac{1}{|J|^2}\int_{J\times J}\|b(x)-b(y)\|_{L_p(\mathcal{M})}^pdxdy\\
			&=\sum_{s=t}^\infty\sum_{J\in\mathcal{D}^0_s(I)}\frac{1}{(2^{n(t-s)}|I|)^2}\int_{J\times J}\|b(x)-b(y)\|_{L_p(\mathcal{M})}^pdxdy\\
			&=\frac{1}{|I|^2}\int_{\mathbb{R}^n\times\mathbb{R}^n}K_I(x,y)\|b(x)-b(y)\|_{L_p(\mathcal{M})}^pdxdy,
		\end{aligned}
	\end{equation*}
	where
	\begin{equation*}
		K_I(x,y)=\sum_{s=t}^\infty\sum_{J\in\mathcal{D}^0_s(I)}2^{2n(s-t)}\mathbbm{1}_J(x)\mathbbm{1}_J(y).
	\end{equation*}
	Clearly,  if $x\notin I$ or $y\notin I$, then $K_I(x,y)=0$. On the other hand, suppose that $x,y\in I$, and $|x-y|>\sqrt{n}\ell(J)$ for some $J\in \mathcal{D}^0_s(I)$. Then $\exists 1\le k\le n$ such that $|x_k-y_k|>\ell(J)$, where $x_k$ is the $k$-th coordinate of $x$. We then deduce that $\mathbbm{1}_J(x)\mathbbm{1}_J(y)=0$. Hence
	\begin{equation*}
		K_I(x,y)\le \mathbbm{1}_I(x)\mathbbm{1}_I(y)\sum_{s=t}^{t+\lfloor\log_2(\sqrt{n}\ell(I)/|x-y|)\rfloor}2^{2n(s-t)}\le \frac{(4n)^n}{4^n-1}\frac{|I|^2}{|x-y|^{2n}}\mathbbm{1}_I(x)\mathbbm{1}_I(y),
	\end{equation*}
	where $\lfloor\cdot\rfloor$ is the floor function.
	
	Therefore, for a given $t\in\mathbb{Z}$, we sum up all $I\in \mathcal{ D}^0_t$, and obtain
	\begin{equation*}
		\sum_{I\in\mathcal{D}^0_t}\sum_{J\in\mathcal{D}^0(I)}\sum_{\eta}\frac{\|\langle H_J^\eta,b\rangle\|_{L_p(\mathcal{M})}^p}{|J|^{p/2}}\le\frac{(4n)^n(2^n-1)}{4^n-1}\int_{\mathbb{R}^n\times\mathbb{R}^n}\frac{\|b(x)-b(y)\|_{L_p(\mathcal{M})}^p}{|x-y|^{2n}}dxdy.
	\end{equation*}
	Finally letting $t\to -\infty$, we have
	\begin{equation*}
		\|b\|^p_{\pmb{B}_p^{0, 2^n}(\mathbb{R}^n,\mathcal{M})}\lesssim_{n}\|b\|^p_{\pmb{B}_p(\mathbb{R}^n,L_p(\mathcal{M}))}.
	\end{equation*}

\end{proof}	

\subsection{Proof of Theorem \ref{thm6.4}}
The following two lemmas will also be needed for the proof of Theorem \ref{thm6.4}. Before formulating them, we introduce some definitions. Let $(e_i)_{i\in \mathbb{N}}$ be the standard orthonormal basis on $\ell_2$. For any $A\in L_p(B(\ell_2)\otimes \mathcal{M})$, denote by $A_{i, j}$ the $(i, j)$-th entry defined as
$$ A_{i, j}= \la e_i, A e_j\ra\in L_p(\M).  $$

\begin{lemma}\label{lem5.1.9AAA}
	Suppose that $(A_\gamma)_{\gamma\in \Gamma}$ is a net of operators in $L_p(B(\ell_2)\otimes \mathcal{M})$ $(1\leq p<\8)$. Let $A\in L_p(B(\ell_2)\otimes \mathcal{M})$. If for any $i, j\in \mathbb{N}$ and $x\in L_{p'}(\M)$
	$$ \lim_{\gamma} \tau\bigg( \big((A_\gamma)_{i, j}-A_{i, j}\big) x\bigg)=0,  $$
	then
	$$ \|A\|_{L_p(B(\ell_2)\otimes \mathcal{M})} \leq \sup_{\gamma}  \|A_\gamma\|_{L_p(B(\ell_2)\otimes \mathcal{M})}. $$
\end{lemma}
\begin{proof}
	Note that for any projection $\rho\in B(\ell_2)$ with finite rank, one has for any $B\in L_{p'}(B(\ell_2)\otimes \mathcal{M})$
	$$  \text{Tr}\otimes \tau ((\rho\otimes 1_{\M})A(\rho\otimes 1_{\M}) B) =\lim_\gamma  \text{Tr}\otimes \tau ((\rho\otimes 1_{\M})A_\gamma(\rho\otimes 1_{\M})B).  $$
	By duality, this implies that
	$$  \|(\rho\otimes 1_{\M})A(\rho\otimes 1_{\M})\|_{L_p(B(\ell_2)\otimes \mathcal{M})} \leq \sup_\gamma  \|(\rho\otimes 1_{\M})A_\gamma(\rho\otimes 1_{\M})\|_{L_p(B(\ell_2)\otimes \mathcal{M})}.   $$
	Therefore, we have
	\begin{equation*}
		\begin{aligned}
			\|A\|_{L_p(B(\ell_2)\otimes \mathcal{M})}&=\sup_{\substack{\rho^2=\rho=\rho* \\ \text{finite rank}}}\|(\rho\otimes 1_{\M})A(\rho\otimes 1_{\M})\|_{L_p(B(\ell_2)\otimes \mathcal{M})} \\
			&\leq \sup_{\substack{\rho^2=\rho=\rho* \\ \text{finite rank}}}\sup_\gamma  \|(\rho\otimes 1_{\M})A_\gamma(\rho\otimes 1_{\M})\|_{L_p(B(\ell_2)\otimes \mathcal{M})}\\
			&\leq \sup_{\gamma}  \|A_\gamma\|_{L_p(B(\ell_2)\otimes \mathcal{M})},
		\end{aligned}
	\end{equation*}
	as desired.
\end{proof}

Let $T\in B(L_2(\mathbb{R}^n))$. Then $T\otimes Id_{L_2(\M)}$ extends to a bounded operator on $L_2(\mathbb{R}^n, L_2(\M))$. We still denote it by $T$ for simplicity. Thus by virtue of continuity and linearity, $T$ satisfies the following properties: for any $f\in \mathcal{S}(L^\8(\mathbb{R}^n)\otimes\M)$ and $x\in \M$
\begin{equation}\label{Tid}
	T(f)x=T(fx) \quad \text{and} \quad \tau (T(fx))= T(\tau(fx)).
\end{equation}

\begin{lemma}\label{lem5.1.10AAA}
	Suppose that $(T_\gamma)_{\gamma\in \Gamma}$ is a bounded net of operators in $B(L_p(\mathbb{R}^n))$ for each $p\in (1, \8)$. Assume that $(T_\gamma)_{\gamma\in \Gamma}$ converges to $T\in B(L_p(\mathbb{R}^n))$ with respect to the weak operator topology for each $p\in (1, \8)$. If $1< p<\8$ and $b\in L_p(\mathbb{R}^n, L_p(\mathcal{M}))$, then
	$$  \| C_{T, b}\|_{L_p(B(L_2(\mathbb{R}^n))\otimes \mathcal{M})} \leq \sup_\gamma  \| C_{T_\gamma, b}\|_{L_p(B(L_2(\mathbb{R}^n))\otimes \mathcal{M})}. $$
\end{lemma}
\begin{proof}
	First we show that for any finite cubes $I, J\subset \mathbb{R}^n $ and $x\in L_{p'}(\M)$, one has
	$$  \lim_\gamma\tau \bigg( \big\la \mathbbm{1}_I, C_{T_\gamma, b}( \mathbbm{1}_J)\big\ra x \bigg)= \tau \bigg( \big\la \mathbbm{1}_I, C_{T, b}( \mathbbm{1}_J)\big\ra x \bigg). $$
	Note that by \eqref{Tid}
	\begin{equation*}
		\begin{aligned}
			\tau \bigg( \big\la \mathbbm{1}_I, C_{T_\gamma, b}( \mathbbm{1}_J)\big\ra x \bigg)&= \tau \bigg( \big\la \mathbbm{1}_I, T_\gamma( b\mathbbm{1}_J)- bT_\gamma( \mathbbm{1}_J) \big\ra x \bigg)\\
			&= \tau \big( \big\la \mathbbm{1}_I, T_\gamma( bx\mathbbm{1}_J)\big\ra\big)- \tau \big( \big\la x^* b^*\mathbbm{1}_J, T_\gamma( \mathbbm{1}_J) \big\ra \big)\\
			&=\big\la \mathbbm{1}_I, T_\gamma( \tau(bx) \mathbbm{1}_J)\big\ra- \big\la \tau(x^* b^*)\mathbbm{1}_J, T_\gamma( \mathbbm{1}_J) \big\ra.
		\end{aligned}
	\end{equation*}
	This implies that
	\begin{equation*}
		\begin{aligned}
			\lim_\gamma\tau \bigg( \big\la \mathbbm{1}_I, C_{T_\gamma, b}( \mathbbm{1}_J)\big\ra x \bigg)&= \big\la \mathbbm{1}_I, T( \tau(bx) \mathbbm{1}_J)\big\ra- \big\la \tau(x^* b^*)\mathbbm{1}_J, T( \mathbbm{1}_J) \big\ra\\
			& =\tau \bigg( \big\la \mathbbm{1}_I, C_{T, b}( \mathbbm{1}_J)\big\ra x \bigg).
		\end{aligned}
	\end{equation*}
	Hence, the desired result follows from Lemma \ref{lem5.1.9AAA}.
\end{proof}

The remaining of the section is devoted to the proof of Theorem \ref{thm6.4}. For simplicity, we will still denote $B(L_2(\mathbb{R}^n))\otimes \mathcal{M}$ by $\mathcal{N}$, where $\mathbb{R}$ is replaced by $\mathbb{R}^n$.

\begin{proof}[Proof of Theorem \ref{thm6.4}]
	From  Proposition \ref{T0est} and Lemma \ref{Comparison}, we have
	\begin{equation*}
		\|[\pi_{T(1)}^{\omega}, M_b]\|_{L_p(\mathcal{N})}\lesssim_{n,p}\|T(1)\|_{BMO(\mathbb{R}^n)} \|b\|_{\pmb{B}_p(\mathbb{R}^n,L_p(\mathcal{M}))}
	\end{equation*}
	and
	\begin{equation*}
		\|[(\pi_{T^*(1)}^{\omega})^*, M_b]\|_{L_p(\mathcal{N})}\lesssim_{n,p} \|T^*(1)\|_{BMO(\mathbb{R}^n)} \|b\|_{\pmb{B}_p(\mathbb{R}^n,L_p(\mathcal{M}))}.
	\end{equation*}
	Hence by Theorem \ref{CZdec}, it remains to estimate $\|[S_\omega^{ij}, M_b]\|_{L_p(\mathcal{N})}$ for any $i, j\in \mathbb{N}\cup\{0\}$. By the triangle inequality
	\begin{equation*}
		\begin{aligned}
			\|[S_\omega^{ij},M_b]\|_{L_p(\mathcal{N})}{}&\le \|[S_\omega^{ij},\pi_b]\|_{L_p(\mathcal{N})}+\|[S_\omega^{ij},\varLambda_b]\|_{L_p(\mathcal{N})}
			+\|[S_\omega^{ij},R_b]\|_{L_p(\mathcal{N})}.
		\end{aligned}
	\end{equation*}
	Here the operators $\pi_b$, $\varLambda_b$ and $R_b$ are associated with the dyadic system $\mathcal{ D}^\omega$.
	
	From Theorem \ref{thm1.2} and Lemma \ref{TLambdab}, we know that
	\begin{equation*}
		\begin{aligned}
			\|\pi_b\|_{L_p(\mathcal{N})}\lesssim_{n, p}\|b\|_{\pmb{B}_p^{\omega, 2^n}(\mathbb{R}^n,\mathcal{M})},\quad \|\varLambda_b\|_{L_p(\mathcal{N})}\lesssim_{n, p}\|b\|_{\pmb{B}_p^{\omega, 2^n}(\mathbb{R}^n,\mathcal{M})}.
		\end{aligned}
	\end{equation*}
	Meanwhile, recall that $S_\omega^{ij}\in B(L_2(\mathbb{R}^n,L_2(\mathcal{M})))$ is contractive. Thus, using Lemma \ref{Comparison}, one gets
	\begin{equation*}
		\begin{aligned}
			\|[S_\omega^{ij},\pi_b]\|_{L_p(\mathcal{N})}&\lesssim \|S_\omega^{ij}\|\|\pi_b\|_{L_p(\mathcal{N})}
			\lesssim_{n,p} \|b\|_{\pmb{B}_p^{\omega, 2^n}(\mathbb{R}^n,\mathcal{M})}\lesssim_{n} \|b\|_{\pmb{B}_p(\mathbb{R}^n,L_p(\mathcal{M}))}.
		\end{aligned}
	\end{equation*}
	Similarly,
	$$ \|[S_\omega^{ij},\varLambda_b]\|_{L_p(\mathcal{N})}\lesssim_{n,p} \|b\|_{\pmb{B}_p(\mathbb{R}^n,L_p(\mathcal{M}))}.$$
	
	For any $i, j\in \mathbb{N}\cup\{0\}$, we will show that $\|[S^{ij}_\omega,R_b]\|_{L_p(\mathcal{N})}$ increases with polynomial growth with respect to $i$ and $j$ uniformly on $\omega$. Then from Theorem \ref{CZdec} and the triangle inequality, the desired result will follow.
	
	Without loss of generality, we assume $\omega=0$.
	Let $\varPhi=[S^{ij}_0,R_b]$. Then
	\begin{equation}\label{Ndef}
		\begin{aligned}
			\varPhi(f)=\sum_{K\in\mathcal{D}^0}\sum_{\substack{I,J\in\mathcal{D}^0;I,J\subseteq K\\ \ell(I)=2^{-i}\ell(K)\\\ell(J)=2^{-j}\ell(K)}}\sum_{\xi,\eta\in\{0,1\}^n_0}a^{\xi\eta}_{IJK}\langle H^\xi_I,R_b(f)\rangle H^\eta_J-\sum_{k\in \mathbb{Z}}b_{k-1}d_k(S^{ij}_0(f)).
		\end{aligned}
	\end{equation}
	Note that for any $k\in\mathbb{Z}$ and $\xi\in\{0,1\}^n_0$, if $I\in\mathcal{D}^0_k$, then $d_{k+1}H^\xi_I=H^\xi_I$. Hence, for any $I\in \mathcal{ D}_k^0$,
	\begin{equation*}
		\begin{aligned}
			\langle H^\xi_I,R_b(f)\rangle{}&=\biggl\langle H^\xi_I,\sum_{l\in\mathbb{Z}}b_{l-1}d_lf\biggr\rangle
			=\biggl\langle d_{k+1}H^\xi_I,\sum_{l\in\mathbb{Z}}b_{l-1}d_lf\biggr\rangle\\
			&=\biggl\langle H^\xi_I,d_{k+1}\biggl(\sum_{l\in\mathbb{Z}}b_{l-1}d_lf\biggr)\biggr\rangle
			=\langle H^\xi_I,b_{k}d_{k+1}f\rangle\\
			&=\langle H^\xi_I,b_{k}\mathbbm{1}_Id_{k+1}f\rangle
			=\biggl\langle H^\xi_I,\biggl\langle \frac{\mathbbm{1}_I}{|I|},b\biggr\rangle\mathbbm{1}_Id_{k+1}f\biggr\rangle\\
			&=\biggl\langle \frac{\mathbbm{1}_I}{|I|},b\biggr\rangle\langle H^\xi_I,d_{k+1}f\rangle=\biggl\langle \frac{\mathbbm{1}_I}{|I|},b\biggr\rangle\langle H^\xi_I,f\rangle.
		\end{aligned}
	\end{equation*}
	For the second term in \eqref{Ndef}, one has
	\begin{equation*}
		\begin{aligned}
			\sum_{k\in \mathbb{Z}}b_{k-1}d_k(S^{ij}_0(f)){}&=\sum_{k\in \mathbb{Z}}b_{k-1}d_k\biggl(\sum_{K\in\mathcal{D}^0}\sum_{\substack{I,J\in\mathcal{D}^0;I,J\subseteq K\\ \ell(I)=2^{-i}\ell(K)\\\ell(J)=2^{-j}\ell(K)}}\sum_{\xi,\eta}a_{IJK}^{\xi\eta}\langle H^\xi_I,f\rangle H^\eta_J\biggr)\\&=\sum_{k\in \mathbb{Z}}\sum_{K\in\mathcal{D}^0_{k-1-j}}\sum_{\substack{I,J\in\mathcal{D}^0;I,J\subseteq K\\ \ell(I)=2^{-i}\ell(K)\\\ell(J)=2^{-j}\ell(K)}}\sum_{\xi,\eta}a_{IJK}^{\xi\eta}b_{k-1}\langle H^\xi_I,f\rangle H^\eta_J\\
			&=\sum_{k\in \mathbb{Z}}\sum_{K\in\mathcal{D}^0_{k-1-j}}\sum_{\substack{I,J\in\mathcal{D}^0;I,J\subseteq K\\ \ell(I)=2^{-i}\ell(K)\\\ell(J)=2^{-j}\ell(K)}}\sum_{\xi,\eta}a_{IJK}^{\xi\eta}\biggl\langle \frac{\mathbbm{1}_J}{|J|},b\biggr\rangle\langle H^\xi_I,f\rangle  H^\eta_J\\
			&=\sum_{K\in\mathcal{D}^0}\sum_{\substack{I,J\in\mathcal{D}^0;I,J\subseteq K\\ \ell(I)=2^{-i}\ell(K)\\\ell(J)=2^{-j}\ell(K)}}\sum_{\xi,\eta}a_{IJK}^{\xi\eta}\biggl\langle \frac{\mathbbm{1}_J}{|J|},b\biggr\rangle \langle H^\xi_I,f\rangle H^\eta_J.
		\end{aligned}
	\end{equation*}
	Therefore,
	\begin{equation}\label{BKdef}
		\begin{aligned}
			\varPhi(f){}&=\sum_{K\in\mathcal{D}^0}\sum_{\substack{I,J\in\mathcal{D}^0;I,J\subseteq K\\ \ell(I)=2^{-i}\ell(K)\\\ell(J)=2^{-j}\ell(K)}}\sum_{\xi,\eta}a_{IJK}^{\xi\eta}\biggl(\biggl\langle \frac{\mathbbm{1}_I}{|I|},b\biggr\rangle-\biggl\langle \frac{\mathbbm{1}_J}{|J|},b\biggr\rangle\biggr)\langle H^\xi_I,f\rangle H^\eta_J
			=:\sum_{K\in\mathcal{D}^0}B_K(f).
		\end{aligned}
	\end{equation}
	Let $b_{IJ}= \langle \frac{\mathbbm{1}_I}{|I|},b\rangle-\langle \frac{\mathbbm{1}_J}{|J|},b\rangle$.
	Since $B_{K_1}$, $B_{K_2}$ have orthogonal ranges when $K_1\neq K_2$, we see
	\[B_{K_1}^*B_{K_2}=0,\quad \forall K_1\neq K_2, K_1,K_2\in\mathcal{D}^0,\]
	which yields $\varPhi^*\varPhi=\sum\limits_{K\in \mathcal{ D}^0}B_K^*B_K$.
	Note that $\forall f\in L_2(\mathbb{R}^n,L_2(\mathcal{M}))$, 
	\begin{equation}\label{BkBK}
		B_K^*B_K(f)=\sum_{\substack{I,\tilde{I},J\in\mathcal{D}^0;I,\tilde{I},J\subseteq K\\ \ell(I)=\ell(\tilde{I})=2^{-i}\ell(K)\\\ell(J)=2^{-j}\ell(K)}}\sum_{\xi,\tilde{\xi},\eta}a_{IJK}^{\xi\eta}\overline{a_{\tilde{I}JK}^{\tilde{\xi}\eta}}b_{\tilde{I}J}^* b_{IJ}\la H^\xi_I,f\rangle H^{\tilde{\xi}}_{\tilde{I}},
	\end{equation}   
	which implies that $\varPhi^*\varPhi$ is a block diagonal matrix with blocks $B_K^*B_K$ for all $K\in \mathcal{ D}^0$.
	Consequently, we have
	\begin{equation}\label{Nf1}
		\begin{aligned}
			\|\varPhi\|_{L_p(\mathcal{N})}^p&=\|\varPhi^*\varPhi\|_{L_{p/2}(\mathcal{N})}^{p/2}
			=\sum_{k\in\mathbb{Z}}\sum_{K\in\mathcal{D}^0_k}\|B^*_KB_K\|_{L_{p/2}(\mathcal{N})}^{p/2}.
		\end{aligned}
	\end{equation}
Denote by 
\begin{equation*}
	[B_K^*B_K]=\biggl((B_K^*B_K)_{(\tilde{I},\tilde{\zeta}),(I,\zeta)}\biggr)_{\tilde{I},I\in \mathcal{D}^0;\tilde{I},I\subseteq K , \ell(\tilde{I})=\ell(I)=2^{-i}\ell(K), \tilde{\zeta},\zeta\in\{0,1\}^n_0}
\end{equation*}
the matrix form of $B_K^*B_K$ with respect to the basis $\{H_I^\zeta\}_{I\in\mathcal{D}^0;I\subseteq K , \ell(I)=2^{-i}\ell(K), \zeta\in\{0,1\}^n_0}$, where $(B_K^*B_K)_{(\tilde{I},\tilde{\zeta}),(I,\zeta)}=\langle  H_{\tilde{I}}^{\tilde{\zeta}},B_K^*B_K H_I^\zeta\rangle$. We also denote the  $2^{in}(2^n-1)\times 2^{in}(2^n-1)$ matrix by
\begin{equation}\label{W}
	W^{K,J,\eta}=\bigg(W_{(\tilde{I},\tilde{\zeta}),(I,\zeta)}^{K,J,\eta}\bigg)_{\tilde{I},I\in \mathcal{D}^0;\tilde{I},I\subseteq K , \ell(\tilde{I})=\ell(I)=2^{-i}\ell(K), \tilde{\zeta},\zeta\in\{0,1\}^n_0},
\end{equation}
where $W_{(\tilde{I},\tilde{\zeta}),(I,\zeta)}^{K,J,\eta}=a_{IJK}^{\zeta\eta}\overline{a_{\tilde{I}JK}^{\tilde{\zeta}\eta}}b_{\tilde{I}J}^* b_{IJ}$. Then by \eqref{BkBK} one has
\begin{equation*}
	[B_K^*B_K]=\sum_{\substack{J\in\mathcal{D}^0;J\subseteq K\\ \ell(J)=2^{-j}\ell(K)}}\sum_{\eta} W^{K,J,\eta}.
\end{equation*}

	Now we divide the proof into two cases: $p\geq 2$ for $n\geq 1$ and $\frac{2}{1+\alpha}< p<2$ for $n=1$.\\
	(1) When $p\geq 2$, using the triangle inequality, one has
	\begin{equation*}
		\begin{aligned}
			\|B^*_KB_K\|_{L_{p/2}(\mathcal{N})}{}&=\|[B_K^*B_K]\|_{L_{p/2}(\mathbb{M}_{2^{in}(2^n-1)}\otimes\mathcal{M})}\le \sum_{\substack{J\in\mathcal{D}^0;J\subseteq K\\ \ell(J)=2^{-j}\ell(K)}}\sum_{\eta}\big\|W^{K,J,\eta}\big\|_{L_{p/2}(\mathbb{M}_{2^{in}(2^n-1)}\otimes\mathcal{M})}.
		\end{aligned}
	\end{equation*}
Notice that
\begin{equation}\label{WVV}
	\begin{aligned}
		\big\|W^{K,J,\eta}\big\|_{L_{p/2}(\mathbb{M}_{2^{in}(2^n-1)}\otimes\mathcal{M})}{}
		&=\bigg\|\sum_{\substack{{I}\in\mathcal{D}^0;{I}\subseteq K \\ \ell({I})=2^{-i}\ell(K)\\ {\zeta}\in\{0,1\}^n_0}}a_{{I}JK}^{{\zeta}\eta}\overline{a_{{I}JK}^{{\zeta}\eta}}b_{{I}J}b_{{I}J}^* \bigg\|_{L_{p/2}(\mathcal{M})}\\
		&\le\sum_{\substack{{I}\in\mathcal{D}^0;{I}\subseteq K \\ \ell({I})=2^{-i}\ell(K)\\ {\zeta}\in\{0,1\}^n_0}}\Big\|a_{{I}JK}^{{\zeta}\eta} b_{{I}J}\Big\|_{L_{p}(\mathcal{M})}^2.
	\end{aligned}
\end{equation}
Besides, note that
$|a_{IJK}^{\xi\eta}|\le 2^{-(i+j)n/2}$ in \eqref{aijkxieta}. Hence by (\ref{Nf1}) and (\ref{WVV}),
\begin{equation}\label{Nf}
	\begin{aligned}
		\|\varPhi\|_{L_p(\mathcal{N})}^p{}&\le \sum_{k\in\mathbb{Z}}\sum_{K\in\mathcal{D}^0_k}\biggl(\sum_{\substack{J\in\mathcal{D}^0;J\subseteq K\\ \ell(J)=2^{-j}\ell(K)}}\sum_{\eta}\big\|W^{K,J,\eta}\big\|_{L_{p/2}(\mathbb{M}_{2^{in}(2^n-1)}\otimes\mathcal{M})}\biggr)^{p/2}\\
		&\le\sum_{k\in\mathbb{Z}}\sum_{K\in\mathcal{D}^0_k}\biggl(\sum_{\substack{J\in\mathcal{D}^0;J\subseteq K\\ \ell(J)=2^{-j}\ell(K)}}\sum_{\eta}\sum_{\substack{I\in\mathcal{D}^0;I\subseteq K \\ \ell(I)=2^{-i}\ell(K)}}\sum_{ \xi}\Big\|a_{IJK}^{\xi\eta} b_{IJ}\Big\|_{L_{p}(\mathcal{M})}^2\biggr)^{p/2}\\
		&\le(2^n-1)^p\sum_{k\in\mathbb{Z}}\sum_{K\in\mathcal{D}^0_k}\biggl(2^{-(i+j)n}\sum_{\substack{J\in\mathcal{D}^0;J\subseteq K\\ \ell(J)=2^{-j}\ell(K)}}\sum_{\substack{I\in\mathcal{D}^0;I\subseteq K\\ \ell(I)=2^{-i}\ell(K)}}\|b_{IJ}\|_{L_p(\mathcal{M})}^2\biggr)^{p/2}.
	\end{aligned}
\end{equation}
	Since $b_{IJ}=b_{IK}-b_{JK}$, by the triangle inequality and the Cauchy-Schwarz inequality, we have
	\begin{equation}
		\begin{aligned}
			{}&\sum_{\substack{J\in\mathcal{D}^0;J\subseteq K\\ \ell(J)=2^{-j}\ell(K)}}\sum_{\substack{I\in\mathcal{D}^0;I\subseteq K\\ \ell(I)=2^{-i}\ell(K)}}\|b_{IJ}\|_{L_p(\mathcal{M})}^2\\
			&\leq \sum_{\substack{J\in\mathcal{D}^0;J\subseteq K\\ \ell(J)=2^{-j}\ell(K)}}\sum_{\substack{I\in\mathcal{D}^0;I\subseteq K\\ \ell(I)=2^{-i}\ell(K)}} 2(\|b_{IK}\|_{L_p(\mathcal{M})}^2+\|b_{JK}\|_{L_p(\mathcal{M})}^2)\\
			&\le 2^{jn+1}\sum_{\substack{I\in\mathcal{D}^0;I\subseteq K\\ \ell(I)=2^{-i}\ell(K)}}\|b_{IK}\|_{L_p(\mathcal{M})}^2+2^{in+1}\sum_{\substack{J\in\mathcal{D}^0;J\subseteq K\\ \ell(J)=2^{-j}\ell(K)}}\|b_{JK}\|_{L_p(\mathcal{M})}^2.
		\end{aligned}
	\end{equation}
	Note that $b_{IK}\cdot \mathbbm{1}_I=(b_{k+i}-b_{k})\cdot \mathbbm{1}_I$, and sum all $I$ and $J$, one has
	\begin{equation}
		\begin{aligned}
			{}&2^{-(i+j)n}\sum_{\substack{J\in\mathcal{D}^0;J\subseteq K\\ \ell(J)=2^{-j}\ell(K)}}\sum_{\substack{I\in\mathcal{D}^0;I\subseteq K\\ \ell(I)=2^{-i}\ell(K)}}\|b_{IJ}\|_{L_p(\mathcal{M})}^2\\
			&\le 2^{kn+1}\bigg(\sum_{\substack{I\in\mathcal{D}^0;I\subseteq K\\ \ell(I)=2^{-i}\ell(K)}}\|b_{IK}\cdot \mathbbm{1}_I\|_{L_2(\mathbb{R}^n,L_p(\mathcal{M}))}^2+\sum_{\substack{J\in\mathcal{D}^0;J\subseteq K\\ \ell(J)=2^{-j}\ell(K)}}\|b_{JK}\cdot \mathbbm{1}_J\|_{L_2(\mathbb{R}^n,L_p(\mathcal{M}))}^2\bigg)\\
			&= 2^{kn+1}\bigg(\|(b_{k+i}-b_{k})\mathbbm{1}_K\|_{L_2(\mathbb{R}^n,L_p(\mathcal{M}))}^2+\|(b_{k+j}-b_{k})\mathbbm{1}_K\|_{L_2(\mathbb{R}^n,L_p(\mathcal{M}))}^2\bigg)\\
			&\le 
			2^{1+2nk/p}\biggl(i\sum_{l=k+1}^{k+i}\|d_lb\cdot \mathbbm{1}_K\|_{L_p(\mathbb{R}^n,L_p(\mathcal{M}))}^2+j\sum_{l=k+1}^{k+j}\|d_l b\cdot \mathbbm{1}_K\|_{L_p(\mathbb{R}^n,L_p(\mathcal{M}))}^2\biggr).
		\end{aligned}
	\end{equation}
	Hence using the convex inequality, we obtain
	\begin{equation*}
		\begin{aligned}
			{}&\|\varPhi\|_{L_p(\mathcal{N})}^p\\
			&\le (2^n-1)^p2^{p}\sum_{k\in\mathbb{Z}}\sum_{K\in\mathcal{D}^0_k}2^{nk}\biggl(i^{p-1}\sum_{l=k+1}^{k+i}\|d_lb\cdot\mathbbm{1}_K\|_{L_p(\mathbb{R}^n,L_p(\mathcal{M}))}^p+j^{p-1}\sum_{l=k+1}^{k+j}\|d_lb\cdot\mathbbm{1}_K\|_{L_p(\mathbb{R}^n,L_p(\mathcal{M}))}^p\biggr)\\
			&=(2^n-1)^p2^{p}(i^{p}+j^{p})\sum_{k\in\mathbb{Z}}2^{nk}\|d_kb\|_{L_p(\mathbb{R}^n,L_p(\mathcal{M}))}^p\\
			&\lesssim_{n,p}(i^{p}+j^{p})\|b\|_{\pmb{B}_p^{0, 2^n}(\mathbb{R}^n,\mathcal{M})}^p\lesssim_{n} (i^{p}+j^{p})\|b\|_{\pmb{B}_p(\mathbb{R}^n,L_p(\mathcal{M}))}^p.
		\end{aligned}
	\end{equation*}
	Since the above estimation is independent of the choose of $\omega$, one has
	$$ \|[S_\omega^{ij}, R_b]\|_{L_p(\mathcal{N})}\lesssim_{n, p} (i^p+j^p)^{1/p}  \|b\|_{\pmb{B}_p(\mathbb{R}^n,L_p(\mathcal{M}))}, $$
	which yields
	$$ \|[S_\omega^{ij}, M_b]\|_{L_p(\mathcal{N})}\lesssim_{n, p} (i^p+j^p+1)^{1/p}  \|b\|_{\pmb{B}_p(\mathbb{R}^n,L_p(\mathcal{M}))}. $$
	Therefore by Lemma \ref{lem5.1.10AAA} and the triangle inequality,
	\begin{equation*}
		\begin{aligned}
			{}&\|[T,M_b]\|_{L_p(\mathcal{N})}\\
			&=\biggl\|\biggl[C_1(T)\mathbb{E}_\omega\sum_{i,j=0\atop \max\{i,j\}>0}^\infty \tau(i,j)S_\omega^{ij}+C_2(T)\mathbb{E}_\omega S_\omega^{00}+\mathbb{E}_\omega \pi^\omega_{T(1)}+\mathbb{E}_\omega (\pi^\omega_{T^*(1)})^*,M_b\biggr]\biggr\|_{L_p(\mathcal{N})}\\
			&\lesssim_T \sum_{i,j=0}^\infty \tau(i,j)\mathbb{E}_\omega\|[S_\omega^{ij},M_b]\|_{L_p(\mathcal{N})}+\mathbb{E}_\omega\|[ \pi^\omega_{T(1)}+ (\pi^\omega_{T^*(1)})^*,M_b]\|_{L_p(\mathcal{N})} \\
			&\lesssim_{n,p, T} \big(1+\|T(1)\|_{BMO(\mathbb{R}^n)}+\|T^*(1)\|_{BMO(\mathbb{R}^n)}\big) \|b\|_{\pmb{B}_p(\mathbb{R}^n,L_p(\mathcal{M}))}.
		\end{aligned}
	\end{equation*}
	This finishes the proof for $p\geq 2$.
\\
(2) When $n=1$ and $\frac{2}{1+\alpha}< p<2$, we use the same method as in the case $p\ge 2$. Note that in this case $p/2<1$, using the triangle inequality for $p/2$-norm,  from \eqref{WVV} and \eqref{Nf} one has
\begin{equation*}
	\begin{aligned}
		\|\varPhi\|_{L_p(\mathcal{N})}^p{}&\le \sum_{k\in\mathbb{Z}}\sum_{K\in\mathcal{D}^0_k}\sum_{\substack{J\in\mathcal{D}^0;J\subseteq K\\ \ell(J)=2^{-j}\ell(K)}}\sum_{\eta}\sum_{\substack{I\in\mathcal{D}^0;I\subseteq K \\ \ell(I)=2^{-i}\ell(K)}}\sum_{ \xi}\Big\|a_{IJK}^{\xi\eta} b_{IJ}\Big\|_{L_{p}(\mathcal{M})}^p.
	\end{aligned}
\end{equation*}
Note that
$|a_{IJK}^{\xi\eta}|\le 2^{-(i+j)/2}$ in \eqref{aijkxieta}, then we estimate 
\begin{equation*}\label{zeta}
	\begin{aligned}
		\|\varPhi\|_{L_p(\mathcal{N})}^p
		{}&\le \sum_{k\in\mathbb{Z}}\sum_{K\in\mathcal{D}^0_k}\frac{1}{2^{(i+j)p/2}}\sum_{\substack{J\in\mathcal{D}^0;J\subseteq K\\ \ell(J)=2^{-j}\ell(K)}}\sum_{\substack{I\in\mathcal{D}^0;I\subseteq K\\ \ell(I)=2^{-i}\ell(K)}}\|b_{IJ}\|_{L_p(\mathcal{M})}^p.
	\end{aligned}
\end{equation*}
Since $b_{IJ}=b_{IK}-b_{JK}$, by the triangle inequality and the Cauchy-Schwarz inequality, we have
\begin{equation}
	\begin{aligned}
		{}&\sum_{\substack{J\in\mathcal{D}^0;J\subseteq K\\ \ell(J)=2^{-j}\ell(K)}}\sum_{\substack{I\in\mathcal{D}^0;I\subseteq K\\ \ell(I)=2^{-i}\ell(K)}}\|b_{IJ}\|_{L_p(\mathcal{M})}^p\\
		&\leq \sum_{\substack{J\in\mathcal{D}^0;J\subseteq K\\ \ell(J)=2^{-j}\ell(K)}}\sum_{\substack{I\in\mathcal{D}^0;I\subseteq K\\ \ell(I)=2^{-i}\ell(K)}} 2^{p-1}(\|b_{IK}\|_{L_p(\mathcal{M})}^p+\|b_{JK}\|_{L_p(\mathcal{M})}^p)\\
		&= 2^{j+p-1}\sum_{\substack{I\in\mathcal{D}^0;I\subseteq K\\ \ell(I)=2^{-i}\ell(K)}}\|b_{IK}\|_{L_p(\mathcal{M})}^p+2^{i+p-1}\sum_{\substack{J\in\mathcal{D}^0;J\subseteq K\\ \ell(J)=2^{-j}\ell(K)}}\|b_{JK}\|_{L_p(\mathcal{M})}^p.
	\end{aligned}
\end{equation}
Note that $b_{IK}\cdot \mathbbm{1}_I=(b_{k+i}-b_{k})\cdot \mathbbm{1}_I$, and sum all $I$ and $J$, one has
\begin{equation}
	\begin{aligned}
		{}&\sum_{\substack{J\in\mathcal{D}^0;J\subseteq K\\ \ell(J)=2^{-j}\ell(K)}}\sum_{\substack{I\in\mathcal{D}^0;I\subseteq K\\ \ell(I)=2^{-i}\ell(K)}}\|b_{IJ}\|_{L_p(\mathcal{M})}^p\\
		&\le 2^{i+j+k+p-1}\big(\|(b_{k+i}-b_{k})\mathbbm{1}_K\|_{L_p(\mathbb{R}^n,L_p(\mathcal{M}))}^p+\|(b_{k+j}-b_{k})\mathbbm{1}_K\|_{L_p(\mathbb{R}^n,L_p(\mathcal{M}))}^p\big).
	\end{aligned}
\end{equation}
Hence we obtain
\begin{equation*}
	\begin{aligned}
		\|\varPhi\|_{L_p(\mathcal{N})}^p{}
		&\le \sum_{k\in\mathbb{Z}}\sum_{K\in\mathcal{D}^0_k}\frac{2^{p-1}i^{p-1}}{2^{(i+j)(p/2-1)-k}}\sum_{l=k+1}^{k+i}\|d_lb\cdot \mathbbm{1}_K\|_{L_p(\mathbb{R}^n,L_p(\mathcal{M}))}^p\\
		&\quad+\sum_{k\in\mathbb{Z}}\sum_{K\in\mathcal{D}^0_k}\frac{2^{p-1}j^{p-1}}{2^{(i+j)(p/2-1)-k}}\sum_{l=k+1}^{k+j}\|d_lb\cdot \mathbbm{1}_K\|_{L_p(\mathbb{R}^n,L_p(\mathcal{M}))}^p\\
		&=2^{p-1} 2^{(i+j)(1-p/2)} (i^p+j^p) \sum_{k\in\mathbb{Z}}2^{k}\|d_kb\|_{L_p(\mathbb{R}^n,L_p(\mathcal{M}))}^p\\
		&\lesssim_{p}2^{(i+j)(1-p/2)}(i^{p}+j^{p})\|b\|_{\pmb{B}_p^{0, 2}(\mathbb{R}^n,\mathcal{M})}^p\lesssim 2^{(i+j)(1-p/2)}(i^{p}+j^{p})\|b\|_{\pmb{B}_p(\mathbb{R}^n,L_p(\mathcal{M}))}^p.
	\end{aligned}
\end{equation*}
Since the above estimation is independent of the choose of $\omega$, one has
$$ \|[S_\omega^{ij}, R_b]\|_{L_p(\mathcal{N})}\lesssim_{p} \big(2^{(i+j)(1-p/2)}(i^p+j^p)\big)^{1/p}  \|b\|_{\pmb{B}_p(\mathbb{R}^n,L_p(\mathcal{M}))}, $$
which yields
\begin{equation*}
	\begin{aligned}
		\|[S_\omega^{ij}, M_b]\|_{L_p(\mathcal{N})}{}&\lesssim_{p} \big(2^{(i+j)(1-p/2)}(i^p+j^p)+1\big)^{1/p}  \|b\|_{\pmb{B}_p(\mathbb{R}^n,L_p(\mathcal{M}))}\\
		&\lesssim 2^{2\max\{i,j\}(1-p/2)/p}(i+j+1) \|b\|_{\pmb{B}_p(\mathbb{R}^n,L_p(\mathcal{M}))}.
	\end{aligned}
\end{equation*}
Since $\frac{2}{1+\alpha}< p<2$, we get
\begin{equation*}
	\begin{aligned}
		\sum_{i,j=0}^\infty &\tau(i,j)\|[S_\omega^{ij}, M_b]\|_{L_p(\mathcal{N})}\\
		&\lesssim_{p} \sum_{i,j=0}^\infty (1+\max\{i,j\})^{2(1+\alpha)+1}2^{\max\{i,j\}\big(2(1-p/2)/p-\alpha\big)}\|b\|_{\pmb{B}_p(\mathbb{R}^n,L_p(\mathcal{M}))}<\infty.
	\end{aligned}
\end{equation*}
Therefore by Lemma \ref{lem5.1.10AAA} and the triangle inequality,
\begin{equation*}
	\begin{aligned}
		{}&\quad\|[T,M_b]\|_{L_p(\mathcal{N})}\\
		&=\biggl\|\biggl[C_1(T)\mathbb{E}_\omega\sum_{i,j=0\atop \max\{i,j\}>0}^\infty \tau(i,j)S_\omega^{ij}+C_2(T)\mathbb{E}_\omega S_\omega^{00}+\mathbb{E}_\omega \pi^\omega_{T(1)}+\mathbb{E}_\omega (\pi^\omega_{T^*(1)})^*,M_b\biggr]\biggr\|_{L_p(\mathcal{N})}\\
		&\lesssim_T \sum_{i,j=0}^\infty \tau(i,j)\mathbb{E}_\omega\|[S_\omega^{ij},M_b]\|_{L_p(\mathcal{N})}+\mathbb{E}_\omega\|[ \pi^\omega_{T(1)}+ (\pi^\omega_{T^*(1)})^*,M_b]\|_{L_p(\mathcal{N})} \\
		&\lesssim_{p, T} \big(1+\|T(1)\|_{BMO(\mathbb{R}^n)}+\|T^*(1)\|_{BMO(\mathbb{R}^n)}\big) \|b\|_{\pmb{B}_p(\mathbb{R}^n,L_p(\mathcal{M}))}.
	\end{aligned}
\end{equation*}
This completes the proof of Theorem \ref{thm6.4}.
\end{proof}

\subsection{Comparison between  Theorem \ref{thm1.5} and Theorem \ref{thm6.4}}
From our proof of Theorem \ref{thm6.4}, we see that when $p\geq 2$ and $\M=\mathbb{C}$, one always has
$$ \|[T,M_b]\|^p_{S_p(L_2(\mathbb{R}^n))}\lesssim_{n, p} \int_{\mathbb{R}^n\times\mathbb{R}^n}\frac{|b(x)-b(y)|^p}{|x-y|^{2n}}dxdy. $$
However, this does not contradict with Theorem \ref{thm1.5} for $p\leq n$ and $n\geq 2$ due to the following fact.
\begin{proposition}\label{0pn}
	Let $n\ge 1, 1\leq p\leq n$ and $X$ be a Banach space. Assume that $b$ is a locally integrable $X$-valued function. Then $b$ is constant if
	$$  \int_{\mathbb{R}^n\times\mathbb{R}^n}\frac{\|b(x)-b(y)\|_{X}^p}{|x-y|^{2n}}dxdy<\8. $$
\end{proposition}
\begin{proof}
	For any $l\in X^*$, let $$h(x)=l( b(x)),\quad \forall x\in \mathbb{R}^n.$$ 
	By assumption,
	\begin{equation*}
		\begin{aligned}
			\int_{\mathbb{R}^n\times\mathbb{R}^n}\frac{|h(x)-h(y)|^p}{|x-y|^{2n}}dxdy\le \|l\|^p_{X^*}\int_{\mathbb{R}^n\times\mathbb{R}^n}\frac{\|b(x)-b(y)\|_{X}^p}{|x-y|^{2n}}dxdy<\infty.
		\end{aligned}
	\end{equation*}
	We first prove that $h$ is constant. We proceed with the proof by contradiction. Assume that $h$ is not constant. Then there exists $\varphi\in C_c^\infty(\mathbb{R}^n)$ such that $h\ast\varphi\in C^\infty(\mathbb{R}^n)$ is not constant either.
	By changing the variables, we have
	\[\int_{\mathbb{R}^n\times\mathbb{R}^n}\frac{|h(x)-h(y)|^p}{|x-y|^{2n}}dxdy=\int_{\mathbb{R}^n}\frac{\|h(x+t)-h(x)\|^p_{L_p(\mathbb{R}^n)}}{|t|^{2n}}dt.\]
	 Since by the Young inequality
	\[\|\varphi\ast h(x+t)-\varphi\ast h(x)\|^p_{L_p(\mathbb{R}^n)}\le \|\varphi\|^p_{L_1(\mathbb{R}^n)}\|h(x+t)-h(x)\|^p_{L_p(\mathbb{R}^n)},\]
	we get
	\[\int_{\mathbb{R}^n}\frac{\|\varphi\ast h(x+t)-\varphi\ast h(x)\|^p_{L_p(\mathbb{R}^n)}}{|t|^{2n}}dt<\infty.\]
	Hence we can assume that $h\in C^\infty(\mathbb{R}^n)$, otherwise we replace $h$ with $h\ast\varphi$.
	
	Since $h$ is not constant, there exists $\tilde{x}=(\tilde{x}_1,\cdots,\tilde{x}_n)\in\mathbb{R}^n$, such that $\nabla h(\tilde{x})\neq 0$. Let $U$ be a unitary matrix in $\mathbb{M}_n$ such that $\nabla h(\tilde{x})\cdot U=(|\nabla h(\tilde{x})|, 0, \cdots, 0)$. We substitute $\tilde{h}(y):=h(y\cdot U)$ for $h$.
	So we can also assume that there exists $\tilde{x}\in\mathbb{R}^n$ with $\nabla h(\tilde{x})=(M, 0, \cdots, 0)$ and $M>0$.
	
	Since $h\in C^\infty(\mathbb{R}^n)$, $\exists\ \delta>0$ such that $\forall \ |y-\tilde{x}|<2\delta$ with
	$$  |\nabla h(y)-\nabla h(\tilde{x})|\leq \dfrac{M}{4}.  $$
	Thus for any $ |x-\tilde{x}|<\delta$ and $|t|<\delta$ with $|t_1|>\frac{|t|}{2}$, by the mean value theorem,
	\begin{equation*}
		\begin{aligned}
			|h(x+t)-h(x)|&=|\nabla h(x+\theta\cdot t)\cdot t| \quad (0<\theta<1)\\
			&\geq |\nabla h(\tilde{x})\cdot t|-|\big(\nabla h(x+\theta\cdot t)-\nabla h(\tilde{x})\big)\cdot t|\\
			&\geq M|t_1|-\dfrac{M|t|}{4}\geq \dfrac{M|t_1|}{2}.
		\end{aligned}
	\end{equation*}
	This yields that
	\begin{equation}\label{phibb66}
		\|h(x+t)-h(x)\|^p_{L_p(\mathbb{R}^n)}\gtrsim _{n, p}\delta^n M^p |t_1|^p.
	\end{equation}
	Consequently, one has
	\begin{equation}\label{bbcont2n}
		\begin{aligned}
			\int_{\mathbb{R}^n}\frac{\|h(x+t)-h(x)\|^p_{L_p(\mathbb{R}^n)}}{|t|^{2n}}dt&\gtrsim_{n, p} \int\limits_{t\in \mathbb{R}^n, |t|<\delta \atop |t_1|>\frac{|t|}{2}} \dfrac{|t_1|^p}{|t|^{2n}}dt\\
			&\gtrsim_{n, p} \int_{0}^\delta \dfrac{r^p}{r^{2n}}\cdot r^{n-1} dr\\
			&= \int_{0}^\delta \dfrac{1}{r^{n+1-p}}dr=\8.
		\end{aligned}
	\end{equation}
	This leads to a contradiction.
Thus $h$ is constant almost everywhere. Namely, there exists a null set $A_l\subset \mathbb{R}^n$, such that
	\begin{equation*}
		\begin{aligned}
			h(x)-h(y)=l(b(x)-b(y))=0,\quad  \forall x,y\in  \mathbb{R}^n\backslash A_l.
		\end{aligned}
	\end{equation*}
	Note that $b$ is a locally integrable $X$-valued function.
	Thus from Pettis measurability theorem in \cite[Theorem II.1.2]{DU}, there exists a closed separable subspace $X_0$ of $X$ such that $b(x)\in X_0$ for all $x\in \mathbb{R}^n\setminus B$ with $B$ a null subset of $\mathbb{R}^n$. Besides, by the Hahn-Banach Theorem, we choose $(l_k)_{k\ge 1}$ in $X^*$ separating the points of $X_0$.   Therefore, $b$ is constant outside the union of the null sets $B$ and $\cup_{k\ge 1}A_{l_k}$.
\end{proof}

\bigskip

\section{Proof of Theorem \ref{thm1.8}}\label{pthm1.8}

In this section we prove Theorem \ref{thm1.8}, which follows the same route as the argument for Theorem \ref{thm6.4}. Denote by $BMO^d_{\mathcal{M}}(\mathbb{R})$ the operator-valued $BMO$ space associated with the $d$-adic martingale consisting of all $\mathcal{M}$-valued functions $b$ that are Bochner integrable on any $d$-adic interval such that
\begin{equation}\label{BMOMd}
	\|b\|_{BMO^d_\mathcal{M}(\mathbb{R})}=\sup_{I\in \mathcal{D}}\biggl(\frac{1}{m(I)}\int_I \Bigl\|b-\big(\frac{1}{m(I)} \int_I b\ d m\big)\Bigr\|_\mathcal{M}^2 dm\biggr)^{1/2}<\infty,
\end{equation} 
where $\mathcal{D}$ is the family of all $d$-adic intervals on $\mathbb{R}$.

During the proof of Theorem \ref{thm1.8}, we also need the $d$-adic martingale square function defined as follows
\begin{equation*}
	S(h)=\biggl(\sum_{k\in\mathbb{Z}} |d_kh|^2\biggr)^{1/2}, \quad \forall h\in L_1(\mathbb{R},L_1(\mathcal{M})), 
\end{equation*}
and the $d$-adic martingale Hardy space $H^{d}_{1,\max}(\mathbb{R})$ defined by
\begin{equation*}
	H^d_{1,\max}(\mathbb{R})=\bigg\{h\in  L_1(\mathbb{R},L_1(\mathcal{M})):\|h\|_{H^d_{1,\max}(\mathbb{R})}=\bigg\|\sup_{m\in\mathbb{Z}}\|\mathbb{E}_mh\|_{L_1(\mathcal{M})}\bigg\|_{L_1(\mathbb{R})}<\infty\bigg\}.
\end{equation*}
Bourgain and Garcia-Cuerva proved independently that $BMO_\mathcal{M}^d(\mathbb{R})$ embeds continuously into the dual of Hardy space $(	H^d_{1,\max}(\mathbb{R}))^*$. We refer the reader to \cite{BJ} for more details.

Firstly we give the following proposition and its corollary, which will be helpful in the proof of Propositions \ref{T2est}.
\begin{proposition}\label{pibsumpibstar}
	Let $1<p<\infty$ and $b\in BMO^d_\mathcal{M}(\mathbb{R})$. Then $\pi_b+(\pi_{b^*})^*$ is bounded on $L_p(\mathbb{R},L_p(\mathcal{M}))$ and 
	\begin{equation*}
		\|\pi_b+(\pi_{b^*})^*\|_{L_p(\mathbb{R},L_p(\mathcal{M}))\to L_p(\mathbb{R},L_p(\mathcal{M}))}\lesssim_{d, p} \|b\|_{BMO^d_\mathcal{M}(\mathbb{R})}.
	\end{equation*}
\end{proposition}
\begin{proof}
	For any $f\in L_p(\mathbb{R},L_p(\mathcal{M}))$ and $g\in L_{p'}(\mathbb{R},L_{p'}(\mathcal{M}))$, by \eqref{pistar}
	\begin{equation*}
		\begin{aligned}
			\langle (\pi_b+(\pi_{b^*})^*)(f),g\rangle{}&=\biggl\langle \sum_{k\in\mathbb{Z}}d_kb\cdot f_{k-1}+\sum_{k\in\mathbb{Z}}\mathbb{E}_{k-1}(d_kb\cdot d_kf),g \biggr\rangle\\
			&=\sum_{k\in\mathbb{Z}}\langle d_kb,d_kg\cdot f_{k-1}^*\rangle+\sum_{k\in\mathbb{Z}}\langle d_kb,g_{k-1}\cdot d_kf^*\rangle\\
			&=\biggl\langle b,\sum_{k\in\mathbb{Z}}d_kg\cdot f_{k-1}^*+\sum_{k\in\mathbb{Z}}g_{k-1}\cdot d_kf^*\biggr\rangle.
		\end{aligned}
	\end{equation*}
Using the same method as in \cite[Theorem 1.1]{Mei2} or \cite{BP08}, we obtain that
\begin{equation*}
	\begin{aligned}
		|\langle (\pi_b+(\pi_{b^*})^*)(f),g\rangle|\lesssim_{d, p} \|b\|_{BMO^d_\mathcal{M}(\mathbb{R})}\|f\|_{L_p(\mathbb{R},L_p(\mathcal{M}))}\|g\|_{L_{p'}(\mathbb{R},L_{p'}(\mathcal{M}))}.
	\end{aligned}
\end{equation*}
	Therefore, one has
	\begin{equation*}
		\|\pi_b+(\pi_{b^*})^*\|_{L_p(\mathbb{R},L_p(\mathcal{M}))\to L_p(\mathbb{R},L_p(\mathcal{M}))}\lesssim_{d, p} \|b\|_{BMO^d_\mathcal{M}(\mathbb{R})}.
	\end{equation*}
\end{proof}

We define
	\begin{equation}\label{Thetab}
	\Theta_b=\pi_b+\varLambda_b.
\end{equation}
 The following corollary is about the boundedness of $\Theta_b$, which has been proved in \cite[Proposition A.2]{HM}, but it seems that the proof there contains a gap. We give a detailed proof here and thus fix that gap.
\begin{corollary}\label{Thetabest}
	If $b\in BMO^d_\mathcal{M}(\mathbb{R})$, then $\Theta_b$ is bounded on $L_2(\mathbb{R},L_2(\mathcal{M}))$ and 
	\begin{equation*}
		\|\Theta_b\|_{L_2(\mathbb{R},L_2(\mathcal{M}))\to L_2(\mathbb{R},L_2(\mathcal{M}))}\lesssim_d \|b\|_{BMO^d_\mathcal{M}(\mathbb{R})}.
	\end{equation*}
\end{corollary}
\begin{proof}
		We use the same notation as that in Lemma \ref{TLambdab}, and the proof of this lemma is also similar to that of Lemma \ref{TLambdab}. 
	By the triangle inequality, one has
	\begin{equation*}
		\begin{aligned}
			{}&\|\Theta_b\|_{L_2(\mathbb{R},L_2(\mathcal{M}))\to L_2(\mathbb{R},L_2(\mathcal{M}))}\\
			&\le \|\pi_b+(\pi_{b^*})^*\|_{L_2(\mathbb{R},L_2(\mathcal{M}))\to L_2(\mathbb{R},L_2(\mathcal{M}))}+\|\varLambda_b-(\pi_{b^*})^*\|_{L_2(\mathbb{R},L_2(\mathcal{M}))\to L_2(\mathbb{R},L_2(\mathcal{M}))}.
		\end{aligned}
	\end{equation*}
From \eqref{Lambdasum}, we write $\varLambda_b$ as follows:
\begin{equation}\label{var}
	\begin{aligned}
		\varLambda_b(f){}&
		=(\pi_{b^*})^*(f)+\tilde{\varLambda}_b(f), \quad \forall f\in L_2(\mathbb{R},L_2(\mathcal{M})),
	\end{aligned}
\end{equation}
where
\begin{equation*}
	\tilde{\varLambda}_b(f)=\sum_{k\in\mathbb{Z}}\sum_{I\in\mathcal{D}_{k-1}}\sum_{l=1}^{d-1}\sum_{\overline{i+j}=l}\langle h_I^i,b\rangle\langle h_I^j,f\rangle \frac{h_I^l}{|I|^{1/2}}
\end{equation*}
is given in \eqref{tildevar}. Since $\tilde{\varLambda}_b$ is a block diagonal matrix with respect to the basis $\{h_I^i\}_{I\in\mathcal{D}, 1\leq i\leq d-1}$,
one has
\begin{equation*}
	\begin{aligned}
		\|\tilde{\varLambda}_b\|_{L_2(\mathbb{R},L_2(\mathcal{M}))\to L_2(\mathbb{R},L_2(\mathcal{M}))}{}&=\sup_{I\in\mathcal{D}}\biggl\|\biggl(a_{i-j}^I\biggr)_{1\le i,j\le d-1}\biggr\|_{L_\infty(\mathbb{M}_{d-1}\otimes \mathcal{M})},
	\end{aligned}
\end{equation*}
where $a^I_{i-j}=|I|^{-1/2}\langle h_I^{i-j},b\rangle$, and $a_0^I=0$.
With $B^I$ in \eqref{BI}, we have 
\[\biggl\|\biggl(a_{i-j}^I\biggr)_{1\le i,j\le d-1}\biggr\|_{L_\infty(\mathbb{M}_{d-1}\otimes \mathcal{M})}\le \|B^I\|_{L_\infty(\mathbb{M}_{d}\otimes \mathcal{M})},\]
which implies
\begin{equation*}
	\begin{aligned}
		\|\tilde{\varLambda}_b\|_{L_2(\mathbb{R},L_2(\mathcal{M}))\to L_2(\mathbb{R},L_2(\mathcal{M}))}{}&\le \sup_{I\in\mathcal{D}}\|B^I\|_{L_\infty(\mathbb{M}_{d}\otimes \mathcal{M})}.
	\end{aligned}
\end{equation*}
Thus using the triangle inequality, one has
	\begin{equation}\label{bbbbbbb}
		\begin{aligned}
			\quad\|\varLambda_b-(\pi_{b^*})^*\|_{L_2(\mathbb{R},L_2(\mathcal{M}))\to L_2(\mathbb{R},L_2(\mathcal{M}))}{}&\le \sup_{I\in\mathcal{D}}\|a_1^IA+a_2^IA^2+\cdots+a_{d-1}^IA^{d-1}\|_{L_\infty(\mathbb{M}_{d}\otimes \mathcal{M})}\\
			&\le\sup_{I\in\mathcal{D}}\sum_{i=1}^{d-1}\|a_i^IA^i\|_{L_\infty(\mathbb{M}_{d}\otimes \mathcal{M})}\\
			&\le\sup_{I\in\mathcal{D}}\sum_{i=1}^{d-1}\|a_i^I\|_{\mathcal{M}}
			=\sup_{I\in \mathcal{D}}\frac{1}{|I|^{1/2}}\sum_{i=1}^{d-1}\|\langle h_I^i,b\rangle\|_\mathcal{M}.
		\end{aligned}
	\end{equation}
	However,
	\begin{equation*}
		\begin{aligned}
			\|\langle h_I^i,b\rangle\|_\mathcal{M}{}&=\bigg\|\biggl\langle h_I^i, b-\bigg\langle \frac{\mathbbm{1}_I}{|I|},b \bigg\rangle \biggl\rangle\bigg\|_{\mathcal{M}}\\
			&\le \frac{1}{|I|^{1/2}}\int_I \bigg\| b(x)-\bigg\langle \frac{\mathbbm{1}_I}{|I|},b \bigg\rangle \bigg\|_{\mathcal{M}}dx\\
			&\le \biggl(\int_I \bigg\| b(x)-\bigg\langle \frac{\mathbbm{1}_I}{|I|},b \bigg\rangle \bigg\|^2_{\mathcal{M}}dx\biggr)^{1/2}.
		\end{aligned}
	\end{equation*}
	This implies
	\begin{equation}\label{lambdapi}
		\begin{aligned}
			\|\varLambda_b-(\pi_{b^*})^*\|_{L_2(\mathbb{R},L_2(\mathcal{M}))\to L_2(\mathbb{R},L_2(\mathcal{M}))}\le (d-1)\|b\|_{BMO^d_\mathcal{M}(\mathbb{R})}.
		\end{aligned}
	\end{equation}
	Therefore from \eqref{lambdapi} and Proposition \ref{pibsumpibstar}
	\begin{equation*}
		\|\Theta_b\|_{L_2(\mathbb{R},L_2(\mathcal{M}))\to L_2(\mathbb{R},L_2(\mathcal{M}))}\lesssim_d \|b\|_{BMO^d_\mathcal{M}(\mathbb{R})},
	\end{equation*}
	as desired.
\end{proof}

\begin{proposition}\label{T2est}
	If $a\in BMO^d(\mathbb{R})$ and $b\in BMO^d_\mathcal{M}(\mathbb{R})$, then $[\pi_a,M_b]$ and $[\pi_a^*,M_b]$ are both bounded on $L_2(\mathbb{R},L_2(\mathcal{M}))$. Moreover, 
	\begin{equation*}
		\|[\pi_a,M_b]\|_{L_2(\mathbb{R},L_2(\mathcal{M}))\to L_2(\mathbb{R},L_2(\mathcal{M}))}\lesssim_d \|a\|_{BMO^d(\mathbb{R})}\|b\|_{BMO^d_{\mathcal{M}}(\mathbb{R})}
	\end{equation*}
and
\begin{equation*}
	\|[\pi_a^*,M_b]\|_{L_2(\mathbb{R},L_2(\mathcal{M}))\to L_2(\mathbb{R},L_2(\mathcal{M}))}\lesssim_d \|a\|_{BMO^d(\mathbb{R})}\|b\|_{BMO^d_{\mathcal{M}}(\mathbb{R})}.
\end{equation*}
\end{proposition}
\begin{proof}	
	Recall that $R_b$ is defined in \eqref{R_b}. Note that by the triangle inequality
	\begin{equation*}
		\begin{aligned}
			{}&\|[\pi_a,M_b]\|_{L_2(\mathbb{R},L_2(\mathcal{M}))\to L_2(\mathbb{R},L_2(\mathcal{M}))}\\
			&\le \|[\pi_a,\Theta_b]\|_{L_2(\mathbb{R},L_2(\mathcal{M}))\to L_2(\mathbb{R},L_2(\mathcal{M}))}+\|[\pi_a,R_b]\|_{L_2(\mathbb{R},L_2(\mathcal{M}))\to L_2(\mathbb{R},L_2(\mathcal{M}))}.
		\end{aligned}
	\end{equation*}
	Note also that from Corollary \ref{Thetabest}
	\begin{equation}\label{a11}
		\begin{aligned}
			\quad\|\pi_{a}\Theta_{b}\|_{L_2(\mathbb{R},L_2(\mathcal{M}))\to L_2(\mathbb{R},L_2(\mathcal{M}))}
			&\le \|\pi_{a}\|_{L_2(\mathbb{R},L_2(\mathcal{M}))\to L_2(\mathbb{R},L_2(\mathcal{M}))}\|\Theta_b\|_{L_2(\mathbb{R},L_2(\mathcal{M}))\to L_2(\mathbb{R},L_2(\mathcal{M}))}\\
			&\lesssim_{d} \|a\|_{BMO^d(\mathbb{R})}\|b\|_{BMO^d_{\mathcal{M}}(\mathbb{R})}.
		\end{aligned}
	\end{equation}
	Hence, one has
	\begin{equation*}
		\|[\pi_a,\Theta_b]\|_{L_2(\mathbb{R},L_2(\mathcal{M}))\to L_2(\mathbb{R},L_2(\mathcal{M}))}\lesssim_d \|a\|_{BMO^d(\mathbb{R})}\|b\|_{BMO^d_{\mathcal{M}}(\mathbb{R})}.
	\end{equation*} 
	Now, we estimate $\|[\pi_a,R_b]\|_{L_2(\mathbb{R},L_2(\mathcal{M}))\to L_2(\mathbb{R},L_2(\mathcal{M}))}$. 
	 By \eqref{piarb} and \eqref{varpsif}, $ \forall f\in L_2(\mathbb{R},L_2(\mathcal{M}))$
	\begin{equation*}
		\begin{aligned}
			[\pi_{a}, R_b](f){}&
			=-\pi_a(\varLambda_{{b}}({f}))+\sum_{k\in\mathbb{Z}} d_ka\cdot \mathbb{E}_{k-1}\biggl(\sum_{j\ge k} d_j{b}\cdot d_j{f}\biggr)-\pi_{a}(\pi_b(f)).
		\end{aligned}
	\end{equation*}
	Define 
	\begin{equation}
		V_{a,b}(f)=\sum_{k\in\mathbb{Z}} d_ka\cdot \mathbb{E}_{k-1}\biggl(\sum_{j\ge k} d_j{b}\cdot d_j{f}\biggr), \quad \ \forall  f\in L_2(\mathbb{R},L_2(\mathcal{M})).
	\end{equation}
	Then we have
	\begin{equation}\label{piarbf}
		[\pi_{a}, R_b](f)=-\pi_{a}(\Theta_b(f))+V_{a,b}(f), \quad \forall f \in L_2(\mathbb{R},L_2(\mathcal{M})).
	\end{equation}
	For any $f\in L_2(\mathbb{R},L_2(\mathcal{M}))$ and $g\in L_2(\mathbb{R},L_2(\mathcal{M}))$, 
	\begin{equation}\label{vabf}
		\begin{aligned}
			\langle V_{a,b}(f),g\rangle{}&=\sum_{k\in\mathbb{Z}}\biggl\langle d_ka\cdot \mathbb{E}_{k-1}\biggl(\sum_{j\ge k}d_jb\cdot d_jf\biggr),g\biggr\rangle\\
			&=\sum_{k\in\mathbb{Z}}\biggl\langle \sum_{j\ge k}d_jb\cdot d_jf, \mathbb{E}_{k-1}\biggl(d_k{a^*}\cdot d_kg\biggr)\biggr\rangle\\
			&=\sum_{k\in\mathbb{Z}}\biggl\langle d_kb,\sum_{j\le k}\mathbb{E}_{j-1}(d_ja^*\cdot d_jg)\cdot d_kf^*  \biggr\rangle\\
			&=\biggl\langle b,\sum_{k\in\mathbb{Z}}\sum_{j\le k}\mathbb{E}_{j-1}(d_ja^*\cdot d_jg)\cdot d_kf^*  \biggr\rangle\\
			&=\biggl\langle b,\sum_{k\in\mathbb{Z}}d_k(d_ka^*\cdot d_kg)\cdot f_{k-1}^*  \biggr\rangle+\langle b, W_{a,f,g}\rangle,
		\end{aligned}
	\end{equation}
	where 
	\begin{equation*}
		\begin{aligned}
			W_{a,f,g}=\sum_{k\in\mathbb{Z}}\sum_{j\le k}\mathbb{E}_{j-1}(d_ja^*\cdot d_jg)\cdot d_kf^*-\sum_{k\in\mathbb{Z}}d_k(d_ka^*\cdot d_kg)\cdot f_{k-1}^*.
		\end{aligned}
	\end{equation*}
	Note that
	\begin{equation*}
		\begin{aligned}
			\biggl\langle b,\sum_{k\in\mathbb{Z}}d_k(d_ka^*\cdot d_kg)\cdot f_{k-1}^*  \biggr\rangle{}&=\sum_{k\in\mathbb{Z}}\langle d_kb,d_kg\cdot d_ka^*\cdot f^*_{k-1}\rangle\\
			&=\sum_{k\in\mathbb{Z}}\langle d_kg^*\cdot d_kb,d_ka^*\cdot f^*_{k-1}\rangle\\
			&=\sum_{k\in\mathbb{Z}}\langle d_k(d_kg^*\cdot d_kb),d_ka^*\cdot f^*_{k-1}\rangle\\
			&=\bigl\langle \big((\varLambda_{b^*}-(\pi_b)^*)(g)\big)^*, \pi_{a^*}(f^*)\big\rangle.
		\end{aligned}
	\end{equation*}
	Then by \eqref{lambdapi},
	\begin{equation}\label{vab1}
		\begin{aligned}
			\bigg|\biggl\langle b,\sum_{k\in\mathbb{Z}}d_k(d_ka^*\cdot d_kg)\cdot f_{k-1}^*  \biggr\rangle\bigg|{}&\le \|(\varLambda_{b^*}-(\pi_b)^*)(g)\|_{L_2(\mathbb{R},L_2(\mathcal{M}))}\|\pi_{a^*}(f^*)\|_{L_2(\mathbb{R},L_2(\mathcal{M}))}\\
			&\lesssim_d \|a\|_{BMO^d(\mathbb{R})}\|b\|_{BMO_{\mathcal{M}}^d(\mathbb{R})}\|g\|_{L_2(\mathbb{R},L_2(\mathcal{M}))}\|f\|_{L_2(\mathbb{R},L_2(\mathcal{M}))}.
		\end{aligned}
	\end{equation}
	We now estimate $W_{a, f, g}$. By duality, one has
	\begin{equation*}
		\begin{aligned}
			|\langle b, W_{a,f,g}\rangle|\lesssim \|b\|_{BMO^d_{\mathcal{M}}(\mathbb{R})}\|W_{a,f,g}\|_{H_{1,\max}^d(\mathbb{R})}.
		\end{aligned}
	\end{equation*}
	We calculate directly that for any $m\in \mathbb{Z}$,
	\begin{equation*}
		\begin{aligned}
			&{}\quad\mathbb{E}_m(W_{a,f,g})\\
			&=\sum_{k\le m}\sum_{j\le k}\mathbb{E}_{j-1}(d_ja^*\cdot d_jg)\cdot d_kf^*-\sum_{k\le m}d_k(d_ka^*\cdot d_kg)\cdot f_{k-1}^*\\
			&=\sum_{j\le m}\mathbb{E}_{j-1}(d_ja^*\cdot d_jg)\cdot (f_m^*-f_{j-1}^*)-\sum_{j\le m}d_j(d_ja^*\cdot d_jg)\cdot f_{j-1}^*\\
			&=\sum_{j\le m}\mathbb{E}_{j-1}(d_ja^*\cdot d_jg)\cdot f_m^*-\sum_{j\le m}(d_ja^*\cdot d_jg)\cdot f_{j-1}^*\\
			&=\mathbb{E}_{m}\biggl(\sum_{j\le m}\mathbb{E}_{j-1}(d_ja^*\cdot d_jg)\biggr)\cdot f_m^*-\sum_{j\le m}(d_ja^*\cdot d_jg)\cdot f_{j-1}^*\\
			&=\mathbb{E}_{m}\biggl(\sum_{j\in\mathbb{Z}}\mathbb{E}_{j-1}(d_ja^*\cdot d_jg)\biggr)\cdot f_m^*-\mathbb{E}_{m}\biggl(\sum_{j\ge m+1}\mathbb{E}_{j-1}(d_ja^*\cdot d_jg)\biggr)\cdot f_m^*-\sum_{j\le m}(d_ja^*\cdot d_jg)\cdot f_{j-1}^*\\
			&=\mathbb{E}_{m}\biggl(\sum_{j\in\mathbb{Z}}\mathbb{E}_{j-1}(d_ja^*\cdot d_jg)\biggr)\cdot f_m^*-\mathbb{E}_{m}\biggl(\sum_{j\ge m+1}d_ja^*\cdot d_jg\biggr)\cdot f_m^*-\sum_{j\le m}(d_ja^*\cdot d_jg)\cdot f_{j-1}^*.
		\end{aligned}
	\end{equation*}
	Hence
	\begin{equation*}
		\begin{aligned}
			{}&\quad\|W_{a,f,g}\|_{H_{1,\max}^d(\mathbb{R})}\\
			&=\bigg\|\sup_{m\in\mathbb{Z}}\|\mathbb{E}_m(W_{a,f,g})\|_{L_1(\mathcal{M})}\bigg\|_{L_1(\mathbb{R})}\\
			&\le \bigg\|\sup_{m\in\mathbb{Z}}\biggl\|\mathbb{E}_{m}\biggl(\sum_{j\in\mathbb{Z}}\mathbb{E}_{j-1}(d_ja^*\cdot d_jg)\biggr)\cdot f_m^*\biggr\|_{L_1(\mathcal{M})}\bigg\|_{L_1(\mathbb{R})}\\
			&\quad+\bigg\|\sup_{m\in\mathbb{Z}}\biggl\|\mathbb{E}_{m}\biggl(\sum_{j\ge m+1}d_ja^*\cdot d_jg\biggr)\cdot f_m^*\biggr\|_{L_1(\mathcal{M})}\bigg\|_{L_1(\mathbb{R})}+\bigg\|\sup_{m\in\mathbb{Z}}\biggl\|\sum_{j\le m}(d_ja^*\cdot d_jg)\cdot f_{j-1}^*\biggr\|_{L_1(\mathcal{M})}\bigg\|_{L_1(\mathbb{R})}\\
			&:=\text{(I)}+\text{(II)}+\text{(III)}.
		\end{aligned}
	\end{equation*}
	For the term $(\text{I})$, from \eqref{pistar}, we have 
	$$  \sum_{j\in\mathbb{Z}}\mathbb{E}_{j-1}(d_ja^*\cdot d_jg)=(\pi_a)^*(g).  $$
	Thus
	\begin{equation}\label{I}
		\begin{aligned}
			(\text{I}){}&\le \biggl\|\sup_{m\in\mathbb{Z}}\biggl\|\mathbb{E}_{m}((\pi_a)^*(g))\biggr\|_{L_2(\mathcal{M})}\cdot \sup_{m\in\mathbb{Z}}\|f_m\|_{L_2(\mathcal{M})}\biggr\|_{L_1(\mathbb{R})}\\
			&\le \biggl\|\sup_{m\in\mathbb{Z}}\biggl\|\mathbb{E}_{m}((\pi_a)^*(g))\biggr\|_{L_2(\mathcal{M})}\biggr\|_{L_2(\mathbb{R})}\cdot \biggl\|\sup_{m\in\mathbb{Z}}\|f_m\|_{L_2(\mathcal{M})}\biggr\|_{L_2(\mathbb{R})}\\
			&\lesssim \|(\pi_{a})^*(g)\|_{L_2(\mathbb{R},L_2(\mathcal{M}))}\|f\|_{L_2(\mathbb{R},L_2(\mathcal{M}))}\\
			&\lesssim_d \|a\|_{BMO^d(\mathbb{R})}\|g\|_{L_2(\mathbb{R},L_2(\mathcal{M}))}\|f\|_{L_2(\mathbb{R},L_2(\mathcal{M}))},
		\end{aligned}
	\end{equation}
	where the first and the second inequalities are both due to the Cauchy-Schwarz inequality, and the third is from the  vector-valued Doob maximal inequality for $L_2(\mathcal{M})$-valued functions.
	
	For the term $(\text{II})$,
	one uses the Cauchy-Schwarz inequality to obtain
	\begin{equation*}
		\begin{aligned}
			{}&\quad\sup_{m\in\mathbb{Z}}\biggl\|\mathbb{E}_{m}\biggl(\sum_{j\ge m+1}d_ja^*\cdot d_jg\biggr)\cdot f_m^*\biggr\|_{L_1(\mathcal{M})}\\
			&
			=\sup_{m\in\mathbb{Z}}\big\|\mathbb{E}_{m}\big((a^*-a^*_m)(g-g_m)\big)\cdot f_m^*\big\|_{L_1(\mathcal{M})}\\
			&\le \sup_{m\in\mathbb{Z}}\big\|\mathbb{E}_{m}\big((a^*-a^*_m)(g-g_m)\big)\big\|_{L_2(\mathcal{M})}\cdot \sup_{m\in\mathbb{Z}}\|f_m\|_{L_2(\mathcal{M})}.
		\end{aligned}
	\end{equation*}
	Let $r=3/2$. We have
	\begin{equation*}
		\begin{aligned}
			{}&\quad\big\|\mathbb{E}_{m}\big((a^*-a^*_m)(g-g_m)\big)\big\|_{L_2(\mathcal{M})}\\
			&\le \mathbb{E}_{m}\big\|(a^*-a^*_m)(g-g_m)\big\|_{L_2(\mathcal{M})}\\
			&=\mathbb{E}_{m}\big(|a-a_m|\cdot\|g-g_m\|_{L_2(\mathcal{M})}\big)\\
			&\le \bigl(\mathbb{E}_{m}\big(|a-a_m|^{r'}\big)\bigr)^{1/r'}\bigl(\mathbb{E}_{m}\big(\|g-g_m\|_{L_2(\mathcal{M})}^{r}\big)\bigr)^{1/r}\\
			&\lesssim \|a\|_{BMO^d(\mathbb{R})}\biggl(\bigl(\mathbb{E}_{m}\|g\|_{L_2(\mathcal{M})}^{r}\bigr)^{1/r}+\bigl(\mathbb{E}_{m}\|g_m\|_{L_2(\mathcal{M})}^{r}\bigr)^{1/r}\biggr)\\
			&=\|a\|_{BMO^d(\mathbb{R})}\biggl(\bigl(\mathbb{E}_{m}\|g\|_{L_2(\mathcal{M})}^{r}\bigr)^{1/r}+\|g_m\|_{L_2(\mathcal{M})}\biggr).
		\end{aligned}
	\end{equation*}
	Hence
	\begin{equation}\label{II}
		\begin{aligned}
			(\text{II}){}&\le \bigg\|\sup_{m\in\mathbb{Z}}\|f_m\|_{L_2(\mathcal{M})}\bigg\|_{L_2(\mathbb{R})}\cdot\biggl\|\sup_{m\in\mathbb{Z}}\big\|\mathbb{E}_{m}\big((a^*-a^*_m)(g-g_m)\big)\big\|_{L_2(\mathcal{M})}\biggr\|_{L_2(\mathbb{R})} \\
			&\lesssim \|f\|_{L_2(\mathbb{R},L_2(\mathcal{M}))}\|a\|_{BMO^d(\mathbb{R})}\cdot \biggl\|\sup_{m\in\mathbb{Z}}\biggl(\bigl(\mathbb{E}_{m}\|g\|_{L_2(\mathcal{M})}^{r}\bigr)^{1/r}+\|g_m\|_{L_2(\mathcal{M})}\biggr)\bigg\|_{L_2(\mathbb{R})}\\
			&\le \|f\|_{L_2(\mathbb{R},L_2(\mathcal{M}))}\|a\|_{BMO^d(\mathbb{R})}\cdot \biggl(\biggl\|\sup_{m\in\mathbb{Z}}\mathbb{E}_{m}\|g\|_{L_2(\mathcal{M})}^{r}\bigg\|^{1/r}_{L_{2/r}(\mathbb{R})}+\biggl\|\sup_{m\in\mathbb{Z}}\|g_m\|_{L_2(\mathcal{M})}\bigg\|_{L_2(\mathbb{R})}\biggr)\\
			&\lesssim_r \|f\|_{L_2(\mathbb{R},L_2(\mathcal{M}))}\|a\|_{BMO^d(\mathbb{R})}\cdot \big(\big\|\|g\|_{L_2(\mathcal{M})}^{r}\big\|^{1/r}_{L_{2/r}(\mathbb{R})}+\|g\|_{L_2(\mathbb{R},L_2(\mathcal{M}))}\big)\\
			&\lesssim \|f\|_{L_2(\mathbb{R},L_2(\mathcal{M}))}\|a\|_{BMO^d(\mathbb{R})}\|g\|_{L_2(\mathbb{R},L_2(\mathcal{M}))},
		\end{aligned}
	\end{equation}
	where in the first inequality we have used the Cauchy-Schwarz inequality, the second and the fourth are both from the vector-valued Doob maximal inequality, and the third is from the triangle inequality.
	
	For the term $\text{(III)}$, note that
	$$ \sum_{j\le m}(d_ja^*\cdot d_jg)\cdot f_{j-1}^*=\sum_{j\le m}d_jg\cdot (d_ja^*\cdot f_{j-1}^*).  $$
	This implies that,
	\begin{equation}\label{III}
		\begin{aligned}
			(\text{III}){}&\leq \biggl\| \sup_{m\in\mathbb{Z}}\biggl\| \biggl(\sum_{j\leq m}|d_jg|^2\biggr)^{1/2}\biggr\|_{L_2(\mathcal{M})}\biggl\|\biggl(\sum_{j\leq m}|d_ja\cdot f_{j-1}|^2\biggr)^{1/2}\biggr\|_{L_2(\mathcal{M})}\biggr\|_{L_1(\mathbb{R})}\\
			&= \biggl\| \biggl\|\biggl(\sum_{j\in\mathbb{Z}}|d_jg|^2\biggr)^{1/2}\biggr\|_{L_2(\mathcal{M})}\biggl\|\biggl(\sum_{j\in\mathbb{Z}}|d_ja\cdot f_{j-1}|^2\biggr)^{1/2}\biggr\|_{L_2(\mathcal{M})}\biggr\|_{L_1(\mathbb{R})}\\
			&\lesssim \|g\|_{L_2(\mathbb{R},L_2(\mathcal{M}))}\|\pi_a(f)\|_{L_2(\mathbb{R},L_2(\mathcal{M}))}\\
			&\lesssim_d \|a\|_{BMO^d(\mathbb{R})}\|g\|_{L_2(\mathbb{R},L_2(\mathcal{M}))}\|f\|_{L_2(\mathbb{R},L_2(\mathcal{M}))},
		\end{aligned}
	\end{equation}
	where in the third inequality we have used the Cauchy-Schwarz inequality.

	Hence from \eqref{I}, \eqref{II} and \eqref{III} we deduce
	\begin{equation}\label{vab2}
		\begin{aligned}
			|\langle b, W_{a,f,g}\rangle|{}&\le \|b\|_{BMO^d_{\mathcal{M}}(\mathbb{R})}\|W_{a,f,g}\|_{H_{1,\max}^d(\mathbb{R})}\\
			&\lesssim_d \|a\|_{BMO^d(\mathbb{R})}\|b\|_{BMO^d_{\mathcal{M}}(\mathbb{R})}\|g\|_{L_2(\mathbb{R},L_2(\mathcal{M}))}\|f\|_{L_2(\mathbb{R},L_2(\mathcal{M}))}.
		\end{aligned}
	\end{equation}
	Then by \eqref{vabf}, \eqref{vab1} and \eqref{vab2}, we get
	\begin{equation*}
		\begin{aligned}
			|\langle V_{a,b}(f),g\rangle|\lesssim_d \|a\|_{BMO^d(\mathbb{R})}\|b\|_{BMO^d_{\mathcal{M}}(\mathbb{R})}\|g\|_{L_2(\mathbb{R},L_2(\mathcal{M}))}\|f\|_{L_2(\mathbb{R},L_2(\mathcal{M}))},
		\end{aligned}
	\end{equation*}
	which yields
	\begin{equation*}
		\begin{aligned}
			\|V_{a,b}\|_{L_2(\mathbb{R},L_2(\mathcal{M}))\to L_2(\mathbb{R},L_2(\mathcal{M}))}\lesssim_d \|a\|_{BMO^d(\mathbb{R})}\|b\|_{BMO^d_{\mathcal{M}}(\mathbb{R})}.
		\end{aligned}
	\end{equation*}
	Therefore
	\begin{equation*}
		\begin{aligned}
			\|[\pi_a,M_b]\|_{L_2(\mathbb{R},L_2(\mathcal{M}))\to L_2(\mathbb{R},L_2(\mathcal{M}))}\lesssim_d \|a\|_{BMO^d(\mathbb{R})}\|b\|_{BMO^d_{\mathcal{M}}(\mathbb{R})}.
		\end{aligned}
	\end{equation*}
Recall that 
	\begin{equation*}
	[\pi_a^*,M_b]^*=-[\pi_a,M_{b^*}].
\end{equation*}
Hence
\begin{equation*}
	\|[\pi_a^*,M_b]\|_{L_2(\mathbb{R},L_2(\mathcal{M}))\to L_2(\mathbb{R},L_2(\mathcal{M}))}\lesssim_d \|a\|_{BMO^d(\mathbb{R})}\|b\|_{BMO^d_{\mathcal{M}}(\mathbb{R})}.
\end{equation*}

\end{proof}

We define the operator-valued martingale $BMO$ space $BMO^{\omega, 2^n}_\M(\mathbb{R}^n)$ on $\mathbb{R}^n$ similarly as in \eqref{BMOMd}. More precisely, $BMO^{\omega, 2^n}_\M(\mathbb{R}^n)$ associated with dyadic system $\mathcal{ D}^\omega$ on $\mathbb{R}^n$ consists of all $\mathcal{M}$-valued functions $b$ that are Bochner integrable on any $d$-adic interval such that
\begin{equation}\label{BMOMd1}
	\|b\|_{BMO^{\omega, 2^n}_\mathcal{M}(\mathbb{R}^n)}=\sup_{Q\in \mathcal{D}^\omega}\biggl(\frac{1}{m(Q)}\int_Q \biggl\|b-\bigg(\frac{1}{m(Q)} \int_Q b\ d m\bigg)\biggr\|_\mathcal{M}^2 dm\biggr)^{1/2}<\infty.
\end{equation} 
Then Corollary \ref{Thetabest} and Proposition \ref{T2est} also hold for the dyadic system $\mathcal{ D}^\omega$. It is easy to verify that if $b\in BMO_\M(\mathbb{R}^n) $, then $b\in BMO^{\omega, 2^n}_\M(\mathbb{R}^n)$ and
$$ 	\|b\|_{BMO^{\omega, 2^n}_\M(\mathbb{R}^n)}\leq 	\|b\|_{BMO_\M(\mathbb{R}^n)}. $$
Thus we come to the proof of Theorem \ref{thm1.8}.

\begin{proof}[Proof of Theorem \ref{thm1.8}]
	We use the same notation as that in the proof of Theorem \ref{thm6.4}. From  Proposition \ref{T2est} we have
	\begin{equation*}
		\|[\pi_{T(1)}^{\omega}, M_b]\|_{L_2(\mathbb{R}^n,L_2(\mathcal{M}))\to L_2(\mathbb{R}^n,L_2(\mathcal{M}))}\lesssim_{n}\|T(1)\|_{BMO(\mathbb{R}^n)} \|b\|_{BMO_{\mathcal{M}}(\mathbb{R}^n)}
	\end{equation*}
	and
	\begin{equation*}
		\|[(\pi_{T^*(1)}^{\omega})^*, M_b]\|_{L_2(\mathbb{R}^n,L_2(\mathcal{M}))\to L_2(\mathbb{R}^n,L_2(\mathcal{M}))}\lesssim_{n}\|T^*(1)\|_{BMO(\mathbb{R}^n)} \|b\|_{BMO_{\mathcal{M}}(\mathbb{R}^n)}.
	\end{equation*}
	It suffices to estimate $\|[S_\omega^{ij}, M_b]\|_{L_2(\mathbb{R}^n,L_2(\mathcal{M}))\to L_2(\mathbb{R}^n,L_2(\mathcal{M}))}$ for any $i, j\in \mathbb{N}\cup\{0\}$.
	Note that by the triangle inequality
	\begin{equation}\label{thm1.88}
		\begin{aligned}
			{}&\quad\|[S_\omega^{ij},M_b]\|_{L_2(\mathbb{R}^n,L_2(\mathcal{M}))\to L_2(\mathbb{R}^n,L_2(\mathcal{M}))}\\&\le \|[S_\omega^{ij},\Theta_b]\|_{L_2(\mathbb{R}^n,L_2(\mathcal{M}))\to L_2(\mathbb{R}^n,L_2(\mathcal{M}))}+\|[S_\omega^{ij},R_b]\|_{L_2(\mathbb{R}^n,L_2(\mathcal{M}))\to L_2(\mathbb{R}^n,L_2(\mathcal{M}))}.
		\end{aligned}
	\end{equation}
	From Corollary \ref{Thetabest} one has
	\begin{equation*}
		\begin{aligned}
			\|[S_\omega^{ij},\Theta_b]\|_{L_2(\mathbb{R}^n,L_2(\mathcal{M}))\to L_2(\mathbb{R}^n,L_2(\mathcal{M}))}\lesssim_{n} \|b\|_{BMO_\mathcal{M}(\mathbb{R}^n)}.
		\end{aligned}
	\end{equation*}
	Now, we estimate $\|[S_\omega^{ij},R_b]\|_{L_2(\mathbb{R}^n,L_2(\mathcal{M}))\to L_2(\mathbb{R}^n,L_2(\mathcal{M}))}$. Take any $f$ with $\|f\|_{L_2(\mathbb{R}^n,L_2(\mathcal{M}))}=1$. From \eqref{BKdef}, \eqref{aijkxieta} and by the Cauchy-Schwarz inequality,
	\begin{equation}\label{form}
		\begin{aligned}
			{}&\quad\ \|[S_\omega^{ij},R_b]f\|^2_{L_2(\mathbb{R}^n,L_2(\mathcal{M}))}\\
			&=\biggl\|\sum_{K\in\mathcal{D}^\omega}\sum_{\substack{I,J\in\mathcal{D}^\omega;I,J\subseteq K\\ \ell(I)=2^{-i}\ell(K)\\\ell(J)=2^{-j}\ell(K)}}\sum_{\xi,\eta}a_{IJK}^{\xi\eta}b_{IJ}\langle H^\xi_I,f\rangle H^\eta_J\biggr\|^2_{L_2(\mathbb{R}^n,L_2(\mathcal{M}))}\\
			&=\sum_{K\in\mathcal{D}^\omega}\sum_{\substack{J\in\mathcal{D}^\omega;J\subseteq K\\ \ell(J)=2^{-j}\ell(K)}}\sum_{\eta}\biggl\|\sum_{\substack{I\in\mathcal{D}^\omega;I\subseteq K\\ \ell(I)=2^{-i}\ell(K)}}\sum_{\xi}a_{IJK}^{\xi\eta}b_{IJ}\langle H^\xi_I,f\rangle \biggr\|^2_{L_2(\mathcal{M})}\\
			&\le \sum_{K\in\mathcal{D}^\omega}\sum_{\substack{J\in\mathcal{D}^\omega;J\subseteq K\\ \ell(J)=2^{-j}\ell(K)}}\sum_{\eta}\biggl(\sum_{\substack{I\in\mathcal{D}^\omega;I\subseteq K\\ \ell(I)=2^{-i}\ell(K)}}\sum_{\xi}|a_{IJK}^{\xi\eta}|^2\biggr)\biggl(\sum_{\substack{I\in\mathcal{D}^\omega;I\subseteq K\\ \ell(I)=2^{-i}\ell(K)}}\sum_{\xi}\|b_{IJ}\langle H_I^\xi,f\rangle\|_{L_2(\mathcal{M})}^2\biggr)\\
			&\leq \sum_{K\in\mathcal{D}^\omega}\sum_{\substack{J\in\mathcal{D}^\omega;J\subseteq K\\ \ell(J)=2^{-j}\ell(K)}} (2^n-1)^2 2^{-jn}\biggl(\sum_{\substack{I\in\mathcal{D}^\omega;I\subseteq K\\ \ell(I)=2^{-i}\ell(K)}}\sum_{\xi}\|b_{IJ}\langle H_I^\xi,f\rangle\|_{L_2(\mathcal{M})}^2\biggr),
		\end{aligned}
	\end{equation}
	where $b_{IJ}= \langle \frac{\mathbbm{1}_I}{|I|},b\rangle-\langle \frac{\mathbbm{1}_J}{|J|},b\rangle$.
	Note that if $\tilde{I}$ is the parent of $I$, then
	\begin{equation*}
		\begin{aligned}
			\bigg\|\biggl\langle \frac{\mathbbm{1}_I}{|I|},b\biggr\rangle-\biggl\langle\frac{\mathbbm{1}_{\tilde{I}}}{|\tilde{I}|},b\biggr\rangle\bigg\|_{\mathcal{M}}{}
			&\le \frac{1}{|I|}\int_I \biggl\|b(t)-\biggl\langle\frac{\mathbbm{1}_{\tilde{I}}}{|\tilde{I}|},b\bigg\rangle\biggr\|_{\mathcal{M}}dt\\
			&\leq\frac{2^{n/2}}{|\tilde{I}|^{1/2}}\biggl(\int_{\tilde{I}}\biggl\|b(t)-\biggl\langle\frac{\mathbbm{1}_{\tilde{I}}}{|\tilde{I}|},b\bigg\rangle\biggr\|^2_{\mathcal{M}}dt\biggr)^{1/2}\\
			&\le 2^{n/2}\|b\|_{BMO_\mathcal{M}(\mathbb{R}^n)}.
		\end{aligned}
	\end{equation*}
	This implies that by the triangle inequality,
	\begin{equation*}
		\|b_{IJ}\|_{\mathcal{M}}\leq\|b_{IK}\|_{\mathcal{M}}+\|b_{JK}\|_{\mathcal{M}} \le  2^{n/2}(i+j)\|b\|_{BMO_\mathcal{M}(\mathbb{R}^n)}.
	\end{equation*}
	As a consequence, one has from \eqref{form} that
	\begin{equation*}
		\begin{aligned}
			&\ \ \ \|[S_\omega^{ij},R_b]f\|^2_{L_2(\mathbb{R}^n,L_2(\mathcal{M}))}\\
			&\leq 2^n(2^n-1)^2(i+j)^2\|b\|^2_{BMO_\mathcal{M}(\mathbb{R}^n)} \sum_{K\in\mathcal{D}^\omega}\biggl(\sum_{\substack{I\in\mathcal{D}^\omega;I\subseteq K\\ \ell(I)=2^{-i}\ell(K)}}\sum_{\xi}\|\langle H_I^\xi,f\rangle\|_{L_2(\mathcal{M})}^2\biggr)\\
			&\lesssim_n (i+j)^2\|b\|^2_{BMO_\mathcal{M}(\mathbb{R}^n)}.
		\end{aligned}
	\end{equation*}
	By \eqref{thm1.88}, we have
	\begin{equation*}
		\begin{aligned}
			\|[S_\omega^{ij},M_b]\|_{L_2(\mathbb{R}^n,L_2(\mathcal{M}))\to L_2(\mathbb{R}^n,L_2(\mathcal{M}))}\lesssim_n (i+j+1)\|b\|_{BMO_\mathcal{M}(\mathbb{R}^n)}.
		\end{aligned}
	\end{equation*}
	Therefore, by the triangle inequality, we conclude
	\begin{equation*}
		\begin{aligned}
			\|[T,M_b]\|_{L_2(\mathbb{R}^n,L_2(\mathcal{M}))\to L_2(\mathbb{R}^n,L_2(\mathcal{M}))}{}
			&\lesssim_T \sum_{i,j=0}^\infty \tau(i, j)\mathbb{E}_\omega\|[S_\omega^{ij},M_b]\|_{L_2(\mathbb{R}^n,L_2(\mathcal{M}))\to L_2(\mathbb{R}^n,L_2(\mathcal{M}))}\\
			&\quad+\mathbb{E}_\omega\|[ \pi^\omega_{T(1)}+ (\pi^\omega_{T^*(1)})^*,M_b]\|_{L_2(\mathbb{R}^n,L_2(\mathcal{M}))\to L_2(\mathbb{R}^n,L_2(\mathcal{M}))}\\
			&\lesssim_{n,T} (1+\|T(1)\|_{BMO(\mathbb{R}^n)}+\|T^*(1)\|_{BMO(\mathbb{R}^n)})\|b\|_{BMO_\mathcal{M}(\mathbb{R}^n)}.
		\end{aligned}
	\end{equation*}
	This completes the proof.
\end{proof}

\bigskip

\section{Complex median method}\label{proofdivide}	

This section is devoted to the proof of the complex median method, i.e. Theorem \ref{divideS}. We proceed the proof with some fundamental lemmas. In the sequel, we will always assume that $(\Omega,\mathcal{F},\mu)$ is a measure space. Besides, suppose that $I\in \mathcal{F}$ is of finite measure, and $b$ is always a measurable function on $I$.
\begin{lem}\label{S1S2S}
	There exists a line $l\subset \mathbb{C}$ that divides $\mathbb{C}$ into two closed half-planes $S_1$ and $S_2$ whose intersection is $l$, such that 
	\begin{equation*}
		\mu(\{x\in I:b(x)\in S_i\})\ge \frac{1}{2}\mu(I),\quad i\in\{1, 2\}.
	\end{equation*}
\end{lem}

\begin{proof}
	For any $x\in \mathbb{R}$, let $l_x$ be the line that is perpendicular to the real axis at $x$. The left side of $l_x$ on the complex plane, including $l_x$, is denoted by $P_x$ (see Figure \ref{picture1}). Let 
	\begin{equation*}
		f(x)=\frac{\mu(\{y\in I:b(y)\in P_x\})}{\mu(I)}, \quad x\in \mathbb{R}.
	\end{equation*}
	It is not hard to check that:\\
    $\bullet$ $f$ is increasing,\\
    $\bullet$ $f$ is right-continuous,\\
	$\bullet$ $\lim_{x\to +\infty}f(x)=1$, $\lim_{x\to -\infty}f(x)=0$.

	Define 
	\begin{equation*}
		\alpha=\inf \Big\{x\in \mathbb{R}: f(x)\ge \frac{1}{2}\Big\}.
	\end{equation*}
	On the one hand, since $f$ is right-continuous, we have 
	\begin{equation}\label{S1122}
		f(\alpha)\ge \frac{1}{2}.
	\end{equation}
	On the other hand, note that for any $\varepsilon>0$, $f(\alpha-\varepsilon)<\frac{1}{2}$. Namely, 
	\begin{equation*}
		\mu(\{x\in I:b(x)\in P_{\alpha-\varepsilon}\})<\frac{1}{2}\mu(I).
	\end{equation*}
	Let $\varepsilon\to 0$, then
	\begin{equation*}
		\mu(\{x\in I:b(x)\in P_{\alpha} \backslash l_{\alpha}\})\le \frac{1}{2}\mu(I).
	\end{equation*}
	Denote by $Q_x$ the right side of $l_x$ on the complex plane, including $l_x$. Thus
	\begin{equation}\label{S2122}
		\mu(\{x\in I:b(x)\in Q_{\alpha}\})\ge \frac{1}{2}\mu(I).
	\end{equation}
	Let $l=l_{\alpha}$, $S_1=P_{\alpha}$ and $S_2=Q_{\alpha}$. Then $l$ divides $\mathbb{C}$ into two parts $S_1$ and $S_2$ whose intersection is $l$. Besides, from \eqref{S1122} and \eqref{S2122} we get
	\begin{equation*}
		\mu(\{x\in I:b(x)\in S_i\})\ge \frac{1}{2}\mu(I),\quad i\in\{1, 2\}.
	\end{equation*}	
\end{proof}

\begin{figure}[H] 
	\centering 
	\includegraphics[width=0.6\textwidth]{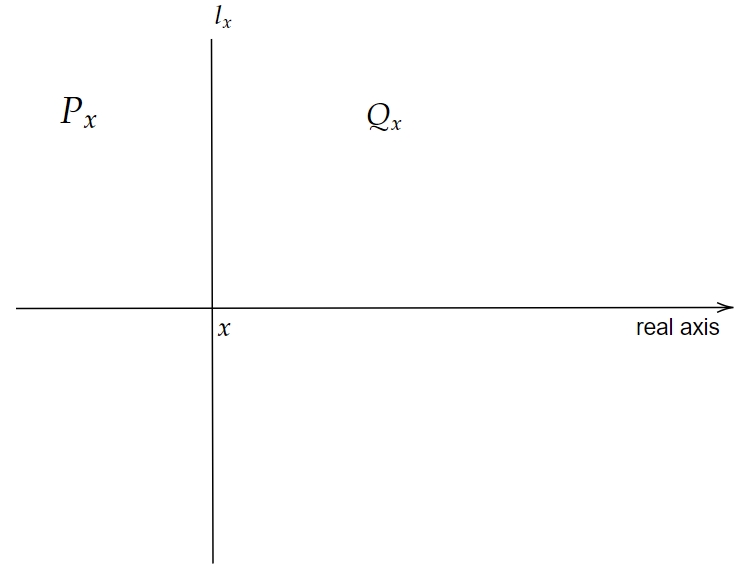} 
	\caption{} 
	\label{picture1} 
\end{figure}

The following lemma is derived by the above one.
\begin{lemma}\label{S1S2S3S4S}
	There exist a line $l$ and two rays $l_1,l_2$ in $\mathbb{C}$ satisfying $l\perp l_1$ and $l\perp l_2$, and they divide $\mathbb{C}$ into four closed quadrants $S_1$, $S_2$, $S_3$ and $S_4$ (see Figure \ref{picture2}), such that
	\begin{equation*}
		\mu(\{x\in I:b(x)\in S_i\})\ge \frac{1}{4}\mu(I),\quad i\in \{1,2,3,4\}.
	\end{equation*}
\end{lemma}

\begin{proof}
	By Lemma \ref{S1S2S}, there exists a line $l$ which divides $\mathbb{C}$ into two closed half-planes $U_1$ and $U_2$, such that 
	\begin{equation*}
		\mu(\{x\in I:b(x)\in U_j\})\ge \frac{1}{2}\mu(I),\quad j\in\{1,2\}.
	\end{equation*} 
	By rotating and translating the axes, we assume that $l$ is the real axis. In terms of $U_1$, repeating the proof of Lemma \ref{S1S2S}, we prove that there exists a ray $l_1$ satisfying $l\perp l_1$, where the origin of $l_1$ is $\alpha_1\in l$, such that $l_1$ divides $U_1$ into two parts $S_1$ and $S_2$ whose intersection is $l_1$. Moreover,
	\begin{equation*}
		\mu(\{x\in I:b(x)\in S_i\})\ge \frac{1}{2}\mu(\{x\in I:b(x)\in U_1\})\ge \frac{1}{4}\mu(I),\quad i\in\{1,2\}.
	\end{equation*}	
	Similarly, in terms of $U_2$, there exists a ray $l_2$ satisfying $l\perp l_2$, where the origin of $l_2$ is $\alpha_2\in l$, such that $l_2$ divides $U_2$ into two parts $S_3$ and $S_4$ whose intersection is $l_2$. Moreover,
	\begin{equation*}
		\mu(\{x\in I:b(x)\in S_i\})\ge \frac{1}{2}\mu(\{x\in I:b(x)\in U_2\})\ge \frac{1}{4}\mu(I),\quad i\in\{3,4\}.
	\end{equation*}	
\end{proof}

\begin{figure}[H]
	\centering 
	\begin{minipage}[b]{0.45\textwidth} 
		\centering 
		\includegraphics[width=1.0\textwidth]{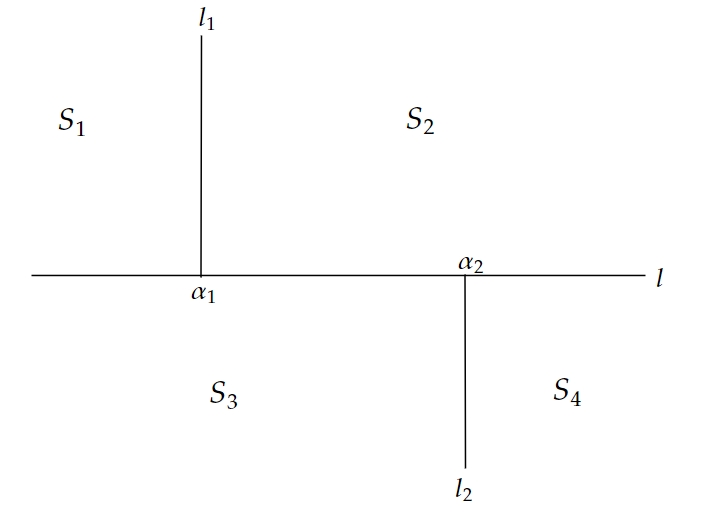} 
		\caption{}
		\label{picture2}
	\end{minipage}
	\begin{minipage}[b]{0.45\textwidth} 
		\centering 
		\includegraphics[width=1.0\textwidth]{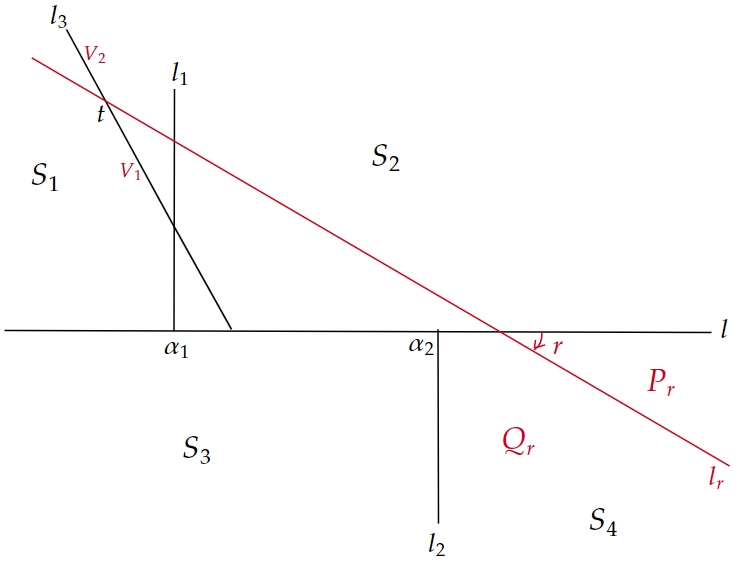}
		\caption{}
		\label{picture7}
	\end{minipage}
\end{figure}

In what follows, we always assume that $\alpha_1\neq \alpha_2$ since Theorem \ref{divideS} clearly holds when $\alpha_1=\alpha_2$. In fact, when $\alpha_1=\alpha_2$, since $l_1\perp l$, we let 
\begin{equation*}
	L_1=l_1,\quad L_2=l,
\end{equation*}
and
\begin{equation*}
	T_j=S_j,\quad j\in\{1,2,3,4\}.
\end{equation*}
Then $L_1$ and $L_2$ are what we need. Besides, we only need to consider the case where $\alpha_1$ is on the left of $\alpha_2$ on the line $l$ as $\alpha_1$ and $\alpha_2$ are symmetric. 

Next, we consider the following particular case which will simplify the proof of Theorem \ref{divideS}.	
\begin{lem}\label{middlecases}
	 Let $S_1$, $S_2$, $S_3$ and $S_4$ be four closed quadrants as in Lemma \ref{S1S2S3S4S}. If there is a ray $l_3$ with origin in $l$ such that 
	\begin{equation*}
		\mu(\{x\in I:b(x)\in l_3\cap S_1\})\ge \frac{1}{2}\mu(\{x\in I:b(x)\in S_1\})
	\end{equation*}
	\begin{equation*}
		(\text{or}\quad	\mu(\{x\in I:b(x)\in l_3\cap S_4\})\ge \frac{1}{2}\mu(\{x\in I:b(x)\in S_4\}),\,\,)
	\end{equation*}
	then there exist two orthogonal lines $l_4$ and $l_5$ such that $l_4$ and $l_5$ divide $\mathbb{C}$ into four closed quadrants $T_1$, $T_2$, $T_3$, $T_4$ (see Figure \ref{picture3}). Moreover,
	\begin{equation*}
		\mu(\{x\in I:b(x)\in T_i\})\ge \frac{1}{16}\mu(I),\quad i\in \{1,2,3,4\}.
	\end{equation*}
\end{lem}

\begin{proof}
	We only consider the case that $l_3\cap S_1\neq \emptyset$. The proof of the case that $l_3\cap S_4\neq \emptyset$ is the same. From the proof of Lemma \ref{S1S2S}, there exists a point $t\in l_3$, such that $t$ divides $l_3\cap S_1$ into two parts $V_{1}$ and $V_{2}$, such that
	\begin{equation*}
		\mu(\{x\in I:b(x)\in V_i\})\ge \frac{1}{2}\mu(\{x\in I:b(x)\in l_{3}\cap S_1\})\ge \frac{1}{16}\mu(I),\quad i\in\{1,2\}.
	\end{equation*}
	Without loss of generality, we assume that $l$ is the real axis. Then by our assumption that $\alpha_1$ is on the left of $\alpha_2$ on the line $l$, we see $\alpha_1< \alpha_2$. 
	
	Let $l_r$ be the line passing through $t$, such that the clockwise angle from $l$ to $l_r$ is $r$, where $-\frac{\pi}{2}< r \le \frac{\pi}{2}$ (see Figure \ref{picture7}). The line $l_r$ divides $S_4$ into two closed parts, and denote by $P_r$ the upper part and by $Q_r$ the lower part (see Figure \ref{picture7}). Let
	\begin{equation*}
		f(r)=
		\begin{cases}
			\mu(\{x\in I:b(x)\in P_r\})\big/\mu(\{x\in I:b(x)\in S_4\}),&  0< r \le \frac{\pi}{2},\\
			0, & -\frac{\pi}{2}<r< 0,
		\end{cases}
	\end{equation*}
	and
	\begin{equation*}
		f(0)=\begin{cases}
			\mu(\{x\in I:b(x)\in P_0\})\big/\mu(\{x\in I:b(x)\in S_4\}), &\mathrm{Im}(t)=0,\\
			0, &\mathrm{Im}(t)>0.
			\end{cases}
		\end{equation*}
	It is not hard to check that:\\
     $\bullet$ $f$ is increasing,\\
     $\bullet$ $f$ is right-continuous,\\
	 $\bullet$ $f(\frac{\pi}{2})=1$,

	Define 
	\begin{equation*}
		\beta=\inf \Big\{r\in (-\frac{\pi}{2},\frac{\pi}{2}]: f(r)\ge \frac{1}{2}\Big\}.
	\end{equation*}
	So $0\leq \beta<\frac{\pi}{2}$. On the one hand, since $f(r)$ is right-continuous, we have 
	\begin{equation}\label{S112}
		f(\beta)\ge \frac{1}{2}.
	\end{equation}
	On the other hand, note that for any $0<\varepsilon<\beta+\frac{\pi}{2}$, $f(\beta-\varepsilon)<\frac{1}{2}$. Namely,
	\begin{equation*}
		\mu(\{x\in I:b(x)\in P_{\beta-\varepsilon}\})<\frac{1}{2}\mu(\{x\in I:b(x)\in S_4\}).
	\end{equation*} 
	Let $\varepsilon\to 0$, then
	\begin{equation*}
		\mu(\{x\in I:b(x)\in P_\beta \backslash l_\beta\})\le \frac{1}{2}\mu(\{x\in I:b(x)\in S_4\}).
	\end{equation*}
Thus this implies
	\begin{equation}\label{S212}
		\mu(\{x\in I:b(x)\in Q_\beta \})\ge \frac{1}{2}\mu(\{x\in I:b(x)\in S_4\}).
	\end{equation}
	Let $l_4=l_\beta$, $S_{41}=P_\beta$ and $S_{42}=Q_\beta$. From \eqref{S112} and \eqref{S212}, $l_4$ passes through $t$ and divides $S_4$ into two closed parts $S_{41}$ and $S_{42}$, such that 
	\begin{equation*}
		\mu(\{x\in I:b(x)\in S_{4i}\})\ge \frac{1}{2}\mu(\{x\in I:b(x)\in S_{4}\})\ge \frac{1}{8}\mu(I),\quad i\in\{1,2\}.
	\end{equation*}
	Now if $\mathrm{Im}(t)>0$, then $\beta\neq 0$ as $t\in l_4$. Suppose $c$ is the intersection point of $l_4$ and $l$ (see Figure \ref{picture3}). Besides, for any given $\tilde{c}\in (\alpha_1,\alpha_2)$, we can find another line $l_5$ passing through $\tilde{c}$, such that $l_5\perp l_4$. Then  $l_4$ and $l_5$ divide $\mathbb{C}$ into four closed quadrants $T_1$, $T_2$, $T_3$, $T_4$. Moreover, we have
	\begin{equation*}
		\begin{aligned}
			{}&\mu(\{x\in I:b(x)\in T_1\})\ge \mu(\{x\in I:b(x)\in V_2\})\ge \frac{1}{16}\mu(I),\\
			&\mu(\{x\in I:b(x)\in T_2\})\ge \mu(\{x\in I:b(x)\in S_{41}\})\ge \frac{1}{8}\mu(I),\\
			&\mu(\{x\in I:b(x)\in T_3\})\ge \mu(\{x\in I:b(x)\in V_1\})\ge \frac{1}{16}\mu(I),\\
			&\mu(\{x\in I:b(x)\in T_4\})\ge \mu(\{x\in I:b(x)\in S_{42}\})\ge \frac{1}{8}\mu(I).			
		\end{aligned}
	\end{equation*}
	If $\mathrm{Im}(t)=0$ and $\beta>0$, then Lemma \ref{middlecases} follows in a similar way. Finally, if $\mathrm{Im}(t)=0$ and $\beta=0$, then $l_4=l$ and choose $l_5$ to be the line through the point $t$ and orthogonal to $l$, and this proves Lemma \ref{middlecases}.
\end{proof}

\begin{figure}[H] 
	\centering 
	\includegraphics[width=0.6\textwidth]{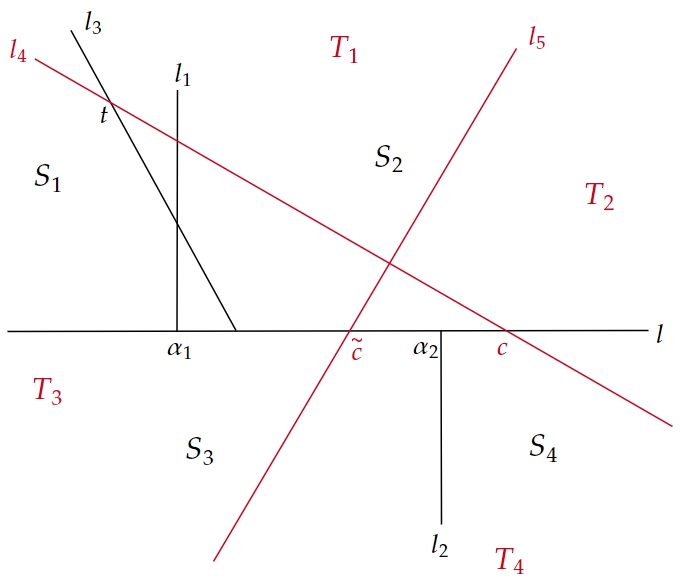} 
	\caption{} 
	\label{picture3} 
\end{figure}

\begin{rem}
	Indeed, we can find a line $l_6$ which passes through the point $t$ and is perpendicular to $l_3$, and then $l_3$ and $l_6$ satisfy Lemma \ref{middlecases}. This is an easy proof of Lemma \ref{middlecases}. But we still keep the previous complicated proof of Lemma \ref{middlecases} since the argument of this complicated proof plays a vital role in the later proof of Theorem \ref{divideS}.
	
	\end{rem}

Now we begin to prove Theorem \ref{divideS}.	
\begin{proof}[Proof of Theorem \ref{divideS}]
	By Lemma \ref{S1S2S3S4S}, there exist a line $l$ and two rays $l_1,l_2$ in $\mathbb{C}$ satisfying $l\perp l_1$ and $l\perp l_2$, and they divide $\mathbb{C}$ into four closed quadrants $S_1$, $S_2$, $S_3$ and $S_4$. Moreover,
	\begin{equation}\label{SiS4}
		\mu(\{x\in I:b(x)\in S_i\})\ge \frac{1}{4}\mu(I),\quad i\in \{1,2,3,4\}.
	\end{equation}
	We denote by $\alpha_1$, $\alpha_2$ the origins of $l_1$ and $l_2$ respectively.
	
	\smallskip
	
	By rotating and translating the axes, we assume that $l$ is the real axis, and $\alpha_1<\alpha_2$. Besides, by the proof of Lemma \ref{S1S2S}, there exists a point $A\le \alpha_1$ in $l$ and a ray $l_A\perp l$, whose origin is $A$, such that $l_A$ divides $S_1$ into two closed parts $R_{1}$ and $R_{2}$ (see Figure \ref{picture4}). Moreover, 
	\begin{equation}\label{muRiS1}
		\mu(\{x\in I:b(x)\in R_{i}\})\ge \frac{1}{2}\mu(\{x\in I:b(x)\in S_1\}),\quad i\in\{1,2\}.
	\end{equation}
	Similarly, there exists a point $B\ge \alpha_2$ in $l$ and a ray $l_B\perp l$, whose origin is $B$, such that $l_B$ divides $S_4$ into two closed parts $R_3$ and $R_4$. Moreover, 
	\begin{equation*}
		\mu(\{x\in I:b(x)\in R_{i}\})\ge \frac{1}{2}\mu(\{x\in I:b(x)\in S_4\}),\quad i\in\{3,4\}.
	\end{equation*}
	Now for any point $A\le x\le B$ and any angle $r\in[0,\pi]$, we
	let $l(x,r)\subset S_1\cup S_2$ be the ray whose origin is $x$, such that the clockwise angle from $l$ to $l(x,r)$ is $r$. Then $l(x, r)$ divides $S_1$ into two closed parts, and denote by $Q_{11}(x,r)$ the lower part and by $Q_{12}(x,r)$ the upper part (see Figure \ref{picture5}).
	
	Similarly, for any point $A\le {x}\le B$ and any angle $\tilde{r}\in[0,\pi]$, let $\tilde{l}({x},\tilde{r})\subset S_3\cup S_4$ be the ray whose origin is ${x}$, such that the clockwise angle from $l$ to $\tilde{l}({x},\tilde{r})$ is $\tilde{r}$. Then $\tilde{l}({x},\tilde{r})$ divides $S_4$ into two closed parts, and denote by $Q_{42}({x},\tilde{r})$ the lower part and by $Q_{41}({x},\tilde{r})$ the upper part (see Figure \ref{picture5}).
	
	\begin{figure}[H]
		\centering 
		\begin{minipage}[b]{0.45\textwidth} 
			\centering 
			\includegraphics[width=1.0\textwidth]{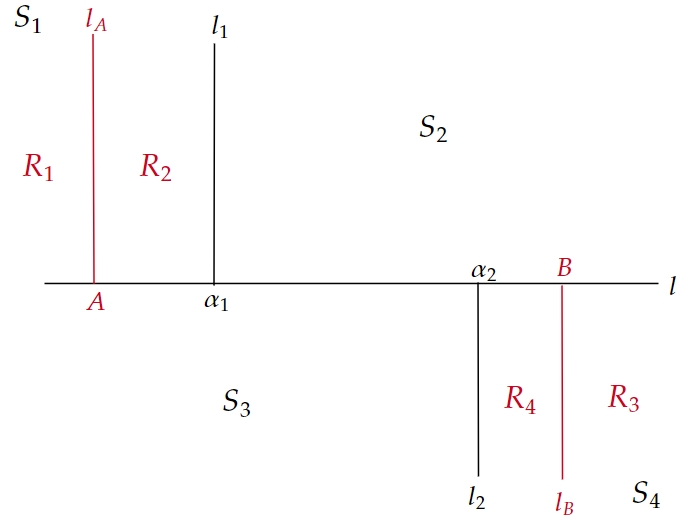} 
			\caption{}
			\label{picture4}
		\end{minipage}
		\begin{minipage}[b]{0.45\textwidth} 
			\centering 
			\includegraphics[width=1.0\textwidth]{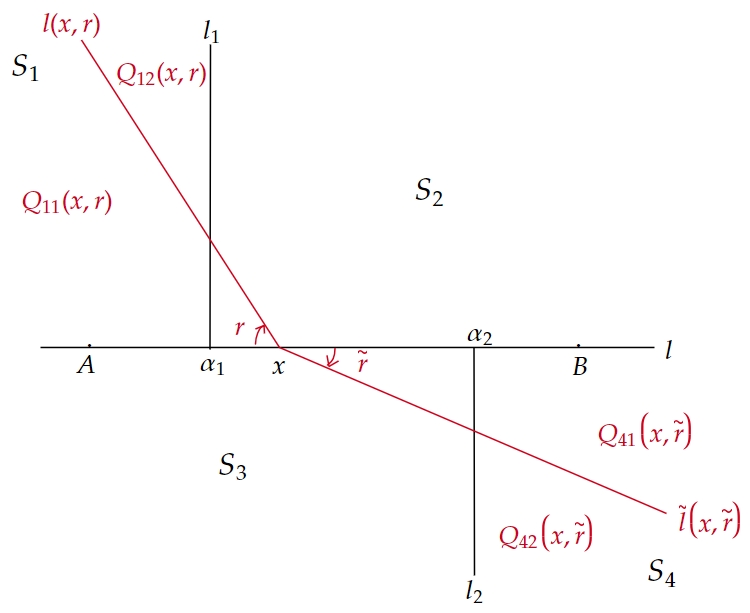}
			\caption{}
			\label{picture5}
		\end{minipage}
	\end{figure}

	Then let
	\begin{equation*}
		f(x,r)=\frac{\mu(\{y\in I:b(y)\in Q_{11}(x,r)\})}{\mu(\{y\in I:b(y)\in S_1\})},\quad g(x,r)=\frac{\mu(\{y\in I:b(y)\in Q_{12}(x,r)\})}{\mu(\{y\in I:b(y)\in S_1\})}
	\end{equation*}
	and
	\begin{equation*}
		\tilde{f}(x,r)=\frac{\mu(\{y\in I:b(y)\in Q_{41}(x,r)\})}{\mu(\{y\in I:b(y)\in S_4\})},\quad \tilde{g}(x,r)=\frac{\mu(\{y\in I:b(y)\in Q_{42}(x,r)\})}{\mu(\{y\in I:b(y)\in S_4\})}.
	\end{equation*}
	It is not hard to check that:\\
	$\bullet$ when the point $x$ is fixed, in terms of the angle $r$, $f$ and $\tilde{f}$ are both increasing and right-continuous functions, $g$ and $\tilde{g}$ are both decreasing and left-continuous functions,\\
	$\bullet$ when the angle $r$ is fixed, in terms of the point $x$, $f$ and $\tilde{g}$ are both increasing and right-continuous functions, $g$ and $\tilde{f}$ are both decreasing and left-continuous functions.

	Besides, define 
	\begin{equation*}
		\begin{aligned}
			r_1(x){}&=\inf\Big\{r:f(x,r)\ge \frac{1}{4}, \ g(x,r)\ge \frac{1}{4}\Big\},\\
			r_2(x)&=\sup\Big\{r:f(x,r)\ge \frac{1}{4}, \ g(x,r)\ge \frac{1}{4}\Big\},\\
			r_3(x)&=\inf\Big\{r:\tilde{f}(x,r)\ge \frac{1}{4}, \  \tilde{g}(x,r)\ge \frac{1}{4}\Big\},\\
			r_4(x)&=\sup\Big\{r:\tilde{f}(x,r)\ge \frac{1}{4}, \  \tilde{g}(x,r)\ge \frac{1}{4}\Big\}.
		\end{aligned}
	\end{equation*}
Note that each $r_i$ is well-defined, since $\Big\{r:f(x,r)\ge \frac{1}{4}, \ g(x,r)\ge \frac{1}{4}\Big\}$ and $\Big\{r:\tilde{f}(x,r)\ge \frac{1}{4}, \  \tilde{g}(x,r)\ge \frac{1}{4}\Big\}$ are not empty by using the same way as in Lemma \ref{S1S2S}. Indeed, these two nonempty sets are even closed intervals. In fact, if $r_1(x)<r_2(x)$, for any $x\in [A,B]$, by the definition of $r_1(x)$, there exists a positive sequence $\{\varepsilon_n\}_{n\ge 1}$ with $\varepsilon_n\to 0$, such that $f(x,r_1(x)+\varepsilon_n)\ge \frac{1}{4}$. Since $f$ is right-continuous with respect to the angle, we then deduce $$f(x,r_1(x))\ge \frac{1}{4}.$$
	Since $r_1(x)< r_2(x)$, one has $$f(x,r_2(x))\ge \frac{1}{4}.$$
	Similarly, we derive that 
	$$g(x,r_2(x))\ge \frac{1}{4} \quad \text{and} \quad g(x,r_1(x))\ge \frac{1}{4}.$$
	Thus for any $x\in [A,B]$,
	\begin{equation}\label{fgr2r1}
		\Big\{r:f(x,r)\ge \frac{1}{4}, \ g(x,r)\ge \frac{1}{4}\Big\}=[r_1(x),r_2(x)].
	\end{equation}
    If $r_1(x)=r_2(x)$, then \eqref{fgr2r1} is trivial.
	Similarly, for any $x\in [A,B]$,
	\begin{equation}\label{fgr3r4}
		\Big\{r:\tilde{f}(x,r)\ge \frac{1}{4}, \ \tilde{g}(x,r)\ge \frac{1}{4}\Big\}=[r_3(x),r_4(x)].
	\end{equation}

	 Furthermore, in the later proof we will always assume that $r_2,r_4< \pi$ and $r_1,r_3>0$ (because the case $r_2=\pi$ or $r_4=\pi$ or $r_1=0$ or $r_3=0$ is the special case as in Lemma \ref{middlecases}, and we omit the details).
	For any given $A\le x_1\le x_2\le B$, since $g$ is decreasing with respect to the angle, by the definition of $r_2(x_1)$, for any $0<\varepsilon<\pi-r_2(x_1)$, we have $g(x_1,r_2(x_1)+\varepsilon)<\frac{1}{4}$. Besides, note that $g$ is decreasing with respect to $x$, then
	\begin{equation*}
		g(x_2,r_2(x_1)+\varepsilon)\le g(x_1,r_2(x_1)+\varepsilon)<\frac{1}{4},
	\end{equation*}
	this implies that $r_2(x_2)\le r_2(x_1)+\varepsilon$. Then by letting $\varepsilon\to 0$, one has $r_2(x_2)\le r_2(x_1)$. Thus $r_2$ is decreasing. Similarly, we obtain:\\
	$\bullet$ $r_1$, $r_2$ are decreasing,\\
	$\bullet$ $r_3$, $r_4$ are increasing.

	\smallskip
	
	In the following we will divide the remaining of the proof into three cases.
	
	\subsubsection*{Case 1: There exists $x_0\in [A,B]$, such that $r_1(x_0)=r_2(x_0)$}	
	Since $f$ is increasing with respect to the angle, by the definition of $r_1(x_0)$, for any $0<\varepsilon<r_1(x_0)$, $f(x_0,r_1(x_0)-\varepsilon)<\frac{1}{4}$. Namely,
	\begin{equation*}
		\mu(\{x\in I:b(x)\in Q_{11}(x_0,r_1(x_0)-\varepsilon)\})<\frac{1}{4}\mu(\{x\in I:b(x)\in S_1\}).
	\end{equation*} 
	Let $\varepsilon\to 0$, then 
	\begin{equation}\label{Q11S1}
		\mu\big(\big\{x\in I:b(x)\in Q_{11}(x_0,r_1(x_0))\big\backslash l(x_0,r_1(x_0))\big\}\big)\le \frac{1}{4}\mu(\{x\in I:b(x)\in S_1\}).
	\end{equation} 
	Similarly, we get
	\begin{equation}\label{Q12S1}
		\mu\big(\big\{x\in I:b(x)\in Q_{12}(x_0,r_2(x_0))\big\backslash l(x_0,r_2(x_0))\big\}\big)\le \frac{1}{4}\mu(\{x\in I:b(x)\in S_1\}).
	\end{equation} 
	Thus from \eqref{Q11S1} and \eqref{Q12S1} one has
	\begin{equation*}
		\mu(\{x\in I:b(x)\in l(x_0,r_1(x_0))\cap S_1 \})\ge \frac{1}{2}\mu(\{x\in I:b(x)\in S_1\}).
	\end{equation*} 
	Hence the desired result is obtained by Lemma \ref{middlecases}.
	
	\bigskip
	
	\subsubsection*{Case 2: There exists $x_0\in [A,B]$, such that $r_3(x_0)=r_4(x_0)$}			
	The proof is the same as in Case 1.
	
	\bigskip
	
	\subsubsection*{Case 3: For any $x\in [A,B]$, $r_1(x)<r_2(x)$ and $r_3(x)<r_4(x)$}
	Our aim is to prove that there exists $y\in [A,B]$, such that
	\begin{equation}\label{achieve}
		\Big\{r:f(y,r)\ge \frac{1}{4}, \ g(y,r)\ge \frac{1}{4}\Big\}\cap \Big\{r:\tilde{f}(y,r)\ge \frac{1}{4}, \  \tilde{g}(y,r)\ge \frac{1}{4}\Big\}\neq \emptyset.
	\end{equation}

	\smallskip
	
	We need the following definition: for any $i\in\{1,2,3,4\}$, define
	\begin{equation*}
		r_i(c-)=\lim_{x\nearrow c}r_i(x),\quad \forall c\in (A,B],
	\end{equation*}
    and
	\begin{equation*}
		r_i(c+)=\lim_{x\searrow c}r_i(x),\quad \forall c\in [A,B).
	\end{equation*}
	Note that $r_i$ is always monotone, thus the above definition makes sense. 
	
	\smallskip
	
	Now we show the following important properties:\\
	$\bullet$ For any $c\in (A,B]$, one has 
	\begin{equation}\label{property1}
		r_1(c-)\le r_2(c)
	\end{equation}
     and 
     \begin{equation}\label{property3}
     	r_3(c)\le r_4(c-),
     \end{equation}
     $\bullet$ For any $c\in [A,B)$, one has
     \begin{equation}\label{property2}
     	r_1(c)\le r_2(c+)
     \end{equation}
     and 
     \begin{equation}\label{property4}
     	r_3(c+)\le r_4(c).
     \end{equation}

	We only prove \eqref{property1}. If $r_1(c-)>r_2(c)$, then for any given $a$ satisfying
	\begin{equation*}
		r_1(c-)>a>r_2(c),
	\end{equation*}
	by the definition of $r_1(c-)$, there exists a sequence $\{x_n\}$ satisfying $x_n\nearrow c$ such that
	\begin{equation*}
		r_1(x_n)>a>r_2(c).
	\end{equation*}
	Since $g$ is decreasing with respect to the angle, by the definition of $r_1(x_n)$, we have $g(x_n,a)\ge \frac{1}{4}$. Then from the fact that $g$ is left-continuous with respect to $x$, one has
	\begin{equation*}
		g(c,a)\ge \frac{1}{4}.
	\end{equation*}
	Besides, since $f$ is increasing with respect to the angle, we have
	\begin{equation*}
		f(c,a)\ge f(c,r_2(c))\ge \frac{1}{4}.
	\end{equation*}
	Thus by the definition of $r_2(c)$, one gets $a\le r_2(c)$. It contradicts $a>r_2(c)$. 
	
	\smallskip

	\bigskip
	
	Now we come back to the proof of \eqref{achieve}. By \eqref{fgr2r1} and \eqref{fgr3r4}, \eqref{achieve} is equivalent to proving that there exists $y\in [A,B]$, such that
	\begin{equation}\label{achieve2}
		[r_1(y),r_2(y)]\cap [r_3(y),r_4(y)]\neq \emptyset. 
	\end{equation}

    \smallskip
    
	If $r_1(x)\le r_4(x)$ for all $x\in [A,B]$, since $l_A$ divides $S_1$ into two closed parts $R_{1}$ and $R_{2}$ satisfying \eqref{muRiS1}, we have $r_2(A)\ge \frac{\pi}{2}$. Besides, note that $A\le \alpha_2$, which implies that $r_3(A)\le \frac{\pi}{2}$. Thus we only have the following four cases:\\
	$\bullet$ $r_3(A)\le r_1(A)\le r_4(A)\le r_2(A)$,\\
	$\bullet$ $r_3(A)\le r_1(A)\le r_2(A)\le r_4(A)$,\\
	$\bullet$ $r_1(A)\le r_3(A)\le r_4(A)\le r_2(A)$,\\
	$\bullet$ $r_1(A)\le r_3(A)\le r_2(A)\le r_4(A)$.\\
	Hence
	\begin{equation*}
		[r_1(A),r_2(A)]\cap [r_3(A),r_4(A)]\neq \emptyset.
	\end{equation*}
    and \eqref{achieve2} is obtained by letting $y=A$.
    
	Now suppose that there exists $x\in [A, B]$ such that $r_1(x)>r_4(x)$.  Then we define 
	\begin{equation*}
		c_0=\sup\{x\in [A,B]:r_1(x)-r_4(x)>0\}.
	\end{equation*}
	We will consider three cases according to the value of $c_0$.
	
	\subsubsection*{Subcase 3.1:  $A<c_0<B$} 
	From the definition of $c_0$ we know that
	\begin{equation}\label{r2r3c0-c0+}
		r_1(c_0-)\ge r_4(c_0-),\quad \text{and}\quad r_1(c_0+)\le r_4(c_0+).
	\end{equation} 
	Besides, for any $c_0<x\le B$, 
	\begin{equation*}
		r_1(x)\le r_4(x).
	\end{equation*}
	There are three subcases in this situation:\\
	(1) if there exists $c_0<x_0\le B$, such that $r_1(x_0)=r_4(x_0)$, then \eqref{achieve2} is obvious by letting $y=x_0$,\\
	(2) if for all $c_0<x\le B$,
		\begin{equation*}
			r_1(x)<r_4(x),
		\end{equation*} 
		but there exists $c_0<x_0\le B$, such that $r_2(x_0)\ge r_3(x_0)$, then \eqref{achieve2} is derived by letting $y=x_0$,\\
	(3) if for any $c_0<x\le B$,
		\begin{equation*}
			r_1(x)<r_4(x)\quad \text{and} \quad r_2(x)<r_3(x),
		\end{equation*} 
		this implies that
		\begin{equation*}
			r_2(c_0+)\le r_3(c_0+).
		\end{equation*}
		On the one hand, from \eqref{property2} and \eqref{property4} one has
		\begin{equation*}
			r_1(c_0)\le r_2(c_0+)\le r_3(c_0+)\le r_4(c_0).
		\end{equation*}
		On the other hand, from \eqref{property1}, \eqref{property3} and \eqref{r2r3c0-c0+} we get
		\begin{equation*}
			r_3(c_0)\le r_4(c_0-)\le r_1(c_0-)\le r_2(c_0).
		\end{equation*}
		Hence \eqref{achieve2} is obtained by letting $y=c_0$.

	\smallskip
	
	\subsubsection*{Subcase 3.2: $c_0=B$} 
	This implies that 
	\begin{equation*}
		r_1(B-)\ge r_4(B-).
	\end{equation*}
	From \eqref{property1} and \eqref{property3} we know that
	\begin{equation*}
		r_3(B)\leq r_4(B-)\le r_1(B-)\leq r_2(B).
	\end{equation*}
	Note that
	\begin{equation*}
		r_1(B)\le \frac{\pi}{2}\le r_4(B).
	\end{equation*}
	Hence \eqref{achieve2} is derived by letting $y=B$.
	
	\smallskip
	
	\subsubsection*{Subcase 3.3:  $c_0=A$} 
	This implies that for any $A<x\le B$,
	\begin{equation*}
		r_1(x)\le r_4(x).
	\end{equation*}
	We need to consider the following three subcases:\\
	(1) if there exists $A<x_0\le B$, such that $r_2(x_0)\ge r_4(x_0)$, then
		\begin{equation*}
			r_1(x_0)\le r_4(x_0)\le r_2(x_0).
		\end{equation*}
		Thus \eqref{achieve2} is obtained by letting $y=x_0$,\\
	(2) if for all $A<x\le B$, 
		\begin{equation*}
			r_2(x)<r_4(x),
		\end{equation*}
		but there exists $A<x_0\le B$, such that $r_2(x_0)\ge r_3(x_0)$, then \eqref{achieve2} is derived by letting $y=x_0$,\\
	(3) if for any $A<x\le B$, 
		\begin{equation*}
			r_2(x)<r_4(x)\quad \text{and} \quad r_2(x)<r_3(x),
		\end{equation*}
		this implies that 
		\begin{equation*}
			r_2(A+)\le r_3(A+).
		\end{equation*}
		Thus from \eqref{property2} and \eqref{property4} we have
		\begin{equation*}
			r_1(A)\leq r_2(A+)\le r_3(A+)\leq r_4(A).
		\end{equation*}
		Note that 
		\begin{equation*}
			r_3(A)\le \frac{\pi}{2}\le r_2(A).
		\end{equation*}
		Hence \eqref{achieve2} is obtained by letting $y=A$.

	Now we have proved \eqref{achieve2}. Thus we choose $r_0(y)\in [r_1(y),r_2(y)]\cap [r_3(y),r_4(y)]\neq \emptyset$, and let the line $L_1$ be the extension of the ray $l({y,r_0(y)})$. Besides, for any given $\tilde{y}\in (\alpha_1,\alpha_2)$, let $L_2$ be the line passing through $\tilde{y}$, such that $L_1\perp L_2$. Then $L_1$ and $L_2$ divide $\mathbb{C}$ into four closed quadrants $T_1$, $T_2$, $T_3$, $T_4$ (see Figure \ref{picture6}). Moreover, from \eqref{SiS4} we have
	\begin{equation*}
		\mu(\{x\in I:b(x)\in T_i\})\ge \frac{1}{4}\mu(\{x\in I:b(x)\in S_i\})\ge \frac{1}{16}\mu(I),\quad i\in \{1,2,3,4\}.
	\end{equation*}
\end{proof}

\begin{figure}[H] 
	\centering 
	\includegraphics[width=0.6\textwidth]{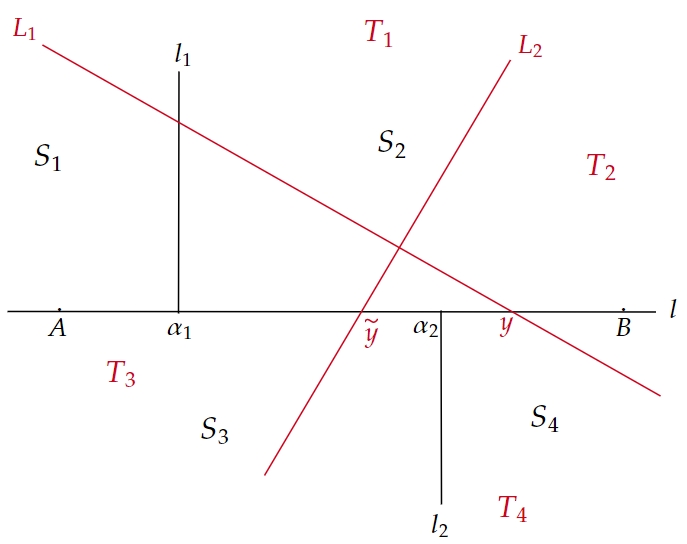} 
	\caption{} 
	\label{picture6} 
\end{figure}

\begin{rem}
	One would try to show Theorem \ref{divideS} first for simple functions, and then use a limit argument for general functions. However, it seems that the proof for simple functions is already complicated. In addition, we do not know how to deal with the limit argument either.
	\end{rem}

\bigskip

\section{Proof of Theorem \ref{nonschatten}: the scalar case}\label{schattenconv}

Our proof of Theorem \ref{nonschatten} relies on the proof of the commutative case $\M=\mathbb{C}$. However, it is also of independent interest to consider the commutative case in its own right. Therefore, we list the statement of Theorem \ref{nonschatten} for $\M=\mathbb{C}$ as well.

\begin{thm}\label{Converse}
	Let $1<p<\infty$ and $T\in B(L_2(\mathbb{R}^n))$ be a singular integral operator with a non-degenerate kernel $K(x,y)$ satisfying \eqref{standard}. Suppose that $b$ is a locally integrable complex-valued function. If $C_{T, b}\in S_p(L_2(\mathbb{R}^n))$, 
	then $b\in \pmb{B}_p(\mathbb{R}^n)$. Furthermore, we have
	\begin{equation*}
		\|b\|_{\pmb{B}_p(\mathbb{R}^n)}\lesssim_{n,p,T} \|C_{T, b}\|_{S_p(L_2(\mathbb{R}^n))}.
	\end{equation*}
	In particular, when $n\geq 2$ and $p\le n$, if $C_{T, b}\in S_p(L_2(\mathbb{R}^n))$, then $b$ is constant.
\end{thm}

In order to prove Theorem \ref{Converse}, we need to describe the Besov space  $\pmb{B}_p(\mathbb{R}^n)$ in terms of the Schatten class membership of commutators involving singular integral operators. In this section, we deal with singular integral operators associated with non-degenerate kernels in Hyt\"{o}nen's sense. We refer the reader to \cite{TH3} for more details about non-degenerate kernels. At first, we show that when $1<p<\8$, $[T, M_b]\in  S_p(L_2(\mathbb{R}^n))$ implies $b\in \pmb{B}^{\omega,2^n}_p(\mathbb{R}^n)$ (see Lemma \ref{conv1}), where $\pmb{B}^{\omega,2^n}_p(\mathbb{R}^n)$ is defined in \eqref{Bpw2nrnm}. Our proof is based on the complex median method, i.e. Theorem \ref{divideS}. Then we show that for $1<p<\8$, $\pmb{B}_p(\mathbb{R}^n)$ is the intersection of several dyadic martingale Besov spaces associated with different translated dyadic systems (see Proposition \ref{pdayun}). This enables us to transfer the martingale setting to the Euclidean setting. We start with non-degenerate kernels.

\subsection{Non-degenerate kernels}
We first give the definition of non-degeneracy of kernels.
\begin{definition}\label{nondege}
	Let $T\in B(L_2(\mathbb{R}^n))$ be a singular integral operator with kernel $K(x,y)$ satisfying standard kernel estimates \eqref{standard}. $K$ is called non-degenerate, if one of the following conditions holds:
	
	(a) for every $y\in \mathbb{R}^n$ and $r>0$, there exists $x\in B(y,r)^c$ such that
	\begin{equation}\label{nondege1}
		|K(x,y)|\ge \frac{1}{c_0r^n},
	\end{equation}
	where $c_0$ is a fixed positive constant.
	
	(b) if $K$ is a homogeneous kernel with
	\begin{equation}\label{nondege2}
		K(x,y)=\frac{\Omega(x-y)}{|x-y|^n},
	\end{equation}
	where $\Omega\in L_1(\mathbb{S}^{n-1})\backslash\{0\}$ and $\Omega(tx)=\Omega(x)$ for all $t>0$ and $x\in \mathbb{R}^n$ (here $\mathbb{S}^{n-1}$ is the sphere of $\mathbb{R}^n$), then there exists a Lebesgue point $\theta_0\in \mathbb{S}^{n-1}$ of $\Omega$ such that
	\begin{equation*}
		\Omega(\theta_0)\neq 0.
	\end{equation*}
	
\end{definition}	

\begin{remark}
	Let $T$ be a Calder\'{o}n-Zygmund transform with the convolution kernel $K(x,y)=\phi(x-y)$. Then the non-degeneracy condition (a) in Definition \ref{nondege} is simplified into the following form: for every $r>0$, there exists $x\in B(0,r)^c$ with
	\begin{equation*}
		|\phi(x)|\ge \frac{1}{c_0r^n},
	\end{equation*}
	We refer to \cite{TH3} for more details.
\end{remark}

In the rest of this section, the kernel $K$ will be always assumed to satisfy \eqref{standard} and to be non-degenerate. From the above definition, we obtain the following property for $K$.
\begin{lem}\label{ball}
	For every $A\ge 3$ and every ball $B=B(x_0,r)$, there is a disjoint ball $\tilde{B}=B(y_0,r)$ at distance $\mathrm{dist}(B,\tilde{B})\approx Ar$ such that
	\begin{equation}\label{kx0y0an}
		|K(y_0,x_0)|\approx \frac{1}{A^nr^n},
	\end{equation}
	and for all $x_1\in B$ and $y_1\in \tilde{B}$, 
	\begin{equation}\label{kx1y1x0y0}
		|K(y_1,x_1)-K(y_0,x_0)|\lesssim\frac{1}{A^{n+\alpha}r^n}.
	\end{equation}
	Furthermore, if $K(y_0,x_0)$ is real and $A$ is sufficiently large,
	then there exists a positive number $ \varrho$ which depends on $A$, $\alpha$ and $n$ and is much less than $1$ such that
	\begin{equation}\label{88888}
		|\mathrm{Im}(K(y_1,x_1))|\le \varrho\mathrm{Re}(K(y_1,x_1))
	\quad\text{and}\quad
		|K(y_1,x_1)|\le 2\mathrm{Re}(K(y_1,x_1))
	\end{equation}
	for any $x_1\in B$ and $y_1\in \tilde{B}$, where $\mathrm{Re}\big(K(y_1,x_1)\big)$ and $\mathrm{Im}\big(K(y_1,x_1)\big)$ are the real and imaginary parts of $K(y_1,x_1)$ respectively.
\end{lem}

\begin{proof}
	(1) Assume that $K$ is as in Definition \ref{nondege} (a). For a fixed ball $B=B(x_0,r)$ and $A\ge 3$, thanks to the standard estimate of $K$ and \eqref{nondege1}, there exists a point $y_0\in B(x_0,Ar)^c$ such that
	\begin{equation*}
		\frac{1}{c_0(Ar)^n}\le |K(y_0,x_0)|\le \frac{C}{|x_0-y_0|^n}.
	\end{equation*}
	This implies that 
	\begin{equation*}
		Ar\le |x_0-y_0|\le (c_0C)^{\frac{1}{n}}Ar,\quad |K(y_0,x_0)|\approx \frac{1}{A^nr^n}.
	\end{equation*}
	Let $\tilde{B}=B(y_0,r)$. Since $A\ge 3$, we have $\mathrm{dist}(B,\tilde{B})\approx |x_0-y_0|$. Thus
	$\mathrm{dist}(B,\tilde{B})\approx Ar$.
	
	(2) Assume that $K$ is as in Definition \ref{nondege} (b). For a fixed ball $B=B(x_0,r)$ and $A\ge 3$, let $y_0=x_0+Ar\theta_0$ and $\tilde{B}=B(y_0,r)$. It is not hard to check that $\mathrm{dist}(B,\tilde{B})\approx Ar$ and
	\begin{equation*}
		|K(y_0,x_0)|=\frac{|\Omega(y_0-x_0)|}{|y_0-x_0|^n}=\frac{|\Omega(Ar\theta_0)|}{|Ar\theta_0|^n}=\frac{|\Omega(\theta_0)|}{(Ar)^n}\approx \frac{1}{(Ar)^n}.
	\end{equation*}
This finishes the proof of \eqref{kx0y0an}. 

Now we show \eqref{kx1y1x0y0} and \eqref{88888}. For $x_1\in B$ and $y_1\in \tilde{B}$, from \eqref{standard} we have  
	\begin{equation*}
		\begin{aligned}
			|K(y_1,x_1)-K(y_0,x_0)|{}&\le |K(y_1,x_1)-K(y_0,x_1)|+|K(y_0,x_1)-K(y_0,x_0)|     \\
			&\lesssim \frac{|y_1-y_0|^\alpha}{|y_1-x_1|^{n+\alpha}}+\frac{|x_1-x_0|^\alpha}{|x_1-y_0|^{n+\alpha}}\\
			&\lesssim \frac{r^{\alpha}}{(Ar)^{n+\alpha} }+\frac{r^{\alpha}}{(Ar)^{n+\alpha} }\approx \frac{1}{A^{n+\alpha}r^n}.
		\end{aligned}
	\end{equation*}
	When $A$ is sufficiently large, $K(y_1,x_1)$ will be very close to $K(y_0,x_0)$. Hence if $K(y_0,x_0)$ is real, we deduce that
	\begin{equation*}
		A^nr^n|\mathrm{Re}(K(y_1,x_1))-K(y_0,x_0)|\lesssim\frac{1}{A^{\alpha}}
	\end{equation*}
and
\begin{equation*}
	A^nr^n|\mathrm{Im}(K(y_1,x_1))|\lesssim\frac{1}{A^{\alpha}}.
\end{equation*}
Note that by \eqref{kx0y0an} that
\begin{equation*}
	A^nr^n|K(y_0,x_0)|\approx 1.
\end{equation*}	
Therefore, we deduce that
	\begin{equation*}
		|\mathrm{Im}(K(y_1,x_1))|\le \varrho\mathrm{Re}(K(y_1,x_1))
	\quad \text{and} \quad
		|K(y_1,x_1)|\le 2\mathrm{Re}(K(y_1,x_1))
	\end{equation*}
	for sufficiently small $\varrho$.
\end{proof}

\subsection{A key lemma}
This subsection establishes a key lemma for the proof of Theorem \ref{Converse}. We will need the following lemma due to Rochberg and Semmes in \cite{RSe}.
\begin{lemma}\label{RSNWO}
	Let $1<p<\infty$. Assume that $\{e_I\}_{I\in\mathcal{D}}$ and $\{f_I\}_{I\in\mathcal{D}}$ are function sequences in $L_2(\mathbb{R}^n)$ satisfying $\mathrm{supp}e_I, \,\mathrm{supp}f_I\subseteq I$ and
	$\|e_I\|_{\infty},\,\|f_I\|_{\infty}\le |I|^{-\frac{1}{2}}$. For any bounded compact operator $V$ on $L_2(\mathbb{R}^n)$, one has
	\begin{equation*}
		\begin{aligned}
			\sum_{I\in\mathcal{D}}\big|\langle e_I,V(f_I)\rangle\big|^p\lesssim_{n,p} \|V\|^p_{S_p(L_2(\mathbb{R}^n))}.
		\end{aligned}
	\end{equation*}
\end{lemma}

In Section \ref{noncomschattenconv}, we will extend Lemma \ref{RSNWO} to the semicommutative setting, that is Theorem \ref{nonRSNWO}. We also give a different but self-contained proof of Theorem \ref{nonRSNWO} in Section \ref{noncomschattenconv}, which directly implies Lemma \ref{RSNWO}. 

The following lemma is explicitly stated in \cite{CJM2013}. In fact, it is due to Mei \cite{Mei4} in the case of $\mathbb{T}^n$. Mei made a remark \cite[Remark 7]{Mei4} for $\mathbb{R}^n$. However, Conde \cite{CJM2013} noted that Mei's remark is not correct, and he finally found the following right substitution.
	\begin{lem}\label{Domegan1}
	There exist $n+1$ dyadic systems $\mathcal{D}^{\omega(1)},\mathcal{D}^{\omega(2)},\cdots,\mathcal{D}^{\omega(n+1)}$ in $\mathbb{R}^n$, where $\omega(i)\in(\{0,1\}^n)^\mathbb{Z}$ for all $ 1\le i\le n+1$, such that for any cube $B\subset\mathbb{R}^n$, there exists some cube $Q\in \bigcup\limits_{i=1}^{n+1} \mathcal{D}^{\omega(i)}$ satisfying
    \begin{equation*}
    	B\subseteq Q\subseteq c_n B,
    \end{equation*}
where $c_n$ only depends on $n$. Moreover, $n+1$ is the optimal number of such dyadic systems.
\end{lem}
\begin{lem}\label{Domegan2}
	Let $\mathcal{D}^{\omega(1)},\mathcal{D}^{\omega(2)},\cdots,\mathcal{D}^{\omega(n+1)}$ be $n+1$ dyadic systems as in Lemma \ref{Domegan1}. For any cube $B\subset \mathbb{R}^n$ of length $2^{-k}$ with $k\in\mathbb{Z}$, let $Q\in\bigcup\limits_{i=1}^{n+1} \mathcal{D}^{\omega(i)}$ such that
	\begin{equation*}
		B\subseteq Q\subseteq c_n B.
	\end{equation*}
  Then for any given $\omega\in(\{0,1\}^n)^\mathbb{Z}$, $Q$ contains only a finite number of dyadic cubes in $\mathcal{D}_k^\omega$, and this number only depends on $n$.
\end{lem}
\begin{proof}
	Fix $k\in\mathbb{Z}$. From Lemma \ref{Domegan1}, we have
	\begin{equation*}
		\ell(Q)\le c_n \ell(B)=c_n2^{-k}.
	\end{equation*}
This implies that $Q$ contain at most $\lfloor c_n\rfloor^n$ dyadic cubes in $\mathcal{D}_k^\omega$.
\end{proof}

Now we come to the following lemma, which is vital for the proof of Theorem \ref{Converse}. It describes the relation between $\|b\|_{\pmb{B}_p^{\omega,2^n}(\mathbb{R}^n)}$ and $\|[T,M_b]\|_{S_p(L_2(\mathbb{R}^n))}$.

\begin{lemma}\label{conv1}
	Let $1<p<\infty$. Suppose that $b$ is a locally integrable complex-valued function. If $[T,M_b]\in S_p(L_2(\mathbb{R}^n))$, then for any $\omega\in(\{0,1\}^n)^\mathbb{Z}$,
	\begin{equation*}
		\begin{aligned}
			\|b\|_{\pmb{B}_p^{\omega,2^n}(\mathbb{R}^n)}\lesssim_{n,p,T}\|[T,M_b]\|_{S_p(L_2(\mathbb{R}^n))}.
		\end{aligned}
	\end{equation*}	
\end{lemma}	

\begin{proof}
Without loss of generality, we assume that $\omega=0$. Recall that
	\begin{equation*}
		\begin{aligned}
			\|b\|^p_{\pmb{B}_p^{0,2^n}(\mathbb{R}^n)}{}&=\sum_{I\in\mathcal{D}^0}\sum_{i=1}^{2^n-1}|I|^{-\frac{p}{2}}|\langle h_I^i,b\rangle|^p.
		\end{aligned}
	\end{equation*}
	For any given $k\in \mathbb{Z}$ and $I\in \mathcal{D}_k^0$, let $c(I)$ be the center of $I$. Let $B=B(c(I),2^{-k}\sqrt{n})$, then $I\subset B$. Due to Lemma \ref{ball}, for any given $A$ which is much greater than $n$, there is a disjoint ball $\tilde{B}=B(y_0,2^{-k}\sqrt{n})$ at distance $\mathrm{dist}(B,\tilde{B})\approx_n 2^{-k}A$, such that
	\begin{equation}\label{minus1}
		|K(y_0,c(I))|\approx_{n,T} \frac{1}{A^n|I|}.
	\end{equation}
	Then we choose a cube $\hat{I}\in \mathcal{D}^0_k$ such that $\hat{I}\subset \tilde{B}$ and $y_0\in \hat{I}$. It is not hard to find that $\mathrm{dist}(I,\hat{I})\approx_n 2^{-k}A$. In the following, we will always assume that $A$ is a sufficiently large number.
	
	By Theorem \ref{divideS}, there exist $\theta={\theta}(\hat{I},b)\in [0,2\pi)$ and $\alpha_{\hat{I}}(b)\in\mathbb{C}$ such that
	if we denote
	\begin{equation*}
		\begin{aligned}
		F_s^I=\bigg\{{x}\in \hat{I}: -\frac{\pi}{4}+\frac{(s-1)\pi}{2}\le \mathrm{arg}\Big(e^{\mathrm{i}\theta}\big(\alpha_{\hat{I}}(b)-b({x})\big)\Big)\le -\frac{\pi}{4}+\frac{s\pi}{2} \quad \text{or}\quad b({x})=\alpha_{
		\hat{I}}(b)\bigg\},
	    \end{aligned}
	\end{equation*}
	where $s\in\{1,2,3,4\}$, and $\mathrm{arg}(z)$ is the argument of a complex number $z$, then $|F_s^I|\ge \frac{1}{16}|\hat{I}|$, and $F_1^I \cup F_2^I \cup F_3^I \cup F_4^I=\hat{I}$.  Note that for any $1\le i\le 2^n-1$,
	\begin{equation*}
		\begin{aligned}
			|I|^{-\frac{1}{2}} |\langle h_I^i,b\rangle|{}
			&=|I|^{-\frac{1}{2}}\bigg|\int_I \overline{h_I^i(x)}b(x)dx\bigg|\\
			&=|I|^{-\frac{1}{2}}\bigg|\int_I \overline{h_I^i(x)}\big(b(x)-\alpha_{\hat{I}}(b)\big)dx\bigg|\\
			&\le \frac{1}{|I|}\int_I |b(x)-\alpha_{\hat{I}}(b)|dx\\
			&=\frac{1}{|I|}\sum_{q=1}^{2^n}\int_{I(q)} |b(x)-\alpha_{\hat{I}}(b)|dx,
		\end{aligned}
	\end{equation*}
    where $I(q)$ is the $q$-th subinterval of $I$.
	Similarly, we define
\begin{equation*}
		E_s^I=\bigg\{x\in I: -\frac{\pi}{4}+\frac{(s-1)\pi}{2}\le \mathrm{arg}\Big(e^{\mathrm{i}{\theta}}\big(b({x})-\alpha_{\hat{I}}(b)\big)\Big)\le -\frac{\pi}{4}+\frac{s\pi}{2} \quad \text{or}\quad b({x})=\alpha_{
			\hat{I}}(b)\bigg\},
	\end{equation*}
	then for any $1\le i\le 2^n-1$,
	\begin{equation}\label{M1M2}
		\begin{aligned}
			|I|^{-\frac{1}{2}} |\langle h_I^i,b\rangle|{}
			&\le\frac{1}{|I|}\sum_{q=1}^{2^n}\int_{I(q)} |b(x)-\alpha_{\hat{I}}(b)|dx\\
			&\le \frac{1}{|I|}\sum_{q=1}^{2^n}\sum_{s=1}^4\int_{I(q)\cap E_s^I} |b(x)-\alpha_{\hat{I}}(b)|dx\\
			&=:\sum_{s=1}^4M_s^I.
		\end{aligned}
	\end{equation}
	Note that for any $s\in \{1,2,3,4\}$, if $x\in I(q)\cap E_s^I$ and $\hat{x}\in F_s^I$, then
	\begin{equation}\label{bxal}
		\begin{aligned}
			|b(x)-\alpha_{\hat{I}}(b)|{}&\le \big|e^{\mathrm{i}{\theta}}\big(b(x)-\alpha_{\hat{I}}(b)\big)\big|+\big|e^{\mathrm{i}{\theta}}\big(\alpha_{\hat{I}}(b)-b(\hat{x})\big)\big|\\
			&\le 2\big|e^{\mathrm{i}{\theta}}\big(b(x)-b(\hat{x})\big)\big|
			= 2|b(x)-b(\hat{x})|.
		\end{aligned}
	\end{equation}
	Thus by \eqref{minus1} and \eqref{bxal} one has
	\begin{equation*}
		\begin{aligned}
			M_s^I{}&\approx \frac{1}{|I|}\sum_{q=1}^{2^n}\int_{I(q)\cap E_s^I} |b(x)-\alpha_{\hat{I}}(b)|dx \cdot \frac{|F_s^I|}{|I|}\\
			&\approx_{n,T} \frac{A^n}{|I|}\sum_{q=1}^{2^n}\int_{I(q)\cap E_s^I} |b(x)-\alpha_{\hat{I}}(b)||K(y_0,c(I))||F_s^I|dx\\
			&= \frac{A^n}{|I|}\sum_{q=1}^{2^n}\int_{I(q)\cap E_s^I}\int_{F_s^I} |b(x)-\alpha_{\hat{I}}(b)||K(y_0,c(I))|d\hat{x}dx\\
			&\le \frac{2A^n}{|I|}\sum_{q=1}^{2^n}\int_{I(q)\cap E_s^I}\int_{F_s^I} |b(x)-b(\hat{x})||K(y_0,c(I))|d\hat{x}dx.
		\end{aligned}
	\end{equation*}
	For any $s\in \{1,2,3,4\}$, $x\in I(q)\cap E_s^I$ and $\hat{x}\in F_s^I$, from Lemma \ref{ball} we have
	\begin{equation*}
		\begin{aligned}
			|K(\hat{x},x)-K(y_0,c(I))|\lesssim_{n,T}  \frac{1}{A^{n+\alpha}|I|}.
		\end{aligned}
	\end{equation*}
	Hence $|K(\hat{x},x)|\ge \frac{1}{2}|K(y_0,c(I))|$ thanks to \eqref{minus1}. Thus
	\begin{equation*}
		\begin{aligned}
			M_s^I{}&\lesssim_{n,T} \frac{4A^n}{|I|}\sum_{q=1}^{2^n}\int_{I(q)\cap E_s^I}\int_{F_s^I} |b(x)-b(\hat{x})||K(\hat{x},x)|d\hat{x}dx.
		\end{aligned}
	\end{equation*}
	
	We first estimate $M_1^I$.
	Let $\theta_1\in [0,2\pi)$ such that $e^{i\theta_1}K(y_0,c(I))$ is  positive. Thus from Lemma \ref{ball} we obtain
	\begin{equation}\label{theta1K}
		|K(\hat{x},x)| \le 2\text{Re}\big(e^{i\theta_1}K(\hat{x},x)\big).
	\end{equation}
	Moreover, note that for any $x\in I(q)\cap E_1^I$ and $\hat{x}\in F_1^I$, the arguments of $e^{\mathrm{i}{\theta}}\big(b({x})-\alpha_{\hat{I}}(b)\big)$ and $e^{\mathrm{i}{\theta}}\big(\alpha_{\hat{I}}(b)-b(\hat{x})\big)$ both belong to $[-\frac{\pi}{4},\frac{\pi}{4}]$, thus
	\begin{equation*}
		-\frac{\pi}{4}\le \mathrm{arg}\Big(e^{\mathrm{i}{\theta}}\big(b(x)-b(\hat{x})\big)\Big)\le \frac{\pi}{4}.
	\end{equation*}
	This implies that
	\begin{equation}\label{b-b1}
		\begin{aligned}
			\big|e^{\mathrm{i}{\theta}}\big(b(x)-b(\hat{x})\big)\big|{}&
			\le 2\mathrm{Re}\Big(e^{\mathrm{i}{\theta}}\big(b(x)-b(\hat{x})\big)\Big)
		\end{aligned}
	\end{equation} 
	and
	\begin{equation}\label{imreb1}
		\Big|\mathrm{Im}\Big(e^{\mathrm{i}{\theta}}\big(b(x)-b(\hat{x})\big)\Big)\Big|\le \mathrm{Re}\Big(e^{\mathrm{i}{\theta}}\big(b(x)-b(\hat{x})\big)\Big).
	\end{equation}
	Thus from \eqref{theta1K} and \eqref{b-b1} we deduce that
	\begin{equation*}
		\begin{aligned}
			M_1^I{}&
			\lesssim_{n,T}\frac{4A^n}{|I|}\sum_{q=1}^{2^n}\int_{I(q)\cap E_1^I}\int_{F_1^I} \big|e^{\mathrm{i}{\theta}}\big(b(x)-b(\hat{x})\big)\big||K(\hat{x},x)|d\hat{x}dx\\
			&\le \frac{8A^n}{|I|}\sum_{q=1}^{2^n}\int_{I(q)\cap E_1^I}\int_{F_1^I} \mathrm{Re}\Big(e^{\mathrm{i}{\theta}}\big(b(x)-b(\hat{x})\big)\Big)|K(\hat{x},x)|d\hat{x}dx\\
			&\le  \frac{16 A^n}{|I|}\sum_{q=1}^{2^n}\int_{I(q)\cap E_1^I}\int_{F_1^I} \mathrm{Re}\Big(e^{\mathrm{i}{\theta}}\big(b(x)-b(\hat{x})\big)\Big)\cdot\text{Re}\big(e^{i\theta_1}K(\hat{x},x)\big)d\hat{x}dx.
		\end{aligned}
	\end{equation*}
	Notice that from Lemma \ref{ball} one has 
	\begin{equation}\label{eitheta1}
		\big|\mathrm{Im}\big(e^{i\theta_1}K(\hat{x},x)\big)\big|\le \varrho\mathrm{Re}\big(e^{i\theta_1}K(\hat{x},x)\big),
	\end{equation}
	where $\varrho$ is a positive number which depends on $A$, $\alpha$ and $n$ and is much less than $1$.
	Then from \eqref{imreb1} and \eqref{eitheta1} we derive that
	\begin{equation*}
		\begin{aligned}
			\mathrm{Re}\Big(e^{\mathrm{i}{\theta}}\big(b(x)-b(\hat{x})\big)\Big)\cdot\text{Re}\big(e^{i\theta_1}K(\hat{x},x)\big)
			&\le 2\mathrm{Re}\Big(e^{\mathrm{i}{\theta}}\big(b(x)-b(\hat{x})\big)\Big)\cdot\text{Re}\big(e^{i\theta_1}K(\hat{x},x)\big)\\
			& \quad \quad \quad \quad-2\mathrm{Im}\Big(e^{\mathrm{i}{\theta}}\big(b(x)-b(\hat{x})\big)\Big)\cdot\text{Im}\big(e^{i\theta_1}K(\hat{x},x)\big)\\
			&=2\mathrm{Re}\Big(e^{\mathrm{i}{\theta}}(b(x)-b(\hat{x}))\cdot e^{i\theta_1}K(\hat{x},x)\Big).
		\end{aligned}
	\end{equation*}
	Hence
	\begin{equation*}\label{M1I}
		\begin{aligned}
			M_1^I{}&
			\lesssim_{n,T} \frac{32A^n}{|I|}\sum_{q=1}^{2^n}\int_{I(q)\cap E_1^I}\int_{F_1^I} \mathrm{Re}\Big(e^{\mathrm{i}{\theta}}(b(x)-b(\hat{x}))\cdot e^{i\theta_1}K(\hat{x},x)\Big)d\hat{x}dx\\
			&\le \frac{32A^n}{|I|}\sum_{q=1}^{2^n}\bigg|\int_{I(q)\cap E_1^I}\int_{F_1^I} e^{\mathrm{i}{\theta}}(b(x)-b(\hat{x}))\cdot e^{i\theta_1}K(\hat{x},x)d\hat{x}dx\bigg|\\
			&=\frac{32A^n}{|I|}\sum_{q=1}^{2^n}\bigg|\int_{I(q)\cap E_1^I}\int_{F_1^I} (b(x)-b(\hat{x}))K(\hat{x},x)d\hat{x}dx\bigg|.
		\end{aligned}
	\end{equation*}
Similarly, the other three terms $M_2^I$, $M_3^I$ and $M_4^I$ can be dealt with in the same way as $M_1^I$ by rotation, and we obtain
	\begin{equation}\label{M2I}
		\begin{aligned}
			M_s^I{}&
			\lesssim_{n,T} \frac{A^n}{|I|}\sum_{q=1}^{2^n}\bigg|\int_{I(q)\cap E_s^I}\int_{F_s^I} (b(x)-b(\hat{x}))K(\hat{x},x)d\hat{x}dx\bigg|,\quad s\in\{1,2,3,4\}.
		\end{aligned}
	\end{equation}
	Hence from \eqref{M1M2} and \eqref{M2I} one has
	\begin{equation}\label{second2}
		\begin{aligned}
			\|b\|^p_{\pmb{B}_p^{0,2^n}(\mathbb{R}^n)}{}
			&= \sum_{k\in\mathbb{Z}}\sum_{I\in\mathcal{D}^0_k}\sum_{i=1}^{2^n-1} |I|^{-\frac{p}{2}} |\langle h_I^i,b\rangle|^p\\
			&\le (2^n-1) \cdot 4^{p-1} \sum_{I\in\mathcal{D}^0}\sum_{s=1}^4 (M_s^I)^p\\
			&\lesssim_{n,p,T} A^{np}\sum_{s=1}^4\sum_{q=1}^{2^n}\sum_{I\in\mathcal{D}^0}\bigg|\Big\langle \frac{|I(q)|^{\frac{1}{2}}\mathbbm{1}_{F_s^I}}{|I|}, [T,M_b]\frac{\mathbbm{1}_{I(q)\cap E_s^I}}{|I(q)|^{\frac{1}{2}}}\Big\rangle\bigg|^p.
		\end{aligned}
	\end{equation}
	For any $I\in\mathcal{D}^0_k$, note that $\mathrm{dist}(I,\hat{I})=2^{-k}C_nA$, where $C_n$ is a constant only depending on $n$. Let $c(\hat{I})$ be the center of $\hat{I}$. We consider the cube $Q(I)$, where the center of $Q(I)$ is $\frac{c(I)+c(\hat{I})}{2}$, and the length of $Q(I)$ is $2^{-k+1}C_nA$. This implies that $I,\hat{I}\subset Q(I)$. Besides, from Lemma \ref{Domegan1} we know that there exists some cube $J(I)\in\bigcup\limits_{i=1}^{n+1} \mathcal{D}^{\omega(i)}$ such that
	\begin{equation*}
		Q(I)\subseteq J(I)\subseteq c_n Q(I).
	\end{equation*}
	Notice that $I\subseteq J(I)$ and
	\begin{equation*}
		\ell(I)\leq \ell(J(I))\le c_n\ell(Q(I))\le 2c_nC_nA\ell(I).
	\end{equation*}
    Now for any $s\in\{1,2,3,4\}$ and $1\le q\le 2^n$, let
	\begin{equation*}
		e_{J(I),s,q}=\frac{|I(q)|^{\frac{1}{2}}\mathbbm{1}_{F_s^I}}{|I|} \quad \text{and} \quad f_{J(I),s,q}=\frac{\mathbbm{1}_{I(q)\cap E_s^I}}{|I(q)|^{\frac{1}{2}}}.
	\end{equation*}
	Then  $\mathrm{supp}e_{J(I),s,q},\,\mathrm{supp}f_{J(I),s,q}\subseteq J(I)$ and $\|e_{J(I),s,q}\|_{\infty},\,\|f_{J(I),s,q}\|_{\infty}\le C|J(I)|^{-\frac{1}{2}}$, where 
	the constant $C$ only depends on $n$ and $A$. Note that from \eqref{second2} one has
	\begin{align*}
		\|b\|^p_{\pmb{B}_p^{0,2^n}(\mathbb{R}^n)}&\lesssim_{n,p,T}A^{np}\sum_{s=1}^4\sum_{q=1}^{2^n}\sum_{I\in\mathcal{D}^0}\Big|\big\langle e_{J(I),s,q}, [T,M_b](f_{J(I),s,q})\big\rangle\Big|^p.
	\end{align*}
	From Lemma \ref{Domegan2}, each $J(I)$ contains only a finite number of dyadic cubes in $\mathcal{D}_k^0$, and this number only depends on $n$ and $A$. Therefore, from Lemma \ref{RSNWO} we get
	\begin{equation*}
		\begin{aligned}
			\|b\|_{\pmb{B}_p^{0,2^n}(\mathbb{R}^n)}
			\lesssim_{n,p,T} \|[T,M_b]\|_{S_p(L_2(\mathbb{R}^n))},
		\end{aligned}
	\end{equation*}
	as long as we fix $A$. (Here we apply Lemma \ref{RSNWO} for dyadic systems $ \mathcal{D}^{\omega(i)}$.)
\end{proof}

\smallskip

\subsection{Proof of Theorem \ref{Converse}}


As mentioned before, we will show that for $1<p<\infty$, the operator-valued Besov space $\pmb{B}_p(\mathbb{R}^n,L_p(\mathcal{M}))$ is the intersection of finite well-chosen martingale Besov spaces, where the number of chosen martingales only depends on $n$. We will use the dyadic covering result in Lemma \ref{Domegan1}.
\begin{prop}\label{pdayun}
	Let $1<p<\infty$. Let $\mathcal{D}^{\omega(1)}, \mathcal{D}^{\omega(2)}, \cdots, \mathcal{D}^{\omega(n+1)}$ be $n+1$ dyadic systems as in Lemma \ref{Domegan1}. Then
	$$  \pmb{B}_p(\mathbb{R}^n,L_p(\mathcal{M}))=\bigcap_{i=1}^{n+1}  \pmb{B}_p^{\omega(i),2^n}(\mathbb{R}^n,\mathcal{M}), $$
	where the norm for the intersection on the right hand side is the maximum of the involved norms.
\end{prop}

\begin{proof}
	By the standard limit argument, it suffices to show that if $b$ is a locally integrable $L_p(\M)$-valued function, and $b\in \pmb{B}_p^{\omega(i),2^n}(\mathbb{R}^n,\mathcal{M})$ for any $1\le i\le n+1$, then
	  \begin{equation}\label{hzhang}
		\|b\|_{\pmb{B}_p(\mathbb{R}^n,L_p(\mathcal{M}))}\approx_{n,p} \sum_{i=1}^{n+1} \|b\|_{\pmb{B}_p^{\omega(i),2^n}(\mathbb{R}^n,\mathcal{M})}.
	\end{equation}
	Note that
	\begin{equation*}
		\begin{aligned}
			\|b\|^p_{\pmb{B}_p(\mathbb{R}^n,L_p(\mathcal{M}))}{}&=\int_{\mathbb{R}^n\times\mathbb{R}^n} \frac{\|b(x)-b(y)\|^p_{L_p(\mathcal{M})}}{|x-y|^{2n}}dxdy\\
			&= \sum_{k\in \mathbb{Z}}\int_{2^{-(k+2)}<|x-y|\le 2^{-(k+1)}} \frac{\|b(x)-b(y)\|^p_{L_p(\mathcal{M})}}{|x-y|^{2n}}dxdy\\
			&\approx_{n} \sum_{k\in \mathbb{Z}}\int_{2^{-(k+2)}<|x-y|\le 2^{-(k+1)}} 2^{2nk}\|b(x)-b(y)\|^p_{L_p(\mathcal{M})}dxdy\\
			&=\sum_{k\in \mathbb{Z}}\sum_{I\in \mathcal{D}^0_k}\frac{1}{|I|^2}\int_{I}\int_{2^{-(k+2)}<|x-y|\le 2^{-(k+1)}} \|b(x)-b(y)\|^p_{L_p(\mathcal{M})}dxdy.
		\end{aligned}
	\end{equation*}
	For any given $k\in \mathbb{Z}$ and $I\in\mathcal{D}^0_k$, we consider the cube $2I$. From Lemma \ref{Domegan1}, we know that there exists $Q(I)\in\bigcup\limits_{i=1}^{n+1} \mathcal{D}^{\omega(i)}$ such that
	\begin{equation*}
		2I\subseteq Q(I)\subseteq c_n\cdot 2I.
	\end{equation*}
	Besides, note that
	\begin{equation*}
		\ell(Q(I))\le 2c_n\ell(I)=c_n2^{-k+1}.
	\end{equation*}
From Lemma \ref{Domegan2} it implies that $Q(I)$ contains only a finite number of dyadic cubes in $\mathcal{D}_k^0$, and the number only depends on $n$.
	Thus
	\begin{equation*}
		\begin{aligned}
			{}&\frac{1}{|I|^2}\int_{I}\int_{2^{-(k+2)}<|x-y|\le 2^{-(k+1)}} \|b(x)-b(y)\|^p_{L_p(\mathcal{M})}dxdy\\
			&\lesssim_n \frac{1}{|Q(I)|^2}\int_{Q(I)\times Q(I)} \|b(x)-b(y)\|^p_{L_p(\mathcal{M})}dxdy.
		\end{aligned}
	\end{equation*}
	Then
	\begin{equation*}
		\begin{aligned}
			\|b\|^p_{\pmb{B}_p(\mathbb{R}^n,L_p(\mathcal{M}))}{}&
			\lesssim_n \sum_{k\in \mathbb{Z}}\sum_{I\in \mathcal{D}^0_k}\frac{1}{|Q(I)|^2}\int_{Q(I)\times Q(I)} \|b(x)-b(y)\|^p_{L_p(\mathcal{M})}dxdy\\
			&\lesssim_n\sum_{i=1}^{n+1}\sum_{Q\in \mathcal{D}^{\omega(i)}}\frac{1}{|Q|^2}\int_{Q\times Q} \|b(x)-b(y)\|^p_{L_p(\mathcal{M})}dxdy.
		\end{aligned}
	\end{equation*}
	Note that for any given $1\le i\le n+1$ and $Q\in \mathcal{D}^{\omega(i)}$,
	\begin{equation*}
		\begin{aligned}
			  {}&\int_{Q\times Q} \|b(x)-b(y)\|^p_{L_p(\mathcal{M})}dxdy\\
			  &=\int_{Q\times Q} \bigg\|b(x)-\bigg\langle \frac{\mathbbm{1}_Q}{|Q|},b\bigg\rangle+\bigg\langle \frac{\mathbbm{1}_Q}{|Q|},b\bigg\rangle-b(y)\bigg\|^p_{L_p(\mathcal{M})}dxdy\\
			  &\lesssim_p |Q|\int_{Q} \bigg\|b(x)-\bigg\langle \frac{\mathbbm{1}_Q}{|Q|},b\bigg\rangle\bigg\|^p_{L_p(\mathcal{M})}dx.
		\end{aligned}
	\end{equation*}
	Then from Proposition \ref{equbbk} we obtain
	\begin{equation*}
		\begin{aligned}
			\|b\|^p_{\pmb{B}_p(\mathbb{R}^n,L_p(\mathcal{M}))}
			&\lesssim_{n,p} \sum_{i=1}^{n+1}\sum_{Q\in \mathcal{D}^{\omega(i)}}\frac{1}{|Q|}\int_{Q} \bigg\|b(x)-\bigg\langle \frac{\mathbbm{1}_Q}{|Q|},b\bigg\rangle\bigg\|^p_{L_p(\mathcal{M})}dx\\
			&=\sum_{i=1}^{n+1}\sum_{k\in\mathbb{Z}}\sum_{Q\in \mathcal{D}_k^{\omega(i)}}\frac{1}{|Q|}\int_{Q} \|b(x)-b_k^{\omega(i)}(x)\|^p_{L_p(\mathcal{M})}dx\\
			&=\sum_{i=1}^{n+1}\sum_{k\in\mathbb{Z}} 2^{nk}\|b-b_k^{\omega(i)}\|^p_{L_p(\mathbb{R}^n,L_p(\mathcal{M}))}\approx_{n,p} \sum_{i=1}^{n+1}\|b\|^p_{\pmb{B}_p^{\omega(i),2^n}(\mathbb{R}^n,\mathcal{M})},
		\end{aligned}
	\end{equation*}
where $$b_k^{\omega(i)}=\sum_{Q\in\mathcal{D}_k^{\omega(i)}}\bigg\langle \frac{\mathbbm{1}_Q}{|Q|},b\bigg\rangle \mathbbm{1}_Q.$$
Therefore, from Lemma \ref{Comparison} and the above inequality, we derive \eqref{hzhang} as desired.
\end{proof}

	 
Now we give the proof of Theorem \ref{Converse}.
\begin{proof}[Proof of Theorem \ref{Converse}]
 For any $\omega\in(\{0,1\}^n)^\mathbb{Z}$, from Lemma \ref{conv1} one has
	\begin{equation*}
		\|b\|_{\pmb{B}_p^{\omega,2^n}(\mathbb{R}^n)}\lesssim_{n,p,T} \|[T,M_{b}]\|_{S_p(L_2(\mathbb{R}^n))}<\infty.
	\end{equation*}
	Hence from Proposition \ref{pdayun} we obtain that $b\in \pmb{B}_p(\mathbb{R}^n)$ and
	\begin{equation*}
		\|b\|_{\pmb{B}_p(\mathbb{R}^n)}\lesssim_{n,p,T} \|[T,M_{b}]\|_{S_p(L_2(\mathbb{R}^n))}<\infty.
	\end{equation*}
	In particular, when $n\geq 2$ and $0<p\le n$, if $[T,M_{b}]\in S_p(L_2(\mathbb{R}^n))$, then $[T,M_{b}]\in S_n(L_2(\mathbb{R}^n))$ as $S_p(L_2(\mathbb{R}^n))\subset S_q(L_2(\mathbb{R}^n))$ for $0<p\leq q\leq \8$. This implies that $b\in \pmb{B}_n(\mathbb{R}^n)$. Hence by Proposition \ref{0pn}, we see that $b$ is constant.	
\end{proof}

\bigskip

\section{Proof of Theorem \ref{nonschatten}}\label{noncomschattenconv}

This section aims to describe operator-valued Besov space  $\pmb{B}_p(\mathbb{R}^n, L_p(\mathcal{M}))$ in terms of the Schatten class membership of operator-valued commutators. Our goal is to prove Theorem \ref{nonschatten}. Without doubt, this semicommutative setting is much more involved. Similar to the previous section, we first show that $[T,M_b]\in L_p(B(L_2(\mathbb{R}^n))\otimes \mathcal{M})$ implies $b\in \pmb{B}_p^{\omega,2^n}(\mathbb{R}^n,\mathcal{M})$ (see Lemma \ref{nonconv1}). Our main ingredients are the duality method, Lemma \ref{conv1} and the semicommutative version of Lemma \ref{RSNWO}. Then Theorem \ref{nonschatten} is derived from Proposition \ref{pdayun}. We proceed with the proof of Theorem \ref{nonRSNWO}, the semicommutative version of Lemma \ref{RSNWO}.
\begin{thm}\label{nonRSNWO}
	Let $1<p<\infty$. Assume that $\{e_I\}_{I\in\mathcal{D}}$ and $\{f_I\}_{I\in\mathcal{D}}$ are function sequences in $L_2(\mathbb{R}^n)$ satisfying $\mathrm{supp}e_I, \,\mathrm{supp}f_I\subseteq I$ and
	$\|e_I\|_{\infty},\,\|f_I\|_{\infty}\le |I|^{-\frac{1}{2}}$.  For any $V\in L_p(B(L_2(\mathbb{R}^n))\otimes \mathcal{M})$, one has
	\begin{equation*}
		\begin{aligned}
			\sum_{I\in\mathcal{D}}\big\|\langle e_I,V(f_I)\rangle\big\|^p_{L_p(\mathcal{M})}\lesssim_{n,p} \|V\|^p_{L_p(B(L_2(\mathbb{R}^n))\otimes \mathcal{M})},
		\end{aligned}
	\end{equation*}
	where $\langle e_I,V(f_I)\rangle= \mathrm{Tr}\otimes Id_{L_p(\M)}  (V\cdot (f_I\otimes e_I) \otimes 1_\M)$ is a partial trace.
\end{thm}
\begin{proof}
	For any sequence $\{\lambda_I\}_{I\in\mathcal{D}}\subset L_{p'}(\mathcal{M})$, let
	$$A=\sum_{I\in \mathcal{D}}\lambda_I\cdot f_I\otimes e_I. $$
	From Lemma \ref{nonNWOpre} one has
	\begin{equation*}
		\begin{aligned}
			\tau\bigg(\sum_{I\in \mathcal{D}}\lambda_I\langle e_I,V(f_I)\rangle\bigg){}&=(\mathrm{Tr}\otimes \tau)(AV)\\
			&\le \|A\|_{L_{p'}(B(L_2(\mathbb{R}^n))\otimes \mathcal{M})}\|V\|_{L_p(B(L_2(\mathbb{R}^n))\otimes \mathcal{M})}\\
			&\lesssim_{n, p} \biggl(\sum_{I\in\mathcal{D}}\|\lambda_I\|^{p'}_{L_{p'}(\mathcal{M})}\biggr)^{1/p'}\|V\|_{L_p(B(L_2(\mathbb{R}^n))\otimes \mathcal{M})}.
		\end{aligned}
	\end{equation*}
	Hence by duality we obtain
	\begin{equation*}
		\sum_{I\in\mathcal{D}}\big\|\langle e_I,V(f_I)\rangle\big\|^p_{L_p(\mathcal{M})}\lesssim_{n,p} \|V\|^p_{L_p(B(L_2(\mathbb{R}^n))\otimes \mathcal{M})},
	\end{equation*}
	as desired.
\end{proof}

We proceed with the following main lemma, which reveals the relationship between $\|b\|_{\pmb{B}_p^{\omega,2^n}(\mathbb{R}^n,\mathcal{M})}$ and $\|[T,M_b]\|_{L_p(B(L_2(\mathbb{R}^n))\otimes \mathcal{M})}$.

\begin{lemma}\label{nonconv1}
	Let $1<p<\infty$ and $T\in B(L_2(\mathbb{R}^n))$ be a singular integral operator with a non-degenerate kernel $K(x,y)$ satisfying standard kernel estimates \eqref{standard}. Suppose that $b$ is a locally integrable $L_p(\mathcal{M})$-valued function. If $[T,M_b]\in L_p(B(L_2(\mathbb{R}^n))\otimes \mathcal{M})$, then for any $\omega\in(\{0,1\}^n)^\mathbb{Z}$,
	\begin{equation*}
		\begin{aligned}
			\|b\|_{\pmb{B}_p^{\omega,2^n}(\mathbb{R}^n,\mathcal{M})}\lesssim_{n,p,T}\|[T,M_b]\|_{L_p(B(L_2(\mathbb{R}^n))\otimes \mathcal{M})}.
		\end{aligned}
	\end{equation*}	
\end{lemma}	

\begin{proof}
Without loss of generality, we assume that $\omega=0$. Note that
	\begin{equation*}
		\begin{aligned}
			\|b\|_{\pmb{B}_p^{0,2^n}(\mathbb{R}^n,\mathcal{M})}{}&=\bigg(\sum_{I\in\mathcal{D}^0}\sum_{i=1}^{2^n-1}|I|^{-\frac{p}{2}}\|\langle h_I^i,b\rangle\|^p_{L_p(\mathcal{M})}\bigg)^{\frac{1}{p}}.
		\end{aligned}
	\end{equation*}
    In the following, we dualize $\|b\|_{\pmb{B}_p^{0,2^n}(\mathbb{R}^n,\mathcal{M})}$ with
    \begin{equation}\label{lambdaIi1}
        \sum_{I\in\mathcal{D}^0}\sum_{i=1}^{2^n-1}\|\lambda_{I,i}\|^{p'}_{L_{p'}(\mathcal{M})}\le 1
    \end{equation}
     to consider
    \begin{equation*}
    	\begin{aligned}
    		\sum_{I\in\mathcal{D}^0}\sum_{i=1}^{2^n-1} |I|^{-\frac{1}{2}}\tau\big(\lambda_{I,i}\cdot\langle h_I^i,b\rangle \big).
    	\end{aligned}
    \end{equation*}
	Now we fix $I\in \mathcal{D}^0$ and $1\le i\le 2^n-1$. 
	Define 
	\begin{equation*}
		\begin{aligned}
			G_{\lambda_{I,i},b}(x)=\tau\big(\lambda_{I,i}\cdot b(x)\big),\quad \forall x\in \mathbb{R}^n.
		\end{aligned}
	\end{equation*}
	Then $G_{\lambda_{I,i},b}$ is a locally integrable complex-valued function. Similarly to the proof of Lemma \ref{conv1}, where $b$ is replaced by $G_{\lambda_{I,i},b}$, we prove that
	\begin{equation*}
		\begin{aligned}
			{}&|I|^{-\frac{1}{2}}\big|\tau\big(\lambda_{I,i}\cdot\langle h_I^i,b\rangle \big)\big|
			= |I|^{-\frac{1}{2}}|\langle h_I^i,G_{\lambda_{I,i},b} \rangle|\\
			&\lesssim_{n,p,T}  \sum_{s=1}^4 \frac{A^n}{|I|}\sum_{q=1}^{2^n}\bigg|\int_{I(q)\cap E_s^I}\int_{F_s^I} (G_{\lambda_{I,i},b}(x)-G_{\lambda_{I,i},b}(\hat{x}))K(\hat{x},x)d\hat{x}dx\bigg|\\
			&=\sum_{s=1}^4 \frac{A^n}{|I|}\sum_{q=1}^{2^n}\bigg|\tau\biggl(\lambda_{I,i}\cdot \int_{I(q)\cap E_s^I}\int_{F_s^I} (b(x)-b(\hat{x}))K(\hat{x},x)d\hat{x}dx\biggr)\bigg|\\
			&\leq \|\lambda_{I,i}\|_{L_{p'}(\mathcal{M})}\cdot\sum_{s=1}^4 \frac{A^n}{|I|}\sum_{q=1}^{2^n}\bigg\| \int_{I(q)\cap E_s^I}\int_{F_s^I} (b(x)-b(\hat{x}))K(\hat{x},x)d\hat{x}dx\bigg\|_{L_p(\mathcal{M})}\\
			&=\|\lambda_{I,i}\|_{L_{p'}(\mathcal{M})}\cdot A^{n}\sum_{s=1}^4\sum_{q=1}^{2^n}\bigg\|\Big\langle \frac{|I(q)|^{\frac{1}{2}}\mathbbm{1}_{F_s^I}}{|I|}, [T,M_b]\frac{\mathbbm{1}_{I(q)\cap E_s^I}}{|I(q)|^{\frac{1}{2}}}\Big\rangle\bigg\|_{L_p(\mathcal{M})},
		\end{aligned}
	\end{equation*}
	where $E_s^I$ and $F_s^I$ depend on $G_{\lambda_{I,i},b}$, and $A$ is a sufficiently large number. Thus by duality and \eqref{lambdaIi1},
	\begin{equation*}
		\begin{aligned}
			{}&\bigg|\sum_{I\in\mathcal{D}^0}\sum_{i=1}^{2^n-1} |I|^{-\frac{1}{2}}\tau\big(\lambda_{I,i}\cdot\langle h_I^i,b\rangle \big)\bigg|\\
			&\lesssim_{n, p,T} \sum_{I\in\mathcal{D}^0}\sum_{i=1}^{2^n-1}  \|\lambda_{I,i}\|_{L_{p'}(\mathcal{M})}\cdot A^{n}\sum_{s=1}^4\sum_{q=1}^{2^n}\bigg\|\Big\langle \frac{|I(q)|^{\frac{1}{2}}\mathbbm{1}_{F_s^I}}{|I|}, [T,M_b]\frac{\mathbbm{1}_{I(q)\cap E_s^I}}{|I(q)|^{\frac{1}{2}}}\Big\rangle\bigg\|_{L_p(\mathcal{M})}\\
			&\le
			\bigg(\sum_{I\in\mathcal{D}^0}\sum_{i=1}^{2^n-1}\|\lambda_{I,i}\|^{p'}_{L_{p'}(\mathcal{M})}\bigg)^{\frac{1}{p'}}
			 \cdot\\
			& \quad\quad\quad\quad\bigg(\sum_{I\in\mathcal{D}^0}\sum_{i=1}^{2^n-1}\bigg(A^{n}\sum_{s=1}^4\sum_{q=1}^{2^n}\bigg\|\Big\langle \frac{|I(q)|^{\frac{1}{2}}\mathbbm{1}_{F_s^I}}{|I|}, [T,M_b]\frac{\mathbbm{1}_{I(q)\cap E_s^I}}{|I(q)|^{\frac{1}{2}}}\Big\rangle\bigg\|_{L_p(\mathcal{M})}\bigg)^p\bigg)^{\frac{1}{p}}\\
			&\lesssim_{n,p} A^n\bigg(\sum_{I\in\mathcal{D}^0}\sum_{s=1}^4\sum_{q=1}^{2^n}\bigg\|\Big\langle \frac{|I(q)|^{\frac{1}{2}}\mathbbm{1}_{F_s^I}}{|I|}, [T,M_b]\frac{\mathbbm{1}_{I(q)\cap E_s^I}}{|I(q)|^{\frac{1}{2}}}\Big\rangle\bigg\|_{L_p(\mathcal{M})}^p\bigg)^{\frac{1}{p}}
		\end{aligned}
	\end{equation*}
	For each $I\in\mathcal{D}^0_k$, note that $\mathrm{dist}(I,\hat{I})=2^{-k}C_nA$, where $C_n$ is a constant only depending on $n$. Let $c(\hat{I})$ be the center of $\hat{I}$. We consider the cube $Q(I)$, such that the center of $Q(I)$ is $\frac{c(I)+c(\hat{I})}{2}$, and the length of $Q(I)$ is $2^{-k+1}C_nA$. This implies that $I,\hat{I}\in Q(I)$. Besides, from Lemma \ref{Domegan1} we know that there exists some cube $J(I)\in\bigcup\limits_{i=1}^{n+1} \mathcal{D}^{\omega(i)}$ such that
		\begin{equation*}
			Q(I)\subseteq J(I)\subseteq c_n Q(I).
		\end{equation*}
		Now for any $s\in\{1,2,3,4\}$ and $1\le q\le 2^n$, let
		\begin{equation*}
			e_{J(I),s,q}=\frac{|I(q)|^{\frac{1}{2}}\mathbbm{1}_{F_s^I}}{|I|} \quad \text{and} \quad f_{J(I),s,q}=\frac{\mathbbm{1}_{I(q)\cap E_s^I}}{|I(q)|^{\frac{1}{2}}}.
		\end{equation*}
		Then  $\mathrm{supp}e_{J(I),s,q},\,\mathrm{supp}f_{J(I),s,q}\subseteq J(I)$ and $\|e_{J(I),s,q}\|_{\infty},\,\|f_{J(I),s,q}\|_{\infty}\le C|J(I)|^{-\frac{1}{2}}$, where 
		the constant $C$ only depends on $n$ and $A$. Notice that each $J(I)$ contains only a finite number of dyadic cubes in $\mathcal{D}_k^0$ and does not depend on $G_{\lambda_{I,i},b}$. Hence from Theorem \ref{nonRSNWO} we get
	\begin{equation*}
		\begin{aligned}
			\bigg|\sum_{I\in\mathcal{D}^0}\sum_{i=1}^{2^n-1} |I|^{-\frac{1}{2}}\tau\big(\lambda_{I,i}\cdot\langle h_I^i,b\rangle \big)\bigg|
			\lesssim_{n,p,T} \|[T,M_b]\|_{L_p(B(L_2(\mathbb{R}^n))\otimes \mathcal{M})}.
		\end{aligned}
	\end{equation*}
    Therefore, by duality,
	\begin{equation*}
		\begin{aligned}
			\|b\|_{\pmb{B}_p^{0,2^n}(\mathbb{R}^n,\mathcal{M})}\lesssim_{n,p,T}\|[T,M_b]\|_{L_p(B(L_2(\mathbb{R}^n))\otimes \mathcal{M})},
		\end{aligned}
	\end{equation*}
	as desired.
\end{proof}


Now we give the proof of Theorem \ref{nonschatten}.
\begin{proof}[Proof of Theorem \ref{nonschatten}]
	For any $\omega\in(\{0,1\}^n)^\mathbb{Z}$, from Lemma \ref{nonconv1} one has
	\begin{equation*}
		\|b\|_{\pmb{B}_p^{\omega,2^n}(\mathbb{R}^n,\mathcal{M})}\lesssim_{n,p,T}\|[T,M_b]\|_{L_p(B(L_2(\mathbb{R}^n))\otimes \mathcal{M})}.
	\end{equation*}
	Hence from Proposition \ref{pdayun} we obtain that $b\in \pmb{B}_p(\mathbb{R}^n, L_p(\mathcal{M}))$ and
	\begin{equation*}
		\|b\|_{\pmb{B}_p(\mathbb{R}^n, L_p(\mathcal{M}))}\lesssim_{n,p,T} \|[T,M_{b}]\|_{L_p(B(L_2(\mathbb{R}^n))\otimes \mathcal{M})}.
	\end{equation*}
	In particular, when $n\geq 2$ and $1<p\leq n$, if $[T,M_b]\in L_p(B(L_2(\mathbb{R}^n))\otimes \mathcal{M})$, then $b\in \pmb{B}_p(\mathbb{R}^n, L_p(\mathcal{M}))$, and $b$ is constant by Proposition \ref{0pn}.
\end{proof}

\bigskip

\section{Proof of Theorem \ref{nonbound}}\label{noncomboundconv}

This section aims to get a lower bound of boundedness of operator-valued commutators in terms of operator-valued $BMO$ spaces which are defined in Section \ref{pre2}. To achieve this, we will establish a weak-factorization type decomposition in the semicommutative setting following a similar argument as in \cite[Lemma 2.3.1]{TH3}.

For any cube $Q\subset \mathbb{R}^n$, define
\begin{equation*}
	\begin{aligned}
		L_{1,Q}(\mathcal{M},L_2^c(\mathbb{R}^n))&=\bigg\{f\in L_1(\mathcal{M},L_2^c(\mathbb{R}^n)): \text{supp}\,f\subseteq Q\bigg\},\\
		L^0_{1,Q}(\mathcal{M},L_2^c(\mathbb{R}^n))&=\bigg\{f\in L_{1,Q}(\mathcal{M},L_2^c(\mathbb{R}^n)): \int_Q f(x)dx=0 \bigg\}.
	\end{aligned}
\end{equation*}
The following proposition is the weak-factorization type decomposition in the noncommutative setting. Note also that from the proof of Lemma \ref{ball}, the balls $B$ and $\tilde{B}$ can be replaced by cubes $Q$ and $\tilde{Q}$.
\begin{proposition}\label{nondecom1}
	Suppose that $Q$ and $\tilde{Q}$ are two cubes which satisfy the condition in Lemma \ref{ball}. If $f\in 	L^0_{1,Q}(\mathcal{M},L_2^c(\mathbb{R}^n))$, then there is a decomposition
		\begin{equation*}
			f=gT(h)-h(T^*(g))^*+\tilde{f},
		\end{equation*}
	where $g=\mathbbm{1}_{\tilde{Q}}$, $h\in L_{1,Q}(\mathcal{M},L_2^c(\mathbb{R}^n))$ and $\tilde{f}\in L^0_{1,\tilde{Q}}(\mathcal{M},L_2^c(\mathbb{R}^n))$ satisfy
	\begin{equation*}
		\|h\|_{L_{1}(\mathcal{M},L_2^c(\mathbb{R}^n))}\lesssim_{n,T} A^n\|f\|_{L_{1}(\mathcal{M},L_2^c(\mathbb{R}^n))},\quad \|\tilde{f}\|_{L_{1}(\mathcal{M},L_2^c(\mathbb{R}^n))}\lesssim_{n,T} \varepsilon_A\|f\|_{L_{1}(\mathcal{M},L_2^c(\mathbb{R}^n))}.
	\end{equation*}
    In addition, $A$ is sufficiently large, and $\varepsilon_A$ is a positive number which depends on $A$ and is much less than $1$.
\end{proposition}

\begin{proof}
	First we prove that 
	\begin{equation}\label{Tstargx}
		|T^*(g)(x)|\gtrsim_{n,T} \frac{1}{A^n},\quad \forall x\in Q.
	\end{equation}
 For any $x\in Q$,
	\begin{equation*}
		\begin{aligned}
			T^*(g)(x){}&=\int_{\tilde{Q}}(K(y,x))^*dy\\
			&=|{Q}|(K(y_0,x_0))^*+\int_{\tilde{Q}}(K(y,x)-K(y_0,x_0))^*dy,
		\end{aligned}	
	\end{equation*}
where $x_0$ and $y_0$ are the centers of $Q$ and $\tilde{Q}$ respectively.
	On the one hand, from Lemma \ref{ball} we know that
	\begin{equation*}
		|{Q}||K(y_0,x_0)|\approx_{n,T} \frac{|{Q}|}{A^n|{Q}|}=\frac{1}{A^n}.
	\end{equation*}
	On the other hand, using Lemma \ref{ball} again, one has
	\begin{equation*}
		\begin{aligned}
			\bigg|\int_{\tilde{Q}}K(y,x)-K(y_0,x_0)dy\bigg|{}&\le \int_{\tilde{Q}}|K(y,x)-K(y_0,x_0)|dy\lesssim_{n,T} \frac{\varepsilon_A}{A^n},
		\end{aligned}
	\end{equation*}
	where $\varepsilon_A$ is a positive number which depends on $A$ and is much less than $1$. Thus 
	\begin{equation*}
		|T^*(g)(x)|\ge |{Q}||K(y_0,x_0)|-\bigg|\int_{\tilde{Q}}K(y,x)-K(y_0,x_0)dy\bigg|\gtrsim_{n,T} \frac{1}{A^n}.
	\end{equation*}
	
	Now we define
	\begin{equation*}
		h=-\frac{f}{(T^*(g))^*}\quad \text{and}\quad \tilde{f}=-gT(h).
	\end{equation*}
	Then 
	\begin{equation*}
		f=gT(h)-h(T^*(g))^*+\tilde{f}.
	\end{equation*}
    Since $f\in L^0_{1,Q}(\mathcal{M},L_2^c(\mathbb{R}^n))$, we have
	\begin{equation*}
		\|h\|_{L_1(\mathcal{M},L_2^c(\mathbb{R}^n))}\lesssim_{n,T} A^n \|f\|_{L_1(\mathcal{M},L_2^c(\mathbb{R}^n))}
	\end{equation*}
	and $h\in L_{1,Q}(\mathcal{M},L_2^c(\mathbb{R}^n))$.
	
	Finally, we consider $\tilde{f}$. Observe that 
	\begin{equation*}
		\begin{aligned}
			\|\tilde{f}\|_{L_1(\mathcal{M},L_2^c(\mathbb{R}^n))}{}&= 	\|\mathbbm{1}_{\tilde{Q}}\cdot T(h)\|_{L_1(\mathcal{M},L_2^c(\mathbb{R}^n))}.
		\end{aligned}
	\end{equation*}
   We write $T(h)$ as
    \begin{equation*}
    	\begin{aligned}
    		T(h)=T\biggl(-\frac{f}{(T^*(g))^*}\biggr){}&=-T\biggl(\frac{f}{{K(y_0,x_0)}|Q|}\biggr)+T\biggl(\frac{f}{{K(y_0,x_0)}|Q|}-\frac{f}{(T^*(g))^*}\biggr)\\
    		&=:\mathrm{I}+\mathrm{II}.
    	\end{aligned}
    \end{equation*}
    On the one hand, for any $y\in\tilde{Q}$, note that $\int_Q f(x)dx=0$, then by the Cauchy-Schwarz inequality and Lemma \ref{ball},
    \begin{equation*}
    	\begin{aligned}
    		|T(f)(y)|^2{}&=\bigg|\int_{Q} \big(K(y,x)-K(y_0,x_0)\big)f(x)dx\bigg|^2\\
    		&\le \int_{Q} |K(y,x)-K(y_0,x_0)|^2dx \cdot \int_{Q} |f(x)|^2dx\\
    		&\lesssim_{n,T} \int_{Q} \frac{\varepsilon_A^2}{A^{2n}|Q|^2}dx \cdot \int_{Q} |f(x)|^2dx\\
    		&=\frac{\varepsilon_A^2}{A^{2n}|Q|} \cdot \int_{Q} |f(x)|^2dx.
    	\end{aligned}
    \end{equation*}
    Thus
    \begin{equation}\label{Ibou}
    	\begin{aligned}
    		\|\mathbbm{1}_{\tilde{Q}}\cdot\mathrm{I}\|_{L_1(\mathcal{M},L_2^c(\mathbb{R}^n))}{}&=\bigg\|\bigg(\int_{\tilde{Q}} |\mathrm{I}(y)|^2dy\bigg)^{1/2}\bigg\|_{L_1(\mathcal{M})}\\
    		&=\frac{1}{|K(y_0,x_0)||Q|} \bigg\|\bigg(\int_{\tilde{Q}} |T(f)(y)|^2dy\bigg)^{1/2}\bigg\|_{L_1(\mathcal{M})}\\
    		&\lesssim_{n,T} A^n \bigg\|\bigg(\int_{\tilde{Q}} |T(f)(y)|^2dy\bigg)^{1/2}\bigg\|_{L_1(\mathcal{M})}
    		\lesssim_{n,T} \varepsilon_A \|f\|_{L_1(\mathcal{M},L_2^c(\mathbb{R}^n))}.
    	\end{aligned}
    \end{equation}
    On the other hand, for any $x\in Q$, from Lemma \ref{ball} and \eqref{Tstargx} one has
    \begin{equation*}
   	    \begin{aligned}
   		    \bigg|\biggl(\frac{1}{{K(y_0,x_0)}|Q|}-\frac{1}{(T^*(g))^*}\biggr)(x)\bigg|{}
   		    &=\frac{1}{|{K(y_0,x_0)}||Q||{T^*(g)(x)}|}\cdot\Big|(T^*(g)(x))^*-{K(y_0,x_0)}|Q|\Big|\\
   		    &\lesssim_{n,T} A^{2n}\int_{\tilde{Q}} |K(z,x)-K(y_0,x_0)|dz
   		    \lesssim_{n,T} A^n\varepsilon_A.
   	    \end{aligned}
   \end{equation*}
   From the Cauchy-Schwarz inequality and Lemma \ref{ball}, we deduce that for any $y\in \tilde{Q}$, 
    \begin{equation*}
    	\begin{aligned}
    		|\mathrm{II}(y)|^2{}
    		&=\bigg|\int_{Q}K(y,x)\biggl(\frac{1}{{K(y_0,x_0)}|Q|}-\frac{1}{(T^*(g))^*}\biggr)(x)f(x)dx\bigg|^2\\
    		&\le \int_{Q}\bigg|K(y,x)\biggl(\frac{1}{{K(y_0,x_0)}|Q|}-\frac{1}{(T^*(g))^*}\biggr)(x)\bigg|^2 dx\cdot \int_{Q}|f(x)|^2dx\\
    		&\lesssim_{n,T}\int_{Q}\frac{1}{A^{2n}|Q|^2}A^{2n}\varepsilon_A^2 dx\cdot \int_{Q}|f(x)|^2dx\\
    		&=\frac{\varepsilon_A^2}{|Q|}\cdot \int_{Q}|f(x)|^2dx.
    	\end{aligned}
    \end{equation*}
    Thus
    \begin{equation}\label{IIbou}
    	\begin{aligned}
    	    \|\mathbbm{1}_{\tilde{Q}}\cdot\mathrm{II}\|_{L_1(\mathcal{M},L_2^c(\mathbb{R}^n))}=\bigg\|\bigg(\int_{\tilde{Q}} |\mathrm{II}(y)|^2dy\bigg)^{1/2}\bigg\|_{L_1(\mathcal{M})}{}
    	    &\lesssim_{n,T} \varepsilon_A \|f\|_{L_1(\mathcal{M},L_2^c(\mathbb{R}^n))}.
    	\end{aligned}
    \end{equation}
    Hence from \eqref{Ibou} and \eqref{IIbou} we obtain
    \begin{equation*}
    	\|\tilde{f}\|_{L_1(\mathcal{M},L_2^c(\mathbb{R}^n))}\lesssim_{n,T} \varepsilon_A \|f\|_{L_1(\mathcal{M},L_2^c(\mathbb{R}^n))}.
    \end{equation*}
	Besides, note that 
	\begin{equation*}
		\int_Q \tilde{f}(x)dx=-\langle g,T(h)\rangle=-\langle T^*(g),h\rangle=\int_Q f(x)dx=0.
	\end{equation*}
	This implies that $\tilde{f}\in L^0_{1,\tilde{Q}}(\mathcal{M},L_2^c(\mathbb{R}^n))$.
\end{proof}

To prove Theorem \ref{nonbound}, we also need the following proposition.

\begin{proposition}\label{nondecom2}
	Let $T\in B(L_2(\mathbb{R}^n))$ be a singular integral operator with a non-degenerate kernel $K(x,y)$ satisfying standard kernel estimates \eqref{standard}. Suppose that $Q$ and $\tilde{Q}$ are two cubes which satisfy the condition in Lemma \ref{ball}. If $f\in 	L^0_{1,Q}(\mathcal{M},L_2^c(\mathbb{R}^n))$, then there is a decomposition
	\begin{equation*}
		f=g_1T(h_1)-h_1(T^*(g_1))^*+ g_2^*T(h_2)-h_2(T^*(g_2))^*+\tilde{\tilde{f}},
	\end{equation*}
	where $g_1=\mathbbm{1}_{\tilde{Q}}$, $h_2=\mathbbm{1}_{{Q}}$, $g_2\in L_{1,\tilde{Q}}(\mathcal{M},L_2^c(\mathbb{R}^n))$, $h_1\in L_{1,Q}(\mathcal{M},L_2^c(\mathbb{R}^n))$ and $\tilde{\tilde{f}}\in L^0_{1,Q}(\mathcal{M},L_2^c(\mathbb{R}^n))$ satisfy
	\begin{equation}\label{g2h1f}
		\|g_2\|_{L_{1}(\mathcal{M},L_2^c(\mathbb{R}^n))}\lesssim_{n,T} A^n\|f\|_{L_{1}(\mathcal{M},L_2^c(\mathbb{R}^n))},\quad \|h_1\|_{L_{1}(\mathcal{M},L_2^c(\mathbb{R}^n))}\lesssim_{n,T} A^n\|f\|_{L_{1}(\mathcal{M},L_2^c(\mathbb{R}^n))},
	\end{equation}
    and
    \begin{equation}\label{titif}
       \|\tilde{\tilde{f}}\|_{L_{1}(\mathcal{M},L_2^c(\mathbb{R}^n))}\lesssim_{n,T} \varepsilon_A\|f\|_{L_{1}(\mathcal{M},L_2^c(\mathbb{R}^n))},
    \end{equation}
	where $A$ is sufficiently large, and $\varepsilon_A$ is a positive number which depends on $A$ and is much less than $1$.
\end{proposition}

\begin{proof}
	Firstly, due to Proposition \ref{nondecom1}, there exists a decomposition
		\begin{equation*}
		f=g_1T(h_1)-h_1(T^*(g_1))^*+\tilde{f},
	\end{equation*}
	where $g_1=\mathbbm{1}_{\tilde{Q}}$, $h_1\in L_{1,Q}(\mathcal{M},L_2^c(\mathbb{R}^n))$ and $\tilde{f}\in L^0_{1,\tilde{Q}}(\mathcal{M},L_2^c(\mathbb{R}^n))$ satisfy
	\begin{equation*}
		\|h_1\|_{L_{1}(\mathcal{M},L_2^c(\mathbb{R}^n))}\lesssim_{n,T} A^n\|f\|_{L_{1}(\mathcal{M},L_2^c(\mathbb{R}^n))},\quad \|\tilde{f}\|_{L_{1}(\mathcal{M},L_2^c(\mathbb{R}^n))}\lesssim_{n,T} \varepsilon_A\|f\|_{L_{1}(\mathcal{M},L_2^c(\mathbb{R}^n))},
	\end{equation*}
	where $A$ is sufficiently large, and $\varepsilon_A$ is a positive number which depends on $A$ and is much less than $1$.
	
    Secondly, we use Proposition \ref{nondecom1} again, where $T$ is replaced by $T^*$, $f$ is replaced by $\tilde{f}$, and $Q$ and $\tilde{Q}$ are replaced by $\tilde{Q}$ and $Q$ respectively, then
    \begin{equation*}
    	\tilde{f}=\tilde{g}(T^*(\tilde{h}))^*-\tilde{h}^*{T(\tilde{g})}+\tilde{\tilde{f}},
    \end{equation*}
    where $\tilde{g}=\mathbbm{1}_{{Q}}$, $\tilde{h}\in L_{1,\tilde{Q}}(\mathcal{M},L_2^c(\mathbb{R}^n))$ and $\tilde{\tilde{f}}\in L^0_{1,{Q}}(\mathcal{M},L_2^c(\mathbb{R}^n))$ satisfy
    \begin{equation*}
    	\|\tilde{h}\|_{L_{1}(\mathcal{M},L_2^c(\mathbb{R}^n))}\lesssim_{n,T} A^n\|\tilde{f}\|_{L_{1}(\mathcal{M},L_2^c(\mathbb{R}^n))}\lesssim_{n,T} A^n\|f\|_{L_{1}(\mathcal{M},L_2^c(\mathbb{R}^n))}
    \end{equation*}
    and
    \begin{equation*}
	     \|\tilde{\tilde{f}}\|_{L_{1}(\mathcal{M},L_2^c(\mathbb{R}^n))}\lesssim_{n,T} \varepsilon_A\|\tilde{f}\|_{L_{1}(\mathcal{M},L_2^c(\mathbb{R}^n))}\lesssim_{n,T} \varepsilon_A\|f\|_{L_{1}(\mathcal{M},L_2^c(\mathbb{R}^n))}.
    \end{equation*}
    Finally, let $g_2=-\tilde{h}$ and $h_2=\tilde{g}$, then the proof is complete.
\end{proof}

Now we begin to prove Theorem \ref{nonbound}.
\begin{proof}[Proof of Theorem \ref{nonbound}]
	Recall that
	\begin{equation*}
		\|b\|_{BMO_{cr}(\mathbb{R}^n,\mathcal{M})}=\max\big\{\|b\|_{BMO_c(\mathbb{R}^n,\mathcal{M})},\|b\|_{BMO_r(\mathbb{R}^n,\mathcal{M})}\big\}.
	\end{equation*}
	By the duality in Theorem \ref{HBMOdual},
	\begin{equation}\label{bbmomh1}
		\begin{aligned}
			\|b\|_{BMO_r(\mathbb{R}^n,\mathcal{M})}{}&=\|b^*\|_{BMO_c(\mathbb{R}^n,\mathcal{M})}\approx\sup_{\|m\|_{H_{1,c}(\mathbb{R}^n,\mathcal{M})}=1}|\langle b^*,m\rangle|.
		\end{aligned}
	\end{equation}
	According to the atomic decomposition of $H_{1,c}(\mathbb{R}^n,\mathcal{M})$ in Theorem \ref{atomH1},
 we only need to estimate $|\langle b^*,f\rangle|$ when $f$ is an $\mathcal{M}^c$-atom.
	
	Note that $f\in L^0_{1,Q}(\mathcal{M},L_2^c(\mathbb{R}^n))$ for some cube $Q$, then from Proposition \ref{nondecom2}, we decompose $f$ into
	\begin{equation*}
		f=g_1T(h_1)-h_1(T^*(g_1))^*+ g_2^*T(h_2)-h_2(T^*(g_2))^*+\tilde{\tilde{f}},
	\end{equation*}
	where $g_1,g_2,h_1,h_2$ and $\tilde{\tilde{f}}$ satisfy the condition in Proposition \ref{nondecom2}, and $\tilde{Q}$ is another disjoint cube with length $\ell(Q)$ at distance $\mathrm{dist}(Q,\tilde{Q})\approx A\ell(Q)$. In the following $A$ is assumed to be a sufficiently large number. Then we calculate
	\begin{equation}\label{bf000}
		\begin{aligned}
			\langle b^*,f\rangle{}&=\big\langle b^*,g_1T(h_1)-h_1(T^*(g_1))^*\big\rangle+\big\langle b^*, g_2^*T(h_2)-h_2(T^*(g_2))^*\big\rangle+\langle b^*,\tilde{\tilde{f}}\rangle\\
			&=\big\langle g_1,bT(h_1)-T(bh_1)\big\rangle+\big\langle g_2,bT(h_2)-T(bh_2)\big\rangle+\langle b^*,\tilde{\tilde{f}}\rangle\\
			&=-\big\langle g_1,[T,M_b](h_1)\big\rangle-\big\langle g_2,[T,M_b](h_2)\big\rangle+\langle b^*,\tilde{\tilde{f}}\rangle.
		\end{aligned}
	\end{equation} 
	By the Cauchy-Schwarz inequality,
	\begin{equation*}
		\begin{aligned}
			\bigg|\int_{\mathbb{R}^n}  g_1(x)\cdot [T,M_b](h_1)(x)dx\bigg|{}&\le \bigg(\int_{\mathbb{R}^n}|g_1(x)|^2dx\bigg)^{1/2}\cdot \bigg(\int_{\mathbb{R}^n}\big|[T,M_b](h_1)(x)\big|^2dx\bigg)^{1/2}.
		\end{aligned}
	\end{equation*}
	This implies that
	\begin{equation}\label{g1tbh1666}
		\big|\big\langle g_1,[T,M_b](h_1)\big\rangle\big|\le \|g_1\|_{L_2(\mathbb{R}^n)}\|[T,M_b](h_1)\|_{L_1(\mathcal{M},L_2^c(\mathbb{R}^n))}.
	\end{equation}
	By duality,
	\begin{equation*}
		\|[T,M_b](h_1)\|_{L_1(\mathcal{M},L_2^c(\mathbb{R}^n))}\le\sup_{\|s\|_{L_\infty(\mathcal{M},L_2^c(\mathbb{R}^n))}= 1} \big\|[T,M_b]^*(s)\big\|_{L_\infty(\mathcal{M},L_2^c(\mathbb{R}^n))}\|h_1\|_{L_1(\mathcal{M},L_2^c(\mathbb{R}^n))}.
	\end{equation*}
    We regard $\int_{\mathbb{R}^n} \big|[T,M_b]^*(s)(x)\big|^2dx$ as a left multiplication operator on $L_2(\mathcal{M})$, then
    \begin{equation*}
    	\begin{aligned}
    		\|[T,M_b]^*(s)\|_{L_\infty(\mathcal{M},L_2^c(\mathbb{R}^n))}{}&=\bigg\|\int_{\mathbb{R}^n} \big|[T,M_b]^*(s)(x)\big|^2 dx\bigg\|_{\mathcal{M}}^{1/2}\\
    		&=\sup_{\|t\|_{L_2(\mathcal{M})}= 1}\bigg(\int_{\mathbb{R}^n} \Big\langle\big|[T,M_b]^*(s)(x)\big|^2t,t\Big\rangle_{L_2(\mathcal{M})} dx\bigg)^{1/2}\\
    		&=\sup_{\|t\|_{L_2(\mathcal{M})}= 1}\bigg(\int_{\mathbb{R}^n} \big\|[T,M_b]^*(s)(x)t\big\|^2_{L_2(\mathcal{M})} dx\bigg)^{1/2}.
    	\end{aligned}
    \end{equation*}
	Note that $[T,M_b]^*(s)(x)t=[T,M_b]^*(st)(x)$, where we regard $t$ as a constant operator on $\mathbb{R}^n$, thus by the assumption that $[T,M_b]^*$ is bounded on $L_2(\mathbb{R}^n,L_2(\mathcal{M}))$, 
	\begin{equation*}
		\begin{aligned}
			{}&\|[T,M_b]^*(s)\|_{L_\infty(\mathcal{M},L_2^c(\mathbb{R}^n))}{}\\
			&= \sup_{\|t\|_{L_2(\mathcal{M})}= 1}\bigg(\int_{\mathbb{R}^n} \big\|[T,M_b]^*(st)(x)\big\|^2_{L_2(\mathcal{M})} dx\bigg)^{1/2}\\
			&\le \|[T,M_{b}]\|_{L_2(\mathbb{R}^n,L_2(\mathcal{M}))\to L_2(\mathbb{R}^n,L_2(\mathcal{M}))}\sup_{\|t\|_{L_2(\mathcal{M})}= 1}\bigg(\int_{\mathbb{R}^n} \|s(x)t\|^2_{L_2(\mathcal{M})} dx\bigg)^{1/2}\\
			&=\|[T,M_{b}]\|_{L_2(\mathbb{R}^n,L_2(\mathcal{M}))\to L_2(\mathbb{R}^n,L_2(\mathcal{M}))}\bigg\|\int_{\mathbb{R}^n} |s(x)|^2 dx\bigg\|_{\mathcal{M}} ^{1/2}.
		\end{aligned}
	\end{equation*} 
	Thus 
	\begin{equation}\label{g1tbh1777}
		\|[T,M_b](h_1)\|_{L_1(\mathcal{M},L_2^c(\mathbb{R}^n))}\le \|h_1\|_{L_1(\mathcal{M},L_2^c(\mathbb{R}^n))}\|[T,M_{b}]\|_{L_2(\mathbb{R}^n,L_2(\mathcal{M}))\to L_2(\mathbb{R}^n,L_2(\mathcal{M}))}.
	\end{equation}
	Then from \eqref{g2h1f}, \eqref{g1tbh1666} and \eqref{g1tbh1777}  we have
	\begin{equation}\label{bf111}
		\begin{aligned}
			\big|\big\langle g_1,[T,M_b](h_1)\big\rangle\big|{}&\le \|g_1\|_{L_2(\mathbb{R}^n)}\|h_1\|_{L_1(\mathcal{M},L_2^c(\mathbb{R}^n))}\|[T,M_{b}]\|_{L_2(\mathbb{R}^n,L_2(\mathcal{M}))\to L_2(\mathbb{R}^n,L_2(\mathcal{M}))}\\
			&\lesssim_{n,T} A^n|Q|^{1/2}\|f\|_{L_1(\mathcal{M},L_2^c(\mathbb{R}^n))}\|[T,M_{b}]\|_{L_2(\mathbb{R}^n,L_2(\mathcal{M}))\to L_2(\mathbb{R}^n,L_2(\mathcal{M}))}\\
			&\le A^n\|[T,M_{b}]\|_{L_2(\mathbb{R}^n,L_2(\mathcal{M}))\to L_2(\mathbb{R}^n,L_2(\mathcal{M}))},
		\end{aligned}
	\end{equation}
	where the last inequality is from the definition of $\mathcal{M}^c$-atoms. Similarly, one has
	\begin{equation}\label{bf222}
		\big|\big\langle g_2,[T,M_b](h_2)\big\rangle\big|\lesssim_{n,T} A^n\|[T,M_{b}]\|_{L_2(\mathbb{R}^n,L_2(\mathcal{M}))\to L_2(\mathbb{R}^n,L_2(\mathcal{M}))}.
	\end{equation}
     Moreover, since $\tilde{\tilde{f}}\in L^0_{1,Q}(\mathcal{M},L_2^c(\mathbb{R}^n))$, 
     by duality and \eqref{titif},
	\begin{equation}\label{bf333}
		\begin{aligned}
			|\langle b^*,\tilde{\tilde{f}}\rangle|{}&=\bigg|\Big\langle b^*-\big\langle \frac{\mathbbm{1}_Q}{m(Q)},b^*\big\rangle,\tilde{\tilde{f}}\Big\rangle\bigg|\\
			&\le \Big\|\left(b^*-\big\langle \frac{\mathbbm{1}_Q}{m(Q)},b^*\big\rangle\right)\mathbbm{1}_Q\Big\|_{L_\infty(\mathcal{M},L_2^c(\mathbb{R}^n))}\|\tilde{\tilde{f}}\|_{L_1(\mathcal{M},L_2^c(\mathbb{R}^n))}\\ &=|Q|^{1/2} \big\|MO(b^*;Q)\big\|_{\mathcal{M}}\|\tilde{\tilde{f}}\|_{L_1(\mathcal{M},L_2^c(\mathbb{R}^n))}\\
			&\lesssim_{n,T} \varepsilon_A |Q|^{1/2}\big\|MO(b^*;Q)\big\|_{\mathcal{M}}\|f\|_{L_1(\mathcal{M},L_2^c(\mathbb{R}^n))}\\
			&\le \varepsilon_A \big\|MO(b^*;Q)\big\|_{\mathcal{M}},
		\end{aligned}
	\end{equation}
     where $MO(b^*;Q)$ is defined in \eqref{MOBQ}, and $\varepsilon_A$ is a positive number which depends on $A$ and is much less than $1$.
	Thus from \eqref{bf000}, \eqref{bf111}, \eqref{bf222} and \eqref{bf333} we have
	\begin{equation*}
		\begin{aligned}
			|\langle b^*,f\rangle|{}&\lesssim_{n,T} A^n\|[T,M_{b}]\|_{L_2(\mathbb{R}^n,L_2(\mathcal{M}))\to L_2(\mathbb{R}^n,L_2(\mathcal{M}))}+\varepsilon_A \big\|MO(b^*;Q)\big\|_{\mathcal{M}}.
		\end{aligned}
	\end{equation*}
	From \eqref{bbmomh1}, we obtain
	\begin{equation*}
		\begin{aligned}
			\big\|MO(b^*;Q)\big\|_{\mathcal{M}}{}&\lesssim \sup_{\|f\|_{H_{1,c}(\mathbb{R}^n,\mathcal{M})}=1}|\langle b^*,f\rangle|\\
		    &\lesssim_{n,T} A^n\|[T,M_{b}]\|_{L_2(\mathbb{R}^n,L_2(\mathcal{M}))\to L_2(\mathbb{R}^n,L_2(\mathcal{M}))}+\varepsilon_{A}\big\|MO(b^*;Q)\big\|_{\mathcal{M}}.
		\end{aligned}
	\end{equation*}
   By letting $\varepsilon_{A}$ be sufficiently less than 1, we deduce
     \begin{equation*}
     	\begin{aligned}
     	    \|b\|_{BMO_r(\mathbb{R}^n,\mathcal{M})}=\sup_{Q\subset \mathbb{R}^n \atop Q \operatorname{cube}}\big\|MO(b^*;Q)\big\|_{\mathcal{M}}\lesssim_{n,T} \|[T,M_{b}]\|_{L_2(\mathbb{R}^n,L_2(\mathcal{M}))\to L_2(\mathbb{R}^n,L_2(\mathcal{M}))}.	
     	\end{aligned}
     	\end{equation*}
In the same way, we have
	\begin{equation*}
		\|b\|_{BMO_c(\mathbb{R}^n,\mathcal{M})}\lesssim_{n,T} \|[T,M_{b}]\|_{L_2(\mathbb{R}^n,L_2(\mathcal{M}))\to L_2(\mathbb{R}^n,L_2(\mathcal{M}))}.
	\end{equation*}
	Therefore,
	\begin{equation*}
		\|b\|_{BMO_{cr}(\mathbb{R}^n,\mathcal{M})}\lesssim_{n,T} \|[T,M_{b}]\|_{L_2(\mathbb{R}^n,L_2(\mathcal{M}))\to L_2(\mathbb{R}^n,L_2(\mathcal{M}))}.
	\end{equation*}	
	This completes the proof.
\end{proof}

\bigskip

\section{Appendix}\label{appendix}

We give a new proof of the following theorem in \cite[Theorem 1.1]{HLW} or \cite[Theorem 3.1]{CPP2012}.
\begin{thm}\label{thm1.7}
	Let $1<p<\infty$ and $T\in B(L_2(\mathbb{R}^n))$ be a singular integral operator with a kernel $K(x,y)$ satisfying the standard estimates \eqref{standard}. If $b\in BMO(\mathbb{R}^n)$, then $C_{T, b}$ is bounded on $L_p(\mathbb{R}^n)$ and
	$$ \|C_{T, b}\|_{L_p(\mathbb{R}^n)\rightarrow L_p(\mathbb{R}^n)}\lesssim_{n, p,T}\big(1+\|T(1)\|_{BMO(\mathbb{R}^n)}+\|T^*(1)\|_{BMO(\mathbb{R}^n)}\big)\|b\|_{BMO(\mathbb{R}^n)}. $$
\end{thm}
The idea of the proof is the same as that of Theorem \ref{thm6.4}. 
Via Hyt\"{o}nen's dyadic martingale technique, we will derive Theorem \ref{thm1.7}. We also establish the boundedness of commutators involving martingale paraproducts and pointwise multiplication operator (see Proposition \ref{T1est}).

From \eqref{BMOd} we see that
\begin{equation*}
	\begin{aligned}
		\|b\|_{BMO^d(\mathbb{R})}{}&
		=\sup_{I\in\mathcal{D}}\frac{1}{|I|^{1/2}}\biggl(\sum_{J\subseteq I}\sum_{i=1}^{d-1}|\langle h_J^i,b\rangle|^2 \biggr)^{1/2}.
	\end{aligned}
\end{equation*}
In \cite{CP}, Chao and Peng showed that for $1<p<\8$, $\pi_b$ is bounded from $L_p(\mathbb{R})$ to $L_p(\mathbb{R})$ if and only if $b\in BMO^d(\mathbb{R})$.

Recall that the operator $\varLambda_b$ is defined in Lemma \ref{TLambdab}. 

\begin{lem}\label{TLambdab1}
	Let $1<p<\8$. If $b\in BMO^d(\mathbb{R})$, then $\varLambda_b$ is bounded on $L_p(\mathbb{R})$.
\end{lem}
\begin{proof}
	We use the same notation as that in Corollary \ref{Thetabest}, and the proof of this lemma is also similar to that of Corollary \ref{Thetabest}. 
	It has been shown in \cite{CP} that
	\begin{equation}\label{bbbbb}
		\|(\pi_{b^*})^*\|_{L_p(\mathbb{R})\to L_p(\mathbb{R})}=	\|\pi_{b^*}\|_{L_{p'}(\mathbb{R})\to L_{p'}(\mathbb{R})}\approx_{d, p} \|b\|_{BMO^d(\mathbb{R})}.
	\end{equation}
	It remains to estimate $\|\varLambda_b-(\pi_{b^*})^*\|_{L_p(\mathbb{R})\to L_p(\mathbb{R})}$. 
	
	At first, we show the boundedness of $\varLambda_b-(\pi_{b^*})^*$ for $p=2$.
	Using the same notation as in Corollary \ref{Thetabest} and from \eqref{var} and \eqref{bbbbbbb}, we have
	\begin{equation}\label{br}
		\begin{aligned}
			\quad\|\tilde{\varLambda}_b\|_{L_2(\mathbb{R})\to L_2(\mathbb{R})}{}&\le \sup_{I\in\mathcal{D}}\|a_1^IA+a_2^IA^2+\cdots+a_{d-1}^IA^{d-1}\|_{S_\infty(\mathbb{M}_{d})}\\&\le\sup_{I\in\mathcal{D}}\sum_{i=1}^{d-1}\|a_i^IA^i\|_{S_\infty(\mathbb{M}_{d})}\le\sup_{I\in\mathcal{D}}\sum_{i=1}^{d-1}|a_i^I|\\
			&=\sup_{I\in \mathcal{D}}\frac{1}{|I|^{1/2}}\sum_{i=1}^{d-1}|\langle h_I^i,b\rangle|\le (d-1)\|b\|_{BMO^d(\mathbb{R})}.
		\end{aligned}
	\end{equation}
	Hence, we obtain that $\tilde{\varLambda}_b$ is bounded on $L_2(\mathbb{R})$.
	
	Next, we prove that $\tilde{\varLambda}_b$ satisfies weak type (1,1). Assume $f\in L_1(\mathbb{R})$ and let $\lambda>0$. In the same way as in \cite[Lemma 2.7]{Per}, we have the Calder\'{o}n-Zygmund decomposition $f=g+h$ with
	\begin{enumerate}
		\item $\|g\|_{L_\infty(\mathbb{R})}\le d\lambda$, $\|g\|_{L_1(\mathbb{R})}\le \|f\|_{L_1(\mathbb{R})}$;
		\item $h=\sum\limits_{j}h_j$, where $h_j=\bigg(f-\biggl\langle \frac{\mathbbm{1}_{I_j}}{|I_j|},f\biggr\rangle\bigg)\mathbbm{1}_{I_j}=\sum\limits_{J\subseteq I_j}\sum\limits_{l=1}^{d-1}\langle h_J^l,f\rangle h_J^l$ and $\{I_j\}$ form a sequence of disjoint $d$-adic intervals such that $\sum\limits_{j}|I_j|\le \frac{\|f\|_{L_1(\mathbb{R})}}{\lambda}$.
	\end{enumerate}
	We see that $\tilde{\varLambda}_b$ is of strong type (2,2). In particular, \eqref{br} implies that
	\begin{equation}\label{lamg}
		\begin{aligned}
			|\{|\tilde\varLambda_b(g)|>\lambda/2\}|{}&\le 4(d-1)\|b\|_{BMO^d(\mathbb{R})} \frac{\|g\|_{L_2(\mathbb{R})}^2}{\lambda^2}\\
			&\le 4(d-1)\|b\|_{BMO^d(\mathbb{R})} \frac{\|g\|_{L_\infty(\mathbb{R})}\|f\|_{L_1(\mathbb{R})}}{\lambda^2}\\
			&\le 4d(d-1)\|b\|_{BMO^d(\mathbb{R})}\frac{\|f\|_{L_1(\mathbb{R})}}{\lambda}.
		\end{aligned}
	\end{equation}
	On the other hand, from \eqref{tildevar} we deduce that $\mathrm{supp} \tilde{\varLambda}(h_j)\subseteq I_j$, and
	\begin{equation}\label{lamh}
		|\{|\tilde{\varLambda}_b(h)|>\lambda/2\}|\le |\cup_jI_j|\le \frac{\|f\|_{L_1(\mathbb{R})}}{\lambda}.
	\end{equation}
	Then from \eqref{lamg} and \eqref{lamh}, we conclude that
	\begin{equation*}
		\begin{aligned}
			|\{|\tilde{\varLambda}_b(f)|>\lambda\}|{}&\le |\{|\tilde{\varLambda}_b(g)|>\lambda/2\}|+|\{|\tilde{\varLambda}_b(h)|>\lambda/2\}| \\
			&\le (4d(d-1)\|b\|_{BMO^d(\mathbb{R})}+1)\frac{\|f\|_{L_1(\mathbb{R})}}{\lambda}.
		\end{aligned}
	\end{equation*}
	Hence $\tilde{\varLambda}_b$ is of weak type (1,1). Using interpolation and duality argument, we obtain that $\tilde{\varLambda}_b$ is bounded on $ L_p(\mathbb{R})$ for $1<p<\8$.
\end{proof}

\begin{rem}
	There is another easy proof for the boundedness of $\varLambda_b$ on $L_p(\mathbb{R})$ when $1<p\neq 2<\8$. By the duality between the $d$-adic martingale Hardy space $H^d_1(\mathbb{R})$ (see the definition in \eqref{defnh1d}) and the $d$-adic martingale $BMO$ space $BMO^d(\mathbb{R})$, we see that if $b\in BMO^d(\mathbb{R})$, then $\varLambda_b$ is bounded from $H^d_1(\mathbb{R})$ to $L_1(\mathbb{R})$. Using the boundedness of $\varLambda_b$ on $L_2(\mathbb{R})$ and by interpolation, we conclude that $\varLambda_b$ is bounded on $L_p(\mathbb{R})$ for $1<p\leq 2$. The boundedness of $\varLambda_b$ on $L_p(\mathbb{R})$ for $2\leq p<\8$ follows from the duality.
\end{rem}

We now provide the following useful lemma.
\begin{lem}\label{supk}
	Let $1<p<\infty$, $f\in L_p(\mathbb{R})$ and $b\in BMO^d(\mathbb{R})$. Then 
	\begin{equation*}
		\biggl\|\sup_{k\in\mathbb{Z}}\bigl|\mathbb{E}_{k-1}\bigl(\sum_{j\ge k} d_jb\cdot d_jf\bigr)\bigr|\biggr\|_{L_p(\mathbb{R})}\lesssim_{p} \|b\|_{BMO^d(\mathbb{R})}\|f\|_{L_p(\mathbb{R})}.
	\end{equation*}
\end{lem}
\begin{proof}
	By the Hölder inequality, one has
	\begin{equation*}
		\begin{aligned}
			\biggl|\mathbb{E}_{k-1}\biggl(\sum\limits_{j\ge k}d_j{b}\cdot d_j{f}\biggr)\biggr|{}&=|\mathbb{E}_{k-1}(b-b_{k-1})(f-f_{k-1})|\\
			&\le \big(\mathbb{E}_{k-1}|b-b_{k-1}|^q\big)^{1/q}\cdot \big(\mathbb{E}_{k-1}|f-f_{k-1}|^{q'}\big)^{1/{q'}},
		\end{aligned}
	\end{equation*}
	where $q=\frac{2p}{p-1}$ and $\frac{1}{q}+\frac{1}{q'}=1$.
	From the martingale John-Nirenberg inequality we have
	\begin{equation}\label{BMOeq}
		\|b\|_{BMO^d(\mathbb{R})}\approx_q \sup_{k\in \mathbb{Z}}\big\|\mathbb{E}_{k-1}|b-b_{k-1}|^q\big\|_{\infty}^{1/q}.
	\end{equation}
	Hence by \eqref{BMOeq},
	\begin{equation*}
		\begin{aligned}
			{}&\quad\biggl\|\sup_{k\in\mathbb{Z}}\biggl|\mathbb{E}_{k-1}\biggl(\sum_{j\ge k} d_jb\cdot d_jf\biggr)\biggr|\biggr\|_{L_p(\mathbb{R})}\\
			&
			\le \biggl\|\sup_{k\in\mathbb{Z}}\big(\mathbb{E}_{k-1}|b-b_{k-1}|^q\big)^{1/q}\cdot \sup_{k\in\mathbb{Z}}\big(\mathbb{E}_{k-1}|f-f_{k-1}|^{q'}\big)^{1/{q'}}\biggr\|_{L_p(\mathbb{R})}\\
			&\lesssim_p \|b\|_{BMO^d(\mathbb{R})}\cdot \biggl\|\sup_{k\in\mathbb{Z}}\big(\mathbb{E}_{k-1}|f-f_{k-1}|^{q'}\big)^{1/{q'}}\biggr\|_{L_p(\mathbb{R})}.
		\end{aligned}
	\end{equation*}
	Note that $|f|^{q'}\in L_{p/{q'}}(\mathbb{R})$ and $p/{q'}>1$, by the Doob maximal inequality, 
	\begin{equation*}
		\begin{aligned}
			\biggl\|\sup_{k\in\mathbb{Z}}\big(\mathbb{E}_{k-1}|f-f_{k-1}|^{q'}\big)^{1/{q'}}\biggr\|_{L_p(\mathbb{R})}{}
			&\lesssim_p \biggl\|\sup_{k\in\mathbb{Z}}\big(\mathbb{E}_{k-1}|f|^{q'}\big)^{1/{q'}}\biggr\|_{L_p(\mathbb{R}^n)}+\bigl\|\sup_{k\in\mathbb{Z}}|f_{k-1}|\bigr\|_{L_p(\mathbb{R})}\\
			&\lesssim_p \|f\|_{L_p(\mathbb{R})}.
		\end{aligned}
	\end{equation*}
	Therefore
	\begin{equation*}
		\biggl\|\sup_{k\in\mathbb{Z}}\biggl|\mathbb{E}_{k-1}\biggl(\sum_{j\ge k} d_jb\cdot d_jf\biggr)\biggr|\biggr\|_{L_p(\mathbb{R})}\lesssim_{p} \|b\|_{BMO^d(\mathbb{R})}\|f\|_{L_p(\mathbb{R})},
	\end{equation*}
	as desired.
\end{proof}

Before proving Theorem \ref{thm1.7}, we give the following two propositions concerning the boundedness of commutators involving martingale paraproducts.

\begin{proposition}\label{T1est}
	Let $1<p<\infty$. If $a,b\in BMO^d(\mathbb{R})$, then $[\pi_a,M_b]$ is bounded on $L_p(\mathbb{R})$ and
	\begin{equation*}
		\|[\pi_a,M_b]\|_{L_p(\mathbb{R})\to L_p(\mathbb{R})}\lesssim_{d,p} \|a\|_{BMO^d(\mathbb{R})} \|b\|_{BMO^d(\mathbb{R})}.
	\end{equation*}
Moreover, $[\pi_a^*,M_b]$ is bounded on $L_p(\mathbb{R})$ and
\begin{equation*}
	\|[\pi_a^*,M_b]\|_{L_p(\mathbb{R})\to L_p(\mathbb{R})}\lesssim_{d,p} \|a\|_{BMO^d(\mathbb{R})}\|b\|_{BMO^d(\mathbb{R})}.
\end{equation*}
\end{proposition}
\begin{proof}
	We use the same notation as that in the proof of Proposition \ref{T2est}. 
	Recall that $R_b$ is defined in \eqref{R_b}. We will first focus on the estimate of the norm $\|[\pi_{a}, R_b]\|_{L_p(\mathbb{R})\to L_p(\mathbb{R})}$. From \eqref{piarbf} we have
	\begin{equation*}
		[\pi_{a}, R_b]=-\pi_{a}\Theta_b+V_{a,b}.
	\end{equation*}
	By Lemma \ref{TLambdab1}, one has
	\begin{equation}\label{a1}
		\begin{aligned}
			\|\pi_{a}\Theta_{b}\|_{L_p(\mathbb{R})\to L_p(\mathbb{R})}{}&\le \|\pi_{a}\|_{L_p(\mathbb{R}^n)\to L_p(\mathbb{R})}\bigl(\|\pi_{b}\|_{L_p(\mathbb{R})\to L_p(\mathbb{R})}+\|\varLambda_{b}\|_{L_p(\mathbb{R})\to L_p(\mathbb{R})}\bigr)\\
			&\lesssim_{d,p} \|a\|_{BMO^d(\mathbb{R})}\|b\|_{BMO^d(\mathbb{R})}.
		\end{aligned}
	\end{equation}
	For any $f\in L_p(\mathbb{R})$ and $g\in L_{p'}(\mathbb{R})$, 
	\begin{equation*}
		\begin{aligned}
			\langle V_{a,b}(f),g\rangle{}&=\sum_{k\in\mathbb{Z}}\biggl\langle d_ka\cdot \mathbb{E}_{k-1}\biggl(\sum_{j\ge k}d_jb\cdot d_jf\biggr),g\biggr\rangle\\
			&=\sum_{k\in\mathbb{Z}}\biggl\langle d_ka,d_kg\cdot \mathbb{E}_{k-1}\biggl(\sum_{j\ge k}d_j\overline{b}\cdot d_j\overline{f}\biggr)\biggr\rangle\\
			&= \biggl\langle a,\sum_{k\in\mathbb{Z}}d_kg\cdot \mathbb{E}_{k-1}\biggl(\sum_{j\ge k}d_j\overline{b}\cdot d_j\overline{f}\biggr)\biggr\rangle.
		\end{aligned}
	\end{equation*}
	To use duality, we need to estimate the following
	\begin{equation}\label{SV3}
		\begin{aligned}
			\biggl\|\sum_{k\in\mathbb{Z}}d_kg\cdot \mathbb{E}_{k-1}\biggl(\sum_{j\ge k}d_j\overline{b}\cdot d_j\overline{f}\biggr)\biggr\|_{H_1^d(\mathbb{R})}{}&=\biggl\|\biggl(\sum_{k\in\mathbb{Z}} |d_kg|^2\biggl|\mathbb{E}_{k-1}\biggl(\sum\limits_{j\ge k}d_j{b}\cdot d_j{f}\biggr)\biggr|^2\biggr)^{1/2}\biggr\|_{L_1(\mathbb{R})}\\
			&\le \biggl\|\biggl(\sum_{k\in\mathbb{Z}} |d_kg|^2\biggr)^{1/2} \cdot\sup_{k\in\mathbb{Z}}\biggl|\mathbb{E}_{k-1}\biggl(\sum\limits_{j\ge k}d_j{b}\cdot d_j{f}\biggr)\biggr|\biggr\|_{L_1(\mathbb{R})}\\
			&
			\le \|S(g)\|_{L_{p'}(\mathbb{R})}\cdot\biggl\|\sup_{k\in\mathbb{Z}}\biggl|\mathbb{E}_{k-1}\biggl(\sum\limits_{j\ge k}d_j{b}\cdot d_j{f}\biggr)\biggr|\biggr\|_{L_p(\mathbb{R})}\\
			&\lesssim_{p} \|b\|_{BMO^d(\mathbb{R})} \|g\|_{L_{p'}(\mathbb{R})}\|f\|_{L_p(\mathbb{R})},
		\end{aligned}
	\end{equation}
	where the third inequality is from the Hölder inequality, and the fourth is from Lemma \ref{supk}.
	Therefore 
	\begin{equation*}
		\begin{aligned}
			|\langle V_{a,b} (f),g\rangle|{}&\lesssim \|a\|_{BMO^d(\mathbb{R})}\biggl\|\sum_{k\in\mathbb{Z}}d_kg\cdot \mathbb{E}_{k-1}\biggl(\sum_{j\ge k}d_j\overline{b}\cdot d_j\overline{f}\biggr)\biggr\|_{H_1^d(\mathbb{R})}\\
			&\lesssim_{p} \|a\|_{BMO^d(\mathbb{R})}\|b\|_{BMO^d(\mathbb{R})} \|g\|_{L_{p'}(\mathbb{R})}\|f\|_{L_p(\mathbb{R})}.
		\end{aligned}
	\end{equation*}
	This implies that
	\begin{equation}\label{a2}
		\begin{aligned}
			\|V_{a,b}\|_{L_p(\mathbb{R})\to L_p(\mathbb{R})}\lesssim_{p} \|a\|_{BMO^d(\mathbb{R})}\|b\|_{BMO^d(\mathbb{R})}.
		\end{aligned}
	\end{equation}	
	From \eqref{a1} and \eqref{a2} we have
	\begin{equation*}
		\|[\pi_{a}, R_b]\|_{L_p(\mathbb{R})\to L_p(\mathbb{R})}\lesssim_{d,p} \|a\|_{BMO^d(\mathbb{R})}\|b\|_{BMO^d(\mathbb{R})}.
	\end{equation*}	
	Recall that
	\begin{equation*}
		\begin{aligned}
			[\pi_{{a}},M_b]=[\pi_{{a}},\pi_b]+[\pi_{{a}},\varLambda_b]+[\pi_{{a}},R_b].
		\end{aligned}
	\end{equation*}
	Since $\pi_a$ is bounded on $L_p(\mathbb{R})$, by the triangle inequality we deduce that
	\begin{equation*}
		\begin{aligned}
			\|[\pi_a,M_b]\|_{L_p(\mathbb{R})\to L_p(\mathbb{R})}\lesssim_{d,p} \|a\|_{BMO^d(\mathbb{R})}\|b\|_{BMO^d(\mathbb{R})}.
		\end{aligned}
	\end{equation*}
	Recall that 
\begin{equation*}
	[\pi_a^*,M_b]^*=-[\pi_a,M_{b^*}].
\end{equation*}
Hence
\begin{equation*}
	\begin{aligned}
		\|[\pi_a^*,M_b]\|_{L_p(\mathbb{R})\to L_p(\mathbb{R})}=\|[\pi_a,M_{b^*}]\|_{L_{p'}(\mathbb{R})\to L_{p'}(\mathbb{R})}\lesssim_{d,p} \|a\|_{BMO^d(\mathbb{R})}\|b\|_{BMO^d(\mathbb{R})}.
	\end{aligned}
\end{equation*}
	This completes the proof.
\end{proof}

We can define the martingale $BMO$ space $BMO^{\omega, 2^n}(\mathbb{R}^n)$ on $\mathbb{R}^n$ by virtue of $H_I^\eta$ similarly as in \eqref{BMOd}. More precisely, $BMO^{\omega, 2^n}(\mathbb{R}^n)$ associated with the dyadic system $\mathcal{ D}^\omega$ on $\mathbb{R}^n$ is the space consisting of all locally integrable functions $b$ such that
\begin{equation}
\|b\|_{BMO^{\omega, 2^n}(\mathbb{R}^n)}=\sup_{I\in\mathcal{D}^\omega}\frac{1}{|I|^{1/2}}\biggl(\sum_{J\subseteq I}\sum_{\eta\in\{0,1\}^n_0}|\langle H_J^\eta,b\rangle|^2 \biggr)^{1/2}<\infty.
\end{equation}
Then Lemma \ref{TLambdab1} and Proposition \ref{T1est} also hold for the dyadic system $\mathcal{ D}^\omega$. It is straightforward to verify that if $b\in BMO(\mathbb{R}^n) $, then $b\in BMO^{\omega, 2^n}(\mathbb{R}^n)$ and
$$ 	\|b\|_{BMO^{\omega, 2^n}(\mathbb{R}^n)}\leq 	\|b\|_{BMO(\mathbb{R}^n)}. $$

We come to the proof of Theorem \ref{thm1.7}.
\begin{proof}[Proof of Theorem \ref{thm1.7}]
We use the same notation as that in the proof of Theorem \ref{thm6.4}. From Proposition \ref{T1est}, we have
\begin{equation*}
	\|[\pi_{T(1)}^{\omega}, M_b]\|_{L_p(\mathbb{R}^n)\to L_p(\mathbb{R}^n)}\lesssim_{n,p}\|T(1)\|_{BMO(\mathbb{R}^n)} \|b\|_{BMO(\mathbb{R}^n)}
\end{equation*}
and
\begin{equation*}
	\|[(\pi_{T^*(1)}^{\omega})^*, M_b]\|_{L_p(\mathbb{R}^n)\to L_p(\mathbb{R}^n)}\lesssim_{n,p} \|T^*(1)\|_{BMO(\mathbb{R}^n)} \|b\|_{BMO(\mathbb{R}^n)}.
\end{equation*}
By Theorem \ref{CZdec}, it suffices to estimate $\|[S_\omega^{ij}, M_b]\|_{L_p(\mathbb{R}^n)\to L_p(\mathbb{R}^n)}$ for any $i, j\in \mathbb{N}\cup\{0\}$.
Note that by the triangle inequality
\begin{equation*}
	\begin{aligned}
		{}&\quad\|[S_\omega^{ij},M_b]\|_{L_p(\mathbb{R}^n)\to L_p(\mathbb{R}^n)}\\&\le \|[S_\omega^{ij},\pi_b]\|_{L_p(\mathbb{R}^n)\to L_p(\mathbb{R}^n)}+\|[S_\omega^{ij},\varLambda_b]\|_{L_p(\mathbb{R}^n)\to L_p(\mathbb{R}^n)}
		+\|[S_\omega^{ij},R_b]\|_{L_p(\mathbb{R}^n)\to L_p(\mathbb{R}^n)}.
	\end{aligned}
\end{equation*}
Here $\pi_b$, $\varLambda_b$ and $R_b$ are with respect to the dyadic system $\mathcal{ D}^\omega$. From \cite{CP} and Lemma \ref{TLambdab1}, we know that
\begin{equation*}
	\begin{aligned}
		\|\pi_b\|_{L_p(\mathbb{R}^n)\to L_p(\mathbb{R}^n)}\lesssim_{n,p}\|b\|_{BMO^{\omega, 2^n}(\mathbb{R}^n)},\quad \|\varLambda_b\|_{L_p(\mathbb{R}^n)\to L_p(\mathbb{R}^n)}\lesssim_{n,p }\|b\|_{BMO^{\omega, 2^n}(\mathbb{R}^n)}.
	\end{aligned}
\end{equation*}
However, $S_\omega^{ij}$ is bounded on $L_p(\mathbb{R}^n)$. Therefore one has
\begin{equation*}
	\begin{aligned}
		\|[S_\omega^{ij},\pi_b]\|_{L_p(\mathbb{R}^n)\to L_p(\mathbb{R}^n)}&\lesssim \|S_\omega^{ij}\|_{L_p(\mathbb{R}^n)\to L_p(\mathbb{R}^n)}\|\pi_b\|_{L_p(\mathbb{R}^n)\to L_p(\mathbb{R}^n)}\\
		&\lesssim_{n,p} (i+j)\|b\|_{BMO^{\omega, 2^n}(\mathbb{R}^n)}\lesssim_{n,p} (i+j)\|b\|_{BMO(\mathbb{R}^n)}.
	\end{aligned}
\end{equation*}
Analogously, we have
$$ \|[S_\omega^{ij},\varLambda_b]\|_{L_p(\mathbb{R}^n)\to L_p(\mathbb{R}^n)}\lesssim_{n,p} (i+j)\|b\|_{BMO(\mathbb{R}^n)}.$$
It remains to estimate $\|[S_\omega^{ij}, R_b]\|_{L_p(\mathbb{R}^n)\to L_p(\mathbb{R}^n)}$ for any $i, j\in \mathbb{N}\cup\{0\}$. We will show that $\|[S^{ij}_\omega,R_b]\|_{L_p(\mathbb{R}^n)\to L_p(\mathbb{R}^n)}$ increases with polynomial growth with respect to $i$ and $j$ uniformly on $\omega$. Then from Theorem \ref{CZdec} and the triangle inequality, the desired result will follow.

We first prove
\begin{equation}\label{22}
	\|[S_\omega^{ij},R_b]\|_{L_2(\mathbb{R}^n)\to L_2(\mathbb{R}^n)}\lesssim_{n} \|b\|_{BMO(\mathbb{R}^n)} .
\end{equation}
Without loss of generality, we can assume $\omega=0$.
Let $\varPhi=[S^{ij}_0,R_b]$. The forms of $\varPhi$ and $B_K$ ($K\in\mathcal{D}^0$) have been given in \eqref{Ndef} and \eqref{BKdef} respectively. From \eqref{BkBK} we know $\varPhi^*\varPhi$ is a block diagonal matrix with blocks $B_K^*B_K$ for all $K\in \mathcal{ D}^0$. Hence
\begin{equation}\label{Nf2}
	\begin{aligned}
		\|\varPhi\|_{L_2(\mathbb{R}^n)\to L_2(\mathbb{R}^n)}=\|\varPhi^*\varPhi\|_{L_2(\mathbb{R}^n)\to L_2(\mathbb{R}^n)}^{1/2}=\sup_{k\in\mathbb{Z}}\sup_{K\in\mathcal{D}^0_k}\|B^*_KB_K\|_{L_2(\mathbb{R}^n)\to L_2(\mathbb{R}^n)}^{1/2}.
	\end{aligned}
\end{equation}
Now we fix $k\in\mathbb{Z}$ and $K\in\mathcal{D}^0_k$. We write $B_K^*B_K$ in the matrix form $[B_K^*B_K]$ with respect to the basis $\{H_Q^\zeta\}$, where $Q\in\mathcal{D}^0, Q\subseteq K$ , $\ell(Q)=2^{-i}\ell(K)$ and $\zeta\in\{0,1\}^n_0$. Then using the triangle inequality, one has
\begin{equation*}
	\begin{aligned}
		\|B^*_KB_K\|_{L_2(\mathbb{R}^n)\to L_2(\mathbb{R}^n)}{}&=\|[B_K^*B_K]\|_{S_\infty(\mathbb{M}_{2^{in}(2^n-1)})}\\&\le \sum_{\substack{J\in\mathcal{D}^0;J\subseteq K\\ \ell(J)=2^{-j}\ell(K)}}\sum_{\eta}\big\|W^{K,J,\eta}\big\|_{S_\infty(\mathbb{M}_{2^{in}(2^n-1)})},
	\end{aligned}
\end{equation*}
where $W^{K,J,\eta}$ is defined in \eqref{W}.
Analogously to \eqref{WVV}, we deduce
\begin{equation*}
	\begin{aligned}
		\big\|W^{K,J,\eta}\big\|_{S_\infty(\mathbb{M}_{2^{in}(2^n-1)})}{}
		&=\bigg|\sum_{\substack{{Q}\in\mathcal{D}^0;{Q}\subseteq K \\ \ell({Q})=2^{-i}\ell(K)\\ {\zeta}\in\{0,1\}^n_0}}a_{{Q}JK}^{{\zeta}\eta}\overline{a_{{Q}JK}^{{\zeta}\eta}}b_{{Q}J}\overline{b_{{Q}J}} \bigg|
		\le\sum_{\substack{{Q}\in\mathcal{D}^0;{Q}\subseteq K \\ \ell({Q})=2^{-i}\ell(K)\\ {\zeta}\in\{0,1\}^n_0}}\bigg|a_{{Q}JK}^{{\zeta}\eta}b_{{Q}J} \bigg|^2,
	\end{aligned}
\end{equation*}
where $b_{IJ}= \langle \frac{\mathbbm{1}_I}{|I|},b\rangle-\langle \frac{\mathbbm{1}_J}{|J|},b\rangle$ for any $I, J\in \mathcal{ D}^0$.
This implies that
\begin{equation*}
	\begin{aligned}
		\|\varPhi\|_{L_2(\mathbb{R}^n)\to L_2(\mathbb{R}^n)}{}&\le \sup_{k\in\mathbb{Z}}\sup_{K\in\mathcal{D}^0_k}\biggl(\sum_{\substack{J\in\mathcal{D}^0;J\subseteq K\\ \ell(J)=2^{-j}\ell(K)}}\sum_{\eta}\sum_{\substack{I\in\mathcal{D}^0;I\subseteq K \\ \ell(I)=2^{-i}\ell(K)}}\sum_{ \xi}\bigg|a_{IJK}^{\xi\eta} b_{IJ}\bigg|^2\biggr)^{1/2}.
	\end{aligned}
\end{equation*}
Note that
$|a_{IJK}^{\xi\eta}|\le 2^{-(i+j)n/2}$ in \eqref{aijkxieta}, $b_{IJ}=b_{IK}-b_{JK}$ and $b_{IK}\cdot \mathbbm{1}_I=(b_{k+i}-b_{k})\cdot \mathbbm{1}_I$. Then by the Cauchy-Schwarz inequality
\begin{equation}
	\begin{aligned}
		{}&\quad\sum_{\substack{J\in\mathcal{D}^0;J\subseteq K\\ \ell(J)=2^{-j}\ell(K)}}\sum_{\eta}\sum_{\substack{I\in\mathcal{D}^0;I\subseteq K\\ \ell(I)=2^{-i}\ell(K)}}\sum_{\xi}\biggl|a_{IJK}^{\xi\eta}b_{IJ}\biggr|^2\\
		&
		\le \frac{(2^n-1)^2}{2^{(i+j)n}}\sum_{\substack{J\in\mathcal{D}^0;J\subseteq K\\ \ell(J)=2^{-j}\ell(K)}}\sum_{\substack{I\in\mathcal{D}^0;I\subseteq K\\ \ell(I)=2^{-i}\ell(K)}}|b_{IJ}|^2\\
		&\leq \frac{(2^n-1)^2}{2^{(i+j)n}}\sum_{\substack{J\in\mathcal{D}^0;J\subseteq K\\ \ell(J)=2^{-j}\ell(K)}}\sum_{\substack{I\in\mathcal{D}^0;I\subseteq K\\ \ell(I)=2^{-i}\ell(K)}} 2(|b_{IK}|^2+|b_{JK}|^2)\\
		&= \frac{(2^{n}-1)^2}{2^{in-1}}\sum_{\substack{I\in\mathcal{D}^0;I\subseteq K\\ \ell(I)=2^{-i}\ell(K)}}|b_{IK}|^2+\frac{(2^{n}-1)^2}{2^{jn-1}}\sum_{\substack{J\in\mathcal{D}^0;J\subseteq K\\ \ell(J)=2^{-j}\ell(K)}}|b_{JK}|^2\\
		&=(2^{n}-1)^22^{kn+1}\bigg(\|(b_{k+i}-b_{k})\mathbbm{1}_K\|_{L_2(\mathbb{R}^n)}^2+\|(b_{k+j}-b_{k})\mathbbm{1}_K\|_{L_2(\mathbb{R}^n)}^2\bigg).
	\end{aligned}
\end{equation}
We also notice that
\begin{equation*}
	\begin{aligned}
		\|(b_{k+i}-b_{k})\mathbbm{1}_K\|_{L_2(\mathbb{R}^n)}^2=\int_K \mathbb{E}_k(|b_{k+i}(t)-b_k(t)|^2) dt\leq |K|\|b\|^2_{BMO^{0,2^n}(\mathbb{R}^n)}.
	\end{aligned}
\end{equation*}
Hence 
\begin{equation}\label{Nest}
	\begin{aligned}
		\|\varPhi\|_{L_2(\mathbb{R}^n)\to L_2(\mathbb{R}^n)}{}&\lesssim_{n} \|b\|_{BMO^{0, 2^n}(\mathbb{R}^n)}\leq\|b\|_{BMO(\mathbb{R}^n)}.
	\end{aligned}
\end{equation}

Now we prove that $\varPhi$ is of weak type (1,1). Assume $f\in L_1(\mathbb{R}^n)$ and let $\lambda>0$. We let $A_{IJK}^{\xi\eta}=a_{IJK}^{\xi\eta}b_{IJ}$,
then by \eqref{BKdef}
\begin{equation}
	\begin{aligned}
		\varPhi(f){}&
		=\sum_{K\in\mathcal{D}^0}\sum_{\substack{I,J\in\mathcal{D}^0;I,J\subseteq K\\ \ell(I)=2^{-i}\ell(K)\\\ell(J)=2^{-j}\ell(K)}}\sum_{\xi,\eta}A_{IJK}^{\xi\eta}\langle H^\xi_I,f\rangle H^\eta_J.
	\end{aligned}
\end{equation}
Note that if $\tilde{I}$ is the parent of $I$, then
\begin{equation*}
	\begin{aligned}
		\bigg|\biggl\langle \frac{\mathbbm{1}_I}{|I|},b\biggr\rangle-\biggl\langle\frac{\mathbbm{1}_{\tilde{I}}}{|\tilde{I}|},b\biggr\rangle\bigg|{}
		&\le \frac{1}{|I|}\int_I \biggl|b(t)-\biggl\langle\frac{\mathbbm{1}_{\tilde{I}}}{|\tilde{I}|},b\bigg\rangle\biggr|dt\\
		&\le\frac{2^n}{|\tilde{I}|}\int_{\tilde{I}} \biggl|b(t)-\biggl\langle\frac{\mathbbm{1}_{\tilde{I}}}{|\tilde{I}|},b\bigg\rangle\biggr|dt\\
		&\le 2^n\|b\|_{BMO^{0, 2^n}(\mathbb{R}^n)}.
	\end{aligned}
\end{equation*}
Together with the triangle inequality, this implies that 
\begin{equation*}
	|A_{IJK}^{\xi\eta}|=|a_{IJK}^{\xi\eta}||b_{IJ}|\le |a_{IJK}^{\xi\eta}||b_{IK}|+|a_{IJK}^{\xi\eta}||b_{JK}|\le 2^n(i+j)\|b\|_{BMO^{0, 2^n}(\mathbb{R}^n)}|a_{IJK}^{\xi\eta}|.
\end{equation*}
Thus the operator $\varPhi$ can be written as a multiple of $S_0^{ij}$. Recall that $S_0^{ij}$ is also of weak type (1,1), and hence for any $\lambda>0$
\begin{equation*}
	|\{|\varPhi(f)|>\lambda\}|\lesssim_n i(i+j) \|b\|_{BMO^{0, 2^n}(\mathbb{R}^n)}\frac{\|f\|_{L_1(\mathbb{R}^n)}}{\lambda}.
\end{equation*}

Therefore using interpolation and duality, $\varPhi=[S^{ij}_0,R_b]$ is bounded on $L_p(\mathbb{R}^n)$. Since the above estimation is independent of the choose of $\omega$, one has
$$ \|[S_\omega^{ij}, R_b]\|_{L_p(\mathbb{R}^n)\to L_p(\mathbb{R}^n)}\lesssim_{n,p} (i+j)^2\|b\|_{BMO(\mathbb{R}^n)}, $$
which yields
$$ \|[S_\omega^{ij}, M_b]\|_{L_p(\mathbb{R}^n)\to L_p(\mathbb{R}^n)}\lesssim_{n,p}  (i+j+1)^2\|b\|_{BMO(\mathbb{R}^n)}. $$
As a consequence, 
\begin{equation*}
	\begin{aligned}
		{}&\quad\|[T,M_b]\|_{L_p(\mathbb{R}^n)\to L_p(\mathbb{R}^n)}\\&=\biggl\|\biggl[C_1(T)\mathbb{E}_\omega\sum_{i,j=0\atop \max\{i,j\}>0}^\infty \tau(i, j)S_\omega^{ij}+C_2(T)\mathbb{E}_\omega S_\omega^{00}+\mathbb{E}_\omega \pi^\omega_{T(1)}+\mathbb{E}_\omega (\pi^\omega_{T^*(1)})^*,M_b\biggr]\biggr\|_{L_p(\mathbb{R}^n)\to L_p(\mathbb{R}^n)}\\
		&\lesssim_T \sum_{i,j=0}^\infty \tau(i, j)\mathbb{E}_\omega\|[S_\omega^{ij},M_b]\|_{L_p(\mathbb{R}^n)\to L_p(\mathbb{R}^n)}+\mathbb{E}_\omega\|[ \pi^\omega_{T(1)}+ (\pi^\omega_{T^*(1)})^*,M_b]\|_{L_p(\mathbb{R}^n)\to L_p(\mathbb{R}^n)}\\
		&\lesssim_{n,p, T} \big(1+\|T(1)\|_{BMO(\mathbb{R}^n)}+\|T^*(1)\|_{BMO(\mathbb{R}^n)}\big)\|b\|_{BMO(\mathbb{R}^n)}.
	\end{aligned}
\end{equation*}
This completes the proof.
\end{proof}	

\bigskip {\textbf{Acknowledgments.}} We thank Professor Quanhua Xu for proposing this subject to us and for many helpful discussions and suggestions. We also thank Professor Tuomas Hyt\"{o}nen and Professor \'{E}ric Ricard for valuable comments and suggestions.

\bibliographystyle{myrefstyle}
\bibliography{ref999}
\end{document}